%% file: main.tex
\documentclass{article}
\usepackage{authblk}

\usepackage[english]{babel}
\usepackage[letterpaper,top=2cm,bottom=2cm,left=3cm,right=3cm,marginparwidth=1.75cm]{geometry}
\usepackage{amsmath}
\usepackage{amssymb}
\usepackage{graphicx}
\usepackage{subcaption}
\usepackage[colorlinks=true, allcolors=blue]{hyperref}
\usepackage{blindtext,alltt}
\usepackage{algorithm}
\usepackage{algpseudocode}
\usepackage{float}
\usepackage{tikz}
\usepackage{pgfplots}
\pgfplotsset{compat=1.17}
\newtheorem{theorem}{Theorem}
\newtheorem{corollary}{Corollary}
\newtheorem{proof}{Proof}

\begin{document}

\title{An Age-Structured Vaccination Strategy for Epidemic Containment: A Model Predictive Control Approach}

\author[1]{Candy Sonveaux}
\author[2]{Morgane Dumont\footnote{corresponding author : morgane.dumont@uliege.be}}
\author[3]{Mirko Fiacchini}
\author[3]{Mohamad Ajami}
\affil[1]{University of Namur, Department of Mathematics and naXys, Belgium}
\affil[2]{QuantOM, HEC - Management School of the University of Liège, Belgium\\
               }
\affil[3]{ Univ. Grenoble Alpes, CNRS, Grenoble INP, GIPSA-lab, Grenoble, France}

\maketitle

\begin{abstract}
\input{sections/00_Abstract}
\end{abstract}

\input{sections/01_Introduction}
\input{sections/02_Model_Overview}
\input{sections/03_Model_Predictive_Control}
\input{sections/04_Results}
\input{sections/05_Conclusion}

\section*{Acknowledgments}

\begin{itemize}
    \item Computational resources have been provided by the Consortium des Équipements de Calcul Intensif (CÉCI), funded by the Fonds de la Recherche Scientifique de Belgique (F.R.S.-FNRS) under Grant No. 2.5020.11 and by the Walloon Region.
    \item The "AVIQ, l'Agence Wallonne pour la Santé, la Protection sociale, le Handicap et les Familles" has given access to valuable data for the calibration of the parameters.
\end{itemize}

\section*{Data availability}
The sources of data and the parameters used are explicitly mentioned in the paper, allowing the reader to reproduce the simulations. The Python code can be made available upon request.

\section*{Funding} 
No specific funding is linked to this research.

\section*{Conflicts of interest}
The authors declare that they have no conflict of interest.

\section*{Ethics declaration}

The article process does not affect any humans, birds, or animals.

\bibliographystyle{apalike}
\bibliography{bibliography}

\end{document}

%% file: sections/00_Abstract.tex
This work presents a novel Model Predictive Control (MPC) approach to develop an optimal age-structured vaccination strategy for the containment of COVID-19 in Wallonia, Belgium. The proposed MPC framework is designed to minimize deaths, achieve early disease eradication, and adhere to operational constraints. By incorporating an age-structured Susceptible-Infected-Recovered-Deceased (SIRD) model with an additional term for vaccination, the MPC strategy dynamically adapts to the evolving epidemic state. A detailed proof of the asymptotic stability and recursive feasibility of the proposed MPC algorithm is provided. This ensures that the optimal cost at each step provides an upper bound on the minimal number obtainable of deaths at the end of the pandemic. Moreover, simulations demonstrate that the proposed MPC approach outperforms the decreasing age vaccination strategy adopted by the Belgian government during the first wave of vaccinations. The results highlight the potential of MPC-based vaccination strategies to reduce the total number of deaths, accelerate disease eradication, and optimize vaccine administration.\\
\textbf{Keywords}: Model predictive control \and Epidemic control \and Asymptotic stability \and Vaccination strategy optimization

%% file: sections/01_Introduction.tex
\section{Introduction}
\label{sec: intro}
Over the last decades, the number of major epidemics and pandemics has increased significantly. An epidemic can be defined as the widespread outbreak of an infectious disease in a community at a particular time. A pandemic is an epidemic that spreads across different regions. Throughout history, epidemics and pandemics have shown their ability to claim countless lives and exhaust healthcare systems. One can cite the Ebola virus disease that occurred in West Africa between 2014 and 2016 and more recently the COVID-19 pandemic. Until the $5^{th}$ of January, 2025, more than 7 million lives have been lost to COVID-19 according to \cite{who_covid_dashboard}, with half of these deaths occurring during the first year of the pandemic. This highlights the need for reliable and cost-effective strategies that would enable nations, regardless of their economic power, to combat spread and save as many lives as possible.

According to the objectives of the World Health Organization, \cite{world2014infection}, immediate action is required once an epidemic emerges, since the disease is already circulating by the time it is detected. Hence, scientists and decision makers must unite their efforts to develop and enforce measures to control spread without crippling the economy. In the early stages of an epidemic, scientists focus on rapidly testing and developing vaccines, while decision makers enforce regulations to limit transmission. Once vaccines become available, decision makers deploy vaccination strategies to effectively contain the outbreak.

Decision makers often use mathematical models to support their decisions. Compartmental models in epidemiology divide the population under study into disjoint compartments, each representing a certain disease status, referred to as states \cite{brauer2008compartmental}. These models describe how the population dynamically evolves between states over time. In \cite{tolles2020modeling}, the authors present a comprehensive overview of compartmental models in epidemiology. The first known compartmental model is the SIR model \cite{kermack1991contributions}, which consists of three compartments: Susceptible individuals (S) who can contract the disease, Infected individuals (I) who can transmit the disease, and Recovered individuals (R) who are assumed to be immune against the disease. Depending on the objectives of the study and the complexity of the disease, additional compartments may be added. However, increasing the complexity of the model can lead to problems of identifiability and forecasting capability \cite{FIACCHINI2021500}.

In this paper, an age-dependent vaccination strategy based on Model Predictive Control (MPC) is proposed. The strategy is applied to COVID-19 data for Wallonia, Belgium. The MPC employs an age-structured SIRD model that incorporates an additional compartment describing the total number of Deceased individuals (D). An additional input term for vaccination is included in the model to study the impact of vaccine administration on population dynamics. Although the traditional SIRD model does not consider age, the age-dependent nature of COVID-19 requires models that take into account the demographic factor. In fact, parameters such as death and recovery rates have been shown to be age-dependent \cite{davies2020age}. In addition, similar models have been adopted in the literature to develop strategies against pandemics. In~\cite{grundel2022much}, a social distancing strategy is developed to achieve the containment of COVID-19. In Belgium, several studies were aimed at developing variants of the SIRD model. In \cite{franco2021covid}, the authors present an age-structured SEIR-QD model (Susceptible-Exposed-Infectious-Recovered-Quarantined-Deceased) that includes nursing home dynamics to forecast the progression of COVID-19 and assess the effects of long-term public health strategies. In \cite{alleman2021assessing}, a social contact model integrates public mobility data to evaluate the impact of non-pharmaceutical interventions on virus transmission in different settings. A similar model is used in \cite{abrams2021modelling} to analyze the effects of early lockdown on hospitalizations and project future epidemic scenarios based on various exit strategies. To our knowledge, no study has proposed a dynamic framework to guide the COVID-19 vaccination strategy using an age-stratified SIRD model, accompanied by a stability and recursive feasibility proof.

Model predictive control for strategic epidemic planning has been extensively examined in the literature. In \cite{MORATO2020417}, an MPC is developed to adapt the social distancing policies for COVID-19. Moreover, in \cite{grundel2022much}, the authors used an SIRD model within an MPC framework to optimize mass testing and develop an age-specific social distancing strategy. Similarly, the authors in \cite{sauerteig2022model} introduced an MPC approach using an age-independent SEIR model to minimize social distancing and quarantine measures during a pandemic while maintaining a strict infection cap. On the other hand, the authors in \cite{parino2023model} developed an age-dependent MPC two-dose scheduling strategy for COVID-19 vaccine administration to balance health requirements and socioeconomic costs. Furthermore, a vaccination strategy is proposed in \cite{cartocci2021compartment}, with given objectives such as minimizing total infections, total deaths, and total quality adjusted life years (QALY). However, exhaustive full-grid simulations were performed in that study without providing any guarantee on the obtained solution. A proof of recursive feasibility and asymptotic convergence is crucial even for inherently stable systems, as the case of pandemic dynamics, since it also provides guaranteed bounds on the infinite-horizon optimal solution.

To date, there has been no consensus on the definition of an optimal vaccination strategy. For example, Belgium adopted a decreasing-age vaccination strategy, starting with the elderly. The argument was that COVID-19 posed a higher mortality risk in older age groups compared to younger ones. In contrast, \cite{rahmandad2022behavioral} argues that it is more effective to base the vaccination strategy on the level of contact, suggesting that the level of exposure is more critical than age to eradicate the disease as quickly as possible. This dilemma raises the need for a dynamic framework to guide the vaccination strategy in an optimal manner. This work proposes a novel age-structured MPC approach to design a vaccination strategy aimed at minimizing fatalities. A proof of asymptotic stability and recursive feasibility of the proposed algorithm is provided, ensuring that the final death toll minimization is upper bounded by the MPC cost.

This paper is organized as follows: Section \ref{sec: model} presents the epidemiological model used in this study, detailing the age-structured SIRD model and its parameters. Section \ref{sec: mpc} introduces the Model Predictive Control (MPC) framework, including the proof of recursive feasibility and asymptotic stability. Section \ref{sec: results} discusses the implementation of the MPC-based vaccination strategy for COVID-19 in Wallonia and compares its performance with the Belgium national vaccination strategy through simulations. Finally, Section \ref{sec: conclusion} concludes the paper by summarizing the key findings and suggesting directions for future research.

%% file: sections/02_Model_Overview.tex
\section{Epidemiological Model}
\label{sec: model}
This work focuses on an age-structured Susceptible-Infected-Recovered-Deceased (SIRD) model with an additional input term for vaccination. This model is used to analyze the impact of the administration of the COVID-19 vaccine on the dynamics of the Walloon population. Compartmental models, such as the SIRD model, are widely utilized in the literature to study the effects of various epidemic containment measures on population dynamics. This section provides a comprehensive overview of the model structure and parameters.
The continuous-time SIRD model is presented in (\ref{eq:cont_model}) as follows:
\begin{equation}
    \left\{
    \begin{aligned}
        \frac{dS_k(t)}{dt} &= -\lambda_k S_k(t) \sum_{j=1}^{n_a} C_{kj} I_j(t) - u_k(t)\\
        \frac{dI_k(t)}{dt} &= \lambda_k S_k(t) \sum_{j=1}^{n_a} C_{kj} I_j(t) - (\gamma_{R_k} + \gamma_{D_k}) I_k(t) \\
        \frac{dR_k(t)}{dt} &= \gamma_{R_k} I_k(t) + u_k(t)\\
        \frac{dD_k(t)}{dt} &= \gamma_{D_k} I_k(t)
    \end{aligned}
    \right.
    \label{eq:cont_model}
\end{equation}
where $S_k(t)$, $I_k(t)$, $R_k(t)$, and $D_k(t)$ denote the number of susceptible,  infected, recovered, and deceased individuals, respectively, in the $k^{th}$ age group, with $k \in \{1, 2, \dots, n_a\}$, at time $t$; $n_a$ is the total number of considered age groups. The term $\lambda_k$ represents the probability of disease transmission when an individual in the $k^{th}$ age group has a contact. The term $C_{ij}$ represents the probability that an individual in the $i^{th}$ age group encounters someone in the $j^{th}$ age group. The parameters $\gamma_{R_k}$ and $\gamma_{D_k}$ refer to the recovery and death rates, respectively, for the $k^{th}$ age group. The input $u_k(t)$ is defined as the vaccination rate among susceptible individuals in the $k^{th}$ age group at time $t$. \autoref{fig:sird_cont} provides a graphical representation of the continuous-time SIRD model in (\ref{eq:cont_model}) for the $k^{th}$ age group. The total population considered in the model is assumed to be constant, as derived from the model equations, where the sum of the derivatives of the compartments, representing the derivative of the total population, is zero.
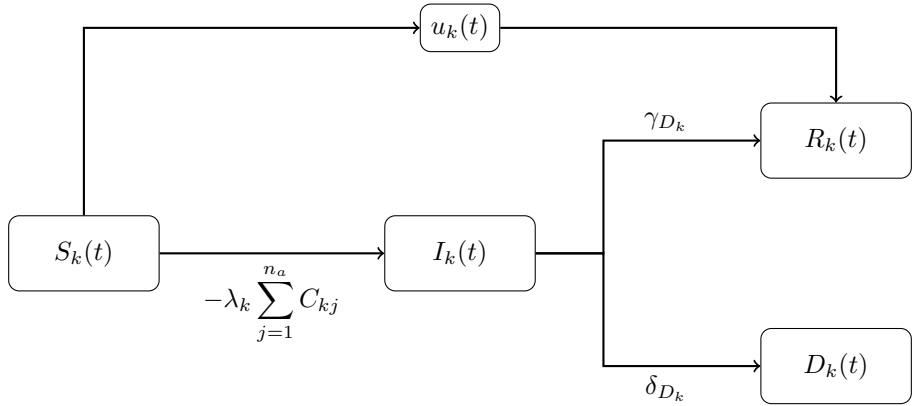
\begin{figure}[h!]
\centering
\begin{tikzpicture}[
    box/.style={rectangle,  rounded corners, draw, minimum width=2cm, minimum height=1cm, align=center},
    smallbox/.style={rectangle, rounded corners, draw, inner sep=4pt, align=center},
    arrow/.style={->, thick},
    ]
   \node[box] (S) at (0, 0) {$S_k(t)$};
   \node[box] (I) at (5, 0) {$I_k(t)$};
   \node[box] (R) at (10, 1.5) {$R_k(t)$};
   \node[box] (D) at (10, -1.5) {$D_k(t)$};
   \node[smallbox] (T) at (5, 3) {$u_k(t)$};
   \draw[arrow] (S) -- node[below] {$-\lambda_k\displaystyle\sum_{j=1}^{n_a}C_{kj}$} (I);
 \draw[arrow] (I.east) -- ++(0.9, 0) |- node[above,  pos=0.7 ] {$\gamma_{D_k}$} (R.west);
   \draw[arrow] (I.east) -- ++(0.9, 0) |- node[below, pos=0.7] {$\delta_{D_k}$} (D.west);
   \draw[arrow] (S.north) -- ++(0, 0)|- (T.west);
    \draw[arrow] (T.east)  -| (R.north);
\end{tikzpicture}
    \caption{Compartmental scheme of the age-structured SIRD model with an additional vaccination term for the $k^{th}$ age group}
    \label{fig:sird_cont}
\end{figure}

In the context of Model Predictive Control, the continuous-time model in (\ref{eq:cont_model}) is discretized using the Forward Euler Method to account for the discrete-time nature of the optimization problem. To align with the frequency of data updates and the meeting of decision makers, the model is discretized at a sampling time $ts=1\text{ (in \textit{days})}$.
The discrete-time age-structured SIRD model is represented in (\ref{eq:disc_model}) as follows:
\begin{equation}
    \left\{
    \begin{aligned}
        S_k(n+1) &= S_k(n) - \lambda_k S_k(n) \sum_{j=1}^{n_a} C_{kj} I_j(n) - u_k(n)\\
        I_k(n+1) &= I_k(n) + \lambda_k S_k(n) \sum_{j=1}^{n_a} C_{kj} I_j(n) - (\gamma_{R_k} + \gamma_{D_k}) I_k(n) \\
        R_k(n+1) &= R_k(n) + \gamma_{R_k} I_k(n) + u_k(n)\\
        D_k(n+1) &= D_k(n) + \gamma_{D_k} I_k(n)
    \end{aligned}
    \right.
    \label{eq:disc_model}
\end{equation}
Here, $n \in \mathbb{N}$ represents the discrete-time step, where each $n$ corresponds to the $n^{th}$ day since the beginning of the COVID-19 outbreak.

%% file: sections/03_Model_Predictive_Control.tex
\section{Model Predictive Control}
\label{sec: mpc}
The primary goal of this work is to prove that for complex dynamics and objectives, such as finding the best vaccination strategy for the COVID-19 pandemic, optimization-based methods might outperform alternative approaches, based on the intuition, which could fail to capture the inherent complexity of the problem. To fulfill it, an optimal age-dependent vaccination strategy, based on model predictive control, is proposed to minimize the final deaths toll and eradicate disease while adhering to operational constraints. MPC operates by minimizing a predefined cost over a finite time horizon, called the prediction horizon $N$, at every sampled instant; see the schematic illustration in \autoref{fig_MPC_illu}. The plant model is used to forecast the future trajectories of the system and select the optimal one. From the calculated solution, only the first control action is used and the process is repeated for the next time step. This iterative nature of the MPC enables it to predict, anticipate, and adapt to changing conditions.

\begin{figure}[h]
	\begin{center}
		\begin{tikzpicture}[
			thick, scale=0.6,
			>=stealth,
			dot/.style = {
				draw,
				fill = white,
				circle,
				inner sep = 0pt,
				minimum size = 4pt
			}, domain=0:25
			]
			\coordinate (O) at (0,0);
			\draw[-,gray!70] (0,0) -- (0,10); 
			\draw[-,gray!70] (-6,0) -- (10,0);
			\draw[-,black] (-6,1) -- (-3,1); 
			\draw[-,black] (-3,1) -- (-3,-1); 
			\draw[-,black] (-3,-1) -- (-2,-1); 
			\draw[-,black] (-2,-1) -- (-2,2); 
			\draw[-,black] (-2,2) -- (-1,2); 
			\draw[-,black] (-1,2) -- (-1,0.5);
			\draw[-,black] (-1,0.5) -- (0,0.5);  
			\draw[black] (-2.5,1.8) node {$u$};
			\draw[-,dashed,black,thick] (0,1.5) -- (1,1.5); 
			\draw[-,dashed,black,thick] (1,1.5) -- (1,2.5); 	
			\draw[-,dashed,black,thick] (1,2.5) -- (2,2.5); 
			\draw[-,dashed,black,thick] (2,2.5) -- (2,1); 
			\draw[-,dashed,black,thick] (2,1) -- (3,1); 
			\draw[-,dashed,black,thick] (3,1) -- (3,1.5);
			\draw[-,dashed,black,thick] (3,1.5) -- (4,1.5);  
			\draw[-,dashed,black,thick] (4,1.5) -- (4,3);
			\draw[-,dashed,black,thick] (4,3) -- (5,3);
			\draw[-,dashed,black,thick] (5,3) -- (5,2.75);
			\draw[-,dashed,black,thick] (5,2.75) -- (6,2.75);
			\draw[-,dashed,black,thick] (6,2.75) -- (6,2.5);  
			\draw[-,dashed,black,thick] (6,2.5) -- (10,2.5);  
			\draw[black] (9,2.8) node {$\textbf{u}^{\textbf{o}}$};
			\draw[<->,gray!70] (0,3.5) -- (6,3.5);  
			\draw[gray!70] (3,3.9) node {Control horizon, $N$}; 
			\draw[<->,gray!70] (0,0.2) -- (10,0.2);  
			\draw[gray!70] (5,0.5) node {Prediction horizon, $P$};
			
			\filldraw[black] (-5,4.2) circle (2pt);
			\filldraw[black] (-4,4.2) circle (2pt);
			\filldraw[black] (-3,4) circle (2pt);
			\filldraw[black] (-2,4.5) circle (2pt);
			\filldraw[black] (-1,5) circle (2pt);
			\filldraw[black] (0,6) circle (2pt);
			\draw[black] (-2.5,4.75) node {$D$};
			
			\draw[black] (1,6) circle (2pt);
			\draw[black] (2,6) circle (2pt);
			\draw[black] (3,6.3) circle (2pt);
			\draw[black] (4,6.8) circle (2pt);
			\draw[black] (5,7.5) circle (2pt);
			\draw[black] (6,7.7) circle (2pt);
			\draw[black] (7,7.8) circle (2pt);
			\draw[black] (8,7.9) circle (2pt);
			\draw[black] (9,8) circle (2pt);
			\draw[black] (10,7.9) circle (2pt);
			\draw[black] (1.5,6.5) node {$\hat{D}$};
			\draw[-,dashed,gray!70, thick] (0,8) -- (10,8);  
			\draw[gray!70] (5,8.4) node {Target};
			\draw[->,gray!70, thick] (0,9) -- (-2,9);  
			\draw[gray!70] (-1,9.4) node {Past};
			\draw[->,gray!70, thick] (0,9) -- (2,9);  
			\draw[gray!70] (1,9.4) node {Future};
			\filldraw[black] (6,6.5) circle (2pt);
			\filldraw[black] (6.4,6.5) circle (2pt);
			\filldraw[black] (6.8,6.5) circle (2pt);
			\draw[black] (9,6.5) node {Past outputs $D$};
			
			\draw[black] (6,6) circle (2pt);
			\draw[black] (6.4,6) circle (2pt);
			\draw[black] (6.8,6) circle (2pt);
			\draw[black] (9.6,6) node {Predicted outputs $\hat{D}$};
			
			\draw[-,black] (6,5.5) -- (6.8,5.5);
			\draw[black] (9.7,5.5) node {Past control action $u$};
			
			\draw[-,dashed,black,thick] (6,5) -- (6.8,5); 
			\draw[black] (9.9,5) node {Optimal control action}; \draw[black] (9.6,4.5) node {at time $k$, $\textbf{u}^{\textbf{o}}$};
			\draw[-,gray!70] (-1,0) -- (-1,-0.25);
			\draw[black] (-1,-0.5) node {$n-1$};
			\draw[-,gray!70] (0,0) -- (0,-0.25);
			\draw[black] (0,-0.5) node {$n$};
			\draw[-,gray!70] (1,0) -- (1,-0.25);
			\draw[black] (1,-0.5) node {$n+1$};
			\draw[-,gray!70] (6,0) -- (6,-0.25);
			\draw[black] (6,-0.5) node {$n+N-1$};
			\draw[-,gray!70] (10,0) -- (10,-0.25);
			\draw[black] (10,-0.5) node {$n+P$};
			
			\draw[gray!70] (5,-0.9) node {Sampling instants};
		\end{tikzpicture}		
	\end{center}
	\caption{Illustration of the MPC scheme, inspired by \cite{seborg2016process}}
	\label{fig_MPC_illu}
\end{figure}
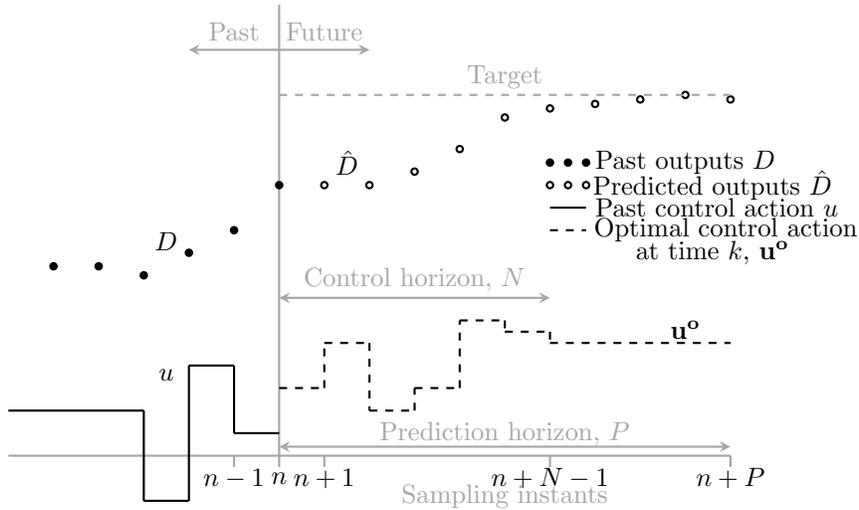

Moreover, this work provides proofs of the recursive feasibility and asymptotic stability of the problem. In general, the conditions for them to hold are already well established in the literature. However, ensuring their satisfaction for nonlinear systems is not straightforward, in general. It is important to recall that, even for inherently asymptotically stable systems, a proof of recursive feasibility and asymptotic stability is fundamental since it provides bounds on the optimal infinite-horizon solution attainable. In particular, the resulting optimal cost provides an upper bound on the minimal final number of deaths obtainable at the end of the vaccination procedure.

\subsection{Recursive feasibility and stability conditions for general MPC}

Before proceeding to design the MPC-based vaccination strategy, some general theoretical results are recalled. The proof of recursive feasibility and asymptotic stability is achieved by introducing notations from the classical MPC theory in \cite{rawlings2017model}. Consider the general dynamical system given by:
\begin{equation}
    x(n+1) = f(x(n), u(n)), \quad \text{with } \quad x(0) = x_0
\label{eq:diff_eq}
\end{equation}
where $x \in \mathbb{X}$ is the current state, $u \in \mathbb{U}$ is the current control action, and $f$ is a non-linear function that describes the evolution of the next state. The sets $\mathbb{X}$ and $\mathbb{U}$ represent the state and input constraints. In what follows, the order operators (i.e. $<, \leq$, etc.) are considered to be elementwise when applied to vectors.

The proofs of recursive feasibility, asymptotic convergence of the closed-loop solution, and the upper bound on the future death toll are based on standard results from the classical MPC theory. The general form of the MPC is structured as follows:
\begin{equation}
\begin{aligned}
    \mathbb{P}_N(x): & \quad \min_{\mathbf{u}} V_N(x, \mathbf{u}) \\
    \text{s.t.} & \quad x(n+1) = f(x(n), u(n)), \quad \forall n \in \{0, 1, \dots, N-1\}, \\
                & \quad x(0) = x, \\
                & \quad (x(n), u(n)) \in \mathbb{Z}, \quad \forall n \in \{0, 1, \dots, N-1\}, \\
                & \quad x(N) \in \mathbb{X}_f \subseteq \mathbb{X}.
\end{aligned}
\label{eq:opt_prob}
\end{equation}
with 
\begin{equation}
    V_N(x, \mathbf{u})=\sum_{n=0}^{N-1}l(x(n), u(n)) + V_f(x(N))
\label{eq:general_cost_fn}
\end{equation}
where $N \in \mathbb{N}$ is the prediction horizon, $\mathbb{Z}$ is a subset of $\mathbb{X} \times \mathbb{U}$ representing the constraints on future trajectories, and $\mathbb{X}_f$ and $V_f(\cdot)$ are additional parameters designed to ensure stability and recursive feasibility properties. The performance cost to be minimized at each step, $V_N(x, {\bf u})$, represents the total cost over the prediction horizon, where $l(x(n), u(n))$ is a function that captures the stage cost  at time $n$ and $V_f(x(N))$ is the terminal cost.

By defining $\mathcal{U}_N(x) \in \mathbb{U}^N$ as the set of feasible control sequences that satisfy the input, state and terminal constraints of problem (\ref{eq:opt_prob}) at $x$, the optimal control problem can be rewritten as:
\begin{equation}
    \mathbb{P}_N(x): \min_{\mathbf{u} \in \mathcal{U}_N(x)} V_N(x, \mathbf{u}) := V_N^0 (x).
\label{eq:opt_prob_simp}
\end{equation}
Moreover, the feasibility region of the problem, i.e. the subset of $\mathbb{X}$ where (\ref{eq:opt_prob}) admits a solution, is denoted as $\mathcal{X}_N$, that is $\mathcal{X}_N = \{ x \in \mathbb{X}: \mathcal{U}_N(x) \neq \emptyset\}$. The optimal control sequence $\mathbf{u}^0(x)$ can hence be obtained by solving the optimal control problem in (\ref{eq:opt_prob_simp}), which yields the following:
\begin{equation*}
    \mathbf{u}^0(x)=(u^0(0,x), \dots, u^0(N-1, x)) = \text{arg}\min_{\mathbf{u}\in \mathcal{U}_N(x)}V_N(x, \textbf{u})
\end{equation*}
where only the first element of the input sequence $\kappa_N(x)=u^0(0, x)$ is used to control the system. The classical proof that guarantees closed-loop recursive feasibility and asymptotic stability of the system using MPC, see Theorem $2.19$ in \cite{rawlings2017model}, is recalled hereafter.

\begin{theorem}[\cite{rawlings2017model}]
Given the dynamical system (\ref{eq:diff_eq}), the optimization problem (\ref{eq:opt_prob}) and the equilibrium set $X^* \subseteq \mathbb{X}$, suppose that:
\begin{enumerate}
    \item the functions $f: \mathbb{Z} \rightarrow \mathbb{X}$, $l: \mathbb{Z} \rightarrow \mathbb{R}_+$ and $V_{f}: \mathbb{X} \rightarrow \mathbb{R}_+$ are continuous;
    \item $f(x^*,0)=0$, $l(x^*,0)=0$ and $V_f(x^*,0)=0$ $\forall x^* \in X^*$;
    \item the set $\mathbb{Z}$ is closed, $\mathbb{X}_f \subseteq \mathbb{X}$ is compact, $X^* \subseteq \mathbb{X}_f$, and $0 \in \mathbb{U}$, where $\mathbb{U}$ is compact;
    \item for every $x \in \mathbb{X}_f$, there exists $u \in \mathbb{U}$,  such that $(x, u) \in \mathbb{Z}$ and:
    \begin{enumerate}
        \item $f(x, u) \in \mathbb{X}_f$,
        \item $V_f(f(x, u)) - V_f(x) \leq -l(x, u)$;
    \end{enumerate}
    \item there exist two $\mathcal{K}_\infty$ functions, $\alpha_1(\cdot)$ and $\alpha_f(\cdot)$, satisfying:
    \begin{enumerate}
        \item $l(x, u) \geq \alpha_1(\|x\|_{X^*}), \quad \forall x \in \mathcal{X}_N$, $\forall u$ such that $(x, u) \in \mathbb{Z}$,
        \item $V_f(x) \leq \alpha_f(\|x\|_{X^*}), \quad \forall x \in \mathbb{X}_f$, with $\|x\|_{X^*} := \min\limits_{x^* \in X^*} \|x-x^*\|$;
    \end{enumerate}
    \item there exists a $\mathcal{K}_\infty$ function $\alpha(\cdot)$ such that $V_N^0(x) \leq \alpha(\|x\|_{X^*})$ $\forall x \in \mathcal{X}_N$.
\end{enumerate}
Then, the optimal control problem in (\ref{eq:opt_prob}) is feasible $\forall x \in \mathcal{X}_N$, and there exist two $\mathcal{K}_\infty$ functions: $\alpha_d(\cdot)$ and $\alpha_u(\cdot)$, such that $\forall x \in \mathcal{X}_N$,
\begin{align*}
    &\alpha_d(\|x\|_{X^*}) \leq V_N^0(x) \leq \alpha_u(\|x\|_{X^*}), \\
    &V_N^0(f(x, \kappa_N(x)) - V_N^0(x) \leq -\alpha_d(\|x\|_{X^*}),
\end{align*}
and hence $V_N^0(\cdot)$ is a Lyapunov function defined in $\mathcal{X}_N$.
Therefore, the set $X^*$ is asymptotically stable in $\mathcal{X}_N$ for $x(n+1)=f(x(n), \kappa_N(x(n)))$.
\label{thm: asymptotic_stab}
\end{theorem}

The recursive feasibility and asymptotic convergence of the proposed MPC vaccination approach can be ensured by satisfying the conditions of Theorem~\ref{thm: asymptotic_stab}. Moreover, an important property, which proof can be found in \cite{rawlings2017model}, states that upon satisfaction of the former conditions, the optimal cost sets an upper bound on the infinite-horizon cost along the resulting trajectory.

\begin{corollary}\label{cor:VN0}
Given the dynamical system (\ref{eq:diff_eq}), the optimization problem (\ref{eq:opt_prob}) and the equilibrium set $X^* \subseteq \mathbb{X}$, if conditions 1.-6. of Theorem~\ref{thm: asymptotic_stab} are satisfied, then
\begin{equation}
V_N^0(x_0) \leq \sum_{n = 0}^{\infty} l(x(n), \kappa_N(x(n))),
\end{equation}
holds $\forall x_0 \in \mathcal{X}_N$ and $x(n+1)=f(x(n), \kappa_N(x(n)))$.
\end{corollary}

This result implies, in the COVID-19 application case, that the resulting optimal cost is an upper bound on the minimum number of deaths that can be achieved at the end of the pandemic.

\subsection{MPC-based vaccination for COVID-19 pandemic}
The aim of the age-structured MPC vaccination strategy is to determine, at each sampled step (every day in the context of this work), the number of individuals from each age group to be vaccinated to minimize the death toll at the end of the pandemic. Here, the number of vaccines administered daily should not exceed the daily vaccination capacity, denoted $\bar{v}$. The  ability of MPC to handle constraints makes it well suited for this task. Therefore, the ideal cost to be minimized, given as the summation of the daily number of deceased individuals, is represented as follows:
\begin{equation}\label{eq:cost_infty}
\lim_{n\rightarrow + \infty} \sum_{k=1}^{n_a} D_k(n) = \sum_{n = 0}^{\infty} \sum_{k=1}^{n_a} D_k(n+1) - D_k(n) = \sum_{n = 0}^{\infty} \sum_{k=1}^{n_a} \gamma_{D_k} I_k(n) = \sum_{n = 0}^{\infty} \gamma_D^{\top} I(n)
\end{equation}

Since it is not possible to solve the infinite-horizon optimal control problem, the MPC approach is employed. The stage cost of the MPC serves as an upper bound for (\ref{eq:cost_infty}), and is defined as the daily number of deceased individuals: $l(x(n), u(n))=\gamma_{D}^\top I(n)$. Additionally, an appropriate final cost is incorporated. The final cost is derived from a local Lyapunov function, defined within the invariant set $\mathbb{X}_f$, as illustrated in \autoref{thm:final_cost}.

\begin{theorem}\label{thm:final_cost}
Given the discrete-time dynamics in (\ref{eq:disc_model}) and the optimization problem in (\ref{eq:opt_prob}) with $x = (S, \, I, \, R, \, D) \in \mathbb{R}^{4n_a}$, the population in every age group $P \in \mathbb{R}^{n_a}$, supposed to be constant in time, and 
\begin{align}
&\mathbb{X} = \{x \in \mathbb{R}^{4n_a}: 0 \leq S \leq P, \ 0 \leq I \leq P, \ 0 \leq R \leq P, \ 0 \leq D \leq P\}, \label{eq:sets1} \\
&\mathbb{U} = \{u \in \mathbb{R}^n_a: u \geq 0, \ \sum_{k=1}^{n_a} u_k \leq \bar{v} \}, \qquad \mathbb{Z}=\mathbb{X} \times \mathbb{U}, \qquad X^*=\{x \in \mathbb{X}: I=0\}. \label{eq:sets2}
\end{align}
Let $l(\cdot,\cdot), V_f(\cdot)$ and $X_f$ be given by:
\begin{align}
& l(x(n),u(n)) = \gamma_D^\top I(n), \\
& V_f(x(N)) = \frac{1}{\epsilon} \gamma_{D}^\top I(N)\label{eq:vf}\\ 
& \mathbb{X}_f = \{x \in \mathbb{X} : C^\top \Lambda S \leq \Gamma \} \cup X^*, \label{eq:xf}
\end{align}
where
\begin{equation}\label{eq:CLambdaGamma}
    C = \begin{pmatrix}
    C_{11} & \cdots & C_{1n_a} \\
    \vdots & \ddots & \vdots \\
    C_{n_a1} & \cdots & C_{n_an_a}
    \end{pmatrix}\!\!, \
    \Lambda = \begin{pmatrix}
    \gamma_{D_1}\lambda_1 & \cdots & 0 \\
     \vdots & \ddots & \vdots \\
     0 & \cdots & \gamma_{D_{n_a}}\lambda_{n_a}
    \end{pmatrix}\!\!, \
    \Gamma = \begin{pmatrix}
        \gamma_{D_1}(\gamma_{R_1} + \gamma_{D_1} - \epsilon) \\
        \vdots \\
        \gamma_{D_{n_a}}(\gamma_{R_{n_a}} + \gamma_{D_{n_a}} - \epsilon)
    \end{pmatrix}
\end{equation}
with $\epsilon \in \mathbb{R}$ such that $0 < \epsilon < \min\limits_{k=1, \dots, n_a} \{\gamma_{R_k}+\gamma_{D_k}\}$. \\
Then, the following assertions hold:
\begin{enumerate}
    \item[i.] $\forall x \in \mathcal{X}_N$, there exists a solution $\mathbf{u}^0$ to the optimal control problem in (\ref{eq:opt_prob});
    \item[ii.] the set $X^*$ is asymptotically stable in $\mathcal{X}_N$ for $x(n+1)=f(x(n), \kappa_N(x(n)))$;
    \item[iii.] the optimal value $V_N^0(x)$ is an upper bound for the total number of future deceased individuals if the MPC control $\kappa_N(x)$ is applied.
\end{enumerate}
\end{theorem}

\begin{proof}
Assertions \emph{i.} and \emph{ii.} are proved by showing that the six conditions of Theorem \ref{thm: asymptotic_stab} are satisfied.
\begin{itemize}
    \item[\emph{1.-2.-3.}] By definition, the sets $\mathbb{X}$, $\mathbb{U}$, $\mathbb{Z}$, $X^*$ and $\mathbb{X}_f$ and the functions $l(\cdot, \cdot)$, $f(\cdot)$ and $V_f(\cdot)$ satisfy conditions \emph{1}, \emph{2}, and \emph{3} of Theorem $\ref{thm: asymptotic_stab}$.
    \item[\emph{4.}] First note that $X^*$ is invariant $\forall u \in \mathbb{U}$ since no one can get infected if there are no infected individuals. Consider the condition under which the daily sum of deceased individuals is decreasing $\forall u \in \mathbb{U}$, such that:
\begin{equation}\label{eq:region_dec}
	\gamma_D^\top I(n+1) - \gamma_D^\top I(n) \leq - \epsilon \gamma_D^\top I(n)
\end{equation}
with $0 < \epsilon < \min\limits_{k=1, \dots, n_a} \{\gamma_{R_k}+\gamma_{D_k}\}$.\\
From the discrete-time dynamics in (\ref{eq:disc_model}), condition (\ref{eq:region_dec}) can be represented as follows:
\begin{equation}
\begin{aligned}
\sum_{k=1}^{n_a} \gamma_{D_k}\lambda_kS_k(n) \sum_{j=1}^{n_a} C_{k,j}I_j(n) - \sum_{j=1}^{n_a} \gamma_{D_j} (\gamma_{R_j} + \gamma_{D_j} - \epsilon)I_j(n) \leq 0\\
\iff \sum_{j=1}^{n_a} \left( \sum_{k=1}^{n_a} \gamma_{D_k} \lambda_k S_k(n)C_{k,j} - \gamma_{D_j}(\gamma_{R_j}+\gamma_{D_j}-\epsilon) \right) I_j(n) \leq 0.
\end{aligned}
\label{eq:region_dec_equiv}
\end{equation}
Note that the one-step decreasing condition does not depend on the input. From the non-negativity of the number of infected individuals $I(n)$, condition (\ref{eq:region_dec_equiv}) is satisfied if (\ref{eq:region_dec_equiv_cond}) holds.
\begin{equation}
\sum_{k=1}^{n_a} \gamma_{D_k} \lambda_k S_k(n) C_{k,j} \leq \gamma_{D_j} (\gamma_{R_j} + \gamma_{D_j} - \epsilon), \quad \forall j=1, \dots, n_a
\label{eq:region_dec_equiv_cond}
\end{equation}
The term in (\ref{eq:region_dec_equiv_cond}) is given in matrix form by $C^\top \Lambda S(n) \leq \Gamma$, see (\ref{eq:xf}) and (\ref{eq:CLambdaGamma}). Moreover, from the non-increasing nature of the number of susceptible individuals $S(n)$ and the positivity of matrices $C$ and $\Lambda$, it follows that $C^\top \Lambda S(n+1) \leq C^\top \Lambda S(n) \leq \Gamma$. Hence, (\ref{eq:xf}) is an invariant set of the system $\forall u \in \mathbb{U}$, satisfying condition \emph{4.(a)}. To prove the satisfaction of condition \emph{4.(b)}, note that (\ref{eq:vf_diff}) holds $\forall x \in \mathbb{X}_f$.
\begin{equation}\label{eq:vf_diff}
V_f(f(x, 0)) - V_f(x) = \frac{1}{\epsilon} (\lambda_D^\top I^+ - \lambda_D^\top I) \leq - \lambda_D^\top I = -l(x, 0)
\end{equation}

\item[\it 5.] Consider $x \in \mathcal{X}_N$ and $u \in \mathcal{U}_N$ such that $(x, u) \in \mathbb{Z}$. By defining $\underline{\gamma_D} = \min\limits_{k=1, \dots, n_1} \gamma_{D_k}$, which is positive by construction, it follows that:
\begin{align*}
l(x,u) = \gamma_D^\top I \geq \underline{\gamma_D} \sum_{k=1}^{n_a} I_k = \underline{\gamma_D} \|I\|_1 = \underline{\gamma_D} \|x\|_{X^*}.
\end{align*}
Then, condition \emph{5.(a)} holds with $\alpha_1(y) = \underline{\gamma_D} y$ of class $\mathcal{K}_\infty$. Condition \emph{5.(b)} holds with $\alpha_f(y) = \frac{1}{\epsilon} \Bar{\gamma}_D y$ of class $\mathcal{K}_\infty$, where $\Bar{\gamma}_D=\max\limits_{k=1, \dots, n_a} \gamma_{D_k}$.

\item[\it 6.] From the discrete-time SIRD model in (\ref{eq:disc_model}), it follows that 
\begin{align*}
	\gamma_D^\top I(n+1) = \gamma_D^\top (\mathbb{I} + S_d(n) \Lambda C - \Gamma_{RD})I(n)
\end{align*}
where $\mathbb{I}$ is an identity matrix of dimensions $n_a \times n_a$, and
\begin{align*}
S_d =
\begin{pmatrix}
S_1 & & \\
& \ddots & \\
& & S_{n_a}
\end{pmatrix}\!\!,
\quad
P_d =
\begin{pmatrix}
P_1 & & \\
& \ddots & \\
& & P_{n_a}
\end{pmatrix}\!\!,
\quad
\Gamma_{RD} =
\begin{pmatrix}
\gamma_{R_1}+\gamma_{D_1} & & \\
& \ddots & \\
& & \gamma_{R_{n_a}} + \gamma_{D_{n_a}}
\end{pmatrix}\!\!.
\end{align*}
From the bounds on the number of infected and susceptible individuals at every step $n$, it follows that $0 \leq S(n) \leq P$, where $P$ is the total population in every age group. Moreover, from the positivity of $\Lambda$ and $C$, it implies that
\begin{equation}
\gamma_D^\top I(n+1) = \gamma_D^\top (\mathbb{I}+S_d(n)\Lambda C-\Gamma_{RD})I(n) \leq \gamma_D^\top (\mathbb{I}+P_d\Lambda C - \Gamma_{RD}) I(n) \leq \eta \gamma_D^\top I(n)
\label{eq:daily_deceased}
\end{equation}
for $\eta > 0$, where $\eta$ is big enough. From (\ref{eq:vf}) and (\ref{eq:daily_deceased}), it follows that:
        \begin{align*}
            &\gamma_D^\top I(n+i) \leq \eta^i \gamma_D^\top I(n), \quad \forall n \in \mathbb{N}, i \in \mathbb{N} \\
            &V_f(x(n+N)) = \frac{1}{\epsilon} \gamma_D^\top I(n+N) \leq \frac{1}{\epsilon} \eta^N \gamma_D^\top I(n), \quad \forall I(n) \in \mathbb{R}^n_+, N \in \mathbb{N},
        \end{align*}
which extends to:
\begin{equation}\label{eq:bound}
\begin{aligned}
V_N^0(x(n)) & = \sum_{i=0}^{N-1} \gamma_D^\top I(n+i) + V_f(x(n+N)) \leq \left( \epsilon^{-1}\eta^N + \sum_{i=0}^{N-1} \eta^i \right) \gamma_D^\top I(n) \\
&\leq \left( \epsilon^{-1} \eta^N + \sum_{i=0}^{N-1} \eta^i \right) \|\gamma_D\|_1 \|I(n)\|_1 = \left( \epsilon^{-1} \eta^N + \sum_{i=0}^{N-1} \eta^i \right) \|\gamma_D\|_1 \|x(n)\|_{X^*}
\end{aligned}
\end{equation}
This indicates that the bound in condition \emph{6.} holds with $\alpha(y)=\left( \epsilon^{-1}\eta^N + \sum_{i=0}^{N-1}\eta^i \right) \|\gamma_D\|_1y$.
\end{itemize}
Finally, Assertion \emph{iii.} is a direct consequence of Corollary \ref{cor:VN0}.
\label{thm: proof_stab}
\end{proof}

\subsection{MPC-based vaccination implementation}
To implement the MPC vaccination strategy, it is essential to gather data on the number of infected and deceased individuals at each sampled step. In the context of this work, the vaccination strategy requires daily data acquisition to update the initial states of the discrete-time model in (\ref{eq:disc_model}) used inside the MPC. This enables the MPC to compute the daily number of vaccines to be administered per age group over the vaccination strategy horizon, $N_v$. The MPC structure is illustrated in Algorithm \ref{alg:mpc_vaccination}.
\begin{algorithm}[h!]
\caption{MPC Vaccination Strategy}
\label{alg:mpc_vaccination}
\begin{algorithmic}[1]
\State \textbf{Define the MPC parameters:} the prediction horizon $N$, the vaccination strategy horizon $N_v$, the maximum daily vaccination capacity $\bar{v}$, and the threshold on infected population to determine pandemic eradication $I_{e} \in \mathbb{R}^{n_a}$;
\State \textbf{Define the initial states and populations:} $S_k(0)$, $I_k(0)$, $R_k(0)$, $D_k(0)$  and $P_k$ $\forall k \in \{1, 2, \dots, n_a\}$;
\For{each day $n = 0, 1, \dots, N_v$}
	\State Obtain the current state values $S_k(n)$, $I_k(n)$, $R_k(n)$ and $D_k(n)$;
	\If{$I_k(n) \leq I_e$}
        \State The disease is eradicated: stop vaccinating;
    \Else
    	\State Solve the optimization problem in (\ref{eq:opt_prob}) with (\ref{eq:sets1})-(\ref{eq:CLambdaGamma});
    	\State Extract the first element of the optimal control sequence: $\kappa_N(x_k(n)) = u_k^0(0,x_k(n))$;
	    \State Administer vaccines according to $u_k(n) = \kappa_N(x_k(n))$;
    \EndIf
	\State Update the state;
\EndFor
\end{algorithmic}
\end{algorithm}
 
Recall that the equilibrium set is defined as $X^*=\{x \in \mathbb{X}: I=0\}$, and note that neither the problem constraints nor the cost function depend on the recovered or past deceased populations. In fact, the cost in (\ref{eq:cost_infty}) representing the sum of future deaths depends only on susceptible and infected populations. This implies that the solution of the MPC problem is not affected by the dynamics of $D$ and $R$. Then, a reduced-order dynamical model can be employed inside the MPC. Hence, the state vector representing the daily number of susceptible and infected individuals in each age group can be defined as $x = \left( S, I\right) \in \mathbb{X}$, where
\begin{align*}
    &\mathbb{X}:=\{x = (S,I) \in \mathbb{R}^{2n_a}: 0 \leq S \leq P, \quad 0 \leq I \leq P\}, \\
    &X^*:=\{x = (S,I) \in \mathbb{R}^{2n_a}: I=0\}.
\end{align*}
Here, $\mathbb{X}$ and $X^*$ represent the projections of the full state sets onto the subspace of susceptible and infected populations. Consequently, the MPC problem can be equivalently expressed in their terms, simplifying the optimization problem while preserving the essential dynamics.

%% file: sections/04_Results.tex
\section{Results and Discussion}
\label{sec: results}
To demonstrate the effectiveness of the MPC vaccination strategy, it is necessary to compare its performance with both the strategy adopted by the Belgian government during the first wave of vaccination and the no vaccination approach. This analysis uses the model in (\ref{eq:disc_model}) to forecast the demographic consequences of the three proposed scenarios on the Walloon population.
\subsection{Simulation parameters}
To simulate the impact of a vaccination strategy on a given population, it is crucial to gather enough demographic information to represent it. To start, the data are categorized into $n_a=6$ different age groups with respective populations $P_k$ for $k \in \{1, 2, \dots, n_a\}$. The Walloon population per age group is obtained from the website of the Belgian statistical office \cite{statbel2018structuur}. The maximum number of vaccines administered during a single day in Wallonia is $\bar{v} = 55191$ shots, which occurred on June 16, 2021 \cite{Sciensano2020}. This number is used to set the threshold for the daily vaccination capacity. The contact matrix $C_{kj}$ is obtained using the online tool Socrates (social contact rates) \cite{willem2020socrates}. For this work, data collected from Belgium in 2010 are used. The elements of the social contact matrix are divided by the total number of individuals in each age group to adjust it to the Walloon population. The remaining simulation parameters: $\lambda_k$, $\gamma_{R_k}$, $\gamma_{D_k}$, and the initial number of infected individuals $I_{0_k}$ are obtained from \cite{Sonveaux2023}, where the author used a genetic algorithm to calibrate the model parameters in (\ref{eq:cont_model}) using data from Wallonia. No drugs are administered and no cases of deceased or recovered individuals are recorded prior to the simulation. The simulation time is set to $t_{sim}=140$ $days$. The simulation time is equivalent to the vaccination strategy horizon, $N_v$. The simulation parameters are summarized in \autoref{tab:sim_param}.
\begin{table}[h!]
    \centering
    \begin{tabular}{c|cccccc}
        $k$ & $range$ & $P_k$ & $\lambda_k$ & $\gamma_{R_{k}}$ & $\gamma_{D_{k}}$ & $I_{0_k}$ \\ \hline
        1 & 0-24 & 1058304 & 0.0769924521 & 0.9216927886 & 0.0004407167 & 4.6595088243 \\
        2 & 25-44 & 0915796 & 0.0290873349 & 0.7230105996 & 0.0018303543 & 4.3296088874 \\
        3 & 45-64 & 0983789 & 0.0136872530 & 0.5707245171 & 0.0232746601 & 4.8417769521 \\
        4 & 65-74 & 0384803 & 0.1149749309 & 0.8482912034 & 0.0397484004 & 0.1709101349 \\
        5 & 75-84 & 0203035 & 0.2326289564 & 0.8200428486 & 0.1006921381 & 1.4936938584 \\
        6 & 85+ & 0099516 & 0.3331837058 & 0.6612236351 & 0.1514435560 & 1.6144863665 \\
    \end{tabular}
    \caption{Age-structured SIRD model simulation parameters}
    \label{tab:sim_param}
\end{table}
\subsection{No vaccination strategy}
At the beginning of every epidemic, a fundamental question arises: is vaccination necessary? The primary goal of vaccination is to reduce the severe effects and limit the spread of viruses and diseases. However, some might argue that the immune system of humans is capable of adapting to and defeating new diseases. In the case of COVID-19, a notable number of deaths have been reported among individuals with weak immunity and pre-existing comorbidity. To forecast the demographic consequences of choosing not to vaccinate the Walloon population, the evolution of population dynamics is simulated using the model in (\ref{eq:disc_model}) with zero input.

Graphically, \autoref{fig:ol_states} illustrates the evolution of population dynamics over time when no vaccine is administered. \autoref{fig:ol_inf} shows that the disease does not eradicate within the specified simulation time, while \autoref{fig:ol_dec} depicts the intolerable growth in the COVID-19 death toll. When no vaccine is administered, more than $16\%$ of the Walloon population becomes infected, with a total death toll of $15497$ individuals during the simulation period. In summary, these alarming findings underscore the urgent need for an effective vaccination strategy to minimize deaths.
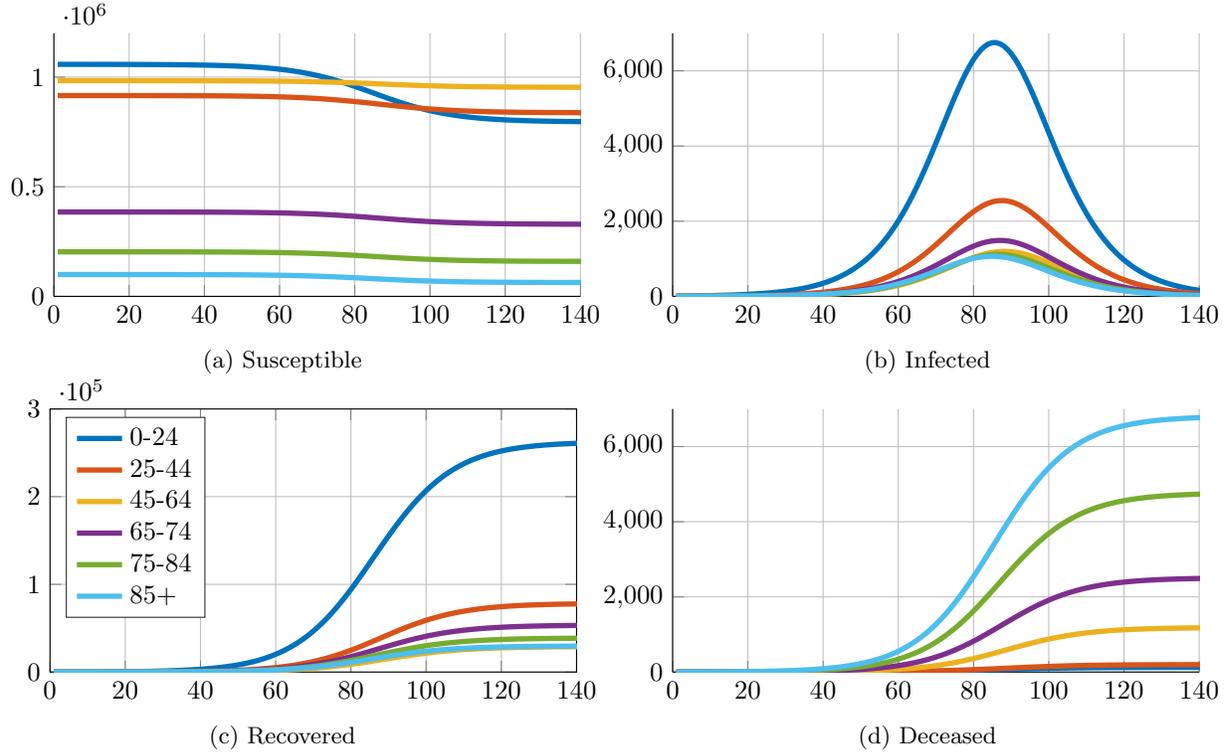
\begin{figure}[h!]
    \centering
        \centering
        \begin{subfigure}{0.45\textwidth}
            \centering
            \hspace*{-0.3cm}
            \input{Figures/OL_state_1}
            \caption{Susceptible}
            \label{fig:ol_sus}
        \end{subfigure}
        \hfill
        \begin{subfigure}{0.45\textwidth}
            \centering
            \hspace*{-1cm}
            \input{Figures/OL_state_2}
            \caption{Infected}
            \label{fig:ol_inf}
        \end{subfigure}
        \vfill
        \begin{subfigure}{0.45\textwidth}
            \centering
            \input{Figures/OL_state_3}
            \caption{Recovered}
            \label{fig:ol_rec}
        \end{subfigure}
        \hfill
        \begin{subfigure}{0.45\textwidth}
            \hspace*{-1cm}
            \centering
            \input{Figures/OL_state_4}
            \caption{Deceased}
            \label{fig:ol_dec}
        \end{subfigure}
        \caption{Populations evolution over time (in days) in the absence of vaccination}
        \label{fig:ol_states}
\end{figure}
\subsection{National vaccination strategy}
During the first wave of vaccination, decision makers in Belgium adopted a decreasing age strategy to limit virus transmission and save as many lives as possible. Given that COVID-19 has been shown to be more fatal in older age groups, the national strategy prioritizes vaccinating the elderly first, then progressively moving to younger age groups. The objective of this hierarchical approach is to first drive the number of susceptible individuals in the oldest age group to zero before moving down the age hierarchy, systematically covering the whole population until the disease is eradicated. Here, disease eradication is defined as $I_{e, k}<1$ for all $k \in \{1, 2, \dots, n_a\}$.
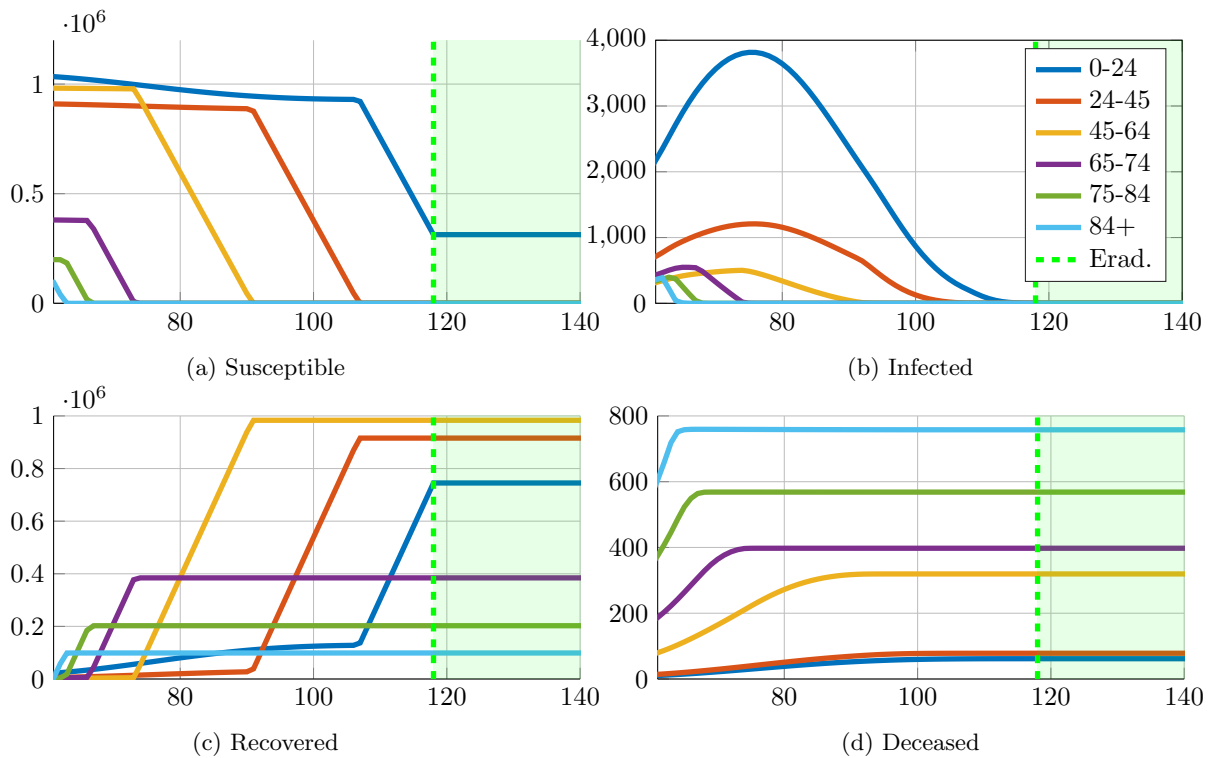
\begin{figure}[h!]
    \centering
        \centering
        \begin{subfigure}{0.45\textwidth}
            \centering
            \input{Figures/NS_state_1}
            \caption{Susceptible}
            \label{fig:ns_sus}
        \end{subfigure}
        \hfill
        \begin{subfigure}{0.45\textwidth}
            \centering
            \hspace*{-1cm}
            \input{Figures/NS_state_2}
            \caption{Infected}
            \label{fig:ns_inf}
        \end{subfigure}
        \vfill
        \begin{subfigure}{0.45\textwidth}
            \centering
            \input{Figures/NS_state_3}
            \caption{Recovered}
            \label{fig:ns_rec}
        \end{subfigure}
        \hfill
        \begin{subfigure}{0.45\textwidth}
            \centering
            \hspace*{-0.7cm}
            \input{Figures/NS_state_4}
            \caption{Deceased}
            \label{fig:ns_dec}
        \end{subfigure}
        \caption{Populations evolution over time (in days) under the national vaccination strategy}
        \label{fig:ns_states}
\end{figure}

To comply with realistic settings, the national vaccination strategy is implemented two months after the COVID-19 outbreak, assuming that this is the period required for the development and testing of the vaccine. In \autoref{fig:ns_inf}, it is shown that the disease is eradicated within 57 days from the start of vaccination. In addition, the fraction of infected individuals decreases dramatically to $5.96\%$ after vaccination. \autoref{fig:ns_dec} highlights a significant reduction in the total death toll since the beginning of vaccination, which decreases by $79.9\%$ under the national strategy compared to no administration. The green-shaded area represents the safe zone: the period of time during which the disease is no longer present. A larger safe zone indicates faster disease eradication.

Despite its benefits, the national vaccination strategy is not optimal in terms of the final death toll, as shown in the following sections. Moreover, the decreasing age approach does not account for additional factors, such as exposure levels and pandemic spreading rates. As proved in the following, taking explicitly into account the pandemics dynamics allows to develop a more advanced strategy capable of further reducing the final death toll, eradicating the disease faster while consuming fewer vaccines. 

\subsection{MPC vaccination strategy}
An MPC vaccination strategy aimed at minimizing deaths is proposed here, with the goal of developing a systematic and optimal strategy to vaccinate the Walloon population while adhering to operational constraints. The MPC algorithm is summarized in Algorithm \ref{alg:mpc_vaccination}. The values of the MPC parameters used for this simulation are $N=40$ days and $\epsilon=0.1$.
\begin{figure}[h!]
    \centering
        \centering
        \begin{subfigure}{0.45\textwidth}
            \centering
            \input{Figures/MPC_state_1}
            \caption{Susceptible}
            \label{fig:mpc_sus}
        \end{subfigure}
        \hfill
        \begin{subfigure}{0.45\textwidth}
            \centering
            \hspace*{-1cm}
            \input{Figures/MPC_state_2}
            \caption{Infected}
            \label{fig:mpc_inf}
        \end{subfigure}
        \vfill
        \begin{subfigure}{0.45\textwidth}
            \centering
            \input{Figures/MPC_state_3}
            \caption{Recovered}
            \label{fig:mpc_rec}
        \end{subfigure}
        \hfill
        \begin{subfigure}{0.45\textwidth}
            \centering
            \hspace*{-0.7cm}
            \input{Figures/MPC_state_4}
            \caption{Deceased}
            \label{fig:mpc_dec}
        \end{subfigure}
        \caption{Populations evolution over time (in days) under the MPC vaccination strategy}
        \label{fig:mpc_states}
    \hfill
\end{figure}
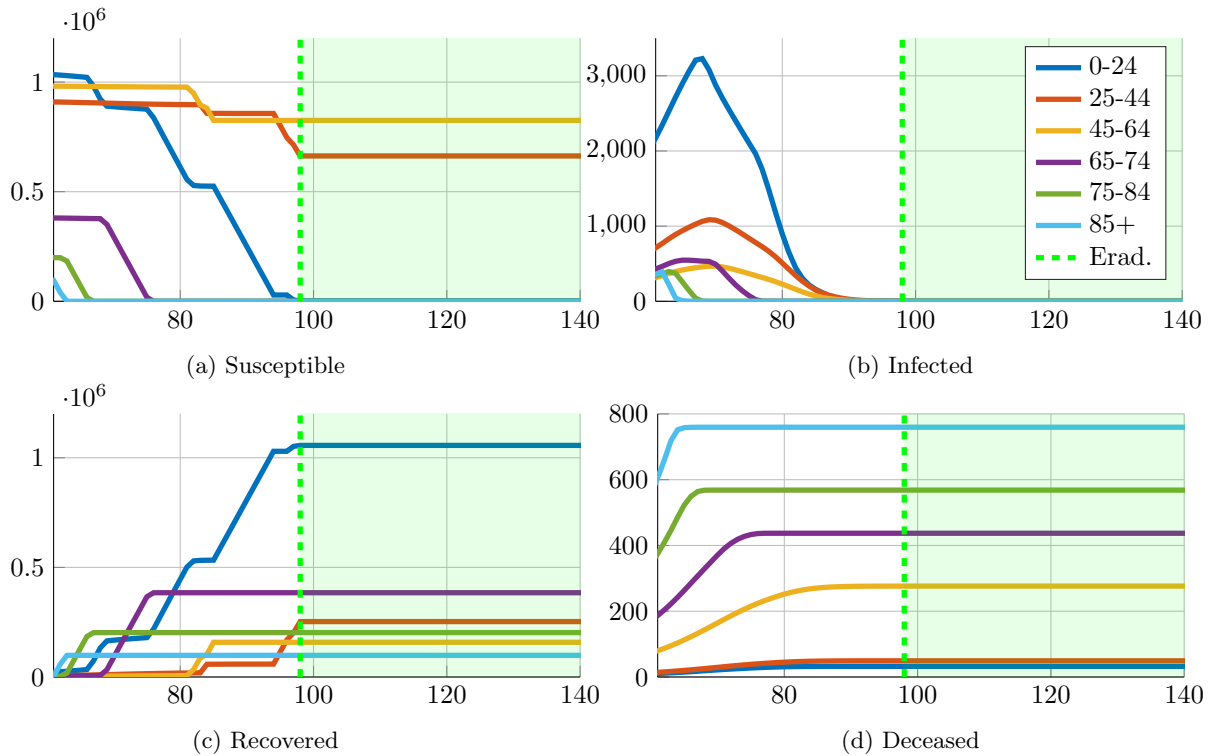

As for the former approach, the MPC is tested following the availability of vaccines: two months after the COVID-19 outbreak. \autoref{fig:mpc_states} illustrates the demographic improvements using the MPC approach. In \autoref{fig:mpc_inf}, the eradication time decreases to 37 days from the beginning of vaccination, and the fraction of infected individuals is reduced to $3.67\%$. Consequently, the total disease death toll since the beginning of vaccination in \autoref{fig:mpc_dec} decreases by $80.3\%$ under the MPC strategy compared to the absence of administration. This improvement is attributed to the model-based nature of the MPC. The MPC uses the mathematical model in (\ref{eq:disc_model}) to guide its decisions, incorporating information such as contact, death, and recovery rates for each age group. This represents an improvement over the national strategy, which solely bases its decisions on age.

\subsection{Discussion}
Comparing both vaccination strategies, \autoref{MPC_NS_ID} demonstrates that the MPC provides an improved performance in eradicating the disease and saving as many lives as possible. Graphically, the safe zone occupies a significantly larger area in \autoref{fig:mpc_ID} compared to \autoref{fig:ns_ID}. A larger safe zone is equivalent to faster disease eradication time. Numerically, it can be concluded that the MPC approach is $35\%$ faster than the national strategy in eradicating the disease. This results in the MPC decreasing the total number of infected individuals by $49\%$ compared to the national strategy. Moreover, since vaccination is carried out daily at maximum capacity, the speed of disease eradication is inversely proportional to the number of vaccines consumed. This suggests that the MPC approach is more cost effective than the national strategy, as it requires $35\%$ fewer vaccines to achieve disease eradication. Furthermore, while not graphically prominent in \autoref{MPC_NS_ID}, the MPC approach saves $6.6\%$ more lives compared to the national strategy. Ethically, this alone establishes the MPC as the preferred method. The offset in the total number of deceased individuals between the two approaches is better represented in \autoref{fig:mpc_cost}. Their respective performance is summarized in \autoref{tab:mpc_vs_nat}.
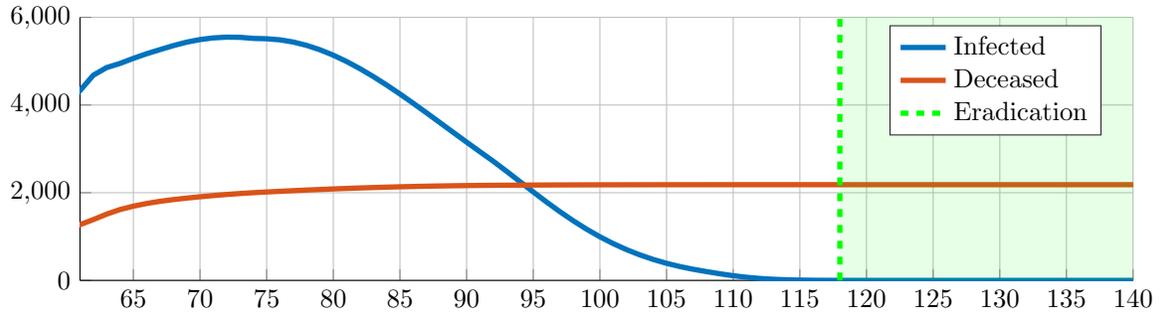
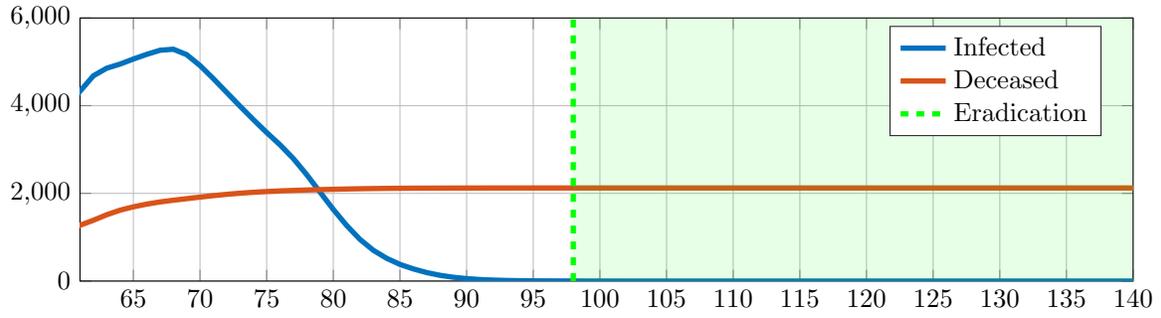
\begin{figure}
        \centering
        \begin{subfigure}{\textwidth}
            \input{Figures/NS_ID}
            \caption{National vaccination strategy}
            \label{fig:ns_ID}
        \end{subfigure}
    \vfill
        \centering
        \begin{subfigure}{\textwidth}
            \input{Figures/MPC_ID}
            \caption{MPC vaccination strategy}
            \label{fig:mpc_ID}
        \end{subfigure}
    \caption{Daily infected and total deceased evolution over time (in days)}
    \label{MPC_NS_ID}
\end{figure}
\begin{table}
    \centering
    \begin{tabular}{c|cccc}
         & Deceased & Infected & Eradication & Injected Vaccines \\ \hline
        National Strategy & 911 & 169578 & 57 days & 3145887 \\
        MPC Strategy & 851 & 86314 & 37 days & 2042017 \\  \hline
        Improvement & $6.6\%$ & $49.1\%$ & $35.1\%$ & $35.1\%$ \\
        \hline
    \end{tabular}
    \caption{Statistical comparison between the national and the MPC vaccination strategies}
    \label{tab:mpc_vs_nat}
\end{table}

In terms of vaccination, \autoref{fig:mpc_ns_input} illustrates the significant difference between the two approaches. The national strategy in \autoref{fig:ns_strat} follows a decreasing age approach. This strategy prioritizes vaccination of the older age groups. In contrast, the MPC strategy in \autoref{fig:mpc_strat} employs a more dynamic approach. The MPC strategy allows to simultaneously vaccinate multiple age groups at once. This stems from the model-based nature of the MPC that accounts for different factors to take its decisions. Consequently, the MPC is better also for minimizing deaths compared to the national strategy.
\begin{figure}
    \centering
    \begin{subfigure}{\linewidth}
        \centering
        \input{Figures/NS_input}
        \caption{National vaccination strategy}
        \label{fig:ns_strat}
    \end{subfigure}
    \hfill
    \begin{subfigure}{\linewidth}
        \centering
        \input{Figures/MPC_input}
        \caption{MPC vaccination strategy}
        \label{fig:mpc_strat}
    \end{subfigure}
    \caption{Vaccine administration evolution over time (in days)}
    \label{fig:mpc_ns_input}
\end{figure}
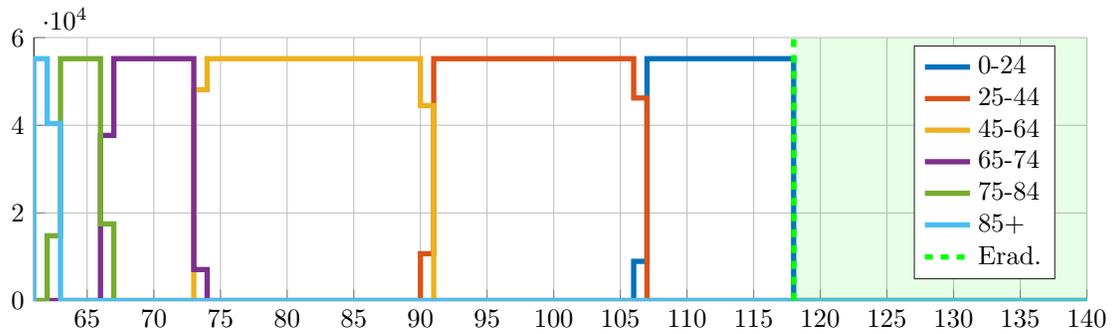
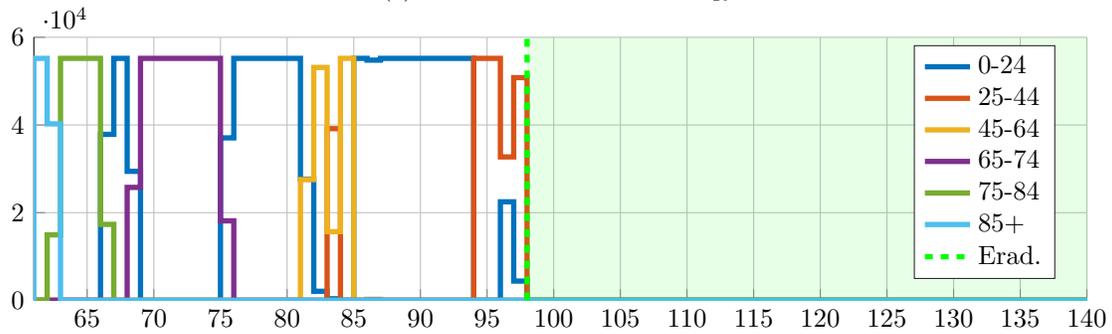

As previously mentioned, the MPC approach requires $35\%$ fewer vaccines compared to the national strategy to achieve disease eradication. This gives the MPC its cost-effective characteristics, making it suitable for implementation in nations with weak economies. Moreover, \autoref{fig:mpc_cost} highlights the performance improvement of the proposed MPC approach, providing a comparison of the evolution of the total death toll after vaccine administration using both strategies. Note moreover that, as expected from Theorem \ref{thm:final_cost}, the MPC results in a final death toll that is upper-bounded by the MPC cost, whereas the national strategy violates such upper bound. Indeed, the MPC cost in \autoref{fig:mpc_cost} is an upper bound for the minimum overall number of deaths achievable at the end of the simulation. In addition, it is interesting to note that the MPC-based strategy temporarily generates a higher number of deceased individuals, see values around day 75, although the final death toll is lower than that obtained under the national strategy. This counterintuitive observation confirms that the design of the optimal vaccination strategy is not a trivial problem to be addressed.

\begin{figure}
    \centering
    \input{Figures/MPC_cost}
    \caption{Evolution of the MPC cost versus the total number of deceased individuals under the MPC and national strategies since the start of vaccination until disease eradication}
    \label{fig:mpc_cost}
\end{figure}
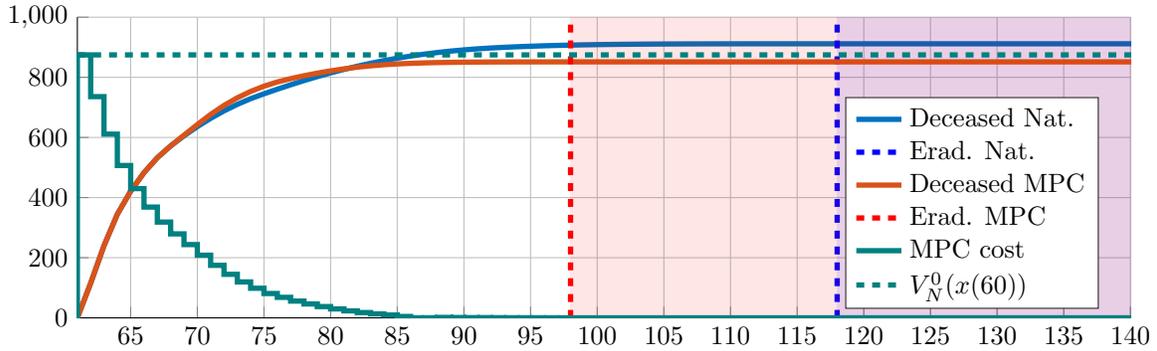

%% file: Figures/OL_state_1.tex
%
%
\definecolor{mycolor1}{rgb}{0.00000,0.44700,0.74100}%
\definecolor{mycolor2}{rgb}{0.85000,0.32500,0.09800}%
\definecolor{mycolor3}{rgb}{0.92900,0.69400,0.12500}%
\definecolor{mycolor4}{rgb}{0.49400,0.18400,0.55600}%
\definecolor{mycolor5}{rgb}{0.46600,0.67400,0.18800}%
\definecolor{mycolor6}{rgb}{0.30100,0.74500,0.93300}%
\begin{tikzpicture}

\begin{axis}[%
width=7cm,
height=3.5cm,
at={(0cm,0cm)},
scale only axis,
xmin=0,
xmax=140,
ymin=0,
ymax=1200000,
xmajorgrids,
ymajorgrids,
axis background/.style={fill=white},
axis x line*=bottom,
axis y line*=left
]
\addplot [color=mycolor1, line width=2.0pt, forget plot]
  table[row sep=crcr]{%
1	1058299.34049118\\
2	1058292.44626517\\
3	1058283.78077337\\
4	1058273.61241241\\
5	1058262.04982195\\
6	1058249.10573796\\
7	1058234.73290958\\
8	1058218.84387911\\
9	1058201.32163147\\
10	1058182.02508278\\
11	1058160.79161252\\
12	1058137.43785495\\
13	1058111.75941957\\
14	1058083.52991054\\
15	1058052.49944864\\
16	1058018.39280639\\
17	1057980.9072135\\
18	1057939.70985877\\
19	1057894.43509538\\
20	1057844.68134481\\
21	1057790.00768631\\
22	1057729.9301135\\
23	1057663.91743528\\
24	1057591.38679507\\
25	1057511.69877958\\
26	1057424.1520859\\
27	1057327.97771375\\
28	1057222.332648\\
29	1057106.29299513\\
30	1056978.8465363\\
31	1056838.88465923\\
32	1056685.19363109\\
33	1056516.44517546\\
34	1056331.18631846\\
35	1056127.8284717\\
36	1055904.63572501\\
37	1055659.71232743\\
38	1055390.98934441\\
39	1055096.21048994\\
40	1054772.91714741\\
41	1054418.43261123\\
42	1054029.8456047\\
43	1053603.99315784\\
44	1053137.44296392\\
45	1052626.47537521\\
46	1052067.06524862\\
47	1051454.86391135\\
48	1050785.18158549\\
49	1050052.97069169\\
50	1049252.81054321\\
51	1048378.89404652\\
52	1047425.01714036\\
53	1046384.57183296\\
54	1045250.54383415\\
55	1044015.515924\\
56	1042671.67834603\\
57	1041210.84765812\\
58	1039624.49560582\\
59	1037903.78969438\\
60	1036039.64720974\\
61	1034022.80446188\\
62	1031843.90297417\\
63	1029493.59419949\\
64	1026962.66408259\\
65	1024242.178385\\
66	1021323.64912073\\
67	1018199.22169971\\
68	1014861.88143083\\
69	1011305.6768999\\
70	1007525.95642771\\
71	1003519.61237032\\
72	999285.326511916\\
73	994823.808313339\\
74	990138.016433998\\
75	985233.352883275\\
76	980117.818532393\\
77	974802.118682044\\
78	969299.708067791\\
79	963626.766186715\\
80	957802.096174977\\
81	951846.943606579\\
82	945784.735376779\\
83	939640.743046923\\
84	933441.679351163\\
85	927215.240640888\\
86	920989.611500949\\
87	914792.950279623\\
88	908652.875578678\\
89	902595.973713577\\
90	896647.34577617\\
91	890830.210349501\\
92	885165.574391614\\
93	879671.980660265\\
94	874365.335669193\\
95	869258.817915713\\
96	864362.862315347\\
97	859685.213658447\\
98	855231.039607785\\
99	851003.092330887\\
100	847001.907268583\\
101	843226.027679553\\
102	839672.244326334\\
103	836335.840819775\\
104	833210.836555919\\
105	830290.220717435\\
106	827566.172351282\\
107	825030.262984619\\
108	822673.639540345\\
109	820487.18642628\\
110	818461.666583911\\
111	816587.841996273\\
112	814856.574684003\\
113	813258.909585047\\
114	811786.140941637\\
115	810429.863933506\\
116	809182.013323264\\
117	808034.890840422\\
118	806981.182943931\\
119	806013.970485174\\
120	805126.731656978\\
121	804313.339469344\\
122	803568.054846839\\
123	802885.516301113\\
124	802260.726998707\\
125	801689.039921231\\
126	801166.141703526\\
127	800688.035635924\\
128	800251.024229028\\
129	799851.691663159\\
130	799486.88637886\\
131	799153.704008756\\
132	798849.470803727\\
133	798571.727666696\\
134	798318.214874442\\
135	798086.857540895\\
136	797875.751853439\\
137	797683.152096142\\
138	797507.458459936\\
139	797347.205628871\\
140	797201.052123294\\
141	797067.77037465\\
};
\addplot [color=mycolor2, line width=2.0pt, forget plot]
  table[row sep=crcr]{%
1	915791.670391113\\
2	915788.901498256\\
3	915786.091690194\\
4	915783.132068456\\
5	915779.950175584\\
6	915776.492532322\\
7	915772.714557096\\
8	915768.574937074\\
9	915764.032453362\\
10	915759.04411254\\
11	915753.563958054\\
12	915747.542225712\\
13	915740.924663755\\
14	915733.651920725\\
15	915725.658947598\\
16	915716.874383396\\
17	915707.219905194\\
18	915696.609529594\\
19	915684.948855782\\
20	915672.13424177\\
21	915658.051906044\\
22	915642.576946937\\
23	915625.572271926\\
24	915606.887428743\\
25	915586.357329784\\
26	915563.800860863\\
27	915539.019364923\\
28	915511.794990813\\
29	915481.888896859\\
30	915449.039298545\\
31	915412.959349352\\
32	915373.334843598\\
33	915329.821730101\\
34	915282.043425655\\
35	915229.587917726\\
36	915172.00464652\\
37	915108.801157736\\
38	915039.439518927\\
39	914963.33249471\\
40	914879.839478979\\
41	914788.262186238\\
42	914687.840109055\\
43	914577.745754883\\
44	914457.079683137\\
45	914324.865372849\\
46	914180.043962563\\
47	914021.468917787\\
48	913847.900697485\\
49	913658.001510099\\
50	913450.330271679\\
51	913223.337904108\\
52	912975.363140264\\
53	912704.629035312\\
54	912409.240419031\\
55	912087.182562778\\
56	911736.321375737\\
57	911354.405487427\\
58	910939.070615502\\
59	910487.846657563\\
60	909998.167980195\\
61	909467.387404114\\
62	908892.794396792\\
63	908271.637977925\\
64	907601.154812585\\
65	906878.602905263\\
66	906101.301208428\\
67	905266.675315064\\
68	904372.309210307\\
69	903416.002809082\\
70	902395.834703814\\
71	901310.229192402\\
72	900158.026260838\\
73	898938.552772743\\
74	897651.692692815\\
75	896297.953773094\\
76	894878.527797046\\
77	893395.341248036\\
78	891851.093188294\\
79	890249.277240698\\
80	888594.184888541\\
81	886890.887862634\\
82	885145.198165721\\
83	883363.605262136\\
84	881553.191082796\\
85	879721.524686019\\
86	877876.539581636\\
87	876026.397770429\\
88	874179.345378715\\
89	872343.565300644\\
90	870527.032447006\\
91	868737.377020988\\
92	866981.760716719\\
93	865266.769917185\\
94	863598.328932465\\
95	861981.635161281\\
96	860421.116876896\\
97	858920.413223708\\
98	857482.375038305\\
99	856109.08433118\\
100	854801.889710976\\
101	853561.454707138\\
102	852387.815834328\\
103	851280.447313611\\
104	850238.329583018\\
105	849260.019052128\\
106	848343.71694115\\
107	847487.335458649\\
108	846688.559983598\\
109	845944.906304455\\
110	845253.772314923\\
111	844612.483864049\\
112	844018.33470369\\
113	843468.620669646\\
114	842960.66837744\\
115	842491.858814964\\
116	842059.646278509\\
117	841661.573132348\\
118	841295.280881599\\
119	840958.518039209\\
120	840649.145245717\\
121	840365.138069289\\
122	840104.587876791\\
123	839865.701127266\\
124	839646.797399069\\
125	839446.306422791\\
126	839262.764354877\\
127	839094.809492351\\
128	838941.177597605\\
129	838800.696973985\\
130	838672.283407954\\
131	838554.935071754\\
132	838447.727461623\\
133	838349.808430434\\
134	838260.393359924\\
135	838178.76050614\\
136	838104.246542173\\
137	838036.242314337\\
138	837974.188821532\\
139	837917.573422329\\
140	837865.926270219\\
141	837818.816974232\\
};
\addplot [color=mycolor3, line width=2.0pt, forget plot]
  table[row sep=crcr]{%
1	983784.158223048\\
2	983783.023485627\\
3	983781.909278624\\
4	983780.76542737\\
5	983779.556847025\\
6	983778.257717261\\
7	983776.847487018\\
8	983775.308245556\\
9	983773.623068632\\
10	983771.774986174\\
11	983769.746318778\\
12	983767.518223245\\
13	983765.070351849\\
14	983762.380569837\\
15	983759.42469881\\
16	983756.176266809\\
17	983752.606253255\\
18	983748.682820979\\
19	983744.371029843\\
20	983739.632527662\\
21	983734.425214798\\
22	983728.70287909\\
23	983722.414797912\\
24	983715.505304104\\
25	983707.913312493\\
26	983699.571803527\\
27	983690.40726046\\
28	983680.339056318\\
29	983669.278786734\\
30	983657.129544611\\
31	983643.785132405\\
32	983629.129207786\\
33	983613.034358352\\
34	983595.361101145\\
35	983575.956802793\\
36	983554.654516383\\
37	983531.2717315\\
38	983505.609034454\\
39	983477.448676443\\
40	983446.553048471\\
41	983412.663063139\\
42	983375.496445196\\
43	983334.745934862\\
44	983290.077410699\\
45	983241.12794206\\
46	983187.503785267\\
47	983128.778342466\\
48	983064.490107939\\
49	982994.140633484\\
50	982917.192552434\\
51	982833.067711126\\
52	982741.145467132\\
53	982640.761225429\\
54	982531.205296857\\
55	982411.72217759\\
56	982281.510363753\\
57	982139.722831353\\
58	981985.468327959\\
59	981817.813638172\\
60	981635.786999115\\
61	981438.3828535\\
62	981224.568134846\\
63	980993.29028022\\
64	980743.487158281\\
65	980474.099081869\\
66	980184.083042559\\
67	979872.429256577\\
68	979538.180045261\\
69	979180.450986864\\
70	978798.454169303\\
71	978391.523246022\\
72	977959.139852166\\
73	977500.960780475\\
74	977016.845153584\\
75	976506.88067198\\
76	975971.407877976\\
77	975411.041270753\\
78	974826.686051914\\
79	974219.549290572\\
80	973591.144384624\\
81	972943.287868766\\
82	972278.087881253\\
83	971597.923943327\\
84	970905.418111201\\
85	970203.398005154\\
86	969494.852671359\\
87	968782.882652624\\
88	968070.645997722\\
89	967361.302192663\\
90	966657.956127097\\
91	965963.604202607\\
92	965281.084547824\\
93	964613.033042294\\
94	963961.846492142\\
95	963329.653879042\\
96	962718.296156561\\
97	962129.314630493\\
98	961563.947563664\\
99	961023.134314723\\
100	960507.526069608\\
101	960017.502059641\\
102	959553.190079389\\
103	959114.490111885\\
104	958701.099925663\\
105	958312.54161205\\
106	957948.188166799\\
107	957607.289372833\\
108	957288.996398279\\
109	956992.384676082\\
110	956716.474771297\\
111	956460.251064904\\
112	956222.678186502\\
113	956002.715211703\\
114	955799.327704475\\
115	955611.497731498\\
116	955438.232007202\\
117	955278.568346811\\
118	955131.580612993\\
119	954996.382341865\\
120	954872.129228172\\
121	954758.020639305\\
122	954653.300314832\\
123	954557.256393713\\
124	954469.220896149\\
125	954388.568771862\\
126	954314.716612017\\
127	954247.121108243\\
128	954185.27732962\\
129	954128.716877007\\
130	954077.005963978\\
131	954029.743464578\\
132	953986.558960355\\
133	953947.110812387\\
134	953911.084278275\\
135	953878.189689248\\
136	953848.160698461\\
137	953820.752608223\\
138	953795.74078113\\
139	953772.919137863\\
140	953752.098742589\\
141	953733.106475547\\
};
\addplot [color=mycolor4, line width=2.0pt, forget plot]
  table[row sep=crcr]{%
1	384802.829089865\\
2	384800.972397342\\
3	384798.818823749\\
4	384796.517541259\\
5	384794.076555305\\
6	384791.46606567\\
7	384788.648804423\\
8	384785.587690053\\
9	384782.246722872\\
10	384778.590145894\\
11	384774.581362579\\
12	384770.181981572\\
13	384765.351027502\\
14	384760.044276442\\
15	384754.213667157\\
16	384747.806748983\\
17	384740.766137903\\
18	384733.028960518\\
19	384724.526271146\\
20	384715.182430756\\
21	384704.91443859\\
22	384693.631208598\\
23	384681.232783445\\
24	384667.609479182\\
25	384652.640953732\\
26	384636.195192244\\
27	384618.127402216\\
28	384598.278811047\\
29	384576.47535847\\
30	384552.526276079\\
31	384526.222546029\\
32	384497.335230855\\
33	384465.613666422\\
34	384430.783510139\\
35	384392.544636933\\
36	384350.568876091\\
37	384304.497582932\\
38	384253.93904056\\
39	384198.46568865\\
40	384137.611178493\\
41	384070.867256428\\
42	383997.680481524\\
43	383917.448788002\\
44	383829.517908621\\
45	383733.177682233\\
46	383627.658277205\\
47	383512.126372463\\
48	383385.681349934\\
49	383247.351566175\\
50	383096.090787269\\
51	382930.774889761\\
52	382750.198951581\\
53	382553.07488056\\
54	382338.029754211\\
55	382103.60507253\\
56	381848.25715528\\
57	381570.358945583\\
58	381268.203511641\\
59	380940.009566259\\
60	380583.929347586\\
61	380198.059221277\\
62	379780.453370947\\
63	379329.140936356\\
64	378842.146932837\\
65	378317.517236281\\
66	377753.347840547\\
67	377147.818483865\\
68	376499.230593798\\
69	375806.0493143\\
70	375066.949153361\\
71	374280.862528932\\
72	373447.030201669\\
73	372565.052277817\\
74	371634.938161905\\
75	370657.153559775\\
76	369632.662405116\\
77	368562.961437407\\
78	367450.105126198\\
79	366296.718743156\\
80	365105.99764891\\
81	363881.691294634\\
82	362628.071031579\\
83	361349.881551044\\
84	360052.27660057\\
85	358740.740481891\\
86	357420.997664506\\
87	356098.913573955\\
88	354780.390169602\\
89	353471.260259605\\
90	352177.18457844\\
91	350903.555467121\\
92	349655.410566951\\
93	348437.359306225\\
94	347253.524185755\\
95	346107.498022998\\
96	345002.317466771\\
97	343940.452309069\\
98	342923.809448344\\
99	341953.749833988\\
100	341031.116361189\\
101	340156.270488691\\
102	339329.135305645\\
103	338549.242854351\\
104	337815.783694716\\
105	337127.656943203\\
106	336483.519305292\\
107	335881.831920688\\
108	335320.90413455\\
109	334798.933580757\\
110	334314.042204849\\
111	333864.308059423\\
112	333447.792871613\\
113	333062.565511975\\
114	332706.721589537\\
115	332378.399463059\\
116	332075.792998366\\
117	331797.161420625\\
118	331540.836613283\\
119	331305.228205907\\
120	331088.826775147\\
121	330890.205459179\\
122	330708.020258879\\
123	330541.009270251\\
124	330387.991063938\\
125	330247.862399733\\
126	330119.595437792\\
127	330002.234583994\\
128	329894.89308493\\
129	329796.749468383\\
130	329707.043907841\\
131	329625.074574508\\
132	329550.194027284\\
133	329481.805680062\\
134	329419.360376311\\
135	329362.353093057\\
136	329310.319789832\\
137	329262.834412831\\
138	329219.506060135\\
139	329179.976310396\\
140	329143.916714581\\
141	329111.026448263\\
};
\addplot [color=mycolor5, line width=2.0pt, forget plot]
  table[row sep=crcr]{%
1	203033.506306142\\
2	203031.847499558\\
3	203030.016281485\\
4	203028.080497469\\
5	203026.038662809\\
6	203023.865270194\\
7	203021.528986979\\
8	203018.998252081\\
9	203016.242259281\\
10	203013.230499123\\
11	203009.931973184\\
12	203006.314425502\\
13	203002.343669331\\
14	202997.983003483\\
15	202993.192691894\\
16	202987.929479424\\
17	202982.146121576\\
18	202975.790910739\\
19	202968.807185581\\
20	202961.132813118\\
21	202952.699634934\\
22	202943.432870321\\
23	202933.250469891\\
24	202922.062413597\\
25	202909.769947357\\
26	202896.264752467\\
27	202881.428041965\\
28	202865.129577968\\
29	202847.226603908\\
30	202827.56268546\\
31	202805.966453865\\
32	202782.250245363\\
33	202756.208630506\\
34	202727.616827347\\
35	202696.228992873\\
36	202661.776387624\\
37	202623.965409302\\
38	202582.475492339\\
39	202536.95687195\\
40	202487.028213236\\
41	202432.274108527\\
42	202372.242449391\\
43	202306.441683852\\
44	202234.337974333\\
45	202155.352277903\\
46	202068.857377718\\
47	201974.174903194\\
48	201870.572386691\\
49	201757.260416385\\
50	201633.389958757\\
51	201498.049939787\\
52	201350.265191577\\
53	201188.994890666\\
54	201013.131635627\\
55	200821.501334219\\
56	200612.86409408\\
57	200385.916334657\\
58	200139.294360874\\
59	199871.579659393\\
60	199581.30619426\\
61	199266.969988031\\
62	198927.041274139\\
63	198559.979493129\\
64	198164.251375695\\
65	197738.352305432\\
66	197280.831079943\\
67	196790.318087173\\
68	196265.556781946\\
69	195705.43818487\\
70	195109.03793319\\
71	194475.655195102\\
72	193804.852523301\\
73	193096.495482026\\
74	192350.790650803\\
75	191568.320407594\\
76	190750.072747109\\
77	189897.46432101\\
78	189012.354918516\\
79	188097.051757356\\
80	187154.302237104\\
81	186187.27422003\\
82	185199.523435859\\
83	184194.948229061\\
84	183177.732540195\\
85	182152.278685343\\
86	181123.132113481\\
87	180094.900825091\\
88	179072.172478632\\
89	178059.432360946\\
90	177060.98533833\\
91	176080.884642879\\
92	175122.869909821\\
93	174190.316308427\\
94	173286.195954933\\
95	172413.052117686\\
96	171572.986076662\\
97	170767.655926851\\
98	169998.286150934\\
99	169265.686449748\\
100	168570.278114251\\
101	167912.126142838\\
102	167290.975336851\\
103	166706.288723242\\
104	166157.28683244\\
105	165642.986577831\\
106	165162.238719509\\
107	164713.763131456\\
108	164296.181314495\\
109	163908.045798069\\
110	163547.86624659\\
111	163214.132228526\\
112	162905.332718785\\
113	162619.972489147\\
114	162356.585600653\\
115	162113.746249525\\
116	161890.077238134\\
117	161684.256348558\\
118	161495.020891558\\
119	161321.170691435\\
120	161161.569749646\\
121	161015.146809255\\
122	160880.895019923\\
123	160757.870880366\\
124	160645.192612969\\
125	160542.038104089\\
126	160447.642523962\\
127	160361.295722274\\
128	160282.339479406\\
129	160210.16467917\\
130	160144.208456487\\
131	160083.951362692\\
132	160028.91458199\\
133	159978.657224798\\
134	159932.773717179\\
135	159890.891300113\\
136	159852.66764789\\
137	159817.788611255\\
138	159785.966087991\\
139	159756.936021277\\
140	159730.45652435\\
141	159706.306128573\\
};
\addplot [color=mycolor6, line width=2.0pt, forget plot]
  table[row sep=crcr]{%
1	99514.3855136335\\
2	99512.7730761402\\
3	99511.0192276974\\
4	99509.1475744595\\
5	99507.1581711548\\
6	99505.0347677528\\
7	99502.7530011467\\
8	99500.2851476226\\
9	99497.6022066046\\
10	99494.6744795985\\
11	99491.4714279749\\
12	99487.9612375715\\
13	99484.1102949229\\
14	99479.8826594836\\
15	99475.2395574488\\
16	99470.1388969544\\
17	99464.5347947996\\
18	99458.377102492\\
19	99451.6109199752\\
20	99444.1760868441\\
21	99436.0066423135\\
22	99427.0302464047\\
23	99417.1675556773\\
24	99406.3315474109\\
25	99394.4267864822\\
26	99381.3486293714\\
27	99366.9823597983\\
28	99351.2022505128\\
29	99333.87054575\\
30	99314.8363588744\\
31	99293.9344797919\\
32	99270.9840868399\\
33	99245.7873581139\\
34	99218.1279775855\\
35	99187.769531951\\
36	99154.45379498\\
37	99117.8988972545\\
38	99077.7973806708\\
39	99033.8141389948\\
40	98985.584248193\\
41	98932.7106933104\\
42	98874.7620024355\\
43	98811.2698028964\\
44	98741.7263204084\\
45	98665.5818485642\\
46	98582.2422239751\\
47	98491.0663516659\\
48	98391.3638361262\\
49	98282.3927858354\\
50	98163.3578731762\\
51	98033.4087474469\\
52	97891.6389161039\\
53	97737.0852282407\\
54	97568.7281142971\\
55	97385.4927565749\\
56	97186.2513855491\\
57	96969.8269161364\\
58	96734.9981546325\\
59	96480.5068191435\\
60	96205.0666218251\\
61	95907.37465747\\
62	95586.1253269382\\
63	95240.0269922999\\
64	94867.821509928\\
65	94468.3067148293\\
66	94040.361831431\\
67	93582.9756610013\\
68	93095.2772435565\\
69	92576.5685143771\\
70	92026.3582768119\\
71	91444.3966020133\\
72	90830.7085545253\\
73	90185.6259459633\\
74	89509.8156564854\\
75	88804.3029567212\\
76	88070.4882331647\\
77	87310.1555876556\\
78	86525.4719614664\\
79	85718.9757336789\\
80	84893.5541581798\\
81	84052.4095172455\\
82	83199.0144527066\\
83	82337.0575467732\\
84	81470.3808141673\\
85	80602.911282632\\
86	79738.5892308409\\
87	78881.2958817345\\
88	78034.7833912855\\
89	77202.6098225311\\
90	76388.08146718\\
91	75594.2044047714\\
92	74823.6466181272\\
93	74078.7113665752\\
94	73361.3219083111\\
95	72673.0171074795\\
96	72014.9569966253\\
97	71387.9370143169\\
98	70792.4094103443\\
99	70228.5102038581\\
100	69696.0900801508\\
101	69194.7476999504\\
102	68723.8640483123\\
103	68282.6366453566\\
104	67870.1126570492\\
105	67485.2201631839\\
106	67126.7970479946\\
107	66793.6171668995\\
108	66484.4136050712\\
109	66197.8989773616\\
110	65932.7828246161\\
111	65687.7862403714\\
112	65461.653917239\\
113	65253.163837408\\
114	65061.1348502874\\
115	64884.4323858595\\
116	64721.9725480235\\
117	64572.7248208579\\
118	64435.7136046443\\
119	64310.0187795654\\
120	64194.7754746814\\
121	64089.1731992117\\
122	63992.4544731008\\
123	63903.9130748787\\
124	63822.8920072785\\
125	63748.7812651407\\
126	63681.0154758845\\
127	63619.0714702606\\
128	63562.4658301406\\
129	63510.7524506471\\
130	63463.520145858\\
131	63420.3903204922\\
132	63381.0147242644\\
133	63345.073300849\\
134	63312.2721394885\\
135	63282.3415341084\\
136	63255.0341522482\\
137	63230.1233140901\\
138	63207.401380285\\
139	63186.6782460643\\
140	63167.779938215\\
141	63150.5473108462\\
};
\end{axis}
\end{tikzpicture}%

%% file: Figures/OL_state_2.tex
%
%
\definecolor{mycolor1}{rgb}{0.00000,0.44700,0.74100}%
\definecolor{mycolor2}{rgb}{0.85000,0.32500,0.09800}%
\definecolor{mycolor3}{rgb}{0.92900,0.69400,0.12500}%
\definecolor{mycolor4}{rgb}{0.49400,0.18400,0.55600}%
\definecolor{mycolor5}{rgb}{0.46600,0.67400,0.18800}%
\definecolor{mycolor6}{rgb}{0.30100,0.74500,0.93300}%
\begin{tikzpicture}

\begin{axis}[%
width=7cm,
height=3.5cm,
at={(0cm,0cm)},
scale only axis,
xmin=0,
xmax=140,
ymin=0,
ymax=7000,
xmajorgrids,
ymajorgrids,
axis background/.style={fill=white},
axis x line*=bottom,
axis y line*=left
]
\addplot [color=mycolor1, line width=2.0pt, forget plot]
  table[row sep=crcr]{%
1	4.6595088243\\
2	7.25704562732303\\
3	9.23057250575148\\
4	10.8871132834846\\
5	12.410331803941\\
6	13.910433025333\\
7	15.4559850395323\\
8	17.0925338455949\\
9	18.8531833362379\\
10	20.7645799987799\\
11	22.8503353158311\\
12	25.1330330798772\\
13	27.6354565648139\\
14	30.3813851687548\\
15	33.3961538701621\\
16	36.7070836882444\\
17	40.3438448207566\\
18	44.3387885132268\\
19	48.7272694325119\\
20	53.5479722374286\\
21	58.843251390541\\
22	64.6594905314611\\
23	71.0474861021338\\
24	78.0628589110542\\
25	85.7664966763911\\
26	94.2250301363437\\
27	103.511344957687\\
28	113.705131342684\\
29	124.893472880895\\
30	137.171475775716\\
31	150.642939059574\\
32	165.421065760796\\
33	181.629214161851\\
34	199.401687251714\\
35	218.884557178936\\
36	240.236519904921\\
37	263.629773281045\\
38	289.250909363816\\
39	317.301808867322\\
40	348.00052215358\\
41	381.582116994941\\
42	418.299468420906\\
43	458.423960196214\\
44	502.246060783532\\
45	550.075728948676\\
46	602.242595414716\\
47	659.095857141163\\
48	721.003809921073\\
49	788.352933145968\\
50	861.546427973644\\
51	941.002097059782\\
52	1027.14944095545\\
53	1120.42583390773\\
54	1221.27163105874\\
55	1330.12405114667\\
56	1447.40967533991\\
57	1573.53540573172\\
58	1708.87773862994\\
59	1853.77023081836\\
60	2008.48907448498\\
61	2173.23675173455\\
62	2348.12381572611\\
63	2533.14894533317\\
64	2728.17754582503\\
65	2932.91932000419\\
66	3146.9054109523\\
67	3369.46591452163\\
68	3599.70876865432\\
69	3836.50123469055\\
70	4078.45537524345\\
71	4323.91908125341\\
72	4570.97428062728\\
73	4817.44394317343\\
74	5060.90935261002\\
75	5298.73882200518\\
76	5528.12856928234\\
77	5746.15584429022\\
78	5949.84362784707\\
79	6136.23534839004\\
80	6302.4771489711\\
81	6445.90437191505\\
82	6564.1282084128\\
83	6655.11798420671\\
84	6717.27440500506\\
85	6749.4893221301\\
86	6751.18821446884\\
87	6722.35258264643\\
88	6663.52073269331\\
89	6575.76686691621\\
90	6460.65985329842\\
91	6320.20436289426\\
92	6156.76811741331\\
93	5972.99968333269\\
94	5771.74153925748\\
95	5555.9430354567\\
96	5328.57740928992\\
97	5092.56630149847\\
98	4850.7143375873\\
99	4605.65539915707\\
100	4359.811304032\\
101	4115.36281282898\\
102	3874.23222987233\\
103	3638.07638995283\\
104	3408.28851979293\\
105	3186.00731844618\\
106	2972.1315881291\\
107	2767.33883521813\\
108	2572.10641901975\\
109	2386.73402490882\\
110	2211.36645466967\\
111	2046.01594196089\\
112	1890.58340177086\\
113	1744.87820138925\\
114	1608.63619263103\\
115	1481.53586969887\\
116	1363.21261518751\\
117	1253.27107071753\\
118	1151.29572167717\\
119	1056.85982096629\\
120	969.532797844545\\
121	888.886308098513\\
122	814.499083503906\\
123	745.960734294488\\
124	682.874649968929\\
125	624.860132789181\\
126	571.553885922452\\
127	522.610965230844\\
128	477.705290851015\\
129	436.529802366683\\
130	398.796329841097\\
131	364.235242408667\\
132	332.594926601396\\
133	303.641138120411\\
134	277.156263326247\\
135	252.938520255725\\
136	230.801123403524\\
137	210.571431749239\\
138	192.090095479778\\
139	175.210213467445\\
140	159.79651073541\\
141	145.724542799472\\
};
\addplot [color=mycolor2, line width=2.0pt, forget plot]
  table[row sep=crcr]{%
1	4.3296088874\\
2	3.96022390797655\\
3	3.89949949462591\\
4	4.03260429886922\\
5	4.29150042493362\\
6	4.63848842450323\\
7	5.05429727656978\\
8	5.53035563963166\\
9	6.06421109384361\\
10	6.6569633628338\\
11	7.31187817389316\\
12	8.03366176631637\\
13	8.82809666452643\\
14	9.70187368772757\\
15	10.6625314356469\\
16	11.7184561807865\\
17	12.8789174266314\\
18	14.1541262340792\\
19	15.5553096856015\\
20	17.0947981864076\\
21	18.7861240879687\\
22	20.644131091658\\
23	22.6850944297544\\
24	24.9268521261098\\
25	27.388947812695\\
26	30.092785674487\\
27	33.0617981407625\\
28	36.3216269485436\\
29	39.9003181783317\\
30	43.8285318031401\\
31	48.1397661957305\\
32	52.8705978999381\\
33	58.060936781986\\
34	63.7542964266499\\
35	69.998079318849\\
36	76.8438759393414\\
37	84.3477763869278\\
38	92.5706924996921\\
39	101.578687662451\\
40	111.44331053226\\
41	122.24192776142\\
42	134.058049418619\\
43	146.981639172394\\
44	161.109399375062\\
45	176.545018937896\\
46	193.39936929064\\
47	211.790630746314\\
48	231.844328230881\\
49	253.6932515857\\
50	277.477231528359\\
51	303.342737912812\\
52	331.442262249623\\
53	361.933441669627\\
54	394.977876842139\\
55	430.739592075495\\
56	469.383082313906\\
57	511.070889494892\\
58	555.960650368529\\
59	604.201560163545\\
60	655.930202314117\\
61	711.265704857928\\
62	770.304200194665\\
63	833.112587799084\\
64	899.721630292754\\
65	970.118452869533\\
66	1044.23856493084\\
67	1121.95758079115\\
68	1203.08288245224\\
69	1287.34553953937\\
70	1374.39287592886\\
71	1463.78214411911\\
72	1554.97583003811\\
73	1647.33915419669\\
74	1740.14035019998\\
75	1832.55427856162\\
76	1923.66986326378\\
77	2012.50171359681\\
78	2098.0061115306\\
79	2179.10130795587\\
80	2254.69178940993\\
81	2323.6958679299\\
82	2385.07563535984\\
83	2437.86804028699\\
84	2481.21562382258\\
85	2514.3953209961\\
86	2536.84372242717\\
87	2548.17730997456\\
88	2548.20642962017\\
89	2536.94212851083\\
90	2514.59542973069\\
91	2481.56910578991\\
92	2438.44249224883\\
93	2385.95030967114\\
94	2324.95679597097\\
95	2256.42666538724\\
96	2181.39449322718\\
97	2100.93408111277\\
98	2016.12920308106\\
99	1928.04689545878\\
100	1837.71416479428\\
101	1746.09868042778\\
102	1654.09372011292\\
103	1562.50737090349\\
104	1472.05576829423\\
105	1383.35999190011\\
106	1296.94612676214\\
107	1213.24794158402\\
108	1132.61162133995\\
109	1055.3020124729\\
110	981.50988463161\\
111	911.359774466403\\
112	844.918046555824\\
113	782.200877766332\\
114	723.181939590971\\
115	667.799615130218\\
116	615.963641540294\\
117	567.561114200008\\
118	522.461825535642\\
119	480.522939928244\\
120	441.593027271302\\
121	405.515492576585\\
122	372.13144861386\\
123	341.282083949941\\
124	312.810580867449\\
125	286.563637319025\\
126	262.392645005747\\
127	240.154572429646\\
128	219.71259781289\\
129	200.936532450122\\
130	183.703070626428\\
131	167.895897879204\\
132	153.40568523604\\
133	140.129993204537\\
134	127.973105770126\\
135	116.845811494264\\
136	106.665145998786\\
137	97.354107660514\\
138	88.841356202673\\
139	81.0609020304304\\
140	73.9517925887642\\
141	67.4578006927778\\
};
\addplot [color=mycolor3, line width=2.0pt, forget plot]
  table[row sep=crcr]{%
1	4.8417769521\\
2	3.10050284714218\\
3	2.3730137105326\\
4	2.10729677210785\\
5	2.06414456909919\\
6	2.13717415725919\\
7	2.27792470886025\\
8	2.46408076791808\\
9	2.68559574351201\\
10	2.93843653988166\\
11	3.22167504867669\\
12	3.5360982541545\\
13	3.88353019648157\\
14	4.26649846715358\\
15	4.68807291507581\\
16	5.15179346169021\\
17	5.66164593878754\\
18	6.22206518560832\\
19	6.83795472076256\\
20	7.51471742356909\\
21	8.25829432131283\\
22	9.07520999675987\\
23	9.97262390449867\\
24	10.9583873181901\\
25	12.0411058790925\\
26	13.2302078603086\\
27	14.5360183437703\\
28	15.9698395505314\\
29	17.544037581053\\
30	19.2721358162863\\
31	21.1689152042642\\
32	23.2505216100989\\
33	25.5345803375343\\
34	28.0403178340428\\
35	30.7886904642024\\
36	33.8025200719628\\
37	37.1066358444767\\
38	40.7280217303485\\
39	44.6959683441148\\
40	49.0422278959669\\
41	53.8011702093673\\
42	59.0099373163549\\
43	64.7085934370322\\
44	70.9402663413825\\
45	77.7512751433425\\
46	85.1912384747057\\
47	93.3131557172517\\
48	102.17345252588\\
49	111.831980248838\\
50	122.351957046345\\
51	133.799836539859\\
52	146.245087719713\\
53	159.759867647665\\
54	174.41856628775\\
55	190.297200690996\\
56	207.472633894585\\
57	226.021592468926\\
58	246.019455907087\\
59	267.538791310761\\
60	290.647608459805\\
61	315.407313793891\\
62	341.870347572331\\
63	370.077497030387\\
64	400.054890233951\\
65	431.810691011836\\
66	465.331535154352\\
67	500.578772129522\\
68	537.484604677328\\
69	575.94825013766\\
70	615.832281007736\\
71	656.959336076356\\
72	699.109424849225\\
73	742.018073407167\\
74	785.3755752271\\
75	828.827611353285\\
76	871.977486172342\\
77	914.390184071966\\
78	955.598385932377\\
79	995.110492297049\\
80	1032.4205845972\\
81	1067.02012268008\\
82	1098.41103526604\\
83	1126.11972201572\\
84	1149.71136583636\\
85	1168.80386655893\\
86	1183.08066530977\\
87	1192.30174228982\\
88	1196.31214329742\\
89	1195.04751956353\\
90	1188.5363417935\\
91	1176.89865718581\\
92	1160.34147795279\\
93	1139.15110030764\\
94	1113.68283417053\\
95	1084.34876011129\\
96	1051.60421128797\\
97	1015.93370111104\\
98	977.836985390661\\
99	937.81586957381\\
100	896.362259796939\\
101	853.947824971048\\
102	811.015499818628\\
103	767.972927733909\\
104	725.187826770512\\
105	682.985167965838\\
106	641.645985405443\\
107	601.407591986832\\
108	562.464951739399\\
109	524.972955398699\\
110	489.049356625029\\
111	454.778147571908\\
112	422.213180507799\\
113	391.381873481982\\
114	362.28886989146\\
115	334.91955224436\\
116	309.243338078292\\
117	285.216710095962\\
118	262.785952793254\\
119	241.889584182878\\
120	222.460483897349\\
121	204.427728369907\\
122	187.718150394825\\
123	172.257644633631\\
124	157.97224301895\\
125	144.788984931992\\
126	132.636606860455\\
127	121.446075291852\\
128	111.15098511829\\
129	101.687844025797\\
130	92.9962613716412\\
131	85.0190580344331\\
132	77.7023117388202\\
133	70.9953504676362\\
134	64.8507048166296\\
135	59.2240285416405\\
136	54.0739951045427\\
137	49.3621767430298\\
138	45.0529114650049\\
139	41.1131623916495\\
140	37.5123730325605\\
141	34.2223213588535\\
};
\addplot [color=mycolor4, line width=2.0pt, forget plot]
  table[row sep=crcr]{%
1	0.1709101349\\
2	1.87582768989844\\
3	2.36359200427599\\
4	2.56591118709034\\
5	2.72826638743079\\
6	2.91594742008822\\
7	3.14373187537152\\
8	3.41308783615855\\
9	3.72309784759522\\
10	4.07341648887533\\
11	4.46484463839883\\
12	4.89926678179909\\
13	5.37947791961844\\
14	5.90903953918148\\
15	6.49218769372962\\
16	7.13378608028685\\
17	7.83931259621999\\
18	8.61486992904243\\
19	9.46721362184796\\
20	10.4037933786278\\
21	11.4328049944354\\
22	12.5632513683921\\
23	13.8050117538086\\
24	15.1689188483258\\
25	16.6668436144795\\
26	18.3117879025588\\
27	20.1179850571237\\
28	22.101008746521\\
29	24.2778902728701\\
30	26.6672446041426\\
31	29.2894053220658\\
32	32.1665685986642\\
33	35.3229461971487\\
34	38.7849273344956\\
35	42.5812490366565\\
36	46.7431743545237\\
37	51.30467747951\\
38	56.302634390243\\
39	61.7770171631173\\
40	67.7710894743917\\
41	74.3316000937311\\
42	81.5089703008872\\
43	89.357470129972\\
44	97.9353771406443\\
45	107.305110014271\\
46	117.533327659531\\
47	128.690982673636\\
48	140.853315936287\\
49	154.09977681734\\
50	168.513850973232\\
51	184.18277502818\\
52	201.197114645812\\
53	219.650179690573\\
54	239.637247492694\\
55	261.254562854263\\
56	284.598081616699\\
57	309.761923672349\\
58	336.836501644304\\
59	365.906293560417\\
60	397.047232271967\\
61	430.323691744562\\
62	465.785061351877\\
63	503.461914604186\\
64	543.361798949777\\
65	585.464698845974\\
66	629.718255377511\\
67	676.032862048584\\
68	724.276797146527\\
69	774.271596665068\\
70	825.787915667955\\
71	878.542166644303\\
72	932.194256318718\\
73	986.346762125286\\
74	1040.54589018987\\
75	1094.28453226037\\
76	1147.00768444625\\
77	1198.12040250338\\
78	1246.99834616855\\
79	1293.00081194041\\
80	1335.48597743798\\
81	1373.8278934288\\
82	1407.43457831412\\
83	1435.76641354831\\
84	1458.35392698588\\
85	1474.81400214385\\
86	1484.86357738699\\
87	1488.33000497766\\
88	1485.157421387\\
89	1475.40872331483\\
90	1459.26302638468\\
91	1437.00877791282\\
92	1409.03297228774\\
93	1375.80715056262\\
94	1337.87103414171\\
95	1295.81473380331\\
96	1250.26048722562\\
97	1201.84481720557\\
98	1151.20188263014\\
99	1098.94863324104\\
100	1045.67219718057\\
101	991.919745988871\\
102	938.190910805559\\
103	884.93267737903\\
104	832.536572805281\\
105	781.337876054762\\
106	731.616536080174\\
107	683.599461850508\\
108	637.463852728643\\
109	593.341259308007\\
110	551.322098381984\\
111	511.460385994282\\
112	473.778495266431\\
113	438.271787679255\\
114	404.913005430391\\
115	373.656346991667\\
116	344.441177345515\\
117	317.195348424114\\
118	291.838124224686\\
119	268.28271938982\\
120	246.438470317087\\
121	226.212664743024\\
122	207.512059870433\\
123	190.244121066972\\
124	174.318013482351\\
125	159.64537805923\\
126	146.140921720021\\
127	133.722849295405\\
128	122.313162252171\\
129	111.837846652729\\
130	102.226970163692\\
131	93.414705415059\\
132	85.3392946529991\\
133	77.9429684627352\\
134	71.1718293805746\\
135	64.9757094699495\\
136	59.3080094005608\\
137	54.1255252315881\\
138	49.388267944939\\
139	45.0592797861994\\
140	41.1044506321143\\
141	37.4923368970139\\
};
\addplot [color=mycolor5, line width=2.0pt, forget plot]
  table[row sep=crcr]{%
1	1.4936938584\\
2	1.77720424756331\\
3	1.97208819133525\\
4	2.09210161269225\\
5	2.20766512192823\\
6	2.34838321957669\\
7	2.52242784280443\\
8	2.73067517455431\\
9	2.97243980366291\\
10	3.24737063818024\\
11	3.5559288165295\\
12	3.89940842664589\\
13	4.27984283233873\\
14	4.699907646182\\
15	5.16284983107652\\
16	5.67244583062366\\
17	6.23298434241356\\
18	6.84926842428813\\
19	7.52663251013344\\
20	8.27097108868828\\
21	9.08877681832409\\
22	9.98718662761832\\
23	10.9740349108023\\
24	12.0579133172633\\
25	13.2482368999631\\
26	14.555316563485\\
27	15.9904378630369\\
28	17.5659462676596\\
29	19.2953390245193\\
30	21.1933637524246\\
31	23.2761238541208\\
32	25.5611907684891\\
33	28.0677229834616\\
34	30.8165915945138\\
35	33.8305120167778\\
36	37.134181234456\\
37	40.7544196913252\\
38	44.720316581341\\
39	49.0633768780799\\
40	53.8176679341207\\
41	59.0199628735419\\
42	64.7098772787124\\
43	70.9299948222382\\
44	77.7259765022092\\
45	85.1466469906876\\
46	93.244050290844\\
47	102.073465410331\\
48	111.693371096545\\
49	122.165346851506\\
50	133.553895470843\\
51	145.926170270688\\
52	159.351588038229\\
53	173.901306655308\\
54	189.647544424266\\
55	206.662716539351\\
56	225.018363113808\\
57	244.783842968599\\
58	266.024768351151\\
59	288.801158282396\\
60	313.165292785111\\
61	339.159257326972\\
62	366.8121769348\\
63	396.13715309379\\
64	427.127934142099\\
65	459.755371643776\\
66	493.963741136824\\
67	529.667035281161\\
68	566.745369823296\\
69	605.041676352497\\
70	644.358888203262\\
71	684.457853931489\\
72	725.056232695867\\
73	765.828633203289\\
74	806.408248019053\\
75	846.390203713443\\
76	885.336791238941\\
77	922.784658631823\\
78	958.253940733758\\
79	991.259172516618\\
80	1021.32169174582\\
81	1047.98309455359\\
82	1070.81917809918\\
83	1089.45370319199\\
84	1103.57125113924\\
85	1112.92844475074\\
86	1117.36285983674\\
87	1116.799070336\\
88	1111.2514396225\\
89	1100.82347782729\\
90	1085.70381022726\\
91	1066.15902240852\\
92	1042.52384214958\\
93	1015.18926760693\\
94	984.589344293385\\
95	951.187324717253\\
96	915.461916968678\\
97	877.89425083463\\
98	838.956075385334\\
99	799.099565660082\\
100	758.748973197213\\
101	718.294218864756\\
102	678.086406798392\\
103	638.435141662229\\
104	599.607460796671\\
105	561.828147964443\\
106	525.281173942827\\
107	490.11200729187\\
108	456.430551737682\\
109	424.314490179688\\
110	393.812845186767\\
111	364.949598474862\\
112	337.727244517603\\
113	312.130184167313\\
114	288.127891692932\\
115	265.677812295442\\
116	244.727966716218\\
117	225.219255112153\\
118	207.087464251859\\
119	190.264990731274\\
120	174.682298809924\\
121	160.269135129223\\
122	146.955524460079\\
123	134.67257115799\\
124	123.353090540313\\
125	112.932093242523\\
126	103.347144000229\\
127	94.5386144309741\\
128	86.4498473987163\\
129	79.0272485400472\\
130	72.2203185891839\\
131	65.9816383082595\\
132	60.2668161407425\\
133	55.0344071747858\\
134	50.2458106351049\\
135	45.8651519142828\\
136	41.8591540997929\\
137	38.1970030409501\\
138	34.8502092180919\\
139	31.7924690114731\\
140	28.9995274065096\\
141	26.4490437024923\\
};
\addplot [color=mycolor6, line width=2.0pt, forget plot]
  table[row sep=crcr]{%
1	1.6144863665\\
2	1.91488375921671\\
3	2.11256899616501\\
4	2.26740672199893\\
5	2.41416297476468\\
6	2.57565533323225\\
7	2.76427135446697\\
8	2.98569224151827\\
9	3.24225913212804\\
10	3.53510851648303\\
11	3.8652934316825\\
12	4.23428667919582\\
13	4.64416346595059\\
14	5.09763962630998\\
15	5.59805718477238\\
16	6.14936027127337\\
17	6.75607908726744\\
18	7.42332758020236\\
19	8.15681532377509\\
20	8.96287225740088\\
21	9.84848456635727\\
22	10.8213401860683\\
23	11.8898827804771\\
24	13.0633734052453\\
25	14.3519593623189\\
26	15.7667499714428\\
27	17.3198991324327\\
28	19.0246946398069\\
29	20.8956542482165\\
30	22.9486284797096\\
31	25.200910115944\\
32	27.6713502308905\\
33	30.3804804907944\\
34	33.3506412744961\\
35	36.6061149431006\\
36	40.1732633061732\\
37	44.0806679833478\\
38	48.3592719352008\\
39	53.0425199239504\\
40	58.1664950503123\\
41	63.7700477842033\\
42	69.894913050048\\
43	76.5858099285775\\
44	83.8905173837548\\
45	91.8599181057806\\
46	100.548001073197\\
47	110.011811779509\\
48	120.311337252525\\
49	131.509311040888\\
50	143.670921292971\\
51	156.863402972371\\
52	171.155493235427\\
53	186.61672716962\\
54	203.316549632068\\
55	221.323218060548\\
56	240.70247113991\\
57	261.515939440555\\
58	283.819277011407\\
59	307.659997871502\\
60	333.075008905866\\
61	360.087841347813\\
62	388.705597302223\\
63	418.915646016088\\
64	450.682127032193\\
65	483.942343876723\\
66	518.603162022364\\
67	554.537557475773\\
68	591.581495727301\\
69	629.53135246724\\
70	668.142114113514\\
71	707.126613779833\\
72	746.156062295334\\
73	784.862119589545\\
74	822.840714939756\\
75	859.657762171156\\
76	894.856826836771\\
77	927.968688443754\\
78	958.522607166598\\
79	986.05896018221\\
80	1010.142770251\\
81	1030.37752347546\\
82	1046.41858023897\\
83	1057.98543785473\\
84	1064.87211645452\\
85	1066.95501622996\\
86	1064.19773195145\\
87	1056.65249945785\\
88	1044.4581712037\\
89	1027.83485174446\\
90	1007.07554521378\\
91	982.535353067976\\
92	954.618894177974\\
93	923.766690427398\\
94	890.441267150057\\
95	855.113664567315\\
96	818.250955566307\\
97	780.305232199777\\
98	741.704374920017\\
99	702.844770413368\\
100	664.086008769538\\
101	625.747477574374\\
102	588.106684274128\\
103	551.399080053655\\
104	515.819126798801\\
105	481.522339772859\\
106	448.628047647063\\
107	417.222633412126\\
108	387.36304968197\\
109	359.080435870584\\
110	332.383699418216\\
111	307.262956289328\\
112	283.692755804974\\
113	261.635040640496\\
114	241.041814190434\\
115	221.857504542621\\
116	204.021027337512\\
117	187.467559291344\\
118	172.130040673269\\
119	157.940429094352\\
120	144.830729105101\\
121	132.73382276794\\
122	121.584125966053\\
123	111.318094057063\\
124	101.87459884128\\
125	93.1951968942721\\
126	85.2243072664142\\
127	77.9093144907061\\
128	71.2006108429987\\
129	65.0515899180457\\
130	59.4186018518448\\
131	54.2608789516938\\
132	49.540439095191\\
133	45.2219730251795\\
134	41.2727205913337\\
135	37.6623400594213\\
136	34.3627738132486\\
137	31.3481130981984\\
138	28.5944638854568\\
139	26.0798154593984\\
140	23.7839129349137\\
141	21.6881345855107\\
};
\end{axis}
\end{tikzpicture}%

%% file: Figures/OL_state_3.tex
%
%
\definecolor{mycolor1}{rgb}{0.00000,0.44700,0.74100}%
\definecolor{mycolor2}{rgb}{0.85000,0.32500,0.09800}%
\definecolor{mycolor3}{rgb}{0.92900,0.69400,0.12500}%
\definecolor{mycolor4}{rgb}{0.49400,0.18400,0.55600}%
\definecolor{mycolor5}{rgb}{0.46600,0.67400,0.18800}%
\definecolor{mycolor6}{rgb}{0.30100,0.74500,0.93300}%
\begin{tikzpicture}

\begin{axis}[%
width=7cm,
height=3.5cm,
at={(0cm,0cm)},
scale only axis,
xmin=0,
xmax=140,
ymin=0,
ymax=300000,
xmajorgrids,
ymajorgrids,
axis background/.style={fill=white},
xmajorgrids,
ymajorgrids,
legend pos=north west,
legend style={legend cell align=left, align=left, draw=white!15!black}
]

\addplot [color=mycolor1, line width=2.0pt, mark options={solid, mycolor1}]
  table[row sep=crcr]{%
1	0\\
2	4.29463568177537\\
3	10.9834023030202\\
4	19.4911544162207\\
5	29.5257282182798\\
6	40.9642415461055\\
7	53.7853873518582\\
8	68.0310573035046\\
9	83.7851224878908\\
10	101.161965611055\\
11	120.300529254238\\
12	141.361518531931\\
13	164.5264538773\\
14	189.997854902757\\
15	218.000158520477\\
16	248.781152709582\\
17	282.613807035573\\
18	319.798437871262\\
19	360.665179499164\\
20	405.576752343279\\
21	454.93153219867\\
22	509.166932663109\\
23	568.763118800506\\
24	634.247074389002\\
25	706.19704850482\\
26	785.247409994935\\
27	872.093940777221\\
28	967.499600963008\\
29	1072.30080054838\\
30	1187.41421384591\\
31	1313.84417387\\
32	1452.69068445472\\
33	1605.15808784897\\
34	1772.56442474104\\
35	1956.35152191561\\
36	2158.09583980334\\
37	2379.52010775807\\
38	2622.50576865146\\
39	2889.10624590808\\
40	3181.56103495083\\
41	3502.31060664882\\
42	3854.01209214178\\
43	4239.55569566054\\
44	4662.08175389484\\
45	5124.99832622178\\
46	5631.99915877767\\
47	6187.08181595916\\
48	6794.5657144823\\
49	7459.10972663968\\
50	8185.72893999198\\
51	8979.81006969937\\
52	9847.12491661685\\
53	10793.84114916\\
54	11826.5295604339\\
55	12952.1668157025\\
56	14178.1325615878\\
57	15512.1996214985\\
58	16962.5158575682\\
59	18537.5761458625\\
60	20246.1827993291\\
61	22097.3926952638\\
62	24100.449337258\\
63	26264.6981249527\\
64	28599.483240316\\
65	31114.0248103233\\
66	33817.2753971168\\
67	36717.7554207979\\
68	39823.367855646\\
69	43141.1934687748\\
70	46677.2689902441\\
71	50436.3518982329\\
72	54421.6769339141\\
73	58634.7109652443\\
74	63074.914307152\\
75	67739.517961211\\
76	72623.327322128\\
77	77718.5635588892\\
78	83014.7539627432\\
79	88498.6819278275\\
80	94154.4057975911\\
81	99963.353536114\\
82	105904.497111713\\
83	111954.606744853\\
84	118088.580998179\\
85	124279.844376319\\
86	130500.800011259\\
87	136723.321503017\\
88	142919.265400868\\
89	149060.984406878\\
90	155121.82130763\\
91	161076.564904013\\
92	166901.85168777\\
93	172576.500462673\\
94	178081.771197111\\
95	183401.543751507\\
96	188522.41638116\\
97	193433.727752799\\
98	198127.509388358\\
99	202598.377812871\\
100	206843.377181051\\
101	210861.783819634\\
102	214654.884046691\\
103	218225.735954326\\
104	221578.924727321\\
105	224720.319677482\\
106	227656.839647321\\
107	230396.23189887\\
108	232946.868146903\\
109	235317.560084825\\
110	237517.39562389\\
111	239555.596138111\\
112	241441.394277177\\
113	243183.931364836\\
114	244792.173020042\\
115	246274.841398271\\
116	247640.362325425\\
117	248896.825562172\\
118	250051.956470213\\
119	251113.097434429\\
120	252087.197509975\\
121	252980.808898059\\
122	253800.088998119\\
123	254550.806929706\\
124	255238.353559084\\
125	255867.754199478\\
126	256443.683277753\\
127	256970.480372704\\
128	257452.167130601\\
129	257892.464652254\\
130	258294.811023105\\
131	258662.378724439\\
132	258998.091720721\\
133	259304.642066095\\
134	259584.505913423\\
135	259839.958842646\\
136	260073.090452725\\
137	260285.818183766\\
138	260479.900353895\\
139	260656.94840966\\
140	260818.438399902\\
141	260965.72169149\\
};
\addlegendentry{0-24}

\addplot [color=mycolor2, line width=2.0pt, mark options={solid, mycolor2}]
  table[row sep=crcr]{%
1	0\\
2	3.13035311771256\\
3	5.99363697996894\\
4	8.81301644771832\\
5	11.7286320997933\\
6	14.8314323952082\\
7	18.1851086922459\\
8	21.8394191967353\\
9	25.8379249437466\\
10	30.2224138428075\\
11	35.0354689152852\\
12	40.3220343379938\\
13	46.1304569486418\\
14	52.5132644113878\\
15	59.5278219235952\\
16	67.2369451701362\\
17	75.7095131997929\\
18	85.0211070106206\\
19	95.2546903059363\\
20	106.501344088687\\
21	118.861064375482\\
22	132.443631216484\\
23	147.369556815285\\
24	163.771120540924\\
25	181.793498842764\\
26	201.595998423233\\
27	223.353401437379\\
28	247.257431934985\\
29	273.518353213499\\
30	302.366706183846\\
31	334.055199242422\\
32	368.860760464201\\
33	407.086763153045\\
34	449.065435869127\\
35	495.160467955635\\
36	545.769821254805\\
37	601.328758073296\\
38	662.313094453735\\
39	729.242686343325\\
40	802.685154216735\\
41	883.259848986073\\
42	971.642058473118\\
43	1068.56744916448\\
44	1174.8367322327\\
45	1291.32053567606\\
46	1418.96445567475\\
47	1558.79424962783\\
48	1711.92112055339\\
49	1879.54702732146\\
50	2062.96993726491\\
51	2263.58891680757\\
52	2482.90893163022\\
53	2722.5452003921\\
54	2984.22691506895\\
55	3269.80010663332\\
56	3581.22939737128\\
57	3920.59834115716\\
58	4290.10801140896\\
59	4692.07345458592\\
60	5128.91758687902\\
61	5603.1620757499\\
62	6117.41471949415\\
63	6674.35282115129\\
64	7276.70205279021\\
65	7927.21032818127\\
66	8628.61625247349\\
67	9383.61180342958\\
68	10194.7990266432\\
69	11064.6407028534\\
70	11995.4051732882\\
71	12989.1057905995\\
72	14047.4357963028\\
73	15171.6998035422\\
74	16362.7434731625\\
75	17620.8833911487\\
76	18945.8395588911\\
77	20336.6732601619\\
78	21791.7333308056\\
79	23308.6139874678\\
80	24884.1273307221\\
81	26514.2933932965\\
82	28194.3501360566\\
83	29918.7851012695\\
84	31681.389534823\\
85	33475.3347307399\\
86	35293.2691994047\\
87	37127.4341002483\\
88	38969.7933050201\\
89	40812.1735636043\\
90	42646.4096130895\\
91	44464.4887624905\\
92	46258.6895296165\\
93	48021.7092980274\\
94	49746.7766620386\\
95	51427.7450691376\\
96	53059.1654654327\\
97	54636.336805945\\
98	56155.3344156504\\
99	57613.0171996411\\
100	59007.0155415837\\
101	60335.702361765\\
102	61598.1502156619\\
103	62794.0775080353\\
104	63923.7868991517\\
105	64988.0988228307\\
106	65988.2827600371\\
107	66925.9885567962\\
108	67803.1796785044\\
109	68622.0698859633\\
110	69385.0644267604\\
111	70094.7064769613\\
112	70753.6292539495\\
113	71364.5139574027\\
114	71930.0534830442\\
115	72452.9216908078\\
116	72935.7478909557\\
117	73381.0961327576\\
118	73791.448834245\\
119	74169.1942719936\\
120	74516.6174509127\\
121	74835.8938903393\\
122	75129.0858897741\\
123	75398.1408715665\\
124	75644.8914357158\\
125	75871.05680135\\
126	76078.2453485916\\
127	76267.9580121879\\
128	76441.5923135969\\
129	76600.4468506813\\
130	76745.7260934896\\
131	76878.5453607315\\
132	76999.9358745276\\
133	77110.8498109921\\
134	77212.1652814009\\
135	77304.6911933364\\
136	77389.1719535656\\
137	77466.2919847306\\
138	77536.6800364838\\
139	77600.9132787011\\
140	77659.5211700823\\
141	77712.9890999834\\
};
\addlegendentry{25-44}

\addplot [color=mycolor3, line width=2.0pt, mark options={solid, mycolor3}]
  table[row sep=crcr]{%
1	0\\
2	2.76332081289318\\
3	4.53285380309558\\
4	5.88719090711098\\
5	7.08987683975862\\
6	8.26793475218234\\
7	9.48767244104269\\
8	10.7877399204971\\
9	12.1940512268626\\
10	13.7267865607043\\
11	15.4038243359572\\
12	17.2425132723664\\
13	19.2606512408868\\
14	21.4770771369171\\
15	23.9120724142912\\
16	26.5876705648774\\
17	29.5279254004995\\
18	32.7591655449052\\
19	36.3102506933262\\
20	40.2128390992851\\
21	44.5016725719945\\
22	49.2148836105955\\
23	54.3943284535773\\
24	60.0859494156922\\
25	66.3401697260611\\
26	73.2123240642561\\
27	80.7631280564633\\
28	89.0591901062684\\
29	98.1735690719099\\
30	108.186381448341\\
31	119.185461855576\\
32	131.267080763061\\
33	144.536723481308\\
34	159.109934513798\\
35	175.113231368963\\
36	192.685091866286\\
37	211.97701881112\\
38	233.154685634665\\
39	256.399166169156\\
40	281.908251118668\\
41	309.897852952102\\
42	340.603499839258\\
43	374.281917818236\\
44	411.212698559806\\
45	451.700047810437\\
46	496.07460677053\\
47	544.695335210158\\
48	597.951440945963\\
49	656.264335299236\\
50	720.08958822309\\
51	789.918849824606\\
52	866.281696921875\\
53	949.747353988956\\
54	1040.92622730413\\
55	1140.47117932198\\
56	1249.07845729183\\
57	1367.48817608278\\
58	1496.48424029878\\
59	1636.89357546856\\
60	1789.58452294491\\
61	1955.4642389294\\
62	2135.47492578423\\
63	2330.58871481326\\
64	2541.8010155955\\
65	2770.12214963777\\
66	3016.56709774412\\
67	3282.14321343649\\
68	3567.83579143062\\
69	3874.59143288377\\
70	4203.29921981818\\
71	4554.76980101091\\
72	4929.71260084742\\
73	5328.71148974456\\
74	5752.19939636933\\
75	6200.43249228295\\
76	6673.4647305317\\
77	7171.12366024949\\
78	7692.98855649494\\
79	8238.37198384774\\
80	8806.30593902511\\
81	9395.53367861345\\
82	10004.508222866\\
83	10631.3983305455\\
84	11274.1024650897\\
85	11930.2709291611\\
86	12597.3359514875\\
87	13272.5490928868\\
88	13953.0249289927\\
89	14635.7895992769\\
90	15317.8325177914\\
91	15996.1593475173\\
92	16667.8442653153\\
93	17330.079594991\\
94	17980.221056618\\
95	18615.8271543525\\
96	19234.691576835\\
97	19834.8678825027\\
98	20414.6861534749\\
99	20972.7616947645\\
100	21507.9962040558\\
101	22019.572121925\\
102	22506.9410819602\\
103	22969.8075114548\\
104	23408.1084897816\\
105	23821.990962022\\
106	24211.7873421958\\
107	24577.9904373655\\
108	24921.2284948824\\
109	25242.2410328496\\
110	25541.8559693101\\
111	25820.9684272079\\
112	26080.5214658686\\
113	26321.4888794271\\
114	26544.8601101718\\
115	26751.6272504913\\
116	26942.7740502133\\
117	27119.2668050045\\
118	27282.0469741428\\
119	27432.0253601514\\
120	27570.0776762757\\
121	27697.0413285219\\
122	27813.7132450776\\
123	27920.8485958126\\
124	28019.1602568629\\
125	28109.3188889751\\
126	28191.9535124818\\
127	28267.6524758821\\
128	28336.9647285567\\
129	28400.4013208635\\
130	28458.4370665401\\
131	28511.5123129035\\
132	28560.0347737445\\
133	28604.3813880892\\
134	28644.9001752012\\
135	28681.9120623913\\
136	28715.7126674814\\
137	28746.5740222251\\
138	28774.7462267098\\
139	28800.4590278496\\
140	28823.923317602\\
141	28845.3325485863\\
};
\addlegendentry{45-64}

\addplot [color=mycolor4, line width=2.0pt, mark options={solid, mycolor4}]
  table[row sep=crcr]{%
1	0\\
2	0.144981564007577\\
3	1.73622969244256\\
4	3.74124399809646\\
5	5.91788388681085\\
6	8.23224826380028\\
7	10.705820809838\\
8	13.3726209055639\\
9	16.2679132934087\\
10	19.4261844469212\\
11	22.8816278222187\\
12	26.6691162535201\\
13	30.8251211676301\\
14	35.3884849657269\\
15	40.4010712273574\\
16	45.9083369387699\\
17	51.9598649176146\\
18	58.6098848336909\\
19	65.9178032129327\\
20	73.948757249055\\
21	82.7742036541361\\
22	92.4725515611032\\
23	103.129847183013\\
24	114.840517216603\\
25	127.708177640726\\
26	141.846514467332\\
27	157.380243063599\\
28	174.44615281769\\
29	193.19424412363\\
30	213.788964879216\\
31	236.410553895827\\
32	261.256498783352\\
33	288.543115969162\\
34	318.507260506374\\
35	351.408173188735\\
36	387.529472176316\\
37	427.181295800251\\
38	470.702602399393\\
39	518.463631880883\\
40	570.868532112646\\
41	628.358151158607\\
42	691.412993652765\\
43	760.5563361572\\
44	836.357492026533\\
45	919.435210956603\\
46	1010.46119186158\\
47	1110.16367982149\\
48	1219.33110838043\\
49	1338.81573725891\\
50	1469.53722237896\\
51	1612.48603981061\\
52	1768.72666768482\\
53	1939.40041018832\\
54	2125.72772544507\\
55	2329.00989450011\\
56	2550.62984201749\\
57	2792.05189115745\\
58	3054.82020615697\\
59	3340.55564748586\\
60	3650.95073758186\\
61	3987.76241205249\\
62	4352.80221437401\\
63	4747.92358459394\\
64	5175.00589799959\\
65	5635.93493231229\\
66	6132.57948624456\\
67	6666.76394290169\\
68	7240.23667298683\\
69	7854.63430883296\\
70	8511.44209332641\\
71	9211.95071806155\\
72	9957.21030984189\\
73	10747.9824973371\\
74	11584.69177915\\
75	12467.3777045321\\
76	13395.6496472653\\
77	14368.6461762132\\
78	15385.0011742709\\
79	16442.81890198\\
80	17539.6601167381\\
81	18672.5411236628\\
82	19837.947240644\\
83	21031.8616127889\\
84	22249.8096315391\\
85	23486.9184392451\\
86	24737.9901839148\\
87	25997.5868948613\\
88	27260.1241458401\\
89	28519.9701220669\\
90	29771.5463634745\\
91	31009.4263522035\\
92	32228.4282577155\\
93	33423.6985334078\\
94	34590.7836368049\\
95	35725.6878663509\\
96	36824.9161062724\\
97	37885.5010795445\\
98	38905.0154658318\\
99	39881.5698962045\\
100	40813.7983547713\\
101	41700.8328812796\\
102	42542.2696762807\\
103	43338.1287730269\\
104	44088.8093788487\\
105	44795.0428300682\\
106	45457.8448772087\\
107	46078.4687490275\\
108	46658.3601591643\\
109	47199.1151379195\\
110	47702.4413088047\\
111	48170.1229951022\\
112	48603.9903414287\\
113	49005.8924713233\\
114	49377.67457351\\
115	49721.1587141589\\
116	50038.1281064065\\
117	50330.3145272374\\
118	50599.388551065\\
119	50846.9522646615\\
120	51074.5341355442\\
121	51283.5857220935\\
122	51475.4799356927\\
123	51651.5105906802\\
124	51812.8930050799\\
125	51960.7654425111\\
126	52096.1912123822\\
127	52220.1612707341\\
128	52333.5971874849\\
129	52437.3543670835\\
130	52532.2254286062\\
131	52618.9436681463\\
132	52698.1865410181\\
133	52770.5791139766\\
134	52836.6974484904\\
135	52897.0718852839\\
136	52952.1902080619\\
137	53002.5006707276\\
138	53048.4148776609\\
139	53090.3105109098\\
140	53128.5339015839\\
141	53163.4024454758\\
};
\addlegendentry{65-74}

\addplot [color=mycolor5, line width=2.0pt, mark options={solid, mycolor5}]
  table[row sep=crcr]{%
1	0\\
2	1.22489296657866\\
3	2.68227660029449\\
4	4.29947341840747\\
5	6.01508638444028\\
6	7.82546637978117\\
7	9.75124124476728\\
8	11.8197401583686\\
9	14.0590108071114\\
10	16.4965388109992\\
11	19.1595218795925\\
12	22.0755358757182\\
13	25.2732178697597\\
14	28.7828723775511\\
15	32.6369980318831\\
16	36.8707561142531\\
17	41.5224047517269\\
18	46.6337189871589\\
19	52.2504125766382\\
20	58.4225737406134\\
21	65.2051244328696\\
22	72.6583108652577\\
23	80.8482318368696\\
24	89.8474106857598\\
25	99.7354162706203\\
26	110.599538196994\\
27	122.535521453989\\
28	135.648365669555\\
29	150.053194285241\\
30	165.87619906361\\
31	183.255665446565\\
32	202.343084356264\\
33	223.304356047664\\
34	246.321091556738\\
35	271.592017112046\\
36	299.334486555881\\
37	329.786106315813\\
38	363.206476732527\\
39	399.879052532184\\
40	440.11312386922\\
41	484.245917586925\\
42	532.64481606601\\
43	585.709688162202\\
44	643.875323167414\\
45	707.613954348502\\
46	777.437853295484\\
47	853.901969910989\\
48	937.60658525255\\
49	1029.1999354563\\
50	1129.38075448861\\
51	1238.90067137215\\
52	1358.56638372621\\
53	1489.24151391002\\
54	1631.8480367949\\
55	1787.36714935457\\
56	1956.83943212491\\
57	2141.36413160007\\
58	2342.09737147929\\
59	2560.24908031613\\
60	2797.078404833\\
61	3053.88736361116\\
62	3332.01248711863\\
63	3632.81418959341\\
64	3957.66362905273\\
65	4307.92683688325\\
66	4684.94594150517\\
67	5090.01737489212\\
68	5524.3670393136\\
69	5989.12252681436\\
70	6485.28262661218\\
71	7013.68452481511\\
72	7574.96929309973\\
73	8169.54647155483\\
74	8797.5587654663\\
75	9458.84808230638\\
76	10152.9243159867\\
77	10878.9384202447\\
78	11635.6613803535\\
79	12421.470671595\\
80	13234.3456671264\\
81	14071.8732165626\\
82	14931.2642587049\\
83	15809.3818678489\\
84	16702.7805860323\\
85	17607.7562984496\\
86	18520.4053105709\\
87	19436.6907330713\\
88	20352.5138240235\\
89	21263.7876200823\\
90	22166.5100406456\\
91	23056.8336859202\\
92	23931.1297677167\\
93	24786.0439889665\\
94	25618.542687843\\
95	26425.9481384386\\
96	27205.9625017519\\
97	27956.6804999277\\
98	28676.5914021517\\
99	29364.571332061\\
100	30019.8672161999\\
101	30642.0738855529\\
102	31231.1059229236\\
103	31787.1658315515\\
104	32310.7100037666\\
105	32802.4138139601\\
106	33263.1369688405\\
107	33693.8900390365\\
108	34095.8028856292\\
109	34470.0954954643\\
110	34818.0515586935\\
111	35140.9949660757\\
112	35440.2692744044\\
113	35717.2200860485\\
114	35973.1802114071\\
115	36209.4574284721\\
116	36427.3246184766\\
117	36628.0120374347\\
118	36812.7014769564\\
119	36982.5220710509\\
120	37138.547516039\\
121	37281.7944859551\\
122	37413.2220440691\\
123	37533.7318709649\\
124	37644.1691498456\\
125	37745.3239695958\\
126	37837.9331250368\\
127	37922.6822113974\\
128	38000.2079260781\\
129	38071.1005052\\
130	38135.9062352098\\
131	38195.1299909925\\
132	38249.2377616261\\
133	38298.6591332102\\
134	38343.7897052408\\
135	38384.9934229242\\
136	38422.6048127515\\
137	38456.9311127195\\
138	38488.2542919011\\
139	38516.8329567426\\
140	38542.9041435948\\
141	38566.6849986573\\
};
\addlegendentry{75-84}

\addplot [color=mycolor6, line width=2.0pt, mark options={solid, mycolor6}]
  table[row sep=crcr]{%
1	0\\
2	1.06753654407652\\
3	2.33370294413975\\
4	3.73058349518353\\
5	5.22984641015384\\
6	6.82614802805157\\
7	8.5292322102561\\
8	10.3570337636596\\
9	12.3312440408861\\
10	14.475102410168\\
11	16.8125997139099\\
12	19.3684230875351\\
13	22.1682335176085\\
14	25.239064166563\\
15	28.6097439707015\\
16	32.3113116919143\\
17	36.3774140440252\\
18	40.8446932171313\\
19	45.7531728642508\\
20	51.1466519434767\\
21	57.0731149184523\\
22	63.5851656836453\\
23	70.7404915781311\\
24	78.6023630911511\\
25	87.240174340836\\
26	96.730029081196\\
27	107.155376811026\\
28	118.607703474939\\
29	131.187281221339\\
30	145.003981681138\\
31	160.178157225051\\
32	176.841594619744\\
33	195.138545407538\\
34	215.226837153746\\
35	237.279069410184\\
36	261.48389779975\\
37	288.047408996887\\
38	317.194588518473\\
39	349.170882098255\\
40	384.243849937234\\
41	422.704911235428\\
42	464.871174041799\\
43	511.08734252375\\
44	561.727690161802\\
45	617.198083016708\\
46	677.938031986601\\
47	744.422746758259\\
48	817.165156847042\\
49	896.717856608899\\
50	983.674921304851\\
51	1078.67353014036\\
52	1182.3953196679\\
53	1295.56737707236\\
54	1418.96276778193\\
55	1553.40047580563\\
56	1699.74461858366\\
57	1858.90278152834\\
58	2031.82330164182\\
59	2219.49131569875\\
60	2422.92337786621\\
61	2643.16044601591\\
62	2881.25903742722\\
63	3138.28036545911\\
64	3415.27729171814\\
65	3713.27896602896\\
66	4033.27308182594\\
67	4376.18574979272\\
68	4742.85908934633\\
69	5134.02675640903\\
70	5550.28776569684\\
71	5992.07912315438\\
72	6459.64795319383\\
73	6953.02397705665\\
74	7471.99336082394\\
75	8016.07508946469\\
76	8584.50111990944\\
77	9176.2016038445\\
78	9789.79643327625\\
79	10423.5942359125\\
80	11075.5997259871\\
81	11743.5300005024\\
82	12424.8399721002\\
83	13116.756669562\\
84	13816.3216466632\\
85	14520.4402584219\\
86	15225.9361327416\\
87	15929.6088255277\\
88	16628.2924322568\\
89	17318.91286093\\
90	17998.5415578829\\
91	18664.4437107095\\
92	19314.1193084794\\
93	19945.3358838229\\
94	20556.1522528516\\
95	21144.9330643596\\
96	21710.3544300685\\
97	22251.4013013321\\
98	22767.3575634547\\
99	23257.7900264089\\
100	23722.5276004127\\
101	24161.6369651503\\
102	24575.3959869267\\
103	24964.2660265291\\
104	25328.8641306329\\
105	25669.9359287089\\
106	25988.3298805954\\
107	26284.9733490684\\
108	26560.8508153792\\
109	26816.9844191933\\
110	27054.416890293\\
111	27274.1968482703\\
112	27477.3663771595\\
113	27664.9507324044\\
114	27837.9500050462\\
115	27997.3325496363\\
116	28144.0299752642\\
117	28278.9335005971\\
118	28402.8914816151\\
119	28516.707932819\\
120	28621.141877474\\
121	28716.907378647\\
122	28804.6741194384\\
123	28885.0684171801\\
124	28958.6745719849\\
125	29026.0364645551\\
126	29087.6593314194\\
127	29144.011657669\\
128	29195.5271378047\\
129	29242.6066645276\\
130	29285.6203132823\\
131	29324.9092971913\\
132	29360.7878728155\\
133	29393.5451820384\\
134	29423.4470194285\\
135	29450.7375177684\\
136	29475.6407471689\\
137	29498.3622253818\\
138	29519.0903386781\\
139	29537.9976740322\\
140	29555.242264413\\
141	29570.9687497807\\
};
\addlegendentry{$85+$}

\end{axis}
\end{tikzpicture}%

%% file: Figures/OL_state_4.tex
%
%
\definecolor{mycolor1}{rgb}{0.00000,0.44700,0.74100}%
\definecolor{mycolor2}{rgb}{0.85000,0.32500,0.09800}%
\definecolor{mycolor3}{rgb}{0.92900,0.69400,0.12500}%
\definecolor{mycolor4}{rgb}{0.49400,0.18400,0.55600}%
\definecolor{mycolor5}{rgb}{0.46600,0.67400,0.18800}%
\definecolor{mycolor6}{rgb}{0.30100,0.74500,0.93300}%
\begin{tikzpicture}

\begin{axis}[%
width=7cm,
height=3.5cm,
at={(0cm,0cm)},
scale only axis,
xmin=0,
xmax=140,
ymin=0,
ymax=7000,
xmajorgrids,
ymajorgrids,
axis background/.style={fill=white},
axis x line*=bottom,
axis y line*=left
]
\addplot [color=mycolor1, line width=2.0pt, forget plot]
  table[row sep=crcr]{%
1	0\\
2	0.00205352335266638\\
3	0.00525182455328961\\
4	0.00931989200713513\\
5	0.0141180246459586\\
6	0.0195874651244966\\
7	0.0257180252629923\\
8	0.0325297359848644\\
9	0.0400627010959333\\
10	0.048371613840375\\
11	0.0575229110143233\\
12	0.0675934353886098\\
13	0.0786699827885641\\
14	0.0908493900088023\\
15	0.104238973821805\\
16	0.118957216548155\\
17	0.135134641337862\\
18	0.152914847492578\\
19	0.172455692048125\\
20	0.193930613432432\\
21	0.217530099048604\\
22	0.243463302618713\\
23	0.27195981990942\\
24	0.303271633527648\\
25	0.337675239099494\\
26	0.375473966485274\\
27	0.417000510824364\\
28	0.462619689186677\\
29	0.512731439445091\\
30	0.567774078664699\\
31	0.628227838802702\\
32	0.694618697783339\\
33	0.76752252399592\\
34	0.847569551884924\\
35	0.935449205464931\\
36	1.03191528518579\\
37	1.13779153145777\\
38	1.25397757515994\\
39	1.38145528140676\\
40	1.5212954875148\\
41	1.6746651292366\\
42	1.84283474061763\\
43	2.02718630195185\\
44	2.22922139689045\\
45	2.45056962338697\\
46	2.69299718339932\\
47	2.95841555264993\\
48	3.24889010379286\\
49	3.5666485235887\\
50	3.91408882672011\\
51	4.29378672535344\\
52	4.70850206426271\\
53	5.16118397628744\\
54	5.654974352402\\
55	6.19320915544583\\
56	6.77941703785782\\
57	7.41731465352169\\
58	8.11079798486894\\
59	8.86392894254139\\
60	9.6809164412259\\
61	10.566091118119\\
62	11.5238728476621\\
63	12.5587302269204\\
64	13.6751312707161\\
65	14.8774846757262\\
66	16.1700711998047\\
67	17.5569649677317\\
68	19.0419448663422\\
69	20.6283966358246\\
70	22.3192067995233\\
71	24.1166501935979\\
72	26.0222735421549\\
73	28.0367782428978\\
74	30.1599062399682\\
75	32.3903335088496\\
76	34.7255761966456\\
77	37.1619147768755\\
78	39.6943416182568\\
79	42.3165370674376\\
80	45.0208784606034\\
81	47.7984853915233\\
82	50.6393030948293\\
83	53.5322240172179\\
84	56.4652456533281\\
85	59.4256606620964\\
86	62.4002733228308\\
87	65.3756347137904\\
88	68.3382877602509\\
89	71.275012627945\\
90	74.1730629015017\\
91	77.0203835918699\\
92	79.8058032020102\\
93	82.5191937293818\\
94	85.1515944389213\\
95	87.6952973233557\\
96	90.1438942033302\\
97	92.492287254847\\
98	94.7366662697746\\
99	96.8744570852788\\
100	98.9042463341325\\
101	100.825687984668\\
102	102.639397102841\\
103	104.346835946224\\
104	105.950196967152\\
105	107.452286636243\\
106	108.856413267804\\
107	110.16628129329\\
108	111.385893732529\\
109	112.519463985569\\
110	113.571337528804\\
111	114.545923655197\\
112	115.447637049285\\
113	116.280848727188\\
114	117.049845690007\\
115	117.758798524324\\
116	118.411736123749\\
117	119.012526688913\\
118	119.564864179405\\
119	120.072259430586\\
120	120.538035203245\\
121	120.965324498453\\
122	121.357071538834\\
123	121.716034887068\\
124	122.044792240216\\
125	122.345746502464\\
126	122.621132798149\\
127	122.873026140625\\
128	123.103349520605\\
129	123.313882219961\\
130	123.506268193912\\
131	123.682024396372\\
132	123.84254895043\\
133	123.989129088918\\
134	124.122948809295\\
135	124.245096203052\\
136	124.356570433002\\
137	124.458288342465\\
138	124.55109068898\\
139	124.635748001962\\
140	124.712966069048\\
141	124.783391059931\\
};
\addplot [color=mycolor2, line width=2.0pt, forget plot]
  table[row sep=crcr]{%
1	0\\
2	0.00792471824437081\\
3	0.0151733311032985\\
4	0.0223107967711348\\
5	0.0296918913897686\\
6	0.0375468576459977\\
7	0.0460369348792874\\
8	0.0552880896329352\\
9	0.0654105998584643\\
10	0.0765102547101886\\
11	0.0886948562262939\\
12	0.102078183882955\\
13	0.116782631241678\\
14	0.13294117593241\\
15	0.150699042154799\\
16	0.17021525241692\\
17	0.191664179076784\\
18	0.215237160967964\\
19	0.241144226783254\\
20	0.269615954754126\\
21	0.300905492122249\\
22	0.335290755126997\\
23	0.373076829240376\\
24	0.414598589375783\\
25	0.460223560350273\\
26	0.510355038751714\\
27	0.56543549840999\\
28	0.625950302802667\\
29	0.69243174887093\\
30	0.765463467820007\\
31	0.845685209468571\\
32	0.933798037525921\\
33	1.03056996373564\\
34	1.13684204903658\\
35	1.25353499964457\\
36	1.38165628511757\\
37	1.52230780387181\\
38	1.67669411907706\\
39	1.84613128414785\\
40	2.03205627189918\\
41	2.23603701453813\\
42	2.45978305265654\\
43	2.70515677985952\\
44	2.97418525513976\\
45	3.26907253705632\\
46	3.59221247161288\\
47	3.94620183881129\\
48	4.33385373049752\\
49	4.75821099360553\\
50	5.2225595275264\\
51	5.73044117140642\\
52	6.28566585611891\\
53	6.89232262602924\\
54	7.55478905730304\\
55	8.27773851258592\\
56	9.06614457712155\\
57	9.92528192018206\\
58	10.8607227203739\\
59	11.8783276874067\\
60	12.9842306111187\\
61	14.1848152774243\\
62	15.4866835187535\\
63	16.8966131238879\\
64	18.42150433135\\
65	20.0683136861594\\
66	21.8439741678785\\
67	23.7553007154255\\
68	25.8088805978442\\
69	28.010948524997\\
70	30.3672469688787\\
71	32.8828728792245\\
72	35.5621128209761\\
73	38.4082695178824\\
74	41.4234838223247\\
75	44.6085571949167\\
76	47.9627807986654\\
77	51.4837782046707\\
78	55.16736936991\\
79	59.0074638775763\\
80	62.9959913267289\\
81	67.1228761386501\\
82	71.3760628624078\\
83	75.7415963074139\\
84	80.2037585577858\\
85	84.7452622440766\\
86	89.3474965317617\\
87	93.9908193475343\\
88	98.6548866440087\\
89	103.319007239752\\
90	107.962510173523\\
91	112.56511073109\\
92	117.10726141462\\
93	121.570475115611\\
94	125.937609524503\\
95	130.193104193323\\
96	134.323164442949\\
97	138.315889233624\\
98	142.161342963005\\
99	145.85157371922\\
100	149.380582644925\\
101	152.744250668627\\
102	155.940229896572\\
103	158.967807449784\\
104	161.827749534899\\
105	164.522133140236\\
106	167.054172049858\\
107	169.428042969846\\
108	171.64871655669\\
109	173.72179710804\\
110	175.653373684368\\
111	177.449884522196\\
112	179.117995804238\\
113	180.664495183899\\
114	182.096199923982\\
115	183.419879096795\\
116	184.642188993887\\
117	185.769620693824\\
118	186.808458619713\\
119	187.764748868668\\
120	188.644276098014\\
121	189.45254779433\\
122	190.194784819884\\
123	190.87591721702\\
124	191.500584346891\\
125	192.073138538667\\
126	192.597651524457\\
127	193.077923030532\\
128	193.517490984843\\
129	193.919642883014\\
130	194.287427929211\\
131	194.623669634456\\
132	194.930978613091\\
133	195.211765368707\\
134	195.468252904328\\
135	195.702489028759\\
136	195.916358262265\\
137	196.111593270904\\
138	196.289785780483\\
139	196.452396938826\\
140	196.600767109419\\
141	196.736125090977\\
};
\addplot [color=mycolor3, line width=2.0pt, forget plot]
  table[row sep=crcr]{%
1	0\\
2	0.112690712840141\\
3	0.184853862746458\\
4	0.240084950271744\\
5	0.289131566372381\\
6	0.337173829615426\\
7	0.386915831700138\\
8	0.439933755032252\\
9	0.497284397364492\\
10	0.559790725460741\\
11	0.628181837151906\\
12	0.703165228862507\\
13	0.785466713808157\\
14	0.875854559119352\\
15	0.975155860759522\\
16	1.08426916438193\\
17	1.20417540610817\\
18	1.33594829094\\
19	1.48076474325507\\
20	1.63991581516001\\
21	1.81481830894113\\
22	2.00702730227545\\
23	2.21824973028615\\
24	2.4503591619685\\
25	2.70541190204352\\
26	2.98566454860751\\
27	3.29359313960854\\
28	3.63191402576716\\
29	4.00360661315731\\
30	4.41193812463795\\
31	4.86049053516305\\
32	5.35318984122802\\
33	5.89433782885078\\
34	6.48864650700303\\
35	7.14127537368635\\
36	7.85787167916477\\
37	8.64461384436313\\
38	9.5082581810978\\
39	10.4561890434171\\
40	11.4964725144667\\
41	12.6379136992921\\
42	13.8901176488973\\
43	15.2635538824578\\
44	16.7696244002538\\
45	18.420734986753\\
46	20.2303694880559\\
47	22.2131666070527\\
48	24.3849985892301\\
49	26.7630509680134\\
50	29.365902296615\\
51	32.2136025094385\\
52	35.3277482263393\\
53	38.7315529343103\\
54	42.4499095514307\\
55	46.5094423969074\\
56	50.9385450609718\\
57	55.76740009492\\
58	61.0279758348949\\
59	66.7539950491193\\
60	72.9808694804421\\
61	79.7455937762219\\
62	87.0865917978288\\
63	95.0435079358437\\
64	103.656935889885\\
65	112.968077481423\\
66	123.018324542269\\
67	133.848757856798\\
68	145.499558631388\\
69	158.009330114236\\
70	171.41432987138\\
71	185.747616890442\\
72	201.038122137141\\
73	217.309656373213\\
74	234.579874819822\\
75	252.859224384075\\
76	272.149905319817\\
77	292.444884925331\\
78	313.727005658383\\
79	335.968233283067\\
80	359.129091753225\\
81	383.158329939968\\
82	407.992860615207\\
83	433.558004111113\\
84	459.768057872936\\
85	486.527199125884\\
86	513.730711843608\\
87	541.266512199575\\
88	569.016929988009\\
89	596.860688496759\\
90	624.675013317948\\
91	652.337792689689\\
92	679.729708907835\\
93	706.736262407118\\
94	733.249617069319\\
95	759.170206493843\\
96	784.40805531529\\
97	808.883785892746\\
98	832.52929747024\\
99	855.288120938417\\
100	877.115466539133\\
101	897.977993462375\\
102	917.85333883171\\
103	936.729448925821\\
104	954.603757784829\\
105	971.48225796157\\
106	987.378505599317\\
107	1002.31259781416\\
108	1016.31015509921\\
109	1029.40133566911\\
110	1041.61990276771\\
111	1053.00236031528\\
112	1063.58716712092\\
113	1073.41403538698\\
114	1082.52331546157\\
115	1090.95546576631\\
116	1098.75060450564\\
117	1105.9481380876\\
118	1112.58646006993\\
119	1118.70271380025\\
120	1124.33261165383\\
121	1129.51030380222\\
122	1134.26828969505\\
123	1138.63736584009\\
124	1142.64660396856\\
125	1146.32335423006\\
126	1149.69326864058\\
127	1152.78034058208\\
128	1155.60695670497\\
129	1158.19395810338\\
130	1160.56070810938\\
131	1162.72516448338\\
132	1164.70395416115\\
133	1166.51244905586\\
134	1168.16484170667\\
135	1169.67421981852\\
136	1171.05263895258\\
137	1172.31119280889\\
138	1173.46008069438\\
139	1174.50867189525\\
140	1175.46556677555\\
141	1176.33865450742\\
};
\addplot [color=mycolor4, line width=2.0pt, forget plot]
  table[row sep=crcr]{%
1	0\\
2	0.00679340447442321\\
3	0.0813545545739133\\
4	0.175303555942114\\
5	0.27729442119742\\
6	0.38573864596288\\
7	0.501642891561894\\
8	0.626601204894404\\
9	0.762265986806404\\
10	0.910253170780997\\
11	1.07216496037678\\
12	1.24963539278765\\
13	1.44437341049702\\
14	1.65819905278897\\
15	1.89307392237178\\
16	2.1511279982941\\
17	2.43468458378129\\
18	2.74628471971661\\
19	3.0887120190501\\
20	3.46501861676365\\
21	3.87855276165622\\
22	4.33298847227015\\
23	4.83235761798685\\
24	5.38108475270394\\
25	5.9840250127223\\
26	6.64650538611481\\
27	7.37436966370559\\
28	8.17402738899736\\
29	9.05250713389798\\
30	10.0175144373313\\
31	11.0774947534215\\
32	12.2417017636408\\
33	13.5202714117946\\
34	14.9243020205465\\
35	16.465940841723\\
36	18.1584773779641\\
37	20.0164437881747\\
38	22.0557226510232\\
39	24.2936623063414\\
40	26.7491999200586\\
41	29.442992320031\\
42	32.3975545229293\\
43	35.6374057106406\\
44	39.1892222120978\\
45	43.0819967960091\\
46	47.3472032738224\\
47	52.0189650419779\\
48	57.134225749159\\
49	62.7329197486622\\
50	68.8581393791485\\
51	75.5562954005784\\
52	82.8772660891817\\
53	90.8745295614481\\
54	99.605272851721\\
55	109.130470115814\\
56	119.514921086473\\
57	130.827239587645\\
58	143.139780558448\\
59	156.528492695141\\
60	171.07268256046\\
61	186.854674926518\\
62	203.959353327587\\
63	222.47356444654\\
64	242.485370214378\\
65	264.083132561098\\
66	287.354417830893\\
67	312.384711184828\\
68	339.255936069093\\
69	368.044780202503\\
70	398.820837645093\\
71	431.644586362545\\
72	466.565232170606\\
73	503.618462721343\\
74	542.824168755542\\
75	584.184203433383\\
76	627.680263173195\\
77	673.271983876442\\
78	720.895353362555\\
79	770.461542924201\\
80	821.856256914733\\
81	874.939688274523\\
82	929.54714946322\\
83	985.490422618855\\
84	1042.55984090544\\
85	1100.52707672019\\
86	1159.14857419293\\
87	1218.16952620629\\
88	1277.32826317147\\
89	1336.36089501379\\
90	1395.00603170177\\
91	1453.00940276342\\
92	1510.12820304621\\
93	1566.13500980551\\
94	1620.82114329925\\
95	1673.99937684788\\
96	1725.50593973131\\
97	1775.20179418186\\
98	1822.97320319481\\
99	1868.73163656683\\
100	1912.41308685992\\
101	1953.9768840406\\
102	1993.40410726884\\
103	2030.69569524318\\
104	2065.87035363068\\
105	2098.96235067419\\
106	2130.0192814193\\
107	2159.09986843468\\
108	2186.27185355753\\
109	2211.61002201632\\
110	2235.19438796513\\
111	2257.10855948099\\
112	2277.43829169223\\
113	2296.27022902299\\
114	2313.69083152369\\
115	2329.7854757907\\
116	2344.63771788293\\
117	2358.3287037143\\
118	2370.93671142848\\
119	2382.53681004215\\
120	2393.20061899286\\
121	2402.99615398499\\
122	2411.98774555874\\
123	2420.2360180023\\
124	2427.79791750022\\
125	2434.72677969705\\
126	2441.07242810615\\
127	2446.88129597751\\
128	2452.19656533393\\
129	2457.05831788132\\
130	2461.50369338995\\
131	2465.56705193169\\
132	2469.28013704578\\
133	2472.6722374995\\
134	2475.77034581832\\
135	2478.59931218974\\
136	2481.18199270602\\
137	2483.5393912106\\
138	2485.69079425937\\
139	2487.65389890871\\
140	2489.44493320339\\
141	2491.07876936533\\
};
\addplot [color=mycolor5, line width=2.0pt, forget plot]
  table[row sep=crcr]{%
1	0\\
2	0.150403228269135\\
3	0.329353723796686\\
4	0.527927500303994\\
5	0.738585684808434\\
6	0.960880206144185\\
7	1.19734393360152\\
8	1.45133258629647\\
9	1.72629010807894\\
10	2.0255914272833\\
11	2.35257612004483\\
12	2.71063019551259\\
13	3.10326996731672\\
14	3.53421649283687\\
15	4.00746024260347\\
16	4.52731863078379\\
17	5.09848932972572\\
18	5.72610184990716\\
19	6.41576933196955\\
20	7.17364205210785\\
21	8.00646381519116\\
22	8.92163218574193\\
23	9.92726336088055\\
24	11.0322623996333\\
25	12.246399472573\\
26	13.5803927720856\\
27	15.0459987175852\\
28	16.6561100951696\\
29	18.42486278261\\
30	20.3677517243532\\
31	22.5017568341159\\
32	24.8454795116677\\
33	27.4192904625288\\
34	30.2454895013321\\
35	33.3484779979382\\
36	36.7549445859253\\
37	40.4940646910156\\
38	44.5977143467598\\
39	49.1006986398439\\
40	54.0409949601039\\
41	59.4600110119463\\
42	65.4028572642659\\
43	71.9186331636481\\
44	79.0607259977211\\
45	86.887120757639\\
46	95.4607186951772\\
47	104.849661484066\\
48	115.127656959509\\
49	126.374301306817\\
50	138.675391283023\\
51	152.123218569566\\
52	166.816836658866\\
53	182.862288768066\\
54	200.372783153573\\
55	219.468799887067\\
56	240.278110680968\\
57	262.93569077466\\
58	287.583499295503\\
59	314.370102008337\\
60	343.450108121548\\
61	374.983391030793\\
62	409.134061807454\\
63	446.069164184135\\
64	485.957061109996\\
65	528.965486040999\\
66	575.259237414771\\
67	624.997502653713\\
68	678.330808917261\\
69	735.397611963044\\
70	796.320551994585\\
71	861.202426151511\\
72	930.12195090321\\
73	1003.12941321609\\
74	1080.24233571153\\
75	1161.44130638604\\
76	1246.66614566484\\
77	1335.81260011328\\
78	1428.7297603968\\
79	1525.21839853203\\
80	1625.03040402397\\
81	1727.86946885376\\
82	1833.39312733702\\
83	1941.21619989831\\
84	2050.91562263368\\
85	2162.03657145658\\
86	2274.09971611084\\
87	2386.60937150133\\
88	2499.06225772155\\
89	2610.95654114385\\
90	2721.80081079695\\
91	2831.12264879205\\
92	2938.47648031297\\
93	3043.45043499924\\
94	3145.67201293076\\
95	3244.81241915813\\
96	3340.58950461753\\
97	3432.76932238623\\
98	3521.16637152847\\
99	3605.642652531\\
100	3686.1056963521\\
101	3762.50575274451\\
102	3834.83233342687\\
103	3903.11030354394\\
104	3967.39570299609\\
105	4027.77146024442\\
106	4084.34313770772\\
107	4137.2348222157\\
108	4186.5852481384\\
109	4232.54421628703\\
110	4275.26934953004\\
111	4314.92320692314\\
112	4351.67076229231\\
113	4385.67724063741\\
114	4417.10629624676\\
115	4446.11850970757\\
116	4472.87017667333\\
117	4497.51235889485\\
118	4520.19016723338\\
119	4541.04224678261\\
120	4560.20043550491\\
121	4577.78956966031\\
122	4593.92741154791\\
123	4608.7246775114\\
124	4622.28514664472\\
125	4634.70583307247\\
126	4646.07720700117\\
127	4656.48345189708\\
128	4666.00274711715\\
129	4674.70756709014\\
130	4682.6649897138\\
131	4689.93700800681\\
132	4696.58084024341\\
133	4702.6492348171\\
134	4708.19076694459\\
135	4713.25012504801\\
136	4717.86838525854\\
137	4722.0832729839\\
138	4725.92941088911\\
139	4729.43855296851\\
140	4732.63980464875\\
141	4735.55982906721\\
};
\addplot [color=mycolor6, line width=2.0pt, forget plot]
  table[row sep=crcr]{%
1	0\\
2	0.244503556456279\\
3	0.534500362278705\\
4	0.854435323353285\\
5	1.19781946023111\\
6	1.56342888589301\\
7	1.95349528858806\\
8	2.37212637225748\\
9	2.82429022243462\\
10	3.31530947487756\\
11	3.85067887945964\\
12	4.43605266173708\\
13	5.07730809355792\\
14	5.78063672348677\\
15	6.55264139570166\\
16	7.40043108245494\\
17	8.3317120690617\\
18	9.35487671065472\\
19	10.4790918367534\\
20	11.7143889550212\\
21	13.0717582016558\\
22	14.563247725596\\
23	16.2020699640599\\
24	18.0027160927585\\
25	19.9810798146047\\
26	22.1545915760018\\
27	24.54236425824\\
28	27.1653513724169\\
29	30.0465187804834\\
30	33.2110309647798\\
31	36.6864528670699\\
32	40.5029683094648\\
33	44.6936159877523\\
34	49.2945439862668\\
35	54.3452836957569\\
36	59.8890439140848\\
37	65.973025765296\\
38	72.6487588755495\\
39	79.9724589829873\\
40	88.0054068194712\\
41	96.8143476699469\\
42	106.471910472677\\
43	117.057044651287\\
44	128.655472046011\\
45	141.360150313286\\
46	155.271742965094\\
47	170.499089796311\\
48	187.159669774203\\
49	205.38004651484\\
50	225.296284225982\\
51	247.054319440386\\
52	270.810270992783\\
53	296.73066751729\\
54	324.992568288939\\
55	355.78354955887\\
56	389.301524727323\\
57	425.754362894738\\
58	465.359266714296\\
59	508.341867286253\\
60	554.934991402866\\
61	605.377055166302\\
62	659.910038332378\\
63	718.776996224931\\
64	782.219071321644\\
65	850.471975265043\\
66	923.761924720709\\
67	1002.30103173022\\
68	1086.28217136991\\
69	1175.87337674665\\
70	1271.21184337778\\
71	1372.39766105248\\
72	1479.48742998554\\
73	1592.4879573905\\
74	1711.35026775084\\
75	1835.9641916429\\
76	1966.1538200891\\
77	2101.67412005614\\
78	2242.20899809072\\
79	2387.37107022642\\
80	2536.70334558207\\
81	2689.68295877658\\
82	2845.72699495417\\
83	3004.20034581003\\
84	3164.42542271497\\
85	3325.69344271609\\
86	3487.27690446599\\
87	3648.44279327986\\
88	3808.46600525404\\
89	3966.64246479439\\
90	4122.3014297233\\
91	4274.81653145111\\
92	4423.61517921544\\
93	4568.18605917454\\
94	4708.08457168722\\
95	4842.93616359357\\
96	4972.43761773984\\
97	5096.3564521512\\
98	5214.52865128094\\
99	5326.85499931958\\
100	5433.29631066698\\
101	5533.86785732489\\
102	5628.63328048678\\
103	5717.69824806063\\
104	5801.20408551908\\
105	5879.32156833431\\
106	5952.24502376295\\
107	6020.18685061996\\
108	6083.37252986757\\
109	6142.03616757442\\
110	6196.41658567269\\
111	6246.75395506902\\
112	6293.28694979655\\
113	6336.25038954709\\
114	6375.87333047589\\
115	6412.37755996158\\
116	6445.9764493748\\
117	6476.87411925357\\
118	6505.26487306729\\
119	6531.33285852128\\
120	6555.25191873949\\
121	6577.18559937324\\
122	6597.28728149469\\
123	6615.70041388414\\
124	6632.55882189529\\
125	6647.98707340988\\
126	6662.10088542967\\
127	6675.00755757973\\
128	6686.80642121173\\
129	6697.58929490716\\
130	6707.44093900781\\
131	6716.4395033648\\
132	6724.65696382493\\
133	6732.15954408731\\
134	6739.00812049157\\
135	6745.25860806372\\
136	6750.9623267696\\
137	6756.1663474299\\
138	6760.91381715138\\
139	6765.24426444411\\
140	6769.19388443711\\
141	6772.79580478756\\
};
\end{axis}
\end{tikzpicture}%

%% file: Figures/NS_state_1.tex
%
%
\definecolor{mycolor1}{rgb}{0.00000,0.44700,0.74100}%
\definecolor{mycolor2}{rgb}{0.85000,0.32500,0.09800}%
\definecolor{mycolor3}{rgb}{0.92900,0.69400,0.12500}%
\definecolor{mycolor4}{rgb}{0.49400,0.18400,0.55600}%
\definecolor{mycolor5}{rgb}{0.46600,0.67400,0.18800}%
\definecolor{mycolor6}{rgb}{0.30100,0.74500,0.93300}%
\begin{tikzpicture}

\begin{axis}[%
width=7cm,
height=3.5cm,
at={(0cm,0cm)},
scale only axis,
xmin=61,
xmax=140,
ymin=0,
ymax=1200000,
xmajorgrids,
ymajorgrids,
axis background/.style={fill=white},
axis x line*=bottom,
axis y line*=left
]
\addplot [color=mycolor1, line width=2.0pt, forget plot]
  table[row sep=crcr]{%
1	1058299.34049118\\
2	1058292.44626517\\
3	1058283.78077337\\
4	1058273.61241241\\
5	1058262.04982195\\
6	1058249.10573796\\
7	1058234.73290958\\
8	1058218.84387911\\
9	1058201.32163147\\
10	1058182.02508278\\
11	1058160.79161252\\
12	1058137.43785495\\
13	1058111.75941957\\
14	1058083.52991054\\
15	1058052.49944864\\
16	1058018.39280639\\
17	1057980.9072135\\
18	1057939.70985877\\
19	1057894.43509538\\
20	1057844.68134481\\
21	1057790.00768631\\
22	1057729.9301135\\
23	1057663.91743528\\
24	1057591.38679507\\
25	1057511.69877958\\
26	1057424.1520859\\
27	1057327.97771375\\
28	1057222.332648\\
29	1057106.29299513\\
30	1056978.8465363\\
31	1056838.88465923\\
32	1056685.19363109\\
33	1056516.44517546\\
34	1056331.18631846\\
35	1056127.8284717\\
36	1055904.63572501\\
37	1055659.71232743\\
38	1055390.98934441\\
39	1055096.21048994\\
40	1054772.91714741\\
41	1054418.43261123\\
42	1054029.8456047\\
43	1053603.99315784\\
44	1053137.44296392\\
45	1052626.47537521\\
46	1052067.06524862\\
47	1051454.86391135\\
48	1050785.18158549\\
49	1050052.97069169\\
50	1049252.81054321\\
51	1048378.89404652\\
52	1047425.01714036\\
53	1046384.57183296\\
54	1045250.54383415\\
55	1044015.515924\\
56	1042671.67834603\\
57	1041210.84765812\\
58	1039624.49560582\\
59	1037903.78969438\\
60	1036039.64720974\\
61	1034022.80446188\\
62	1031843.90297417\\
63	1029493.59419949\\
64	1026972.5319605\\
65	1024289.31115379\\
66	1021455.54306075\\
67	1018484.59290657\\
68	1015390.89121291\\
69	1012189.29988798\\
70	1008895.00441435\\
71	1005523.55943901\\
72	1002090.85635053\\
73	998613.009988928\\
74	995106.197770652\\
75	991586.473729026\\
76	988068.875872394\\
77	984568.907581556\\
78	981102.383565846\\
79	977684.953968029\\
80	974331.688906844\\
81	971056.744885781\\
82	967873.1044519\\
83	964792.378951379\\
84	961824.667132328\\
85	958978.464575123\\
86	956260.619852606\\
87	953676.33331559\\
88	951229.193933963\\
89	948921.249050559\\
90	946753.101437955\\
91	944724.027795713\\
92	942832.112823087\\
93	941076.100883799\\
94	939460.517515102\\
95	937989.067842585\\
96	936662.850706548\\
97	935479.957456023\\
98	934435.67359292\\
99	933522.937889787\\
100	932732.896234377\\
101	932055.466043615\\
102	931479.86388914\\
103	930995.069053062\\
104	930590.208717803\\
105	930254.859852261\\
106	929979.269751788\\
107	920790.865782142\\
108	865419.636213644\\
109	810093.496427201\\
110	754807.204174514\\
111	699553.189201182\\
112	644323.266764254\\
113	589109.921107032\\
114	533907.066434076\\
115	478710.295090786\\
116	423516.738102831\\
117	368324.7176416\\
118	313133.35549472\\
119	313133.243108248\\
120	313133.20691973\\
121	313133.195267024\\
122	313133.191514858\\
123	313133.190306664\\
124	313133.189917627\\
125	313133.189792358\\
126	313133.189752022\\
127	313133.189739034\\
128	313133.189734851\\
129	313133.189733505\\
130	313133.189733071\\
131	313133.189732931\\
132	313133.189732887\\
133	313133.189732872\\
134	313133.189732867\\
135	313133.189732866\\
136	313133.189732865\\
137	313133.189732865\\
138	313133.189732865\\
139	313133.189732865\\
140	313133.189732865\\
141	313133.189732865\\
};
\addplot [color=mycolor2, line width=2.0pt, forget plot]
  table[row sep=crcr]{%
1	915791.670391113\\
2	915788.901498256\\
3	915786.091690194\\
4	915783.132068456\\
5	915779.950175584\\
6	915776.492532322\\
7	915772.714557096\\
8	915768.574937074\\
9	915764.032453362\\
10	915759.04411254\\
11	915753.563958054\\
12	915747.542225712\\
13	915740.924663755\\
14	915733.651920725\\
15	915725.658947598\\
16	915716.874383396\\
17	915707.219905194\\
18	915696.609529594\\
19	915684.948855782\\
20	915672.13424177\\
21	915658.051906044\\
22	915642.576946937\\
23	915625.572271926\\
24	915606.887428743\\
25	915586.357329784\\
26	915563.800860863\\
27	915539.019364923\\
28	915511.794990813\\
29	915481.888896859\\
30	915449.039298545\\
31	915412.959349352\\
32	915373.334843598\\
33	915329.821730101\\
34	915282.043425655\\
35	915229.587917726\\
36	915172.00464652\\
37	915108.801157736\\
38	915039.439518927\\
39	914963.33249471\\
40	914879.839478979\\
41	914788.262186238\\
42	914687.840109055\\
43	914577.745754883\\
44	914457.079683137\\
45	914324.865372849\\
46	914180.043962563\\
47	914021.468917787\\
48	913847.900697485\\
49	913658.001510099\\
50	913450.330271679\\
51	913223.337904108\\
52	912975.363140264\\
53	912704.629035312\\
54	912409.240419031\\
55	912087.182562778\\
56	911736.321375737\\
57	911354.405487427\\
58	910939.070615502\\
59	910487.846657563\\
60	909998.167980195\\
61	909467.387404114\\
62	908892.794396792\\
63	908271.637977925\\
64	907610.878716508\\
65	906916.220540593\\
66	906189.711243162\\
67	905434.242440559\\
68	904652.857521702\\
69	903848.117324095\\
70	903022.862499565\\
71	902180.328555471\\
72	901323.951001265\\
73	900457.292968736\\
74	899583.98792225\\
75	898707.686463238\\
76	897831.863898719\\
77	896960.727345113\\
78	896098.569114819\\
79	895249.513135695\\
80	894417.440952547\\
81	893605.944601714\\
82	892818.285620079\\
83	892057.359315071\\
84	891325.666256916\\
85	890625.292688972\\
86	889957.90075332\\
87	889324.728661203\\
88	888726.600320017\\
89	888163.943469647\\
90	887636.815066952\\
91	876437.695515476\\
92	820794.990323058\\
93	765213.12914676\\
94	709691.280080083\\
95	654225.589106038\\
96	598810.717277817\\
97	543440.653385478\\
98	488109.206709527\\
99	432810.336517221\\
100	377538.378399363\\
101	322288.192619412\\
102	267055.248670368\\
103	211835.657678309\\
104	156626.163994834\\
105	101424.107055946\\
106	46227.3637431258\\
107	-2.08365899891942\\
108	-2.08358478881295\\
109	-2.0835268688188\\
110	-2.08348356194019\\
111	-2.08345290771825\\
112	-2.08343249836751\\
113	-2.08341978189477\\
114	-2.08341240460086\\
115	-2.08340844199114\\
116	-2.08340648400852\\
117	-2.08340560078874\\
118	-2.08340524038486\\
119	-2.08340510882801\\
120	-2.08340506646737\\
121	-2.08340505282741\\
122	-2.08340504843539\\
123	-2.08340504702118\\
124	-2.08340504656581\\
125	-2.08340504641918\\
126	-2.08340504637197\\
127	-2.08340504635676\\
128	-2.08340504635187\\
129	-2.08340504635029\\
130	-2.08340504634978\\
131	-2.08340504634962\\
132	-2.08340504634957\\
133	-2.08340504634955\\
134	-2.08340504634955\\
135	-2.08340504634955\\
136	-2.08340504634954\\
137	-2.08340504634954\\
138	-2.08340504634954\\
139	-2.08340504634954\\
140	-2.08340504634954\\
141	-2.08340504634954\\
};
\addplot [color=mycolor3, line width=2.0pt, forget plot]
  table[row sep=crcr]{%
1	983784.158223048\\
2	983783.023485627\\
3	983781.909278624\\
4	983780.76542737\\
5	983779.556847025\\
6	983778.257717261\\
7	983776.847487018\\
8	983775.308245556\\
9	983773.623068632\\
10	983771.774986174\\
11	983769.746318778\\
12	983767.518223245\\
13	983765.070351849\\
14	983762.380569837\\
15	983759.42469881\\
16	983756.176266809\\
17	983752.606253255\\
18	983748.682820979\\
19	983744.371029843\\
20	983739.632527662\\
21	983734.425214798\\
22	983728.70287909\\
23	983722.414797912\\
24	983715.505304104\\
25	983707.913312493\\
26	983699.571803527\\
27	983690.40726046\\
28	983680.339056318\\
29	983669.278786734\\
30	983657.129544611\\
31	983643.785132405\\
32	983629.129207786\\
33	983613.034358352\\
34	983595.361101145\\
35	983575.956802793\\
36	983554.654516383\\
37	983531.2717315\\
38	983505.609034454\\
39	983477.448676443\\
40	983446.553048471\\
41	983412.663063139\\
42	983375.496445196\\
43	983334.745934862\\
44	983290.077410699\\
45	983241.12794206\\
46	983187.503785267\\
47	983128.778342466\\
48	983064.490107939\\
49	982994.140633484\\
50	982917.192552434\\
51	982833.067711126\\
52	982741.145467132\\
53	982640.761225429\\
54	982531.205296857\\
55	982411.72217759\\
56	982281.510363753\\
57	982139.722831353\\
58	981985.468327959\\
59	981817.813638172\\
60	981635.786999115\\
61	981438.3828535\\
62	981224.568134846\\
63	980993.29028022\\
64	980747.981849749\\
65	980492.189966108\\
66	980228.273961777\\
67	979957.400382687\\
68	979680.482152174\\
69	979397.618211536\\
70	979109.361865037\\
71	978816.715967409\\
72	978520.772380803\\
73	978222.664991339\\
74	929819.989385937\\
75	874344.724899122\\
76	818887.035887545\\
77	763448.028750405\\
78	708028.656943608\\
79	652629.544805287\\
80	597250.996662698\\
81	541893.019281155\\
82	486555.344569724\\
83	431237.453963921\\
84	375938.605709311\\
85	320657.865164995\\
86	265394.137514242\\
87	210146.201887666\\
88	154912.745760829\\
89	99692.3985003648\\
90	44483.7630450699\\
91	-7.31816322966915\\
92	-7.31704800863728\\
93	-7.31602185544571\\
94	-7.31509377239636\\
95	-7.3142661302596\\
96	-7.31353728129609\\
97	-7.31290292809672\\
98	-7.31235697240584\\
99	-7.31189213622592\\
100	-7.31150046310892\\
101	-7.3111737336268\\
102	-7.31090380311996\\
103	-7.31068286375556\\
104	-7.31050363413484\\
105	-7.31035948272128\\
106	-7.31024449401688\\
107	-7.31015348796402\\
108	-7.31008200343685\\
109	-7.31002650969596\\
110	-7.30998509341952\\
111	-7.30995579650446\\
112	-7.30993629550013\\
113	-7.30992414611566\\
114	-7.30991709808291\\
115	-7.30991331240067\\
116	-7.3099114418606\\
117	-7.30991059809199\\
118	-7.30991025378926\\
119	-7.30991012811098\\
120	-7.30991008764375\\
121	-7.30991007461357\\
122	-7.30991007041793\\
123	-7.30991006906696\\
124	-7.30991006863195\\
125	-7.30991006849188\\
126	-7.30991006844678\\
127	-7.30991006843226\\
128	-7.30991006842758\\
129	-7.30991006842608\\
130	-7.30991006842559\\
131	-7.30991006842544\\
132	-7.30991006842539\\
133	-7.30991006842537\\
134	-7.30991006842537\\
135	-7.30991006842536\\
136	-7.30991006842536\\
137	-7.30991006842536\\
138	-7.30991006842536\\
139	-7.30991006842536\\
140	-7.30991006842536\\
141	-7.30991006842536\\
};
\addplot [color=mycolor4, line width=2.0pt, forget plot]
  table[row sep=crcr]{%
1	384802.829089865\\
2	384800.972397342\\
3	384798.818823749\\
4	384796.517541259\\
5	384794.076555305\\
6	384791.46606567\\
7	384788.648804423\\
8	384785.587690053\\
9	384782.246722872\\
10	384778.590145894\\
11	384774.581362579\\
12	384770.181981572\\
13	384765.351027502\\
14	384760.044276442\\
15	384754.213667157\\
16	384747.806748983\\
17	384740.766137903\\
18	384733.028960518\\
19	384724.526271146\\
20	384715.182430756\\
21	384704.91443859\\
22	384693.631208598\\
23	384681.232783445\\
24	384667.609479182\\
25	384652.640953732\\
26	384636.195192244\\
27	384618.127402216\\
28	384598.278811047\\
29	384576.47535847\\
30	384552.526276079\\
31	384526.222546029\\
32	384497.335230855\\
33	384465.613666422\\
34	384430.783510139\\
35	384392.544636933\\
36	384350.568876091\\
37	384304.497582932\\
38	384253.93904056\\
39	384198.46568865\\
40	384137.611178493\\
41	384070.867256428\\
42	383997.680481524\\
43	383917.448788002\\
44	383829.517908621\\
45	383733.177682233\\
46	383627.658277205\\
47	383512.126372463\\
48	383385.681349934\\
49	383247.351566175\\
50	383096.090787269\\
51	382930.774889761\\
52	382750.198951581\\
53	382553.07488056\\
54	382338.029754211\\
55	382103.60507253\\
56	381848.25715528\\
57	381570.358945583\\
58	381268.203511641\\
59	380940.009566259\\
60	380583.929347586\\
61	380198.059221277\\
62	379780.453370947\\
63	379329.140936356\\
64	378852.994263984\\
65	378365.367848106\\
66	377879.724716303\\
67	339705.703008444\\
68	284088.570287575\\
69	228545.891909977\\
70	173076.363368675\\
71	117678.349517435\\
72	62349.6253718041\\
73	7087.44142730308\\
74	-7.86941506211133\\
75	-7.86094352338231\\
76	-7.8525656876457\\
77	-7.84427240154382\\
78	-7.8361016937026\\
79	-7.82809306542723\\
80	-7.82028284202391\\
81	-7.8127037129082\\
82	-7.80538456736236\\
83	-7.79835031918997\\
84	-7.79162175551249\\
85	-7.78521544913104\\
86	-7.77914375421632\\
87	-7.77341489062284\\
88	-7.76803311285821\\
89	-7.76299895395265\\
90	-7.75830953104675\\
91	-7.75395889778565\\
92	-7.74993842815692\\
93	-7.74624526222614\\
94	-7.74290239840924\\
95	-7.73991547609565\\
96	-7.73727834749294\\
97	-7.73497639233397\\
98	-7.7329888790416\\
99	-7.73129088897919\\
100	-7.72985497416504\\
101	-7.72865258239223\\
102	-7.72765524244196\\
103	-7.72683550055785\\
104	-7.72616761144266\\
105	-7.72562800094529\\
106	-7.72519552849077\\
107	-7.72485158372091\\
108	-7.7245800539267\\
109	-7.72436828881092\\
110	-7.72420999324442\\
111	-7.72409795657661\\
112	-7.72402336638251\\
113	-7.72397689234223\\
114	-7.72394993141071\\
115	-7.72393544990398\\
116	-7.72392829447819\\
117	-7.72392506681289\\
118	-7.72392374976874\\
119	-7.72392326902616\\
120	-7.72392311423541\\
121	-7.72392306439443\\
122	-7.72392304834603\\
123	-7.72392304317856\\
124	-7.72392304151467\\
125	-7.72392304097891\\
126	-7.7239230408064\\
127	-7.72392304075085\\
128	-7.72392304073297\\
129	-7.72392304072721\\
130	-7.72392304072536\\
131	-7.72392304072476\\
132	-7.72392304072457\\
133	-7.7239230407245\\
134	-7.72392304072449\\
135	-7.72392304072448\\
136	-7.72392304072448\\
137	-7.72392304072448\\
138	-7.72392304072448\\
139	-7.72392304072448\\
140	-7.72392304072448\\
141	-7.72392304072448\\
};
\addplot [color=mycolor5, line width=2.0pt, forget plot]
  table[row sep=crcr]{%
1	203033.506306142\\
2	203031.847499558\\
3	203030.016281485\\
4	203028.080497469\\
5	203026.038662809\\
6	203023.865270194\\
7	203021.528986979\\
8	203018.998252081\\
9	203016.242259281\\
10	203013.230499123\\
11	203009.931973184\\
12	203006.314425502\\
13	203002.343669331\\
14	202997.983003483\\
15	202993.192691894\\
16	202987.929479424\\
17	202982.146121576\\
18	202975.790910739\\
19	202968.807185581\\
20	202961.132813118\\
21	202952.699634934\\
22	202943.432870321\\
23	202933.250469891\\
24	202922.062413597\\
25	202909.769947357\\
26	202896.264752467\\
27	202881.428041965\\
28	202865.129577968\\
29	202847.226603908\\
30	202827.56268546\\
31	202805.966453865\\
32	202782.250245363\\
33	202756.208630506\\
34	202727.616827347\\
35	202696.228992873\\
36	202661.776387624\\
37	202623.965409302\\
38	202582.475492339\\
39	202536.95687195\\
40	202487.028213236\\
41	202432.274108527\\
42	202372.242449391\\
43	202306.441683852\\
44	202234.337974333\\
45	202155.352277903\\
46	202068.857377718\\
47	201974.174903194\\
48	201870.572386691\\
49	201757.260416385\\
50	201633.389958757\\
51	201498.049939787\\
52	201350.265191577\\
53	201188.994890666\\
54	201013.131635627\\
55	200821.501334219\\
56	200612.86409408\\
57	200385.916334657\\
58	200139.294360874\\
59	199871.579659393\\
60	199581.30619426\\
61	199266.969988031\\
62	198927.041274139\\
63	183764.104820067\\
64	128232.821070933\\
65	72812.7878723457\\
66	17497.0123937901\\
67	-28.7714570688404\\
68	-28.7260490798381\\
69	-28.6816299376207\\
70	-28.6377868738725\\
71	-28.5946000185351\\
72	-28.5521809554908\\
73	-28.5106390164638\\
74	-28.4700778019116\\
75	-28.4305934285599\\
76	-28.3915554444541\\
77	-28.3529874275776\\
78	-28.3150794586928\\
79	-28.2780176839289\\
80	-28.2419693527586\\
81	-28.207082419158\\
82	-28.1734855331108\\
83	-28.1412876700123\\
84	-28.1105777448458\\
85	-28.0814244634667\\
86	-28.0538765311283\\
87	-28.0279632532468\\
88	-28.0036955122299\\
89	-27.9810670739003\\
90	-27.9600561600725\\
91	-27.94062721593\\
92	-27.9227327994308\\
93	-27.9063496807182\\
94	-27.8915648328992\\
95	-27.8783926043254\\
96	-27.8667969819633\\
97	-27.8567054870716\\
98	-27.848019445069\\
99	-27.8406225162847\\
100	-27.8343880846144\\
101	-27.829185594278\\
102	-27.8248857835016\\
103	-27.8213647757567\\
104	-27.8185070510125\\
105	-27.8162073840981\\
106	-27.8143718856614\\
107	-27.8129183082225\\
108	-27.8117757870213\\
109	-27.8108883049459\\
110	-27.8102258322044\\
111	-27.8097571938878\\
112	-27.8094452545349\\
113	-27.8092509170663\\
114	-27.8091381830171\\
115	-27.8090776332052\\
116	-27.8090477165608\\
117	-27.8090342225647\\
118	-27.8090287167645\\
119	-27.8090267072553\\
120	-27.8090260603181\\
121	-27.8090258520277\\
122	-27.8090257849629\\
123	-27.8090257633692\\
124	-27.8090257564163\\
125	-27.8090257541776\\
126	-27.8090257534567\\
127	-27.8090257532246\\
128	-27.8090257531499\\
129	-27.8090257531258\\
130	-27.8090257531181\\
131	-27.8090257531156\\
132	-27.8090257531148\\
133	-27.8090257531145\\
134	-27.8090257531144\\
135	-27.8090257531144\\
136	-27.8090257531144\\
137	-27.8090257531144\\
138	-27.8090257531144\\
139	-27.8090257531144\\
140	-27.8090257531144\\
141	-27.8090257531144\\
};
\addplot [color=mycolor6, line width=2.0pt, forget plot]
  table[row sep=crcr]{%
1	99514.3855136335\\
2	99512.7730761402\\
3	99511.0192276974\\
4	99509.1475744595\\
5	99507.1581711548\\
6	99505.0347677528\\
7	99502.7530011467\\
8	99500.2851476226\\
9	99497.6022066046\\
10	99494.6744795985\\
11	99491.4714279749\\
12	99487.9612375715\\
13	99484.1102949229\\
14	99479.8826594836\\
15	99475.2395574488\\
16	99470.1388969544\\
17	99464.5347947996\\
18	99458.377102492\\
19	99451.6109199752\\
20	99444.1760868441\\
21	99436.0066423135\\
22	99427.0302464047\\
23	99417.1675556773\\
24	99406.3315474109\\
25	99394.4267864822\\
26	99381.3486293714\\
27	99366.9823597983\\
28	99351.2022505128\\
29	99333.87054575\\
30	99314.8363588744\\
31	99293.9344797919\\
32	99270.9840868399\\
33	99245.7873581139\\
34	99218.1279775855\\
35	99187.769531951\\
36	99154.45379498\\
37	99117.8988972545\\
38	99077.7973806708\\
39	99033.8141389948\\
40	98985.584248193\\
41	98932.7106933104\\
42	98874.7620024355\\
43	98811.2698028964\\
44	98741.7263204084\\
45	98665.5818485642\\
46	98582.2422239751\\
47	98491.0663516659\\
48	98391.3638361262\\
49	98282.3927858354\\
50	98163.3578731762\\
51	98033.4087474469\\
52	97891.6389161039\\
53	97737.0852282407\\
54	97568.7281142971\\
55	97385.4927565749\\
56	97186.2513855491\\
57	96969.8269161364\\
58	96734.9981546325\\
59	96480.5068191435\\
60	96205.0666218251\\
61	95907.37465747\\
62	40395.1253269382\\
63	-146.262708686438\\
64	-145.801109592756\\
65	-145.426498247357\\
66	-145.079982203848\\
67	-144.748770007528\\
68	-144.42983334187\\
69	-144.112399466173\\
70	-143.791805892655\\
71	-143.468715825862\\
72	-143.144249653445\\
73	-142.819582863036\\
74	-142.495895876126\\
75	-142.174352797879\\
76	-141.854512734905\\
77	-141.537594890171\\
78	-141.225186126732\\
79	-140.91879352948\\
80	-140.61979731687\\
81	-140.329442757679\\
82	-140.048833352165\\
83	-139.778923662088\\
84	-139.520513417548\\
85	-139.274244133383\\
86	-139.040598855675\\
87	-138.819905197278\\
88	-138.612341507122\\
89	-138.417945814265\\
90	-138.236627064519\\
91	-138.068178104711\\
92	-137.912289853049\\
93	-137.769022452567\\
94	-137.639881726321\\
95	-137.525217860885\\
96	-137.424731866865\\
97	-137.337726547134\\
98	-137.263250852577\\
99	-137.200197903511\\
100	-137.147379735753\\
101	-137.103586865706\\
102	-137.067635288239\\
103	-137.038401926315\\
104	-137.014849410799\\
105	-136.996041350205\\
106	-136.981149540109\\
107	-136.969454720739\\
108	-136.960342498131\\
109	-136.953321438138\\
110	-136.948095394571\\
111	-136.944402359202\\
112	-136.941945243107\\
113	-136.940414801491\\
114	-136.939527129587\\
115	-136.939050418472\\
116	-136.938814915204\\
117	-136.938708707751\\
118	-136.938665382004\\
119	-136.938649573297\\
120	-136.938644485851\\
121	-136.938642848246\\
122	-136.938642321043\\
123	-136.938642151305\\
124	-136.938642096654\\
125	-136.938642079057\\
126	-136.938642073391\\
127	-136.938642071567\\
128	-136.93864207098\\
129	-136.93864207079\\
130	-136.93864207073\\
131	-136.93864207071\\
132	-136.938642070704\\
133	-136.938642070702\\
134	-136.938642070701\\
135	-136.938642070701\\
136	-136.938642070701\\
137	-136.938642070701\\
138	-136.938642070701\\
139	-136.938642070701\\
140	-136.938642070701\\
141	-136.938642070701\\
};

\addplot[area legend, draw=none, fill=green, fill opacity=0.1, line width=2.0pt, forget plot]
table[row sep=crcr] {%
x	y\\
118	0\\
150	0\\
150	1200000\\
118	1200000\\
}--cycle;
\addplot [color=green, dashed, line width=2.0pt, forget plot]
  table[row sep=crcr]{%
118	0\\
118	1200000\\
};
\end{axis}
\end{tikzpicture}%

%% file: Figures/NS_state_2.tex
%
%
\definecolor{mycolor1}{rgb}{0.00000,0.44700,0.74100}%
\definecolor{mycolor2}{rgb}{0.85000,0.32500,0.09800}%
\definecolor{mycolor3}{rgb}{0.92900,0.69400,0.12500}%
\definecolor{mycolor4}{rgb}{0.49400,0.18400,0.55600}%
\definecolor{mycolor5}{rgb}{0.46600,0.67400,0.18800}%
\definecolor{mycolor6}{rgb}{0.30100,0.74500,0.93300}%
\begin{tikzpicture}

\begin{axis}[%
width=7cm,
height=3.5cm,
at={(0cm,0cm)},
scale only axis,
xmin=61,
xmax=140,
ymin=0,
ymax=4000,
xmajorgrids,
ymajorgrids,
axis background/.style={fill=white},
xmajorgrids,
ymajorgrids,
legend pos=north east,
legend style={legend cell align=left, align=left, draw=white!15!black}
]

\addplot [color=mycolor1, line width=2.0pt, mark options={solid, mycolor1}]
  table[row sep=crcr]{%
1	4.6595088243\\
2	7.25704562732303\\
3	9.23057250575148\\
4	10.8871132834846\\
5	12.410331803941\\
6	13.910433025333\\
7	15.4559850395323\\
8	17.0925338455949\\
9	18.8531833362379\\
10	20.7645799987799\\
11	22.8503353158311\\
12	25.1330330798772\\
13	27.6354565648139\\
14	30.3813851687548\\
15	33.3961538701621\\
16	36.7070836882444\\
17	40.3438448207566\\
18	44.3387885132268\\
19	48.7272694325119\\
20	53.5479722374286\\
21	58.843251390541\\
22	64.6594905314611\\
23	71.0474861021338\\
24	78.0628589110542\\
25	85.7664966763911\\
26	94.2250301363437\\
27	103.511344957687\\
28	113.705131342684\\
29	124.893472880895\\
30	137.171475775716\\
31	150.642939059574\\
32	165.421065760796\\
33	181.629214161851\\
34	199.401687251714\\
35	218.884557178936\\
36	240.236519904921\\
37	263.629773281045\\
38	289.250909363816\\
39	317.301808867322\\
40	348.00052215358\\
41	381.582116994941\\
42	418.299468420906\\
43	458.423960196214\\
44	502.246060783532\\
45	550.075728948676\\
46	602.242595414716\\
47	659.095857141163\\
48	721.003809921073\\
49	788.352933145968\\
50	861.546427973644\\
51	941.002097059782\\
52	1027.14944095545\\
53	1120.42583390773\\
54	1221.27163105874\\
55	1330.12405114667\\
56	1447.40967533991\\
57	1573.53540573172\\
58	1708.87773862994\\
59	1853.77023081836\\
60	2008.48907448498\\
61	2173.23675173455\\
62	2348.12381572611\\
63	2533.14894533317\\
64	2718.30966791458\\
65	2894.88605205364\\
66	3059.18272247767\\
67	3209.15798941828\\
68	3343.58757723721\\
69	3461.94476929507\\
70	3563.86497765878\\
71	3648.95064872982\\
72	3716.83408483396\\
73	3767.26320316939\\
74	3800.15579851952\\
75	3815.62885297036\\
76	3814.7075004887\\
77	3797.00619220726\\
78	3762.18357825176\\
79	3710.37764547338\\
80	3642.17916245086\\
81	3558.54774551288\\
82	3460.73207304629\\
83	3350.20057614526\\
84	3228.58019445719\\
85	3097.60077980546\\
86	2959.04403722056\\
87	2814.69692385678\\
88	2666.30996473061\\
89	2515.56109414092\\
90	2364.02553720877\\
91	2213.15202420575\\
92	2064.24536298947\\
93	1916.7474899041\\
94	1764.83377696153\\
95	1608.87109245649\\
96	1451.49428843147\\
97	1295.91602284219\\
98	1145.19230122715\\
99	1001.90781338632\\
100	868.056704851594\\
101	745.022723569764\\
102	633.614462430706\\
103	534.132173258688\\
104	446.451335296963\\
105	370.112466075639\\
106	304.409460851591\\
107	248.471010441887\\
108	199.577135115928\\
109	150.680158375991\\
110	107.025188440828\\
111	71.3486496001655\\
112	44.478106174458\\
113	25.8090114402987\\
114	13.8643302091198\\
115	6.85091008392979\\
116	3.09044430891534\\
117	1.26110329698239\\
118	0.460344572490373\\
119	0.148231890500189\\
120	0.0477308161961883\\
121	0.0153693367313214\\
122	0.00494892247414634\\
123	0.00159354934906025\\
124	0.000513120932734478\\
125	0.000165224043416223\\
126	5.32017555740423e-05\\
127	1.71308037934233e-05\\
128	5.51605239476668e-06\\
129	1.7761417864848e-06\\
130	5.71906663525096e-07\\
131	1.84149495559414e-07\\
132	5.92943378275439e-08\\
133	1.9092049135848e-08\\
134	6.14734559709217e-09\\
135	1.97932627223417e-09\\
136	6.3729490109012e-10\\
137	2.05189461558217e-10\\
138	6.60631101827879e-11\\
139	2.12691220334283e-11\\
140	6.84736050742336e-12\\
141	2.20432431731452e-12\\
};
\addlegendentry{0-24}

\addplot [color=mycolor2, line width=2.0pt, mark options={solid, mycolor2}]
  table[row sep=crcr]{%
1	4.3296088874\\
2	3.96022390797655\\
3	3.89949949462591\\
4	4.03260429886922\\
5	4.29150042493362\\
6	4.63848842450323\\
7	5.05429727656978\\
8	5.53035563963166\\
9	6.06421109384361\\
10	6.6569633628338\\
11	7.31187817389316\\
12	8.03366176631637\\
13	8.82809666452643\\
14	9.70187368772757\\
15	10.6625314356469\\
16	11.7184561807865\\
17	12.8789174266314\\
18	14.1541262340792\\
19	15.5553096856015\\
20	17.0947981864076\\
21	18.7861240879687\\
22	20.644131091658\\
23	22.6850944297544\\
24	24.9268521261098\\
25	27.388947812695\\
26	30.092785674487\\
27	33.0617981407625\\
28	36.3216269485436\\
29	39.9003181783317\\
30	43.8285318031401\\
31	48.1397661957305\\
32	52.8705978999381\\
33	58.060936781986\\
34	63.7542964266499\\
35	69.998079318849\\
36	76.8438759393414\\
37	84.3477763869278\\
38	92.5706924996921\\
39	101.578687662451\\
40	111.44331053226\\
41	122.24192776142\\
42	134.058049418619\\
43	146.981639172394\\
44	161.109399375062\\
45	176.545018937896\\
46	193.39936929064\\
47	211.790630746314\\
48	231.844328230881\\
49	253.6932515857\\
50	277.477231528359\\
51	303.342737912812\\
52	331.442262249623\\
53	361.933441669627\\
54	394.977876842139\\
55	430.739592075495\\
56	469.383082313906\\
57	511.070889494892\\
58	555.960650368529\\
59	604.201560163545\\
60	655.930202314117\\
61	711.265704857928\\
62	770.304200194665\\
63	833.112587799084\\
64	889.997726369785\\
65	939.549101334924\\
66	985.034731917636\\
67	1026.51001981295\\
68	1063.83843672069\\
69	1097.4649670605\\
70	1127.23223799456\\
71	1152.70209143317\\
72	1173.55396212255\\
73	1189.57202129322\\
74	1200.62654913274\\
75	1206.66471499322\\
76	1207.84727645942\\
77	1203.48665803095\\
78	1193.30847111174\\
79	1177.40559973778\\
80	1156.04598484531\\
81	1129.59286127082\\
82	1098.47667582355\\
83	1063.1820992903\\
84	1024.23723042684\\
85	982.201707247966\\
86	937.653620496569\\
87	891.175967904537\\
88	843.343470422333\\
89	794.710435226312\\
90	745.80016797729\\
91	697.096259347708\\
92	643.517534179478\\
93	567.930847151909\\
94	487.120376830005\\
95	408.726552269882\\
96	336.336636458845\\
97	271.609960395855\\
98	215.182613564809\\
99	167.079634991492\\
100	126.931590845187\\
101	94.1121554077817\\
102	67.8397599523775\\
103	47.2577156948946\\
104	31.4970714466605\\
105	19.7236430219275\\
106	11.1704516200004\\
107	5.15730981118458\\
108	1.41900623798125\\
109	0.390394482858723\\
110	0.107377266627494\\
111	0.0295151720361097\\
112	0.00810095723219599\\
113	0.00221633519176265\\
114	0.000602467383299615\\
115	0.000161811740770117\\
116	4.25659816188567e-05\\
117	1.08291951192013e-05\\
118	2.61934712342784e-06\\
119	5.89180196530904e-07\\
120	1.19757625471263e-07\\
121	1.93124352859557e-08\\
122	9.21972490440989e-10\\
123	-1.16052520659308e-09\\
124	-7.74700334246518e-10\\
125	-3.59793221718935e-10\\
126	-1.4621354920358e-10\\
127	-5.54342956018017e-11\\
128	-2.01482606514789e-11\\
129	-7.12011706943747e-12\\
130	-2.4666605709471e-12\\
131	-8.42129030328846e-13\\
132	-2.84332363224909e-13\\
133	-9.51765797557371e-14\\
134	-3.16427930382884e-14\\
135	-1.0462788805965e-14\\
136	-3.44426560612037e-15\\
137	-1.12972067848164e-15\\
138	-3.69441444352877e-16\\
139	-1.20514445851241e-16\\
140	-3.92308034682983e-17\\
141	-1.27482774889161e-17\\
};
\addlegendentry{24-45}

\addplot [color=mycolor3, line width=2.0pt, mark options={solid, mycolor3}]
  table[row sep=crcr]{%
1	4.8417769521\\
2	3.10050284714218\\
3	2.3730137105326\\
4	2.10729677210785\\
5	2.06414456909919\\
6	2.13717415725919\\
7	2.27792470886025\\
8	2.46408076791808\\
9	2.68559574351201\\
10	2.93843653988166\\
11	3.22167504867669\\
12	3.5360982541545\\
13	3.88353019648157\\
14	4.26649846715358\\
15	4.68807291507581\\
16	5.15179346169021\\
17	5.66164593878754\\
18	6.22206518560832\\
19	6.83795472076256\\
20	7.51471742356909\\
21	8.25829432131283\\
22	9.07520999675987\\
23	9.97262390449867\\
24	10.9583873181901\\
25	12.0411058790925\\
26	13.2302078603086\\
27	14.5360183437703\\
28	15.9698395505314\\
29	17.544037581053\\
30	19.2721358162863\\
31	21.1689152042642\\
32	23.2505216100989\\
33	25.5345803375343\\
34	28.0403178340428\\
35	30.7886904642024\\
36	33.8025200719628\\
37	37.1066358444767\\
38	40.7280217303485\\
39	44.6959683441148\\
40	49.0422278959669\\
41	53.8011702093673\\
42	59.0099373163549\\
43	64.7085934370322\\
44	70.9402663413825\\
45	77.7512751433425\\
46	85.1912384747057\\
47	93.3131557172517\\
48	102.17345252588\\
49	111.831980248838\\
50	122.351957046345\\
51	133.799836539859\\
52	146.245087719713\\
53	159.759867647665\\
54	174.41856628775\\
55	190.297200690996\\
56	207.472633894585\\
57	226.021592468926\\
58	246.019455907087\\
59	267.538791310761\\
60	290.647608459805\\
61	315.407313793891\\
62	341.870347572331\\
63	370.077497030387\\
64	395.560198765489\\
65	416.389649807208\\
66	432.970544757251\\
67	446.659976510588\\
68	458.26254848772\\
69	468.91891238262\\
70	478.637810752773\\
71	486.973242616871\\
72	493.655123789712\\
73	498.531775901942\\
74	501.521343913035\\
75	487.882565096159\\
76	464.769734435626\\
77	436.70403173317\\
78	405.674003000374\\
79	372.816117328384\\
80	338.911792977278\\
81	304.575848348354\\
82	270.33275646546\\
83	236.645927357851\\
84	203.926695829586\\
85	172.534950613585\\
86	142.776982664389\\
87	114.903199014534\\
88	89.1069201788487\\
89	65.5247433740299\\
90	44.2385550185843\\
91	25.279052966697\\
92	10.2622010830519\\
93	4.16543593026655\\
94	1.69024233195955\\
95	0.685412135370207\\
96	0.277549041953904\\
97	0.112050786201266\\
98	0.0449467557022209\\
99	0.0177835836173756\\
100	0.00682847646397782\\
101	0.00244563758073501\\
102	0.000723000363203692\\
103	7.25993779467981e-05\\
104	-0.000149754213539224\\
105	-0.00020495174747545\\
106	-0.0001981992825114\\
107	-0.000171475124632302\\
108	-0.000141103568867853\\
109	-0.000112781905944145\\
110	-8.72058230508395e-05\\
111	-6.47025509716883e-05\\
112	-4.57702932654308e-05\\
113	-3.07321611894327e-05\\
114	-1.95253154867588e-05\\
115	-1.17129763942876e-05\\
116	-6.62601811800945e-06\\
117	-3.53393742379132e-06\\
118	-1.77908422662191e-06\\
119	-8.47987944606818e-07\\
120	-3.8475103006881e-07\\
121	-1.69239410431824e-07\\
122	-7.29069820315174e-08\\
123	-3.09512678477651e-08\\
124	-1.30012452590378e-08\\
125	-5.41858478338824e-09\\
126	-2.245050732406e-09\\
127	-9.26014441044807e-10\\
128	-3.80638516999983e-10\\
129	-1.5604511167453e-10\\
130	-6.3839203361298e-11\\
131	-2.60748492245185e-11\\
132	-1.06366630021684e-11\\
133	-4.33467328068411e-12\\
134	-1.7650898432513e-12\\
135	-7.18304867771904e-13\\
136	-2.92172206067556e-13\\
137	-1.18795929126299e-13\\
138	-4.82871775861786e-14\\
139	-1.96226350818678e-14\\
140	-7.97259871921315e-15\\
141	-3.23874540881022e-15\\
};
\addlegendentry{45-64}

\addplot [color=mycolor4, line width=2.0pt, mark options={solid, mycolor4}]
  table[row sep=crcr]{%
1	0.1709101349\\
2	1.87582768989844\\
3	2.36359200427599\\
4	2.56591118709034\\
5	2.72826638743079\\
6	2.91594742008822\\
7	3.14373187537152\\
8	3.41308783615855\\
9	3.72309784759522\\
10	4.07341648887533\\
11	4.46484463839883\\
12	4.89926678179909\\
13	5.37947791961844\\
14	5.90903953918148\\
15	6.49218769372962\\
16	7.13378608028685\\
17	7.83931259621999\\
18	8.61486992904243\\
19	9.46721362184796\\
20	10.4037933786278\\
21	11.4328049944354\\
22	12.5632513683921\\
23	13.8050117538086\\
24	15.1689188483258\\
25	16.6668436144795\\
26	18.3117879025588\\
27	20.1179850571237\\
28	22.101008746521\\
29	24.2778902728701\\
30	26.6672446041426\\
31	29.2894053220658\\
32	32.1665685986642\\
33	35.3229461971487\\
34	38.7849273344956\\
35	42.5812490366565\\
36	46.7431743545237\\
37	51.30467747951\\
38	56.302634390243\\
39	61.7770171631173\\
40	67.7710894743917\\
41	74.3316000937311\\
42	81.5089703008872\\
43	89.357470129972\\
44	97.9353771406443\\
45	107.305110014271\\
46	117.533327659531\\
47	128.690982673636\\
48	140.853315936287\\
49	154.09977681734\\
50	168.513850973232\\
51	184.18277502818\\
52	201.197114645812\\
53	219.650179690573\\
54	239.637247492694\\
55	261.254562854263\\
56	284.598081616699\\
57	309.761923672349\\
58	336.836501644304\\
59	365.906293560417\\
60	397.047232271967\\
61	430.323691744562\\
62	465.785061351877\\
63	503.461914604186\\
64	532.514467802375\\
65	547.246946675751\\
66	546.913116771375\\
67	541.266710890486\\
68	486.733156270109\\
69	406.173214617472\\
70	324.003855336784\\
71	243.289451253577\\
72	164.962928984428\\
73	89.6532593884011\\
74	17.9070295038582\\
75	1.99640657928803\\
76	0.215140635856772\\
77	0.0157939447273566\\
78	-0.00640241153198442\\
79	-0.00872544480712123\\
80	-0.00878712766095071\\
81	-0.00856293941008936\\
82	-0.00827785563483224\\
83	-0.00796104016894593\\
84	-0.00761988488896591\\
85	-0.00725943171261681\\
86	-0.00688446376545026\\
87	-0.00649965088427829\\
88	-0.00610948125279775\\
89	-0.00571817884719913\\
90	-0.00532963247517456\\
91	-0.00494734102462779\\
92	-0.0045743758899818\\
93	-0.00420531486778973\\
94	-0.00381369253564634\\
95	-0.003413904840866\\
96	-0.00301935074127389\\
97	-0.00264000286423193\\
98	-0.00228308905902374\\
99	-0.00195360561801674\\
100	-0.00165464127316238\\
101	-0.00138764606531857\\
102	-0.00115270135353088\\
103	-0.000948798784347901\\
104	-0.000774117003002644\\
105	-0.000626280943734612\\
106	-0.000502591117110872\\
107	-0.000400215070458406\\
108	-0.000316338032058016\\
109	-0.000247182447189044\\
110	-0.00018597021121946\\
111	-0.000132857966334203\\
112	-8.94650246510441e-05\\
113	-5.64905798848949e-05\\
114	-3.32856392261313e-05\\
115	-1.82081800861233e-05\\
116	-9.19402084730677e-06\\
117	-4.25703152001732e-06\\
118	-1.79366308363192e-06\\
119	-6.81561813395638e-07\\
120	-2.31098674850928e-07\\
121	-7.57148773809901e-08\\
122	-2.45254708999302e-08\\
123	-7.9133554351079e-09\\
124	-2.54987223255809e-09\\
125	-8.21243922587876e-10\\
126	-2.64456185280508e-10\\
127	-8.51546492015109e-11\\
128	-2.74189755276611e-11\\
129	-8.82849926749288e-12\\
130	-2.84260225194849e-12\\
131	-9.15246476951508e-13\\
132	-2.94680369825918e-13\\
133	-9.48753324554926e-14\\
134	-3.05450984331835e-14\\
135	-9.83359331221644e-15\\
136	-3.1656353115184e-15\\
137	-1.01901757189562e-15\\
138	-3.27995081773721e-16\\
139	-1.05562254762257e-16\\
140	-3.39698815193003e-17\\
141	-1.09297132879195e-17\\
};
\addlegendentry{65-74}

\addplot [color=mycolor5, line width=2.0pt, mark options={solid, mycolor5}]
  table[row sep=crcr]{%
1	1.4936938584\\
2	1.77720424756331\\
3	1.97208819133525\\
4	2.09210161269225\\
5	2.20766512192823\\
6	2.34838321957669\\
7	2.52242784280443\\
8	2.73067517455431\\
9	2.97243980366291\\
10	3.24737063818024\\
11	3.5559288165295\\
12	3.89940842664589\\
13	4.27984283233873\\
14	4.699907646182\\
15	5.16284983107652\\
16	5.67244583062366\\
17	6.23298434241356\\
18	6.84926842428813\\
19	7.52663251013344\\
20	8.27097108868828\\
21	9.08877681832409\\
22	9.98718662761832\\
23	10.9740349108023\\
24	12.0579133172633\\
25	13.2482368999631\\
26	14.555316563485\\
27	15.9904378630369\\
28	17.5659462676596\\
29	19.2953390245193\\
30	21.1933637524246\\
31	23.2761238541208\\
32	25.5611907684891\\
33	28.0677229834616\\
34	30.8165915945138\\
35	33.8305120167778\\
36	37.134181234456\\
37	40.7544196913252\\
38	44.720316581341\\
39	49.0633768780799\\
40	53.8176679341207\\
41	59.0199628735419\\
42	64.7098772787124\\
43	70.9299948222382\\
44	77.7259765022092\\
45	85.1466469906876\\
46	93.244050290844\\
47	102.073465410331\\
48	111.693371096545\\
49	122.165346851506\\
50	133.553895470843\\
51	145.926170270688\\
52	159.351588038229\\
53	173.901306655308\\
54	189.647544424266\\
55	206.662716539351\\
56	225.018363113808\\
57	244.783842968599\\
58	266.024768351151\\
59	288.801158282396\\
60	313.165292785111\\
61	339.159257326972\\
62	366.8121769348\\
63	396.13715309379\\
64	371.683565842721\\
65	258.494701376881\\
66	145.265064498161\\
67	40.2858943383112\\
68	3.14785396152633\\
69	0.205095543909438\\
70	-0.0275861627323931\\
71	-0.0453734728932767\\
72	-0.0460155919766466\\
73	-0.0451893655371029\\
74	-0.0441431502125044\\
75	-0.0429833807403971\\
76	-0.0424450623518582\\
77	-0.0419324253082961\\
78	-0.0412317431346358\\
79	-0.0403300094318175\\
80	-0.0392450899042682\\
81	-0.0379976961738364\\
82	-0.0366087739397762\\
83	-0.0350996580517246\\
84	-0.0334921000288782\\
85	-0.0318080331332898\\
86	-0.0300691965077666\\
87	-0.0282967131426223\\
88	-0.0265106803605133\\
89	-0.0247298077609606\\
90	-0.0229711223688402\\
91	-0.0212497504625883\\
92	-0.0195787782522143\\
93	-0.0179350308311381\\
94	-0.0162064682764303\\
95	-0.0144568344972616\\
96	-0.0127415435408351\\
97	-0.0111014535099161\\
98	-0.00956599886263507\\
99	-0.00815517781140869\\
100	-0.00688085194795411\\
101	-0.00574790115760497\\
102	-0.00475541823814401\\
103	-0.00389794603475433\\
104	-0.00316669548851421\\
105	-0.00255067507437978\\
106	-0.00203767773042937\\
107	-0.00161509399129916\\
108	-0.00127054164785811\\
109	-0.000988191576004691\\
110	-0.000740801759984048\\
111	-0.000527357977934811\\
112	-0.000353740390050547\\
113	-0.000222376705258916\\
114	-0.000130360741703842\\
115	-7.08828578899487e-05\\
116	-3.55351750764062e-05\\
117	-1.63106922250265e-05\\
118	-6.79866742304007e-06\\
119	-2.54840563507313e-06\\
120	-8.48936662867433e-07\\
121	-2.75581380157027e-07\\
122	-8.89086728162316e-08\\
123	-2.86410837986258e-08\\
124	-9.22312917591194e-09\\
125	-2.96981480415895e-09\\
126	-9.56247511941767e-10\\
127	-3.07897890631434e-10\\
128	-9.9137691133251e-11\\
129	-3.19201899061487e-11\\
130	-1.0277443642559e-11\\
131	-3.30899198985465e-12\\
132	-1.0653562044946e-12\\
133	-3.42988382349877e-13\\
134	-1.10419451831924e-13\\
135	-3.55458202236456e-14\\
136	-1.14420063217751e-14\\
137	-3.68280597332551e-15\\
138	-1.18524686289871e-15\\
139	-3.81399279745086e-16\\
140	-1.22709032020218e-16\\
141	-3.94711004667069e-17\\
};
\addlegendentry{75-84}

\addplot [color=mycolor6, line width=2.0pt, mark options={solid, mycolor6}]
  table[row sep=crcr]{%
1	1.6144863665\\
2	1.91488375921671\\
3	2.11256899616501\\
4	2.26740672199893\\
5	2.41416297476468\\
6	2.57565533323225\\
7	2.76427135446697\\
8	2.98569224151827\\
9	3.24225913212804\\
10	3.53510851648303\\
11	3.8652934316825\\
12	4.23428667919582\\
13	4.64416346595059\\
14	5.09763962630998\\
15	5.59805718477238\\
16	6.14936027127337\\
17	6.75607908726744\\
18	7.42332758020236\\
19	8.15681532377509\\
20	8.96287225740088\\
21	9.84848456635727\\
22	10.8213401860683\\
23	11.8898827804771\\
24	13.0633734052453\\
25	14.3519593623189\\
26	15.7667499714428\\
27	17.3198991324327\\
28	19.0246946398069\\
29	20.8956542482165\\
30	22.9486284797096\\
31	25.200910115944\\
32	27.6713502308905\\
33	30.3804804907944\\
34	33.3506412744961\\
35	36.6061149431006\\
36	40.1732633061732\\
37	44.0806679833478\\
38	48.3592719352008\\
39	53.0425199239504\\
40	58.1664950503123\\
41	63.7700477842033\\
42	69.894913050048\\
43	76.5858099285775\\
44	83.8905173837548\\
45	91.8599181057806\\
46	100.548001073197\\
47	110.011811779509\\
48	120.311337252525\\
49	131.509311040888\\
50	143.670921292971\\
51	156.863402972371\\
52	171.155493235427\\
53	186.61672716962\\
54	203.316549632068\\
55	221.323218060548\\
56	240.70247113991\\
57	261.515939440555\\
58	283.819277011407\\
59	307.659997871502\\
60	333.075008905866\\
61	360.087841347813\\
62	388.705597302223\\
63	219.080020064214\\
64	40.5792764388159\\
65	7.22721849301416\\
66	1.00737909732136\\
67	-0.142497040392368\\
68	-0.345631036494566\\
69	-0.382181908606014\\
70	-0.392188783968245\\
71	-0.396559893312626\\
72	-0.398754851128675\\
73	-0.399366656733443\\
74	-0.398501464496706\\
75	-0.39619547694192\\
76	-0.394060474542489\\
77	-0.390738300306981\\
78	-0.385606866780462\\
79	-0.378629414737117\\
80	-0.369925924404229\\
81	-0.359653821694707\\
82	-0.347984366163532\\
83	-0.33509857884436\\
84	-0.321185202572998\\
85	-0.306437810340645\\
86	-0.291051133471953\\
87	-0.275217084763543\\
88	-0.259120879702289\\
89	-0.242937535096092\\
90	-0.226828920583288\\
91	-0.210941458640338\\
92	-0.19540450762271\\
93	-0.179873075766009\\
94	-0.162836854774896\\
95	-0.145168550833608\\
96	-0.127680826412039\\
97	-0.110924127584749\\
98	-0.0952554229523057\\
99	-0.080897415010493\\
100	-0.0679729077454926\\
101	-0.0565264257840368\\
102	-0.046540831586195\\
103	-0.0379519866331686\\
104	-0.0306621677753493\\
105	-0.0245520906103981\\
106	-0.0194912221943837\\
107	-0.0153461647728997\\
108	-0.0119870627602578\\
109	-0.00926663013027556\\
110	-0.00696198741805348\\
111	-0.00499724402823361\\
112	-0.00339326385517521\\
113	-0.00216611126548475\\
114	-0.00129345561134145\\
115	-0.000719017788494923\\
116	-0.000370198889344597\\
117	-0.000175557851451101\\
118	-7.62134917409845e-05\\
119	-3.00859944340828e-05\\
120	-1.07235406516105e-05\\
121	-3.64647578108546e-06\\
122	-1.21030724655504e-06\\
123	-3.96468444768285e-07\\
124	-1.28922787984073e-07\\
125	-4.17481273885174e-08\\
126	-1.34866408648872e-08\\
127	-4.35080782648042e-09\\
128	-1.40244995211135e-09\\
129	-4.51857087042445e-10\\
130	-1.4554400543615e-10\\
131	-4.68721198476541e-11\\
132	-1.50934548728695e-11\\
133	-4.85994107247039e-12\\
134	-1.5647636110533e-12\\
135	-5.03783997428172e-13\\
136	-1.62187470107796e-13\\
137	-5.22113117033205e-14\\
138	-1.68066564702918e-14\\
139	-5.40954158000872e-15\\
140	-1.74097593618667e-15\\
141	-5.60229918694676e-16\\
};
\addlegendentry{84+}

\addplot[area legend, draw=none, fill=green, fill opacity=0.1, line width=2.0pt, forget plot]
table[row sep=crcr] {%
x	y\\
118	0\\
150	0\\
150	4000\\
118	4000\\
}--cycle;

\addplot [color=green, dashed, line width=2.0pt, mark options={solid, green}]
  table[row sep=crcr]{%
118	0\\
118	4000\\
};
\addlegendentry{Erad.}
\end{axis}
\end{tikzpicture}%

%% file: Figures/NS_state_3.tex
%
%
\definecolor{mycolor1}{rgb}{0.00000,0.44700,0.74100}%
\definecolor{mycolor2}{rgb}{0.85000,0.32500,0.09800}%
\definecolor{mycolor3}{rgb}{0.92900,0.69400,0.12500}%
\definecolor{mycolor4}{rgb}{0.49400,0.18400,0.55600}%
\definecolor{mycolor5}{rgb}{0.46600,0.67400,0.18800}%
\definecolor{mycolor6}{rgb}{0.30100,0.74500,0.93300}%
\begin{tikzpicture}

\begin{axis}[%
width=7cm,
height=3.5cm,
at={(0cm,0cm)},
scale only axis,
xmin=61,
xmax=140,
ymin=0,
ymax=1000000,
xmajorgrids,
ymajorgrids,
axis background/.style={fill=white},
axis x line*=bottom,
axis y line*=left
]
\addplot [color=mycolor1, line width=2.0pt, forget plot]
  table[row sep=crcr]{%
1	0\\
2	4.29463568177537\\
3	10.9834023030202\\
4	19.4911544162207\\
5	29.5257282182798\\
6	40.9642415461055\\
7	53.7853873518582\\
8	68.0310573035046\\
9	83.7851224878908\\
10	101.161965611055\\
11	120.300529254238\\
12	141.361518531931\\
13	164.5264538773\\
14	189.997854902757\\
15	218.000158520477\\
16	248.781152709582\\
17	282.613807035573\\
18	319.798437871262\\
19	360.665179499164\\
20	405.576752343279\\
21	454.93153219867\\
22	509.166932663109\\
23	568.763118800506\\
24	634.247074389002\\
25	706.19704850482\\
26	785.247409994935\\
27	872.093940777221\\
28	967.499600963008\\
29	1072.30080054838\\
30	1187.41421384591\\
31	1313.84417387\\
32	1452.69068445472\\
33	1605.15808784897\\
34	1772.56442474104\\
35	1956.35152191561\\
36	2158.09583980334\\
37	2379.52010775807\\
38	2622.50576865146\\
39	2889.10624590808\\
40	3181.56103495083\\
41	3502.31060664882\\
42	3854.01209214178\\
43	4239.55569566054\\
44	4662.08175389484\\
45	5124.99832622178\\
46	5631.99915877767\\
47	6187.08181595916\\
48	6794.5657144823\\
49	7459.10972663968\\
50	8185.72893999198\\
51	8979.81006969937\\
52	9847.12491661685\\
53	10793.84114916\\
54	11826.5295604339\\
55	12952.1668157025\\
56	14178.1325615878\\
57	15512.1996214985\\
58	16962.5158575682\\
59	18537.5761458625\\
60	20246.1827993291\\
61	22097.3926952638\\
62	24100.449337258\\
63	26264.6981249527\\
64	28599.483240316\\
65	31104.9296584145\\
66	33773.125256411\\
67	36592.7519107284\\
68	39550.6096870533\\
69	42632.3702450454\\
70	45823.2197734362\\
71	49108.0084228884\\
72	52471.2199217799\\
73	55896.9990941941\\
74	59369.2584213134\\
75	62871.8346163654\\
76	66388.6722141222\\
77	69904.660607941\\
78	73404.333833568\\
79	76871.911307032\\
80	80291.7396258474\\
81	83648.7098946676\\
82	86928.5976895956\\
83	90118.3294845991\\
84	93206.1851959957\\
85	96181.9442786437\\
86	99036.9805793522\\
87	101764.310129608\\
88	104358.595986422\\
89	106816.114653086\\
90	109134.689172839\\
91	111313.59446255\\
92	113353.440723336\\
93	115256.040788304\\
94	117022.693127316\\
95	118649.327692619\\
96	120132.212576323\\
97	121470.044394665\\
98	122664.48084755\\
99	123719.996333151\\
100	124643.447539591\\
101	125443.529144549\\
102	126130.211216206\\
103	126714.209096981\\
104	127206.514869233\\
105	127618.005845437\\
106	127959.13583639\\
107	137203.344098113\\
108	192623.358036613\\
109	247998.306842819\\
110	303328.187658179\\
111	358617.832002564\\
112	413874.593538376\\
113	469106.588688088\\
114	524321.376667813\\
115	579525.155320986\\
116	634722.469755406\\
117	689916.318195639\\
118	745108.480545453\\
119	745108.904841726\\
120	745109.041465991\\
121	745109.08545914\\
122	745109.099624946\\
123	745109.104186333\\
124	745109.105655096\\
125	745109.106128035\\
126	745109.106280321\\
127	745109.106329357\\
128	745109.106345146\\
129	745109.10635023\\
130	745109.106351867\\
131	745109.106352394\\
132	745109.106352564\\
133	745109.106352619\\
134	745109.106352636\\
135	745109.106352642\\
136	745109.106352644\\
137	745109.106352645\\
138	745109.106352645\\
139	745109.106352645\\
140	745109.106352645\\
141	745109.106352645\\
};
\addplot [color=mycolor2, line width=2.0pt, forget plot]
  table[row sep=crcr]{%
1	0\\
2	3.13035311771256\\
3	5.99363697996894\\
4	8.81301644771832\\
5	11.7286320997933\\
6	14.8314323952082\\
7	18.1851086922459\\
8	21.8394191967353\\
9	25.8379249437466\\
10	30.2224138428075\\
11	35.0354689152852\\
12	40.3220343379938\\
13	46.1304569486418\\
14	52.5132644113878\\
15	59.5278219235952\\
16	67.2369451701362\\
17	75.7095131997929\\
18	85.0211070106206\\
19	95.2546903059363\\
20	106.501344088687\\
21	118.861064375482\\
22	132.443631216484\\
23	147.369556815285\\
24	163.771120540924\\
25	181.793498842764\\
26	201.595998423233\\
27	223.353401437379\\
28	247.257431934985\\
29	273.518353213499\\
30	302.366706183846\\
31	334.055199242422\\
32	368.860760464201\\
33	407.086763153045\\
34	449.065435869127\\
35	495.160467955635\\
36	545.769821254805\\
37	601.328758073296\\
38	662.313094453735\\
39	729.242686343325\\
40	802.685154216735\\
41	883.259848986073\\
42	971.642058473118\\
43	1068.56744916448\\
44	1174.8367322327\\
45	1291.32053567606\\
46	1418.96445567475\\
47	1558.79424962783\\
48	1711.92112055339\\
49	1879.54702732146\\
50	2062.96993726491\\
51	2263.58891680757\\
52	2482.90893163022\\
53	2722.5452003921\\
54	2984.22691506895\\
55	3269.80010663332\\
56	3581.22939737128\\
57	3920.59834115716\\
58	4290.10801140896\\
59	4692.07345458592\\
60	5128.91758687902\\
61	5603.1620757499\\
62	6117.41471949415\\
63	6674.35282115129\\
64	7276.70205279021\\
65	7920.17984257547\\
66	8599.48380168527\\
67	9311.67435383587\\
68	10053.8519787562\\
69	10823.0184447672\\
70	11616.4972486416\\
71	12431.4981049225\\
72	13264.9139352098\\
73	14113.4058890269\\
74	14973.4790694095\\
75	15841.5447905937\\
76	16713.9761696971\\
77	17587.2625532752\\
78	18457.3961635088\\
79	19320.1708367151\\
80	20171.4475653539\\
81	21007.2810660221\\
82	21823.9886779534\\
83	22618.1989579872\\
84	23386.890885079\\
85	24127.4252591826\\
86	24837.5675044681\\
87	25515.5010108404\\
88	26159.8306817442\\
89	26769.576949963\\
90	27344.1610182443\\
91	38590.6193998055\\
92	94285.6273842554\\
93	149941.897382496\\
94	205543.517404826\\
95	261086.710600555\\
96	316573.224230185\\
97	372007.399183378\\
98	427394.776063701\\
99	482741.355374158\\
100	538053.155721234\\
101	593335.92860684\\
102	648594.972692751\\
103	703835.02155827\\
104	759060.189387631\\
105	814273.962104143\\
106	869479.222507111\\
107	915714.66260516\\
108	915718.391394819\\
109	915719.41735137\\
110	915719.699610719\\
111	915719.777245621\\
112	915719.798585403\\
113	915719.804442481\\
114	915719.806044915\\
115	915719.806480505\\
116	915719.806597497\\
117	915719.806628273\\
118	915719.806636102\\
119	915719.806637996\\
120	915719.806638422\\
121	915719.806638509\\
122	915719.806638523\\
123	915719.806638523\\
124	915719.806638522\\
125	915719.806638522\\
126	915719.806638522\\
127	915719.806638521\\
128	915719.806638521\\
129	915719.806638521\\
130	915719.806638521\\
131	915719.806638521\\
132	915719.806638521\\
133	915719.806638521\\
134	915719.806638521\\
135	915719.806638521\\
136	915719.806638521\\
137	915719.806638521\\
138	915719.806638521\\
139	915719.806638521\\
140	915719.806638521\\
141	915719.806638521\\
};
\addplot [color=mycolor3, line width=2.0pt, forget plot]
  table[row sep=crcr]{%
1	0\\
2	2.76332081289318\\
3	4.53285380309558\\
4	5.88719090711098\\
5	7.08987683975862\\
6	8.26793475218234\\
7	9.48767244104269\\
8	10.7877399204971\\
9	12.1940512268626\\
10	13.7267865607043\\
11	15.4038243359572\\
12	17.2425132723664\\
13	19.2606512408868\\
14	21.4770771369171\\
15	23.9120724142912\\
16	26.5876705648774\\
17	29.5279254004995\\
18	32.7591655449052\\
19	36.3102506933262\\
20	40.2128390992851\\
21	44.5016725719945\\
22	49.2148836105955\\
23	54.3943284535773\\
24	60.0859494156922\\
25	66.3401697260611\\
26	73.2123240642561\\
27	80.7631280564633\\
28	89.0591901062684\\
29	98.1735690719099\\
30	108.186381448341\\
31	119.185461855576\\
32	131.267080763061\\
33	144.536723481308\\
34	159.109934513798\\
35	175.113231368963\\
36	192.685091866286\\
37	211.97701881112\\
38	233.154685634665\\
39	256.399166169156\\
40	281.908251118668\\
41	309.897852952102\\
42	340.603499839258\\
43	374.281917818236\\
44	411.212698559806\\
45	451.700047810437\\
46	496.07460677053\\
47	544.695335210158\\
48	597.951440945963\\
49	656.264335299236\\
50	720.08958822309\\
51	789.918849824606\\
52	866.281696921875\\
53	949.747353988956\\
54	1040.92622730413\\
55	1140.47117932198\\
56	1249.07845729183\\
57	1367.48817608278\\
58	1496.48424029878\\
59	1636.89357546856\\
60	1789.58452294491\\
61	1955.4642389294\\
62	2135.47492578423\\
63	2330.58871481326\\
64	2541.8010155955\\
65	2767.55691901992\\
66	3005.20070083157\\
67	3252.30760590668\\
68	3507.22740530858\\
69	3768.76907699925\\
70	4036.39259682788\\
71	4309.56293023556\\
72	4587.49049896869\\
73	4869.23158110752\\
74	53257.3144608651\\
75	108734.544987685\\
76	164203.991529051\\
77	219660.2470113\\
78	275100.484708926\\
79	330523.012808389\\
80	385926.788106918\\
81	441311.213376304\\
82	496676.042280273\\
83	552021.327812163\\
84	607347.387444778\\
85	662654.77340978\\
86	717944.243336151\\
87	773216.729660635\\
88	828473.307733406\\
89	883715.163237396\\
90	938943.559814916\\
91	983452.570887936\\
92	983466.998263233\\
93	983472.855152991\\
94	983475.232469401\\
95	983476.197132139\\
96	983476.588313649\\
97	983476.746717692\\
98	983476.810667823\\
99	983476.836320039\\
100	983476.846469566\\
101	983476.850366745\\
102	983476.85176253\\
103	983476.852175164\\
104	983476.852216598\\
105	983476.85213113\\
106	983476.852014159\\
107	983476.851901042\\
108	983476.851803177\\
109	983476.851722645\\
110	983476.851658278\\
111	983476.851608508\\
112	983476.85157158\\
113	983476.851545458\\
114	983476.851527918\\
115	983476.851516775\\
116	983476.85151009\\
117	983476.851506308\\
118	983476.851504291\\
119	983476.851503276\\
120	983476.851502792\\
121	983476.851502573\\
122	983476.851502476\\
123	983476.851502434\\
124	983476.851502417\\
125	983476.851502409\\
126	983476.851502406\\
127	983476.851502405\\
128	983476.851502404\\
129	983476.851502404\\
130	983476.851502404\\
131	983476.851502404\\
132	983476.851502404\\
133	983476.851502404\\
134	983476.851502404\\
135	983476.851502404\\
136	983476.851502404\\
137	983476.851502404\\
138	983476.851502404\\
139	983476.851502404\\
140	983476.851502404\\
141	983476.851502404\\
};
\addplot [color=mycolor4, line width=2.0pt, forget plot]
  table[row sep=crcr]{%
1	0\\
2	0.144981564007577\\
3	1.73622969244256\\
4	3.74124399809646\\
5	5.91788388681085\\
6	8.23224826380028\\
7	10.705820809838\\
8	13.3726209055639\\
9	16.2679132934087\\
10	19.4261844469212\\
11	22.8816278222187\\
12	26.6691162535201\\
13	30.8251211676301\\
14	35.3884849657269\\
15	40.4010712273574\\
16	45.9083369387699\\
17	51.9598649176146\\
18	58.6098848336909\\
19	65.9178032129327\\
20	73.948757249055\\
21	82.7742036541361\\
22	92.4725515611032\\
23	103.129847183013\\
24	114.840517216603\\
25	127.708177640726\\
26	141.846514467332\\
27	157.380243063599\\
28	174.44615281769\\
29	193.19424412363\\
30	213.788964879216\\
31	236.410553895827\\
32	261.256498783352\\
33	288.543115969162\\
34	318.507260506374\\
35	351.408173188735\\
36	387.529472176316\\
37	427.181295800251\\
38	470.702602399393\\
39	518.463631880883\\
40	570.868532112646\\
41	628.358151158607\\
42	691.412993652765\\
43	760.5563361572\\
44	836.357492026533\\
45	919.435210956603\\
46	1010.46119186158\\
47	1110.16367982149\\
48	1219.33110838043\\
49	1338.81573725891\\
50	1469.53722237896\\
51	1612.48603981061\\
52	1768.72666768482\\
53	1939.40041018832\\
54	2125.72772544507\\
55	2329.00989450011\\
56	2550.62984201749\\
57	2792.05189115745\\
58	3054.82020615697\\
59	3340.55564748586\\
60	3650.95073758186\\
61	3987.76241205249\\
62	4352.80221437401\\
63	4747.92358459394\\
64	5175.00589799959\\
65	5626.73323671958\\
66	6090.95800767213\\
67	44248.8871998632\\
68	99899.0389894049\\
69	155502.930444272\\
70	211038.483609289\\
71	266504.333229639\\
72	321901.713531017\\
73	377232.650132562\\
74	384396.14363116\\
75	384411.334006767\\
76	384413.027540907\\
77	384413.210042816\\
78	384413.22344068\\
79	384413.218009571\\
80	384413.210607853\\
81	384413.203153809\\
82	384413.195889943\\
83	384413.188867911\\
84	384413.182114631\\
85	384413.175650749\\
86	384413.169492637\\
87	384413.163652607\\
88	384413.158139011\\
89	384413.152956391\\
90	384413.148105711\\
91	384413.14358463\\
92	384413.139387844\\
93	384413.135507442\\
94	384413.13194011\\
95	384413.128704988\\
96	384413.125809003\\
97	384413.123247714\\
98	384413.121008223\\
99	384413.119071499\\
100	384413.117414272\\
101	384413.116010654\\
102	384413.114833526\\
103	384413.1138557\\
104	384413.113050842\\
105	384413.112394166\\
106	384413.111862897\\
107	384413.111436554\\
108	384413.111097055\\
109	384413.110828708\\
110	384413.110619025\\
111	384413.110461268\\
112	384413.110348566\\
113	384413.110272674\\
114	384413.110224753\\
115	384413.110196517\\
116	384413.110181072\\
117	384413.110173272\\
118	384413.110169661\\
119	384413.11016814\\
120	384413.110167561\\
121	384413.110167365\\
122	384413.110167301\\
123	384413.11016728\\
124	384413.110167274\\
125	384413.110167272\\
126	384413.110167271\\
127	384413.110167271\\
128	384413.110167271\\
129	384413.110167271\\
130	384413.110167271\\
131	384413.110167271\\
132	384413.110167271\\
133	384413.110167271\\
134	384413.110167271\\
135	384413.110167271\\
136	384413.110167271\\
137	384413.110167271\\
138	384413.110167271\\
139	384413.110167271\\
140	384413.110167271\\
141	384413.110167271\\
};
\addplot [color=mycolor5, line width=2.0pt, forget plot]
  table[row sep=crcr]{%
1	0\\
2	1.22489296657866\\
3	2.68227660029449\\
4	4.29947341840747\\
5	6.01508638444028\\
6	7.82546637978117\\
7	9.75124124476728\\
8	11.8197401583686\\
9	14.0590108071114\\
10	16.4965388109992\\
11	19.1595218795925\\
12	22.0755358757182\\
13	25.2732178697597\\
14	28.7828723775511\\
15	32.6369980318831\\
16	36.8707561142531\\
17	41.5224047517269\\
18	46.6337189871589\\
19	52.2504125766382\\
20	58.4225737406134\\
21	65.2051244328696\\
22	72.6583108652577\\
23	80.8482318368696\\
24	89.8474106857598\\
25	99.7354162706203\\
26	110.599538196994\\
27	122.535521453989\\
28	135.648365669555\\
29	150.053194285241\\
30	165.87619906361\\
31	183.255665446565\\
32	202.343084356264\\
33	223.304356047664\\
34	246.321091556738\\
35	271.592017112046\\
36	299.334486555881\\
37	329.786106315813\\
38	363.206476732527\\
39	399.879052532184\\
40	440.11312386922\\
41	484.245917586925\\
42	532.64481606601\\
43	585.709688162202\\
44	643.875323167414\\
45	707.613954348502\\
46	777.437853295484\\
47	853.901969910989\\
48	937.60658525255\\
49	1029.1999354563\\
50	1129.38075448861\\
51	1238.90067137215\\
52	1358.56638372621\\
53	1489.24151391002\\
54	1631.8480367949\\
55	1787.36714935457\\
56	1956.83943212491\\
57	2141.36413160007\\
58	2342.09737147929\\
59	2560.24908031613\\
60	2797.078404833\\
61	3053.88736361116\\
62	3332.01248711863\\
63	18428.6888626552\\
64	73944.5383021146\\
65	129440.334752226\\
66	184843.311483491\\
67	202459.447454574\\
68	202492.483614126\\
69	202495.064989256\\
70	202495.23317639\\
71	202495.210554554\\
72	202495.173346362\\
73	202495.135611605\\
74	202495.098554389\\
75	202495.062355114\\
76	202495.0271069\\
77	202494.992300131\\
78	202494.957913745\\
79	202494.924101949\\
80	202494.891029613\\
81	202494.858846958\\
82	202494.827687219\\
83	202494.797666456\\
84	202494.768883232\\
85	202494.741418275\\
86	202494.715334325\\
87	202494.690676295\\
88	202494.667471778\\
89	202494.645731884\\
90	202494.625452382\\
91	202494.606615078\\
92	202494.589189372\\
93	202494.573133934\\
94	202494.558426441\\
95	202494.545136442\\
96	202494.533281219\\
97	202494.522832607\\
98	202494.513728939\\
99	202494.50588441\\
100	202494.499196815\\
101	202494.493554222\\
102	202494.488840696\\
103	202494.48494105\\
104	202494.481744567\\
105	202494.479147741\\
106	202494.477056078\\
107	202494.475385095\\
108	202494.474060649\\
109	202494.47301875\\
110	202494.472208391\\
111	202494.471600902\\
112	202494.471168445\\
113	202494.470878363\\
114	202494.470696005\\
115	202494.470589103\\
116	202494.470530976\\
117	202494.470501836\\
118	202494.470488461\\
119	202494.470482885\\
120	202494.470480796\\
121	202494.470480099\\
122	202494.470479873\\
123	202494.4704798\\
124	202494.470479777\\
125	202494.470479769\\
126	202494.470479767\\
127	202494.470479766\\
128	202494.470479766\\
129	202494.470479766\\
130	202494.470479766\\
131	202494.470479766\\
132	202494.470479766\\
133	202494.470479766\\
134	202494.470479766\\
135	202494.470479766\\
136	202494.470479766\\
137	202494.470479766\\
138	202494.470479766\\
139	202494.470479766\\
140	202494.470479766\\
141	202494.470479766\\
};
\addplot [color=mycolor6, line width=2.0pt, forget plot]
  table[row sep=crcr]{%
1	0\\
2	1.06753654407652\\
3	2.33370294413975\\
4	3.73058349518353\\
5	5.22984641015384\\
6	6.82614802805157\\
7	8.5292322102561\\
8	10.3570337636596\\
9	12.3312440408861\\
10	14.475102410168\\
11	16.8125997139099\\
12	19.3684230875351\\
13	22.1682335176085\\
14	25.239064166563\\
15	28.6097439707015\\
16	32.3113116919143\\
17	36.3774140440252\\
18	40.8446932171313\\
19	45.7531728642508\\
20	51.1466519434767\\
21	57.0731149184523\\
22	63.5851656836453\\
23	70.7404915781311\\
24	78.6023630911511\\
25	87.240174340836\\
26	96.730029081196\\
27	107.155376811026\\
28	118.607703474939\\
29	131.187281221339\\
30	145.003981681138\\
31	160.178157225051\\
32	176.841594619744\\
33	195.138545407538\\
34	215.226837153746\\
35	237.279069410184\\
36	261.48389779975\\
37	288.047408996887\\
38	317.194588518473\\
39	349.170882098255\\
40	384.243849937234\\
41	422.704911235428\\
42	464.871174041799\\
43	511.08734252375\\
44	561.727690161802\\
45	617.198083016708\\
46	677.938031986601\\
47	744.422746758259\\
48	817.165156847042\\
49	896.717856608899\\
50	983.674921304851\\
51	1078.67353014036\\
52	1182.3953196679\\
53	1295.56737707236\\
54	1418.96276778193\\
55	1553.40047580563\\
56	1699.74461858366\\
57	1858.90278152834\\
58	2031.82330164182\\
59	2219.49131569875\\
60	2422.92337786621\\
61	2643.16044601591\\
62	58072.2590374272\\
63	98724.4056923973\\
64	98869.2665796419\\
65	98896.0985563185\\
66	98900.8773640021\\
67	98901.5434668708\\
68	98901.4492444598\\
69	98901.2207050494\\
70	98900.9679973385\\
71	98900.7086728451\\
72	98900.4464580709\\
73	98900.1827919388\\
74	98899.9187212663\\
75	98899.6552226793\\
76	98899.3932488658\\
77	98899.1326867664\\
78	98898.8743213671\\
79	98898.6193489929\\
80	98898.368990275\\
81	98898.1243865105\\
82	98897.8865749032\\
83	98897.6564794156\\
84	98897.4349043152\\
85	98897.222529068\\
86	98897.0199051451\\
87	98896.8274552566\\
88	98896.6454752154\\
89	98896.4741383654\\
90	98896.3135023253\\
91	98896.1635176819\\
92	98896.0240382038\\
93	98895.894832125\\
94	98895.775895796\\
95	98895.6682242189\\
96	98895.5722353421\\
97	98895.4878097619\\
98	98895.414464107\\
99	98895.35147897\\
100	98895.2979876872\\
101	98895.253042394\\
102	98895.2156657853\\
103	98895.1848918874\\
104	98895.1597971369\\
105	98895.1395225868\\
106	98895.1232881642\\
107	98895.1104001074\\
108	98895.1002528606\\
109	98895.0923267314\\
110	98895.0861994165\\
111	98895.0815959859\\
112	98895.07829169\\
113	98895.0760479837\\
114	98895.0746156998\\
115	98895.0737604364\\
116	98895.0732850048\\
117	98895.0730402205\\
118	98895.0729241375\\
119	98895.0728737434\\
120	98895.0728538498\\
121	98895.0728467592\\
122	98895.072844348\\
123	98895.0728435477\\
124	98895.0728432856\\
125	98895.0728432003\\
126	98895.0728431727\\
127	98895.0728431638\\
128	98895.0728431609\\
129	98895.07284316\\
130	98895.0728431597\\
131	98895.0728431596\\
132	98895.0728431596\\
133	98895.0728431595\\
134	98895.0728431595\\
135	98895.0728431595\\
136	98895.0728431595\\
137	98895.0728431595\\
138	98895.0728431595\\
139	98895.0728431595\\
140	98895.0728431595\\
141	98895.0728431595\\
};

\addplot[area legend, draw=none, fill=green, fill opacity=0.1, line width=2.0pt, forget plot]
table[row sep=crcr] {%
x	y\\
118	0\\
150	0\\
150	1000000\\
118	1000000\\
}--cycle;
\addplot [color=green, dashed, line width=2.0pt, forget plot]
  table[row sep=crcr]{%
118	0\\
118	1000000\\
};
\end{axis}
\end{tikzpicture}%

%% file: Figures/NS_state_4.tex
%
%
\definecolor{mycolor1}{rgb}{0.00000,0.44700,0.74100}%
\definecolor{mycolor2}{rgb}{0.85000,0.32500,0.09800}%
\definecolor{mycolor3}{rgb}{0.92900,0.69400,0.12500}%
\definecolor{mycolor4}{rgb}{0.49400,0.18400,0.55600}%
\definecolor{mycolor5}{rgb}{0.46600,0.67400,0.18800}%
\definecolor{mycolor6}{rgb}{0.30100,0.74500,0.93300}%
\begin{tikzpicture}

\begin{axis}[%
width=7cm,
height=3.5cm,
at={(0cm,0cm)},
scale only axis,
xmin=61,
xmax=140,
ymin=0,
ymax=800,
xmajorgrids,
ymajorgrids,
axis background/.style={fill=white},
axis x line*=bottom,
axis y line*=left
]
\addplot [color=mycolor1, line width=2.0pt, forget plot]
  table[row sep=crcr]{%
1	0\\
2	0.00205352335266638\\
3	0.00525182455328961\\
4	0.00931989200713513\\
5	0.0141180246459586\\
6	0.0195874651244966\\
7	0.0257180252629923\\
8	0.0325297359848644\\
9	0.0400627010959333\\
10	0.048371613840375\\
11	0.0575229110143233\\
12	0.0675934353886098\\
13	0.0786699827885641\\
14	0.0908493900088023\\
15	0.104238973821805\\
16	0.118957216548155\\
17	0.135134641337862\\
18	0.152914847492578\\
19	0.172455692048125\\
20	0.193930613432432\\
21	0.217530099048604\\
22	0.243463302618713\\
23	0.27195981990942\\
24	0.303271633527648\\
25	0.337675239099494\\
26	0.375473966485274\\
27	0.417000510824364\\
28	0.462619689186677\\
29	0.512731439445091\\
30	0.567774078664699\\
31	0.628227838802702\\
32	0.694618697783339\\
33	0.76752252399592\\
34	0.847569551884924\\
35	0.935449205464931\\
36	1.03191528518579\\
37	1.13779153145777\\
38	1.25397757515994\\
39	1.38145528140676\\
40	1.5212954875148\\
41	1.6746651292366\\
42	1.84283474061763\\
43	2.02718630195185\\
44	2.22922139689045\\
45	2.45056962338697\\
46	2.69299718339932\\
47	2.95841555264993\\
48	3.24889010379286\\
49	3.5666485235887\\
50	3.91408882672011\\
51	4.29378672535344\\
52	4.70850206426271\\
53	5.16118397628744\\
54	5.654974352402\\
55	6.19320915544583\\
56	6.77941703785782\\
57	7.41731465352169\\
58	8.11079798486894\\
59	8.86392894254139\\
60	9.6809164412259\\
61	10.566091118119\\
62	11.5238728476621\\
63	12.5587302269204\\
64	13.6751312707161\\
65	14.8731357371375\\
66	16.1489603648746\\
67	17.497193279022\\
68	18.911522797897\\
69	20.385097681098\\
70	21.910834555404\\
71	23.4814893676033\\
72	25.0896428559744\\
73	26.7277137082899\\
74	28.3880095152222\\
75	30.0628016382316\\
76	31.7444129947375\\
77	33.4256182958181\\
78	35.0990223347272\\
79	36.7570794661286\\
80	38.3923048577953\\
81	39.9974740390795\\
82	41.5657854582743\\
83	43.0909878770915\\
84	44.5674772193483\\
85	45.9903664283348\\
86	47.3555308219281\\
87	48.6596309451666\\
88	49.9001148849489\\
89	51.0752022137821\\
90	52.1838519978403\\
91	53.2257175313147\\
92	54.201090588021\\
93	55.110837992388\\
94	55.9555806208718\\
95	56.7333723391028\\
96	57.4424286976957\\
97	58.082126470562\\
98	58.6532583036262\\
99	59.1579636754884\\
100	59.5995211807082\\
101	59.9820882670833\\
102	60.31043222324\\
103	60.5896766981947\\
104	60.8250776669571\\
105	61.0218362261598\\
106	61.1849509708375\\
107	61.3191093038728\\
108	61.4286146276404\\
109	61.5165716040242\\
110	61.5829788661791\\
111	61.6301466540456\\
112	61.6615911954469\\
113	61.6811934396223\\
114	61.6925679019746\\
115	61.698678143832\\
116	61.7016974543162\\
117	61.7030594647336\\
118	61.703615254017\\
119	61.7038181355578\\
120	61.7038834638274\\
121	61.7039044995952\\
122	61.7039112731186\\
123	61.7039134541914\\
124	61.7039141564952\\
125	61.7039143826362\\
126	61.7039144554532\\
127	61.7039144789001\\
128	61.7039144864499\\
129	61.7039144888809\\
130	61.7039144896637\\
131	61.7039144899157\\
132	61.7039144899969\\
133	61.703914490023\\
134	61.7039144900314\\
135	61.7039144900341\\
136	61.703914490035\\
137	61.7039144900353\\
138	61.7039144900354\\
139	61.7039144900354\\
140	61.7039144900354\\
141	61.7039144900354\\
};
\addplot [color=mycolor2, line width=2.0pt, forget plot]
  table[row sep=crcr]{%
1	0\\
2	0.00792471824437081\\
3	0.0151733311032985\\
4	0.0223107967711348\\
5	0.0296918913897686\\
6	0.0375468576459977\\
7	0.0460369348792874\\
8	0.0552880896329352\\
9	0.0654105998584643\\
10	0.0765102547101886\\
11	0.0886948562262939\\
12	0.102078183882955\\
13	0.116782631241678\\
14	0.13294117593241\\
15	0.150699042154799\\
16	0.17021525241692\\
17	0.191664179076784\\
18	0.215237160967964\\
19	0.241144226783254\\
20	0.269615954754126\\
21	0.300905492122249\\
22	0.335290755126997\\
23	0.373076829240376\\
24	0.414598589375783\\
25	0.460223560350273\\
26	0.510355038751714\\
27	0.56543549840999\\
28	0.625950302802667\\
29	0.69243174887093\\
30	0.765463467820007\\
31	0.845685209468571\\
32	0.933798037525921\\
33	1.03056996373564\\
34	1.13684204903658\\
35	1.25353499964457\\
36	1.38165628511757\\
37	1.52230780387181\\
38	1.67669411907706\\
39	1.84613128414785\\
40	2.03205627189918\\
41	2.23603701453813\\
42	2.45978305265654\\
43	2.70515677985952\\
44	2.97418525513976\\
45	3.26907253705632\\
46	3.59221247161288\\
47	3.94620183881129\\
48	4.33385373049752\\
49	4.75821099360553\\
50	5.2225595275264\\
51	5.73044117140642\\
52	6.28566585611891\\
53	6.89232262602924\\
54	7.55478905730304\\
55	8.27773851258592\\
56	9.06614457712155\\
57	9.92528192018206\\
58	10.8607227203739\\
59	11.8783276874067\\
60	12.9842306111187\\
61	14.1848152774243\\
62	15.4866835187535\\
63	16.8966131238879\\
64	18.42150433135\\
65	20.0505154968012\\
66	21.7702232344907\\
67	23.5731857917055\\
68	25.4520628204632\\
69	27.3992640776202\\
70	29.4080137991788\\
71	31.4712481730907\\
72	33.5811014027644\\
73	35.7291209436175\\
74	37.9064592079512\\
75	40.1040311748505\\
76	42.3126551245966\\
77	44.5234435808074\\
78	46.726250560327\\
79	48.9104278516528\\
80	51.0654972539769\\
81	53.1814709933363\\
82	55.2490261442126\\
83	57.2596276512559\\
84	59.205627578375\\
85	61.0803445973068\\
86	62.8781217156355\\
87	64.594360051822\\
88	66.2255278167327\\
89	67.7691451641971\\
90	69.2237468265685\\
91	70.5888253709664\\
92	71.8647585067774\\
93	73.0426235925882\\
94	74.0821382607754\\
95	74.9737411371238\\
96	75.7218555395951\\
97	76.3374707483851\\
98	76.8346132073185\\
99	77.2284736293421\\
100	77.5342885576912\\
101	77.7666183408005\\
102	77.9388769291334\\
103	78.0630477254732\\
104	78.1495460886036\\
105	78.2071968887634\\
106	78.2432981435802\\
107	78.2637440277358\\
108	78.2731837319252\\
109	78.2757810160946\\
110	78.276495576315\\
111	78.2766921147567\\
112	78.2767461379787\\
113	78.2767609656006\\
114	78.2767650222793\\
115	78.276766125008\\
116	78.2767664211809\\
117	78.2767664990917\\
118	78.2767665189129\\
119	78.2767665237073\\
120	78.2767665247857\\
121	78.2767665250049\\
122	78.2767665250402\\
123	78.2767665250419\\
124	78.2767665250398\\
125	78.2767665250384\\
126	78.2767665250377\\
127	78.2767665250374\\
128	78.2767665250373\\
129	78.2767665250373\\
130	78.2767665250373\\
131	78.2767665250373\\
132	78.2767665250373\\
133	78.2767665250373\\
134	78.2767665250373\\
135	78.2767665250373\\
136	78.2767665250373\\
137	78.2767665250373\\
138	78.2767665250373\\
139	78.2767665250373\\
140	78.2767665250373\\
141	78.2767665250373\\
};
\addplot [color=mycolor3, line width=2.0pt, forget plot]
  table[row sep=crcr]{%
1	0\\
2	0.112690712840141\\
3	0.184853862746458\\
4	0.240084950271744\\
5	0.289131566372381\\
6	0.337173829615426\\
7	0.386915831700138\\
8	0.439933755032252\\
9	0.497284397364492\\
10	0.559790725460741\\
11	0.628181837151906\\
12	0.703165228862507\\
13	0.785466713808157\\
14	0.875854559119352\\
15	0.975155860759522\\
16	1.08426916438193\\
17	1.20417540610817\\
18	1.33594829094\\
19	1.48076474325507\\
20	1.63991581516001\\
21	1.81481830894113\\
22	2.00702730227545\\
23	2.21824973028615\\
24	2.4503591619685\\
25	2.70541190204352\\
26	2.98566454860751\\
27	3.29359313960854\\
28	3.63191402576716\\
29	4.00360661315731\\
30	4.41193812463795\\
31	4.86049053516305\\
32	5.35318984122802\\
33	5.89433782885078\\
34	6.48864650700303\\
35	7.14127537368635\\
36	7.85787167916477\\
37	8.64461384436313\\
38	9.5082581810978\\
39	10.4561890434171\\
40	11.4964725144667\\
41	12.6379136992921\\
42	13.8901176488973\\
43	15.2635538824578\\
44	16.7696244002538\\
45	18.420734986753\\
46	20.2303694880559\\
47	22.2131666070527\\
48	24.3849985892301\\
49	26.7630509680134\\
50	29.365902296615\\
51	32.2136025094385\\
52	35.3277482263393\\
53	38.7315529343103\\
54	42.4499095514307\\
55	46.5094423969074\\
56	50.9385450609718\\
57	55.76740009492\\
58	61.0279758348949\\
59	66.7539950491193\\
60	72.9808694804421\\
61	79.7455937762219\\
62	87.0865917978288\\
63	95.0435079358437\\
64	103.656935889885\\
65	112.86346506524\\
66	122.554792633661\\
67	132.632034896198\\
68	143.027894029755\\
69	153.693799082367\\
70	164.607727382534\\
71	175.747859738813\\
72	187.081996438515\\
73	198.571651651344\\
74	210.174809284511\\
75	221.847548096983\\
76	233.202848968312\\
77	244.020206562068\\
78	254.184344464957\\
79	263.626268996198\\
80	272.303437406817\\
81	280.191494192245\\
82	287.280393537222\\
83	293.572296557852\\
84	299.080150081155\\
85	303.826474611905\\
86	307.842166942806\\
87	311.165252684424\\
88	313.839585585889\\
89	315.91351886561\\
90	317.43858499578\\
91	318.468222327153\\
92	319.056583692603\\
93	319.295432934689\\
94	319.392382040134\\
95	319.431721855897\\
96	319.447674590376\\
97	319.454134449989\\
98	319.456742393951\\
99	319.457788514413\\
100	319.458202421277\\
101	319.458361351746\\
102	319.458418273129\\
103	319.458435100717\\
104	319.458436790443\\
105	319.458433304964\\
106	319.458428534782\\
107	319.458423921761\\
108	319.458419930736\\
109	319.458416646598\\
110	319.458414021638\\
111	319.458411991952\\
112	319.458410486022\\
113	319.458409420734\\
114	319.458408705453\\
115	319.458408251008\\
116	319.458407978393\\
117	319.458407824174\\
118	319.458407741923\\
119	319.458407700516\\
120	319.458407680779\\
121	319.458407671824\\
122	319.458407667885\\
123	319.458407666188\\
124	319.458407665468\\
125	319.458407665165\\
126	319.458407665039\\
127	319.458407664987\\
128	319.458407664965\\
129	319.458407664956\\
130	319.458407664953\\
131	319.458407664951\\
132	319.458407664951\\
133	319.458407664951\\
134	319.45840766495\\
135	319.45840766495\\
136	319.45840766495\\
137	319.45840766495\\
138	319.45840766495\\
139	319.45840766495\\
140	319.45840766495\\
141	319.45840766495\\
};
\addplot [color=mycolor4, line width=2.0pt, forget plot]
  table[row sep=crcr]{%
1	0\\
2	0.00679340447442321\\
3	0.0813545545739133\\
4	0.175303555942114\\
5	0.27729442119742\\
6	0.38573864596288\\
7	0.501642891561894\\
8	0.626601204894404\\
9	0.762265986806404\\
10	0.910253170780997\\
11	1.07216496037678\\
12	1.24963539278765\\
13	1.44437341049702\\
14	1.65819905278897\\
15	1.89307392237178\\
16	2.1511279982941\\
17	2.43468458378129\\
18	2.74628471971661\\
19	3.0887120190501\\
20	3.46501861676365\\
21	3.87855276165622\\
22	4.33298847227015\\
23	4.83235761798685\\
24	5.38108475270394\\
25	5.9840250127223\\
26	6.64650538611481\\
27	7.37436966370559\\
28	8.17402738899736\\
29	9.05250713389798\\
30	10.0175144373313\\
31	11.0774947534215\\
32	12.2417017636408\\
33	13.5202714117946\\
34	14.9243020205465\\
35	16.465940841723\\
36	18.1584773779641\\
37	20.0164437881747\\
38	22.0557226510232\\
39	24.2936623063414\\
40	26.7491999200586\\
41	29.442992320031\\
42	32.3975545229293\\
43	35.6374057106406\\
44	39.1892222120978\\
45	43.0819967960091\\
46	47.3472032738224\\
47	52.0189650419779\\
48	57.134225749159\\
49	62.7329197486622\\
50	68.8581393791485\\
51	75.5562954005784\\
52	82.8772660891817\\
53	90.8745295614481\\
54	99.605272851721\\
55	109.130470115814\\
56	119.514921086473\\
57	130.827239587645\\
58	143.139780558448\\
59	156.528492695141\\
60	171.07268256046\\
61	186.854674926518\\
62	203.959353327587\\
63	222.47356444654\\
64	242.485370214378\\
65	263.65196849938\\
66	285.404159253525\\
67	307.143080802966\\
68	328.657566750632\\
69	348.004431134012\\
70	364.149166700382\\
71	377.027801673452\\
72	386.698168194976\\
73	393.255180747406\\
74	396.818754398741\\
75	397.530530177435\\
76	397.609884145509\\
77	397.618435641646\\
78	397.619063425685\\
79	397.618808940068\\
80	397.618462117594\\
81	397.618112843325\\
82	397.617772480181\\
83	397.617443448661\\
84	397.617127010048\\
85	397.616824131813\\
86	397.616535581014\\
87	397.616261934592\\
88	397.616003583866\\
89	397.615760741759\\
90	397.615533453297\\
91	397.615321608931\\
92	397.615124960039\\
93	397.614943135915\\
94	397.614775981376\\
95	397.614624393198\\
96	397.614488695941\\
97	397.614368681579\\
98	397.614263745688\\
99	397.61417299655\\
100	397.614095343852\\
101	397.614029574508\\
102	397.613974417796\\
103	397.613928599761\\
104	397.613890886527\\
105	397.613860116615\\
106	397.613835222949\\
107	397.613815245756\\
108	397.613799337847\\
109	397.613786763917\\
110	397.61377693881\\
111	397.613769546791\\
112	397.6137642659\\
113	397.613760709808\\
114	397.613758464398\\
115	397.613757141347\\
116	397.613756417601\\
117	397.613756052153\\
118	397.613755882943\\
119	397.613755811648\\
120	397.613755784557\\
121	397.613755775371\\
122	397.613755772362\\
123	397.613755771387\\
124	397.613755771072\\
125	397.613755770971\\
126	397.613755770938\\
127	397.613755770928\\
128	397.613755770924\\
129	397.613755770923\\
130	397.613755770923\\
131	397.613755770923\\
132	397.613755770923\\
133	397.613755770923\\
134	397.613755770923\\
135	397.613755770923\\
136	397.613755770923\\
137	397.613755770923\\
138	397.613755770923\\
139	397.613755770923\\
140	397.613755770923\\
141	397.613755770923\\
};
\addplot [color=mycolor5, line width=2.0pt, forget plot]
  table[row sep=crcr]{%
1	0\\
2	0.150403228269135\\
3	0.329353723796686\\
4	0.527927500303994\\
5	0.738585684808434\\
6	0.960880206144185\\
7	1.19734393360152\\
8	1.45133258629647\\
9	1.72629010807894\\
10	2.0255914272833\\
11	2.35257612004483\\
12	2.71063019551259\\
13	3.10326996731672\\
14	3.53421649283687\\
15	4.00746024260347\\
16	4.52731863078379\\
17	5.09848932972572\\
18	5.72610184990716\\
19	6.41576933196955\\
20	7.17364205210785\\
21	8.00646381519116\\
22	8.92163218574193\\
23	9.92726336088055\\
24	11.0322623996333\\
25	12.246399472573\\
26	13.5803927720856\\
27	15.0459987175852\\
28	16.6561100951696\\
29	18.42486278261\\
30	20.3677517243532\\
31	22.5017568341159\\
32	24.8454795116677\\
33	27.4192904625288\\
34	30.2454895013321\\
35	33.3484779979382\\
36	36.7549445859253\\
37	40.4940646910156\\
38	44.5977143467598\\
39	49.1006986398439\\
40	54.0409949601039\\
41	59.4600110119463\\
42	65.4028572642659\\
43	71.9186331636481\\
44	79.0607259977211\\
45	86.887120757639\\
46	95.4607186951772\\
47	104.849661484066\\
48	115.127656959509\\
49	126.374301306817\\
50	138.675391283023\\
51	152.123218569566\\
52	166.816836658866\\
53	182.862288768066\\
54	200.372783153573\\
55	219.468799887067\\
56	240.278110680968\\
57	262.93569077466\\
58	287.583499295503\\
59	314.370102008337\\
60	343.450108121548\\
61	374.983391030793\\
62	409.134061807454\\
63	446.069164184135\\
64	485.957061109996\\
65	523.382674051331\\
66	549.41105822049\\
67	564.038108156045\\
68	568.09458099224\\
69	568.411545138053\\
70	568.432196646884\\
71	568.429418937176\\
72	568.424850185178\\
73	568.420216776835\\
74	568.415666563\\
75	568.411221694823\\
76	568.406893606313\\
77	568.402619722233\\
78	568.398397456673\\
79	568.394245744299\\
80	568.39018482942\\
81	568.386233157408\\
82	568.382407088137\\
83	568.378720872416\\
84	568.3751866128\\
85	568.371814221639\\
86	568.368611402774\\
87	568.365583671086\\
88	568.362734414539\\
89	568.360064997451\\
90	568.357574900233\\
91	568.355261888807\\
92	568.353122205999\\
93	568.351150776955\\
94	568.349344860354\\
95	568.347712996412\\
96	568.346257306836\\
97	568.344974333575\\
98	568.343856504485\\
99	568.342893283606\\
100	568.342072121316\\
101	568.341379273621\\
102	568.340800505164\\
103	568.340321671934\\
104	568.339929179414\\
105	568.339610318074\\
106	568.339353485147\\
107	568.33914830702\\
108	568.338985679753\\
109	568.338857746198\\
110	568.338758243075\\
111	568.338683650162\\
112	568.338630549359\\
113	568.338594930483\\
114	568.338572538897\\
115	568.338559412595\\
116	568.338552275249\\
117	568.338548697136\\
118	568.338547054778\\
119	568.338546370205\\
120	568.338546113601\\
121	568.33854602812\\
122	568.338546000371\\
123	568.338545991418\\
124	568.338545988534\\
125	568.338545987606\\
126	568.338545987307\\
127	568.33854598721\\
128	568.338545987179\\
129	568.338545987169\\
130	568.338545987166\\
131	568.338545987165\\
132	568.338545987165\\
133	568.338545987165\\
134	568.338545987165\\
135	568.338545987165\\
136	568.338545987165\\
137	568.338545987165\\
138	568.338545987165\\
139	568.338545987165\\
140	568.338545987165\\
141	568.338545987165\\
};
\addplot [color=mycolor6, line width=2.0pt, forget plot]
  table[row sep=crcr]{%
1	0\\
2	0.244503556456279\\
3	0.534500362278705\\
4	0.854435323353285\\
5	1.19781946023111\\
6	1.56342888589301\\
7	1.95349528858806\\
8	2.37212637225748\\
9	2.82429022243462\\
10	3.31530947487756\\
11	3.85067887945964\\
12	4.43605266173708\\
13	5.07730809355792\\
14	5.78063672348677\\
15	6.55264139570166\\
16	7.40043108245494\\
17	8.3317120690617\\
18	9.35487671065472\\
19	10.4790918367534\\
20	11.7143889550212\\
21	13.0717582016558\\
22	14.563247725596\\
23	16.2020699640599\\
24	18.0027160927585\\
25	19.9810798146047\\
26	22.1545915760018\\
27	24.54236425824\\
28	27.1653513724169\\
29	30.0465187804834\\
30	33.2110309647798\\
31	36.6864528670699\\
32	40.5029683094648\\
33	44.6936159877523\\
34	49.2945439862668\\
35	54.3452836957569\\
36	59.8890439140848\\
37	65.973025765296\\
38	72.6487588755495\\
39	79.9724589829873\\
40	88.0054068194712\\
41	96.8143476699469\\
42	106.471910472677\\
43	117.057044651287\\
44	128.655472046011\\
45	141.360150313286\\
46	155.271742965094\\
47	170.499089796311\\
48	187.159669774203\\
49	205.38004651484\\
50	225.296284225982\\
51	247.054319440386\\
52	270.810270992783\\
53	296.73066751729\\
54	324.992568288939\\
55	355.78354955887\\
56	389.301524727323\\
57	425.754362894738\\
58	465.359266714296\\
59	508.341867286253\\
60	554.934991402866\\
61	605.377055166302\\
62	659.910038332378\\
63	718.776996224931\\
64	751.955253512007\\
65	758.100723435808\\
66	759.195239104379\\
67	759.347800177118\\
68	759.326219918601\\
69	759.27387632537\\
70	759.215997338092\\
71	759.156602874025\\
72	759.096546433614\\
73	759.036157580987\\
74	758.975676074344\\
75	758.915325595489\\
76	758.85532434359\\
77	758.795646424046\\
78	758.736471626382\\
79	758.678073951259\\
80	758.620732966285\\
81	758.564710068837\\
82	758.51024281515\\
83	758.457542825306\\
84	758.406794304915\\
85	758.358152875703\\
86	758.311744844012\\
87	758.267667025381\\
88	758.225987171393\\
89	758.186744983937\\
90	758.149953659736\\
91	758.115601881399\\
92	758.083656156795\\
93	758.054063403302\\
94	758.02682278508\\
95	758.002162192745\\
96	757.980177351187\\
97	757.960840912802\\
98	757.944042168475\\
99	757.929616348494\\
100	757.917364956294\\
101	757.907070897433\\
102	757.898510334505\\
103	757.89146202547\\
104	757.885714441657\\
105	757.881070853935\\
106	757.877352598025\\
107	757.874400778025\\
108	757.872076700261\\
109	757.870261336851\\
110	757.868857965432\\
111	757.8678036173\\
112	757.867046816895\\
113	757.86653292895\\
114	757.866204885357\\
115	757.86600899984\\
116	757.865900109229\\
117	757.865844044993\\
118	757.865817457888\\
119	757.865805915845\\
120	757.865801359516\\
121	757.865799735504\\
122	757.865799183269\\
123	757.865798999976\\
124	757.865798939933\\
125	757.865798920409\\
126	757.865798914086\\
127	757.865798912044\\
128	757.865798911385\\
129	757.865798911173\\
130	757.865798911104\\
131	757.865798911082\\
132	757.865798911075\\
133	757.865798911073\\
134	757.865798911072\\
135	757.865798911072\\
136	757.865798911072\\
137	757.865798911072\\
138	757.865798911072\\
139	757.865798911072\\
140	757.865798911072\\
141	757.865798911072\\
};

\addplot[area legend, draw=none, fill=green, fill opacity=0.1, line width=2.0pt, forget plot]
table[row sep=crcr] {%
x	y\\
118	0\\
150	0\\
150	800\\
118	800\\
}--cycle;
\addplot [color=green, dashed, line width=2.0pt, forget plot]
  table[row sep=crcr]{%
118	0\\
118	800\\
};
\end{axis}
\end{tikzpicture}%

%% file: Figures/MPC_state_1.tex
%
%
\definecolor{mycolor1}{rgb}{0.00000,0.44700,0.74100}%
\definecolor{mycolor2}{rgb}{0.85000,0.32500,0.09800}%
\definecolor{mycolor3}{rgb}{0.92900,0.69400,0.12500}%
\definecolor{mycolor4}{rgb}{0.49400,0.18400,0.55600}%
\definecolor{mycolor5}{rgb}{0.46600,0.67400,0.18800}%
\definecolor{mycolor6}{rgb}{0.30100,0.74500,0.93300}%
\begin{tikzpicture}

\begin{axis}[%
width=7cm,
height=3.5cm,
at={(0cm,0cm)},
scale only axis,
xmin=61,
xmax=140,
ymin=0,
ymax=1200000,
xmajorgrids,
ymajorgrids,
axis background/.style={fill=white},
axis x line*=bottom,
axis y line*=left
]
\addplot [color=mycolor1, line width=2.0pt, forget plot]
  table[row sep=crcr]{%
1	1058299.34049118\\
2	1058292.44626517\\
3	1058283.78077337\\
4	1058273.61241241\\
5	1058262.04982195\\
6	1058249.10573796\\
7	1058234.73290958\\
8	1058218.84387911\\
9	1058201.32163147\\
10	1058182.02508278\\
11	1058160.79161252\\
12	1058137.43785495\\
13	1058111.75941957\\
14	1058083.52991054\\
15	1058052.49944864\\
16	1058018.39280639\\
17	1057980.9072135\\
18	1057939.70985877\\
19	1057894.43509538\\
20	1057844.68134481\\
21	1057790.00768631\\
22	1057729.9301135\\
23	1057663.91743528\\
24	1057591.38679507\\
25	1057511.69877958\\
26	1057424.1520859\\
27	1057327.97771375\\
28	1057222.332648\\
29	1057106.29299513\\
30	1056978.8465363\\
31	1056838.88465923\\
32	1056685.19363109\\
33	1056516.44517546\\
34	1056331.18631846\\
35	1056127.8284717\\
36	1055904.63572501\\
37	1055659.71232743\\
38	1055390.98934441\\
39	1055096.21048994\\
40	1054772.91714741\\
41	1054418.43261123\\
42	1054029.8456047\\
43	1053603.99315784\\
44	1053137.44296392\\
45	1052626.47537521\\
46	1052067.06524862\\
47	1051454.86391135\\
48	1050785.18158549\\
49	1050052.97069169\\
50	1049252.81054321\\
51	1048378.89404652\\
52	1047425.01714036\\
53	1046384.57183296\\
54	1045250.54383415\\
55	1044015.515924\\
56	1042671.67834603\\
57	1041210.84765812\\
58	1039624.49560582\\
59	1037903.78969438\\
60	1036039.64720974\\
61	1034022.80446188\\
62	1031843.90297417\\
63	1029493.59419949\\
64	1026972.5317817\\
65	1024289.30484934\\
66	1021455.52443808\\
67	980621.556554458\\
68	922452.84408468\\
69	890197.789663687\\
70	887561.737415302\\
71	885079.528834044\\
72	882739.612503438\\
73	880537.689754827\\
74	878471.961812384\\
75	876540.931533684\\
76	837686.475903961\\
77	780899.239677591\\
78	724371.457367311\\
79	668118.096856923\\
80	612124.525056308\\
81	556356.803407586\\
82	528314.566029051\\
83	525971.065268427\\
84	525413.074157078\\
85	525270.945069597\\
86	469976.238567733\\
87	415122.545951564\\
88	359891.51470458\\
89	304678.243460234\\
90	249475.846392192\\
91	194279.651742886\\
92	139086.775451786\\
93	83895.5679596034\\
94	28705.1455977138\\
95	28705.0623849758\\
96	28705.014675689\\
97	6227.98612020965\\
98	1834.98246282285\\
99	1834.98185267156\\
100	1834.98150921565\\
101	1834.98131320688\\
102	1834.98120076709\\
103	1834.98113608171\\
104	1834.98109880155\\
105	1834.98107729053\\
106	1834.98106486889\\
107	1834.98105769233\\
108	1834.98105354473\\
109	1834.98105114717\\
110	1834.98104976103\\
111	1834.98104895957\\
112	1834.98104849614\\
113	1834.98104822816\\
114	1834.9810480732\\
115	1834.98104798359\\
116	1834.98104793176\\
117	1834.9810479018\\
118	1834.98104788447\\
119	1834.98104787445\\
120	1834.98104786865\\
121	1834.9810478653\\
122	1834.98104786336\\
123	1834.98104786224\\
124	1834.98104786159\\
125	1834.98104786122\\
126	1834.981047861\\
127	1834.98104786087\\
128	1834.9810478608\\
129	1834.98104786076\\
130	1834.98104786074\\
131	1834.98104786072\\
132	1834.98104786071\\
133	1834.98104786071\\
134	1834.98104786071\\
135	1834.9810478607\\
136	1834.9810478607\\
137	1834.9810478607\\
138	1834.9810478607\\
139	1834.9810478607\\
140	1834.9810478607\\
141	1834.9810478607\\
};
\addplot [color=mycolor2, line width=2.0pt, forget plot]
  table[row sep=crcr]{%
1	915791.670391113\\
2	915788.901498256\\
3	915786.091690194\\
4	915783.132068456\\
5	915779.950175584\\
6	915776.492532322\\
7	915772.714557096\\
8	915768.574937074\\
9	915764.032453362\\
10	915759.04411254\\
11	915753.563958054\\
12	915747.542225712\\
13	915740.924663755\\
14	915733.651920725\\
15	915725.658947598\\
16	915716.874383396\\
17	915707.219905194\\
18	915696.609529594\\
19	915684.948855782\\
20	915672.13424177\\
21	915658.051906044\\
22	915642.576946937\\
23	915625.572271926\\
24	915606.887428743\\
25	915586.357329784\\
26	915563.800860863\\
27	915539.019364923\\
28	915511.794990813\\
29	915481.888896859\\
30	915449.039298545\\
31	915412.959349352\\
32	915373.334843598\\
33	915329.821730101\\
34	915282.043425655\\
35	915229.587917726\\
36	915172.00464652\\
37	915108.801157736\\
38	915039.439518927\\
39	914963.33249471\\
40	914879.839478979\\
41	914788.262186238\\
42	914687.840109055\\
43	914577.745754883\\
44	914457.079683137\\
45	914324.865372849\\
46	914180.043962563\\
47	914021.468917787\\
48	913847.900697485\\
49	913658.001510099\\
50	913450.330271679\\
51	913223.337904108\\
52	912975.363140264\\
53	912704.629035312\\
54	912409.240419031\\
55	912087.182562778\\
56	911736.321375737\\
57	911354.405487427\\
58	910939.070615502\\
59	910487.846657563\\
60	909998.167980195\\
61	909467.387404114\\
62	908892.794396792\\
63	908271.637977925\\
64	907610.878540322\\
65	906916.211086578\\
66	906189.690890474\\
67	905434.210652977\\
68	904652.813703764\\
69	903859.88135736\\
70	903079.05120521\\
71	902327.761569162\\
72	901612.619873813\\
73	900937.15934677\\
74	900303.037131489\\
75	899710.710712141\\
76	899159.831009093\\
77	898649.483538798\\
78	898185.841578045\\
79	897777.64721776\\
80	897430.422361326\\
81	897145.21637535\\
82	896918.862088981\\
83	896745.030181943\\
84	857452.144027167\\
85	857358.66654186\\
86	857290.17108192\\
87	857026.215908169\\
88	856990.74452539\\
89	856966.227212327\\
90	856949.707257338\\
91	856938.823086276\\
92	856931.786966222\\
93	856927.308903684\\
94	856924.494658223\\
95	801732.743971016\\
96	746541.733464506\\
97	713828.162961444\\
98	663030.839821429\\
99	663030.665429249\\
100	663030.56617733\\
101	663030.509351481\\
102	663030.476702339\\
103	663030.457901669\\
104	663030.447059568\\
105	663030.440801041\\
106	663030.437186068\\
107	663030.435097169\\
108	663030.433889781\\
109	663030.433191784\\
110	663030.432788223\\
111	663030.432554878\\
112	663030.432419948\\
113	663030.432341923\\
114	663030.432296803\\
115	663030.432270711\\
116	663030.432255623\\
117	663030.432246897\\
118	663030.432241851\\
119	663030.432238933\\
120	663030.432237246\\
121	663030.43223627\\
122	663030.432235706\\
123	663030.432235379\\
124	663030.432235191\\
125	663030.432235081\\
126	663030.432235018\\
127	663030.432234982\\
128	663030.432234961\\
129	663030.432234948\\
130	663030.432234941\\
131	663030.432234937\\
132	663030.432234935\\
133	663030.432234933\\
134	663030.432234933\\
135	663030.432234932\\
136	663030.432234932\\
137	663030.432234932\\
138	663030.432234932\\
139	663030.432234932\\
140	663030.432234932\\
141	663030.432234932\\
};
\addplot [color=mycolor3, line width=2.0pt, forget plot]
  table[row sep=crcr]{%
1	983784.158223048\\
2	983783.023485627\\
3	983781.909278624\\
4	983780.76542737\\
5	983779.556847025\\
6	983778.257717261\\
7	983776.847487018\\
8	983775.308245556\\
9	983773.623068632\\
10	983771.774986174\\
11	983769.746318778\\
12	983767.518223245\\
13	983765.070351849\\
14	983762.380569837\\
15	983759.42469881\\
16	983756.176266809\\
17	983752.606253255\\
18	983748.682820979\\
19	983744.371029843\\
20	983739.632527662\\
21	983734.425214798\\
22	983728.70287909\\
23	983722.414797912\\
24	983715.505304104\\
25	983707.913312493\\
26	983699.571803527\\
27	983690.40726046\\
28	983680.339056318\\
29	983669.278786734\\
30	983657.129544611\\
31	983643.785132405\\
32	983629.129207786\\
33	983613.034358352\\
34	983595.361101145\\
35	983575.956802793\\
36	983554.654516383\\
37	983531.2717315\\
38	983505.609034454\\
39	983477.448676443\\
40	983446.553048471\\
41	983412.663063139\\
42	983375.496445196\\
43	983334.745934862\\
44	983290.077410699\\
45	983241.12794206\\
46	983187.503785267\\
47	983128.778342466\\
48	983064.490107939\\
49	982994.140633484\\
50	982917.192552434\\
51	982833.067711126\\
52	982741.145467132\\
53	982640.761225429\\
54	982531.205296857\\
55	982411.72217759\\
56	982281.510363753\\
57	982139.722831353\\
58	981985.468327959\\
59	981817.813638172\\
60	981635.786999115\\
61	981438.3828535\\
62	981224.568134846\\
63	980993.29028022\\
64	980747.98176831\\
65	980492.190461627\\
66	980228.275036728\\
67	979957.401982764\\
68	979680.483995134\\
69	979400.116966079\\
70	979122.658030932\\
71	978854.668119349\\
72	978599.944279809\\
73	978360.250550593\\
74	978136.259690018\\
75	977928.041355454\\
76	977735.315607735\\
77	977557.589232416\\
78	977395.95982036\\
79	977252.810095536\\
80	977130.02599094\\
81	977028.181333353\\
82	949404.47392098\\
83	896224.787301154\\
84	880517.573928816\\
85	825296.034631666\\
86	825273.643902311\\
87	825199.331685991\\
88	825187.694059419\\
89	825179.586478089\\
90	825174.069791784\\
91	825170.396398103\\
92	825167.996254251\\
93	825166.452992148\\
94	825165.473859332\\
95	825164.859470122\\
96	825164.477474745\\
97	825164.244511842\\
98	825164.105572098\\
99	825164.024163009\\
100	825163.977424015\\
101	825163.950515279\\
102	825163.934998816\\
103	825163.926042548\\
104	825163.920869525\\
105	825163.917880383\\
106	825163.916152678\\
107	825163.915153895\\
108	825163.914576431\\
109	825163.914242535\\
110	825163.914049462\\
111	825163.913937815\\
112	825163.913873253\\
113	825163.913835918\\
114	825163.913814328\\
115	825163.913801843\\
116	825163.913794623\\
117	825163.913790447\\
118	825163.913788032\\
119	825163.913786636\\
120	825163.913785829\\
121	825163.913785362\\
122	825163.913785092\\
123	825163.913784935\\
124	825163.913784845\\
125	825163.913784793\\
126	825163.913784763\\
127	825163.913784745\\
128	825163.913784735\\
129	825163.913784729\\
130	825163.913784726\\
131	825163.913784724\\
132	825163.913784723\\
133	825163.913784722\\
134	825163.913784722\\
135	825163.913784722\\
136	825163.913784721\\
137	825163.913784721\\
138	825163.913784721\\
139	825163.913784721\\
140	825163.913784721\\
141	825163.913784721\\
};
\addplot [color=mycolor4, line width=2.0pt, forget plot]
  table[row sep=crcr]{%
1	384802.829089865\\
2	384800.972397342\\
3	384798.818823749\\
4	384796.517541259\\
5	384794.076555305\\
6	384791.46606567\\
7	384788.648804423\\
8	384785.587690053\\
9	384782.246722872\\
10	384778.590145894\\
11	384774.581362579\\
12	384770.181981572\\
13	384765.351027502\\
14	384760.044276442\\
15	384754.213667157\\
16	384747.806748983\\
17	384740.766137903\\
18	384733.028960518\\
19	384724.526271146\\
20	384715.182430756\\
21	384704.91443859\\
22	384693.631208598\\
23	384681.232783445\\
24	384667.609479182\\
25	384652.640953732\\
26	384636.195192244\\
27	384618.127402216\\
28	384598.278811047\\
29	384576.47535847\\
30	384552.526276079\\
31	384526.222546029\\
32	384497.335230855\\
33	384465.613666422\\
34	384430.783510139\\
35	384392.544636933\\
36	384350.568876091\\
37	384304.497582932\\
38	384253.93904056\\
39	384198.46568865\\
40	384137.611178493\\
41	384070.867256428\\
42	383997.680481524\\
43	383917.448788002\\
44	383829.517908621\\
45	383733.177682233\\
46	383627.658277205\\
47	383512.126372463\\
48	383385.681349934\\
49	383247.351566175\\
50	383096.090787269\\
51	382930.774889761\\
52	382750.198951581\\
53	382553.07488056\\
54	382338.029754211\\
55	382103.60507253\\
56	381848.25715528\\
57	381570.358945583\\
58	381268.203511641\\
59	380940.009566259\\
60	380583.929347586\\
61	380198.059221277\\
62	379780.453370947\\
63	379329.140936356\\
64	378852.994067442\\
65	378365.391110671\\
66	377879.77618993\\
67	377399.771129674\\
68	376926.382477675\\
69	350692.004743713\\
70	295068.155116167\\
71	239528.460475908\\
72	184075.796071919\\
73	128702.108870536\\
74	73396.2643168607\\
75	18146.9895307484\\
76	0.888178040535422\\
77	0.887606440371901\\
78	0.887094325626106\\
79	0.886642264750246\\
80	0.886255254084516\\
81	0.885934771215335\\
82	0.885678037216655\\
83	0.885478847931469\\
84	0.885327796810208\\
85	0.885214803446181\\
86	0.885131797950294\\
87	-28.1149286214945\\
88	-28.1135613176163\\
89	-28.1126102226891\\
90	-28.1119638682101\\
91	-28.1115338608477\\
92	-28.1112530457808\\
93	-28.1110725195356\\
94	-28.1109579746972\\
95	-28.1108860802872\\
96	-28.1108413623538\\
97	-28.1108139989496\\
98	-28.1107976031914\\
99	-28.1107879549726\\
100	-28.1107823853727\\
101	-28.1107791670626\\
102	-28.1107773068019\\
103	-28.1107762313434\\
104	-28.1107756095311\\
105	-28.1107752499852\\
106	-28.1107750420784\\
107	-28.1107749218532\\
108	-28.1107748523298\\
109	-28.1107748121257\\
110	-28.1107747888761\\
111	-28.1107747754311\\
112	-28.110774767656\\
113	-28.1107747631597\\
114	-28.1107747605595\\
115	-28.1107747590558\\
116	-28.1107747581862\\
117	-28.1107747576834\\
118	-28.1107747573925\\
119	-28.1107747572244\\
120	-28.1107747571271\\
121	-28.1107747570709\\
122	-28.1107747570384\\
123	-28.1107747570195\\
124	-28.1107747570087\\
125	-28.1107747570024\\
126	-28.1107747569987\\
127	-28.1107747569966\\
128	-28.1107747569954\\
129	-28.1107747569947\\
130	-28.1107747569943\\
131	-28.1107747569941\\
132	-28.1107747569939\\
133	-28.1107747569939\\
134	-28.1107747569938\\
135	-28.1107747569938\\
136	-28.1107747569938\\
137	-28.1107747569938\\
138	-28.1107747569938\\
139	-28.1107747569938\\
140	-28.1107747569938\\
141	-28.1107747569938\\
};
\addplot [color=mycolor5, line width=2.0pt, forget plot]
  table[row sep=crcr]{%
1	203033.506306142\\
2	203031.847499558\\
3	203030.016281485\\
4	203028.080497469\\
5	203026.038662809\\
6	203023.865270194\\
7	203021.528986979\\
8	203018.998252081\\
9	203016.242259281\\
10	203013.230499123\\
11	203009.931973184\\
12	203006.314425502\\
13	203002.343669331\\
14	202997.983003483\\
15	202993.192691894\\
16	202987.929479424\\
17	202982.146121576\\
18	202975.790910739\\
19	202968.807185581\\
20	202961.132813118\\
21	202952.699634934\\
22	202943.432870321\\
23	202933.250469891\\
24	202922.062413597\\
25	202909.769947357\\
26	202896.264752467\\
27	202881.428041965\\
28	202865.129577968\\
29	202847.226603908\\
30	202827.56268546\\
31	202805.966453865\\
32	202782.250245363\\
33	202756.208630506\\
34	202727.616827347\\
35	202696.228992873\\
36	202661.776387624\\
37	202623.965409302\\
38	202582.475492339\\
39	202536.95687195\\
40	202487.028213236\\
41	202432.274108527\\
42	202372.242449391\\
43	202306.441683852\\
44	202234.337974333\\
45	202155.352277903\\
46	202068.857377718\\
47	201974.174903194\\
48	201870.572386691\\
49	201757.260416385\\
50	201633.389958757\\
51	201498.049939787\\
52	201350.265191577\\
53	201188.994890666\\
54	201013.131635627\\
55	200821.501334219\\
56	200612.86409408\\
57	200385.916334657\\
58	200139.294360874\\
59	199871.579659393\\
60	199581.30619426\\
61	199266.969988031\\
62	198927.041274139\\
63	183618.979493129\\
64	128088.964008759\\
65	72670.1815839733\\
66	17355.6481493852\\
67	0.108825659899594\\
68	0.108653906256459\\
69	0.108484448507589\\
70	0.108316844010182\\
71	0.108155937556789\\
72	0.108006388468854\\
73	0.107869592854743\\
74	0.107745534520759\\
75	0.10763361454413\\
76	0.107533010632207\\
77	0.10744283176447\\
78	0.107361695368073\\
79	0.107289784918022\\
80	0.107227904827955\\
81	0.10717635230891\\
82	0.107134779499648\\
83	0.107102297422069\\
84	0.1070775422473\\
85	0.107058991586308\\
86	0.107045350215617\\
87	-25.8929645723944\\
88	-25.8912497676433\\
89	-25.8900513415357\\
90	-25.8892321459957\\
91	-25.8886837492988\\
92	-25.8883234097451\\
93	-25.8880904128009\\
94	-25.8879417908982\\
95	-25.8878480659054\\
96	-25.8877895254699\\
97	-25.8877536021892\\
98	-25.8877320090787\\
99	-25.8877192517819\\
100	-25.8877118582887\\
101	-25.8877075756568\\
102	-25.8877050963064\\
103	-25.8877036614741\\
104	-25.8877028313255\\
105	-25.8877023511062\\
106	-25.8877020733416\\
107	-25.8877019126903\\
108	-25.8877018197784\\
109	-25.8877017660447\\
110	-25.8877017349696\\
111	-25.8877017169986\\
112	-25.8877017066059\\
113	-25.8877017005957\\
114	-25.8877016971201\\
115	-25.8877016951101\\
116	-25.8877016939477\\
117	-25.8877016932755\\
118	-25.8877016928868\\
119	-25.887701692662\\
120	-25.887701692532\\
121	-25.8877016924568\\
122	-25.8877016924134\\
123	-25.8877016923882\\
124	-25.8877016923737\\
125	-25.8877016923653\\
126	-25.8877016923604\\
127	-25.8877016923576\\
128	-25.887701692356\\
129	-25.887701692355\\
130	-25.8877016923545\\
131	-25.8877016923542\\
132	-25.887701692354\\
133	-25.8877016923539\\
134	-25.8877016923538\\
135	-25.8877016923538\\
136	-25.8877016923538\\
137	-25.8877016923537\\
138	-25.8877016923537\\
139	-25.8877016923537\\
140	-25.8877016923537\\
141	-25.8877016923537\\
};
\addplot [color=mycolor6, line width=2.0pt, forget plot]
  table[row sep=crcr]{%
1	99514.3855136335\\
2	99512.7730761402\\
3	99511.0192276974\\
4	99509.1475744595\\
5	99507.1581711548\\
6	99505.0347677528\\
7	99502.7530011467\\
8	99500.2851476226\\
9	99497.6022066046\\
10	99494.6744795985\\
11	99491.4714279749\\
12	99487.9612375715\\
13	99484.1102949229\\
14	99479.8826594836\\
15	99475.2395574488\\
16	99470.1388969544\\
17	99464.5347947996\\
18	99458.377102492\\
19	99451.6109199752\\
20	99444.1760868441\\
21	99436.0066423135\\
22	99427.0302464047\\
23	99417.1675556773\\
24	99406.3315474109\\
25	99394.4267864822\\
26	99381.3486293714\\
27	99366.9823597983\\
28	99351.2022505128\\
29	99333.87054575\\
30	99314.8363588744\\
31	99293.9344797919\\
32	99270.9840868399\\
33	99245.7873581139\\
34	99218.1279775855\\
35	99187.769531951\\
36	99154.45379498\\
37	99117.8988972545\\
38	99077.7973806708\\
39	99033.8141389948\\
40	98985.584248193\\
41	98932.7106933104\\
42	98874.7620024355\\
43	98811.2698028964\\
44	98741.7263204084\\
45	98665.5818485642\\
46	98582.2422239751\\
47	98491.0663516659\\
48	98391.3638361262\\
49	98282.3927858354\\
50	98163.3578731762\\
51	98033.4087474469\\
52	97891.6389161039\\
53	97737.0852282407\\
54	97568.7281142971\\
55	97385.4927565749\\
56	97186.2513855491\\
57	96969.8269161364\\
58	96734.9981546325\\
59	96480.5068191435\\
60	96205.0666218251\\
61	95907.37465747\\
62	40396.1253269382\\
63	0.858997450770403\\
64	0.856286478435279\\
65	0.854085230756245\\
66	0.852049001019343\\
67	0.850102734923929\\
68	0.848228618736618\\
69	0.846371351352031\\
70	0.844535001905584\\
71	0.842763181409799\\
72	0.841087465354648\\
73	0.839520597413433\\
74	0.838066043308633\\
75	0.836722649256093\\
76	0.835486882298154\\
77	0.834353937642347\\
78	0.833325108169336\\
79	0.832412261798189\\
80	0.831626984967211\\
81	0.830973333940534\\
82	0.830446894279605\\
83	0.8300362529095\\
84	0.829723528650132\\
85	0.829488880932234\\
86	0.829316684956199\\
87	-75.1708087909724\\
88	-75.1626717291323\\
89	-75.1569845791482\\
90	-75.1530999016667\\
91	-75.1505032973371\\
92	-75.1488007530997\\
93	-75.1477026819119\\
94	-75.1470041968277\\
95	-75.1465649657176\\
96	-75.1462913895423\\
97	-75.1461247913201\\
98	-75.1460256711739\\
99	-75.1459676913561\\
100	-75.1459344936579\\
101	-75.1459154176452\\
102	-75.145904431873\\
103	-75.1458980961096\\
104	-75.145894438676\\
105	-75.1458923260545\\
106	-75.1458911052613\\
107	-75.1458903996314\\
108	-75.1458899917004\\
109	-75.1458897558454\\
110	-75.1458896194701\\
111	-75.1458895406119\\
112	-75.1458894950112\\
113	-75.1458894686415\\
114	-75.1458894533923\\
115	-75.145889444574\\
116	-75.1458894394744\\
117	-75.1458894365254\\
118	-75.1458894348199\\
119	-75.1458894338337\\
120	-75.1458894332634\\
121	-75.1458894329335\\
122	-75.1458894327428\\
123	-75.1458894326325\\
124	-75.1458894325687\\
125	-75.1458894325318\\
126	-75.1458894325105\\
127	-75.1458894324982\\
128	-75.145889432491\\
129	-75.1458894324869\\
130	-75.1458894324845\\
131	-75.1458894324831\\
132	-75.1458894324823\\
133	-75.1458894324819\\
134	-75.1458894324816\\
135	-75.1458894324815\\
136	-75.1458894324814\\
137	-75.1458894324813\\
138	-75.1458894324813\\
139	-75.1458894324813\\
140	-75.1458894324813\\
141	-75.1458894324813\\
};

\addplot[area legend, draw=none, fill=green, fill opacity=0.1, line width=2.0pt, forget plot]
table[row sep=crcr] {%
x	y\\
98	0\\
150	0\\
150	1200000\\
98	1200000\\
}--cycle;
\addplot [color=green, dashed, line width=2.0pt, forget plot]
  table[row sep=crcr]{%
98	0\\
98	1200000\\
};
\end{axis}
\end{tikzpicture}%

%% file: Figures/MPC_state_2.tex
%
%
\definecolor{mycolor1}{rgb}{0.00000,0.44700,0.74100}%
\definecolor{mycolor2}{rgb}{0.85000,0.32500,0.09800}%
\definecolor{mycolor3}{rgb}{0.92900,0.69400,0.12500}%
\definecolor{mycolor4}{rgb}{0.49400,0.18400,0.55600}%
\definecolor{mycolor5}{rgb}{0.46600,0.67400,0.18800}%
\definecolor{mycolor6}{rgb}{0.30100,0.74500,0.93300}%
\begin{tikzpicture}

\begin{axis}[%
width=7cm,
height=3.5cm,
at={(0cm,0cm)},
scale only axis,
xmin=61,
xmax=140,
ymin=0,
ymax=3500,
xmajorgrids,
ymajorgrids,
axis background/.style={fill=white},
axis x line*=bottom,
axis y line*=left,
legend pos=north east,
legend style={legend cell align=left, align=left, draw=white!15!black}
]
\addplot [color=mycolor1, solid, line width=2pt, mark options={solid, mycolor1}]
  table[row sep=crcr]{%
1	4.6595088243\\
2	7.25704562732303\\
3	9.23057250575148\\
4	10.8871132834846\\
5	12.410331803941\\
6	13.910433025333\\
7	15.4559850395323\\
8	17.0925338455949\\
9	18.8531833362379\\
10	20.7645799987799\\
11	22.8503353158311\\
12	25.1330330798772\\
13	27.6354565648139\\
14	30.3813851687548\\
15	33.3961538701621\\
16	36.7070836882444\\
17	40.3438448207566\\
18	44.3387885132268\\
19	48.7272694325119\\
20	53.5479722374286\\
21	58.843251390541\\
22	64.6594905314611\\
23	71.0474861021338\\
24	78.0628589110542\\
25	85.7664966763911\\
26	94.2250301363437\\
27	103.511344957687\\
28	113.705131342684\\
29	124.893472880895\\
30	137.171475775716\\
31	150.642939059574\\
32	165.421065760796\\
33	181.629214161851\\
34	199.401687251714\\
35	218.884557178936\\
36	240.236519904921\\
37	263.629773281045\\
38	289.250909363816\\
39	317.301808867322\\
40	348.00052215358\\
41	381.582116994941\\
42	418.299468420906\\
43	458.423960196214\\
44	502.246060783532\\
45	550.075728948676\\
46	602.242595414716\\
47	659.095857141163\\
48	721.003809921073\\
49	788.352933145968\\
50	861.546427973644\\
51	941.002097059782\\
52	1027.14944095545\\
53	1120.42583390773\\
54	1221.27163105874\\
55	1330.12405114667\\
56	1447.40967533991\\
57	1573.53540573172\\
58	1708.87773862994\\
59	1853.77023081836\\
60	2008.48907448498\\
61	2173.23675173455\\
62	2348.12381572611\\
63	2533.14894533317\\
64	2718.30984670963\\
65	2894.89219163544\\
66	3059.19551876016\\
67	3209.17671526629\\
68	3228.59981146918\\
69	3080.45417110083\\
70	2875.91641677293\\
71	2706.14611168202\\
72	2550.63444246941\\
73	2400.53171190691\\
74	2252.64893226511\\
75	2106.43615484574\\
76	1962.47642941023\\
77	1751.0473868593\\
78	1475.13023234829\\
79	1179.22373080691\\
80	896.39381900029\\
81	648.520693277566\\
82	446.735411661379\\
83	307.286481188712\\
84	218.918432507843\\
85	159.175498445665\\
86	116.100939969507\\
87	76.7329893960127\\
88	47.0061758963329\\
89	26.9314504920725\\
90	14.4941256894508\\
91	7.32325606679446\\
92	3.44652737949352\\
93	1.47586118873349\\
94	0.537282027053003\\
95	0.125049006054413\\
96	0.0574464146131343\\
97	0.0330286302646088\\
98	0.00622921046245438\\
99	0.00109519807443467\\
100	0.000428735141363181\\
101	0.000229392879018666\\
102	0.000130301809606528\\
103	7.48315276283553e-05\\
104	4.31070242889658e-05\\
105	2.48676093175402e-05\\
106	1.43579956127805e-05\\
107	8.29456515561576e-06\\
108	4.79346739523827e-06\\
109	2.77081820195441e-06\\
110	1.60189104336241e-06\\
111	9.2619293449325e-07\\
112	5.35547983649758e-07\\
113	3.09680512087021e-07\\
114	1.79077689342885e-07\\
115	1.0355641260639e-07\\
116	5.988494778763e-08\\
117	3.46307361113285e-08\\
118	2.00266344304231e-08\\
119	1.15812557669042e-08\\
120	6.69736969354176e-09\\
121	3.87305356806152e-09\\
122	2.23976850810832e-09\\
123	1.29524828168927e-09\\
124	7.49036680319226e-10\\
125	4.33164899541293e-10\\
126	2.50497560953758e-10\\
127	1.44861771902504e-10\\
128	8.37730090846608e-11\\
129	4.84456132512633e-11\\
130	2.80159151492411e-11\\
131	1.62014982260733e-11\\
132	9.36926553192307e-12\\
133	5.41821108895677e-12\\
134	3.13333115233745e-12\\
135	1.81199366067599e-12\\
136	1.04786914449199e-12\\
137	6.05978800855772e-13\\
138	3.50435270837672e-13\\
139	2.02655404699816e-13\\
140	1.17194861616831e-13\\
141	6.77733496105747e-14\\
};
\addlegendentry{0-24}

\addplot [color=mycolor2, solid, line width=2pt, mark options={solid, mycolor2}]
  table[row sep=crcr]{%
1	4.3296088874\\
2	3.96022390797655\\
3	3.89949949462591\\
4	4.03260429886922\\
5	4.29150042493362\\
6	4.63848842450323\\
7	5.05429727656978\\
8	5.53035563963166\\
9	6.06421109384361\\
10	6.6569633628338\\
11	7.31187817389316\\
12	8.03366176631637\\
13	8.82809666452643\\
14	9.70187368772757\\
15	10.6625314356469\\
16	11.7184561807865\\
17	12.8789174266314\\
18	14.1541262340792\\
19	15.5553096856015\\
20	17.0947981864076\\
21	18.7861240879687\\
22	20.644131091658\\
23	22.6850944297544\\
24	24.9268521261098\\
25	27.388947812695\\
26	30.092785674487\\
27	33.0617981407625\\
28	36.3216269485436\\
29	39.9003181783317\\
30	43.8285318031401\\
31	48.1397661957305\\
32	52.8705978999381\\
33	58.060936781986\\
34	63.7542964266499\\
35	69.998079318849\\
36	76.8438759393414\\
37	84.3477763869278\\
38	92.5706924996921\\
39	101.578687662451\\
40	111.44331053226\\
41	122.24192776142\\
42	134.058049418619\\
43	146.981639172394\\
44	161.109399375062\\
45	176.545018937896\\
46	193.39936929064\\
47	211.790630746314\\
48	231.844328230881\\
49	253.6932515857\\
50	277.477231528359\\
51	303.342737912812\\
52	331.442262249623\\
53	361.933441669627\\
54	394.977876842139\\
55	430.739592075495\\
56	469.383082313906\\
57	511.070889494892\\
58	555.960650368529\\
59	604.201560163545\\
60	655.930202314117\\
61	711.265704857928\\
62	770.304200194665\\
63	833.112587799084\\
64	889.99790255619\\
65	939.558427642101\\
66	985.048196809699\\
67	1026.5251596933\\
68	1063.85463295235\\
69	1085.66157239533\\
70	1079.55975479851\\
71	1048.34026838607\\
72	1003.60200358592\\
73	951.610697013778\\
74	895.966506929711\\
75	838.859708732887\\
76	781.699540314414\\
77	725.439170144253\\
78	663.253110813508\\
79	590.694453579565\\
80	509.75977881774\\
81	425.471000455357\\
82	343.426480997955\\
83	268.328809954908\\
84	204.719254164593\\
85	149.80784000073\\
86	109.716442292592\\
87	80.1446453542447\\
88	57.5239069447738\\
89	40.3455364253497\\
90	27.6213943063104\\
91	18.4844475713195\\
92	12.1222830155227\\
93	7.81361836918725\\
94	4.96423323856809\\
95	3.11664088929939\\
96	1.8680784438086\\
97	1.08452174421988\\
98	0.621555983931183\\
99	0.345418931381893\\
100	0.194297062590131\\
101	0.110288444048834\\
102	0.0629960045250067\\
103	0.0361345909102967\\
104	0.0207848601118446\\
105	0.0119776696638601\\
106	0.00691073733984269\\
107	0.00399045068797529\\
108	0.00230539699344581\\
109	0.0013323472895572\\
110	0.000770168679434047\\
111	0.000445263838261795\\
112	0.000257448429008143\\
113	0.000148864107396223\\
114	8.60810054314208e-05\\
115	4.97778507264933e-05\\
116	2.87854115130525e-05\\
117	1.66461434789599e-05\\
118	9.62626950484704e-06\\
119	5.56678543374551e-06\\
120	3.21923223360766e-06\\
121	1.86166279849575e-06\\
122	1.07658992758271e-06\\
123	6.22586905611781e-07\\
124	3.60039291941642e-07\\
125	2.08209229324306e-07\\
126	1.20406562941431e-07\\
127	6.96306438299466e-08\\
128	4.02671328001575e-08\\
129	2.32863291770953e-08\\
130	1.34663958580489e-08\\
131	7.78756588141854e-09\\
132	4.50352003098249e-09\\
133	2.60436871108416e-09\\
134	1.50609665121018e-09\\
135	8.70970045419952e-10\\
136	5.036787127733e-10\\
137	2.91275512397982e-10\\
138	1.68443537703301e-10\\
139	9.74102669571888e-11\\
140	5.63319925708803e-11\\
141	3.25765803483119e-11\\
};
\addlegendentry{25-44}

\addplot [color=mycolor3, solid, line width=2pt, mark options={solid, mycolor3}]
  table[row sep=crcr]{%
1	4.8417769521\\
2	3.10050284714218\\
3	2.3730137105326\\
4	2.10729677210785\\
5	2.06414456909919\\
6	2.13717415725919\\
7	2.27792470886025\\
8	2.46408076791808\\
9	2.68559574351201\\
10	2.93843653988166\\
11	3.22167504867669\\
12	3.5360982541545\\
13	3.88353019648157\\
14	4.26649846715358\\
15	4.68807291507581\\
16	5.15179346169021\\
17	5.66164593878754\\
18	6.22206518560832\\
19	6.83795472076256\\
20	7.51471742356909\\
21	8.25829432131283\\
22	9.07520999675987\\
23	9.97262390449867\\
24	10.9583873181901\\
25	12.0411058790925\\
26	13.2302078603086\\
27	14.5360183437703\\
28	15.9698395505314\\
29	17.544037581053\\
30	19.2721358162863\\
31	21.1689152042642\\
32	23.2505216100989\\
33	25.5345803375343\\
34	28.0403178340428\\
35	30.7886904642024\\
36	33.8025200719628\\
37	37.1066358444767\\
38	40.7280217303485\\
39	44.6959683441148\\
40	49.0422278959669\\
41	53.8011702093673\\
42	59.0099373163549\\
43	64.7085934370322\\
44	70.9402663413825\\
45	77.7512751433425\\
46	85.1912384747057\\
47	93.3131557172517\\
48	102.17345252588\\
49	111.831980248838\\
50	122.351957046345\\
51	133.799836539859\\
52	146.245087719713\\
53	159.759867647665\\
54	174.41856628775\\
55	190.297200690996\\
56	207.472633894585\\
57	226.021592468926\\
58	246.019455907087\\
59	267.538791310761\\
60	290.647608459805\\
61	315.407313793891\\
62	341.870347572331\\
63	370.077497030387\\
64	395.560280204337\\
65	416.3891059128\\
66	432.969744504746\\
67	446.659126480851\\
68	458.261960491744\\
69	466.421762072395\\
70	466.826554321105\\
71	457.521876742307\\
72	440.478097946012\\
73	418.528199407373\\
74	393.91365389933\\
75	368.147602159905\\
76	342.193977107564\\
77	316.657411582042\\
78	290.192581703419\\
79	260.968151766312\\
80	228.737388938124\\
81	194.712225700622\\
82	160.760736217097\\
83	126.955611003171\\
84	95.7574548651108\\
85	71.4169026143045\\
86	51.3860505779375\\
87	37.1749951353273\\
88	26.7307051848915\\
89	18.9602696289359\\
90	13.214571374434\\
91	9.03852053250604\\
92	6.06979062414116\\
93	4.00760209110192\\
94	2.60622256231039\\
95	1.67251771497748\\
96	1.06103894567824\\
97	0.663745587185904\\
98	0.408420998694932\\
99	0.247228351125963\\
100	0.147113907325265\\
101	0.0866371040695015\\
102	0.0506911982192772\\
103	0.0295369364397587\\
104	0.0171650432715384\\
105	0.00995816360948585\\
106	0.00577072720903801\\
107	0.00334170351538428\\
108	0.00193419822370897\\
109	0.00111918259027773\\
110	0.00064746179350998\\
111	0.000374516331946463\\
112	0.000216615966108403\\
113	0.000125281251501944\\
114	7.24546039699177e-05\\
115	4.1902085179512e-05\\
116	2.4232520617857e-05\\
117	1.40138407683012e-05\\
118	8.10425159996762e-06\\
119	4.68669608217049e-06\\
120	2.71031303517122e-06\\
121	1.56736925432198e-06\\
122	9.06405799470804e-07\\
123	5.24171845532676e-07\\
124	3.03126904541089e-07\\
125	1.75297262936589e-07\\
126	1.01373791938401e-07\\
127	5.86240998304315e-08\\
128	3.39021032834339e-08\\
129	1.96054615986329e-08\\
130	1.13377658668564e-08\\
131	6.55658786276296e-09\\
132	3.79165032942227e-09\\
133	2.19269723288635e-09\\
134	1.26802861634813e-09\\
135	7.33296207198306e-10\\
136	4.24062453267031e-10\\
137	2.45233729766767e-10\\
138	1.41817748009375e-10\\
139	8.20126727780403e-11\\
140	4.74276216264092e-11\\
141	2.74272160159689e-11\\
};
\addlegendentry{45-64}

\addplot [color=mycolor4, solid, line width=2pt, mark options={solid, mycolor4}]
  table[row sep=crcr]{%
1	0.1709101349\\
2	1.87582768989844\\
3	2.36359200427599\\
4	2.56591118709034\\
5	2.72826638743079\\
6	2.91594742008822\\
7	3.14373187537152\\
8	3.41308783615855\\
9	3.72309784759522\\
10	4.07341648887533\\
11	4.46484463839883\\
12	4.89926678179909\\
13	5.37947791961844\\
14	5.90903953918148\\
15	6.49218769372962\\
16	7.13378608028685\\
17	7.83931259621999\\
18	8.61486992904243\\
19	9.46721362184796\\
20	10.4037933786278\\
21	11.4328049944354\\
22	12.5632513683921\\
23	13.8050117538086\\
24	15.1689188483258\\
25	16.6668436144795\\
26	18.3117879025588\\
27	20.1179850571237\\
28	22.101008746521\\
29	24.2778902728701\\
30	26.6672446041426\\
31	29.2894053220658\\
32	32.1665685986642\\
33	35.3229461971487\\
34	38.7849273344956\\
35	42.5812490366565\\
36	46.7431743545237\\
37	51.30467747951\\
38	56.302634390243\\
39	61.7770171631173\\
40	67.7710894743917\\
41	74.3316000937311\\
42	81.5089703008872\\
43	89.357470129972\\
44	97.9353771406443\\
45	107.305110014271\\
46	117.533327659531\\
47	128.690982673636\\
48	140.853315936287\\
49	154.09977681734\\
50	168.513850973232\\
51	184.18277502818\\
52	201.197114645812\\
53	219.650179690573\\
54	239.637247492694\\
55	261.254562854263\\
56	284.598081616699\\
57	309.761923672349\\
58	336.836501644304\\
59	365.906293560417\\
60	397.047232271967\\
61	430.323691744562\\
62	465.785061351877\\
63	503.461914604186\\
64	532.514664344042\\
65	547.223509573944\\
66	546.882281682163\\
67	541.234217188441\\
68	533.985449392348\\
69	531.16295644056\\
70	493.318842596308\\
71	404.926813328385\\
72	308.000170441259\\
73	218.171022495664\\
74	140.271067793364\\
75	74.9795904378307\\
76	21.4960973601816\\
77	2.40728317736323\\
78	0.270032493048976\\
79	0.0306850057844976\\
80	0.00382251607076189\\
81	0.000748453282943668\\
82	0.000340531124775608\\
83	0.000237315284834958\\
84	0.000177621034575398\\
85	0.000132879885431457\\
86	9.78827805061308e-05\\
87	7.13784397100311e-05\\
88	-0.00135931231984898\\
89	-0.00110328407308508\\
90	-0.000769878600908513\\
91	-0.000516203275656614\\
92	-0.000338609390144497\\
93	-0.000218437086612736\\
94	-0.000139001141180419\\
95	-8.74570328416638e-05\\
96	-5.45096574390003e-05\\
97	-3.34663270836752e-05\\
98	-2.01426613844328e-05\\
99	-1.19033991601088e-05\\
100	-6.90230921821564e-06\\
101	-3.99109539516806e-06\\
102	-2.30710528279306e-06\\
103	-1.33376298449923e-06\\
104	-7.71140895811414e-07\\
105	-4.45883173112468e-07\\
106	-2.57828051982412e-07\\
107	-1.49091733584779e-07\\
108	-8.62156938654424e-08\\
109	-4.98568927780344e-08\\
110	-2.88315561225815e-08\\
111	-1.66729928761945e-08\\
112	-9.64185836437224e-09\\
113	-5.57582379655038e-09\\
114	-3.22446791089537e-09\\
115	-1.86469390231063e-09\\
116	-1.07834406090892e-09\\
117	-6.23601788914411e-10\\
118	-3.60626374275627e-10\\
119	-2.0854880444657e-10\\
120	-1.20602963480503e-10\\
121	-6.97442311575129e-11\\
122	-4.03328235615602e-11\\
123	-2.33243192704986e-11\\
124	-1.34883659046796e-11\\
125	-7.80027127258225e-12\\
126	-4.51086758489295e-12\\
127	-2.60861780184363e-12\\
128	-1.50855389499589e-12\\
129	-8.72391064267097e-13\\
130	-5.04500484148938e-13\\
131	-2.91750740272565e-13\\
132	-1.68718360417012e-13\\
133	-9.75691959839805e-14\\
134	-5.64239006623622e-14\\
135	-3.26297304649418e-14\\
136	-1.88696509433967e-14\\
137	-1.09122484827357e-14\\
138	-6.31051243678318e-15\\
139	-3.64934571269589e-15\\
140	-2.11040296085664e-15\\
141	-1.2204381299836e-15\\
};
\addlegendentry{65-74}

\addplot [color=mycolor5, solid, line width=2pt, mark options={solid, mycolor5}]
  table[row sep=crcr]{%
1	1.4936938584\\
2	1.77720424756331\\
3	1.97208819133525\\
4	2.09210161269225\\
5	2.20766512192823\\
6	2.34838321957669\\
7	2.52242784280443\\
8	2.73067517455431\\
9	2.97243980366291\\
10	3.24737063818024\\
11	3.5559288165295\\
12	3.89940842664589\\
13	4.27984283233873\\
14	4.699907646182\\
15	5.16284983107652\\
16	5.67244583062366\\
17	6.23298434241356\\
18	6.84926842428813\\
19	7.52663251013344\\
20	8.27097108868828\\
21	9.08877681832409\\
22	9.98718662761832\\
23	10.9740349108023\\
24	12.0579133172633\\
25	13.2482368999631\\
26	14.555316563485\\
27	15.9904378630369\\
28	17.5659462676596\\
29	19.2953390245193\\
30	21.1933637524246\\
31	23.2761238541208\\
32	25.5611907684891\\
33	28.0677229834616\\
34	30.8165915945138\\
35	33.8305120167778\\
36	37.134181234456\\
37	40.7544196913252\\
38	44.720316581341\\
39	49.0633768780799\\
40	53.8176679341207\\
41	59.0199628735419\\
42	64.7098772787124\\
43	70.9299948222382\\
44	77.7259765022092\\
45	85.1466469906876\\
46	93.244050290844\\
47	102.073465410331\\
48	111.693371096545\\
49	122.165346851506\\
50	133.553895470843\\
51	145.926170270688\\
52	159.351588038229\\
53	173.901306655308\\
54	189.647544424266\\
55	206.662716539351\\
56	225.018363113808\\
57	244.783842968599\\
58	266.024768351151\\
59	288.801158282396\\
60	313.165292785111\\
61	339.159257326972\\
62	366.8121769348\\
63	396.13715309379\\
64	371.415301078135\\
65	258.222663565509\\
66	145.001457450006\\
67	40.0328661785858\\
68	3.17337742372586\\
69	0.251707261446421\\
70	0.0201191839236639\\
71	0.00175565383468763\\
72	0.000288711012491524\\
73	0.000159680296355414\\
74	0.000136715394798773\\
75	0.00012275672421647\\
76	0.000110334225300258\\
77	9.89245115725615e-05\\
78	8.89776491223763e-05\\
79	7.8963264592927e-05\\
80	6.81391142850473e-05\\
81	5.69535668451658e-05\\
82	4.60872344952844e-05\\
83	3.61351828342302e-05\\
84	2.76194305167703e-05\\
85	2.07399155191691e-05\\
86	1.52853203707517e-05\\
87	1.11342011318031e-05\\
88	-0.00171392219850533\\
89	-0.00133428017342557\\
90	-0.000924957275688319\\
91	-0.000621713447697825\\
92	-0.000409619678383617\\
93	-0.000265465453412232\\
94	-0.000169664025486913\\
95	-0.000107173414034885\\
96	-6.70355375883544e-05\\
97	-4.12368533970501e-05\\
98	-2.48617502978992e-05\\
99	-1.47279637768708e-05\\
100	-8.56090542217553e-06\\
101	-4.96121215650353e-06\\
102	-2.87260099821923e-06\\
103	-1.66252899564148e-06\\
104	-9.61928991059578e-07\\
105	-5.56466585389555e-07\\
106	-3.21872992583666e-07\\
107	-1.86164537679407e-07\\
108	-1.07668291831224e-07\\
109	-6.22679583136923e-08\\
110	-3.60107569951439e-08\\
111	-2.08254256167637e-08\\
112	-1.20434693536181e-08\\
113	-6.96477011513693e-09\\
114	-4.02772949744624e-09\\
115	-2.32923183970093e-09\\
116	-1.34699018529477e-09\\
117	-7.78960940793621e-10\\
118	-4.50470784232833e-10\\
119	-2.60505790053622e-10\\
120	-1.50649606842877e-10\\
121	-8.71201352079293e-11\\
122	-5.03812599282515e-11\\
123	-2.91352984789456e-11\\
124	-1.6848835718182e-11\\
125	-9.74361924931458e-12\\
126	-5.63469877020916e-12\\
127	-3.25852529166811e-12\\
128	-1.88439300942949e-12\\
129	-1.08973743443072e-12\\
130	-6.301910828159e-13\\
131	-3.64437145348353e-13\\
132	-2.10752637249341e-13\\
133	-1.21877461238923e-13\\
134	-7.04812795735806e-14\\
135	-4.07590601033712e-14\\
136	-2.35708118528417e-14\\
137	-1.36309122407755e-14\\
138	-7.88270551096398e-15\\
139	-4.55853908157035e-15\\
140	-2.63618608206148e-15\\
141	-1.52449654036885e-15\\
};
\addlegendentry{75-84}

\addplot [color=mycolor6, solid, line width=2pt, mark options={solid, mycolor6}]
  table[row sep=crcr]{%
1	1.6144863665\\
2	1.91488375921671\\
3	2.11256899616501\\
4	2.26740672199893\\
5	2.41416297476468\\
6	2.57565533323225\\
7	2.76427135446697\\
8	2.98569224151827\\
9	3.24225913212804\\
10	3.53510851648303\\
11	3.8652934316825\\
12	4.23428667919582\\
13	4.64416346595059\\
14	5.09763962630998\\
15	5.59805718477238\\
16	6.14936027127337\\
17	6.75607908726744\\
18	7.42332758020236\\
19	8.15681532377509\\
20	8.96287225740088\\
21	9.84848456635727\\
22	10.8213401860683\\
23	11.8898827804771\\
24	13.0633734052453\\
25	14.3519593623189\\
26	15.7667499714428\\
27	17.3198991324327\\
28	19.0246946398069\\
29	20.8956542482165\\
30	22.9486284797096\\
31	25.200910115944\\
32	27.6713502308905\\
33	30.3804804907944\\
34	33.3506412744961\\
35	36.6061149431006\\
36	40.1732633061732\\
37	44.0806679833478\\
38	48.3592719352008\\
39	53.0425199239504\\
40	58.1664950503123\\
41	63.7700477842033\\
42	69.894913050048\\
43	76.5858099285775\\
44	83.8905173837548\\
45	91.8599181057806\\
46	100.548001073197\\
47	110.011811779509\\
48	120.311337252525\\
49	131.509311040888\\
50	143.670921292971\\
51	156.863402972371\\
52	171.155493235427\\
53	186.61672716962\\
54	203.316549632068\\
55	221.323218060548\\
56	240.70247113991\\
57	261.515939440555\\
58	283.819277011407\\
59	307.659997871502\\
60	333.075008905866\\
61	360.087841347813\\
62	388.705597302223\\
63	219.083640865177\\
64	41.0442647996475\\
65	7.6911386618324\\
66	1.44283883889735\\
67	0.27223731857607\\
68	0.05287309776357\\
69	0.0117621333038809\\
70	0.00403978291691938\\
71	0.002528604376958\\
72	0.00214940661568291\\
73	0.00196952231999959\\
74	0.00182351025319701\\
75	0.0016849973503293\\
76	0.00155142224456496\\
77	0.00142357694267162\\
78	0.00129551214036657\\
79	0.00115553829936618\\
80	0.00100174706638911\\
81	0.000841311118431611\\
82	0.000684044835903704\\
83	0.000538785410628347\\
84	0.000413656443735292\\
85	0.000312139141422511\\
86	0.000230669878165164\\
87	0.000168687964804624\\
88	-0.00810546104978984\\
89	-0.00720556877006097\\
90	-0.00523451691883828\\
91	-0.00357720108727936\\
92	-0.0023726713651214\\
93	-0.00154255037914557\\
94	-0.000987455379671298\\
95	-0.000624213900027382\\
96	-0.000390511918504393\\
97	-0.000239753916769072\\
98	-0.000144033920890525\\
99	-8.49620968258729e-05\\
100	-4.91138864454671e-05\\
101	-2.82766549728705e-05\\
102	-1.62829173842122e-05\\
103	-9.38608812361305e-06\\
104	-5.41575578116437e-06\\
105	-3.12717027541516e-06\\
106	-1.80661480370005e-06\\
107	-1.0440681503326e-06\\
108	-6.03519178533008e-07\\
109	-3.48913998409757e-07\\
110	-2.01738304255904e-07\\
111	-1.16650400886999e-07\\
112	-6.74531651323095e-08\\
113	-3.90059027917165e-08\\
114	-2.25562088301346e-08\\
115	-1.30438847377478e-08\\
116	-7.54312161448261e-09\\
117	-4.36211818852306e-09\\
118	-2.52258121665124e-09\\
119	-1.45879346013455e-09\\
120	-8.4361260204896e-10\\
121	-4.87857180973438e-10\\
122	-2.821256685145e-10\\
123	-1.63152099755417e-10\\
124	-9.43502038598574e-11\\
125	-5.45623528022427e-11\\
126	-3.15531979789882e-11\\
127	-1.82470938490122e-11\\
128	-1.0552225144567e-11\\
129	-6.10231194810132e-12\\
130	-3.52894402124053e-12\\
131	-2.04077506504434e-12\\
132	-1.18017256025404e-12\\
133	-6.82489365608863e-13\\
134	-3.94681041990523e-13\\
135	-2.28242567744378e-13\\
136	-1.31991821917696e-13\\
137	-7.63303762413144e-14\\
138	-4.41415707141098e-14\\
139	-2.55269050392401e-14\\
140	-1.47621135931053e-14\\
141	-8.53687501244184e-15\\
};
\addlegendentry{85+}

\addplot[area legend, draw=none, fill=green, fill opacity=0.1, line width=2.0pt, forget plot]
table[row sep=crcr] {%
x	y\\
98	0\\
150	0\\
150	3500\\
98	3500\\
}--cycle;

\addplot [color=green, dashed, line width=2.0pt, mark options={solid, green}]
  table[row sep=crcr]{%
98	0\\
98	3500\\
};
\addlegendentry{Erad.}
\end{axis}
\end{tikzpicture}%

%% file: Figures/MPC_state_3.tex
%
%
\definecolor{mycolor1}{rgb}{0.00000,0.44700,0.74100}%
\definecolor{mycolor2}{rgb}{0.85000,0.32500,0.09800}%
\definecolor{mycolor3}{rgb}{0.92900,0.69400,0.12500}%
\definecolor{mycolor4}{rgb}{0.49400,0.18400,0.55600}%
\definecolor{mycolor5}{rgb}{0.46600,0.67400,0.18800}%
\definecolor{mycolor6}{rgb}{0.30100,0.74500,0.93300}%
\begin{tikzpicture}

\begin{axis}[%
width=7cm,
height=3.5cm,
at={(0cm,0cm)},
scale only axis,
xmin=61,
xmax=140,
ymin=0,
ymax=1200000,
xmajorgrids,
ymajorgrids,
axis background/.style={fill=white},
axis x line*=bottom,
axis y line*=left
]
\addplot [color=mycolor1, line width=2.0pt, forget plot]
  table[row sep=crcr]{%
1	0\\
2	4.29463568177537\\
3	10.9834023030202\\
4	19.4911544162207\\
5	29.5257282182798\\
6	40.9642415461055\\
7	53.7853873518582\\
8	68.0310573035046\\
9	83.7851224878908\\
10	101.161965611055\\
11	120.300529254238\\
12	141.361518531931\\
13	164.5264538773\\
14	189.997854902757\\
15	218.000158520477\\
16	248.781152709582\\
17	282.613807035573\\
18	319.798437871262\\
19	360.665179499164\\
20	405.576752343279\\
21	454.93153219867\\
22	509.166932663109\\
23	568.763118800506\\
24	634.247074389002\\
25	706.19704850482\\
26	785.247409994935\\
27	872.093940777221\\
28	967.499600963008\\
29	1072.30080054838\\
30	1187.41421384591\\
31	1313.84417387\\
32	1452.69068445472\\
33	1605.15808784897\\
34	1772.56442474104\\
35	1956.35152191561\\
36	2158.09583980334\\
37	2379.52010775807\\
38	2622.50576865146\\
39	2889.10624590808\\
40	3181.56103495083\\
41	3502.31060664882\\
42	3854.01209214178\\
43	4239.55569566054\\
44	4662.08175389484\\
45	5124.99832622178\\
46	5631.99915877767\\
47	6187.08181595916\\
48	6794.5657144823\\
49	7459.10972663968\\
50	8185.72893999198\\
51	8979.81006969937\\
52	9847.12491661685\\
53	10793.84114916\\
54	11826.5295604339\\
55	12952.1668157025\\
56	14178.1325615878\\
57	15512.1996214985\\
58	16962.5158575682\\
59	18537.5761458625\\
60	20246.1827993291\\
61	22097.3926952638\\
62	24100.449337258\\
63	26264.6981249527\\
64	28599.483240316\\
65	31104.9298232086\\
66	33773.1310800134\\
67	74455.7695285721\\
68	132603.644564376\\
69	165005.421727883\\
70	167844.654122999\\
71	170495.365544955\\
72	172989.60090099\\
73	175340.502272969\\
74	177553.055040639\\
75	179629.305316755\\
76	218626.792330323\\
77	275624.592703108\\
78	332427.520452073\\
79	388976.137349474\\
80	445252.019358305\\
81	501267.219077023\\
82	529510.955923275\\
83	531993.708730616\\
84	532639.932464362\\
85	532841.708004896\\
86	588179.418913935\\
87	643072.428313054\\
88	698333.152556028\\
89	753566.477809372\\
90	808781.300333077\\
91	863984.659464202\\
92	919181.409256508\\
93	974374.585895939\\
94	1029565.94618655\\
95	1029566.44139552\\
96	1029566.55665229\\
97	1052043.60960024\\
98	1056436.64004249\\
99	1056436.64578391\\
100	1056436.64679334\\
101	1056436.6471885\\
102	1056436.64739993\\
103	1056436.64752003\\
104	1056436.647589\\
105	1056436.64762873\\
106	1056436.64765165\\
107	1056436.64766489\\
108	1056436.64767253\\
109	1056436.64767695\\
110	1056436.64767951\\
111	1056436.64768098\\
112	1056436.64768184\\
113	1056436.64768233\\
114	1056436.64768261\\
115	1056436.64768278\\
116	1056436.64768288\\
117	1056436.64768293\\
118	1056436.64768296\\
119	1056436.64768298\\
120	1056436.64768299\\
121	1056436.647683\\
122	1056436.647683\\
123	1056436.647683\\
124	1056436.647683\\
125	1056436.64768301\\
126	1056436.64768301\\
127	1056436.64768301\\
128	1056436.64768301\\
129	1056436.64768301\\
130	1056436.64768301\\
131	1056436.64768301\\
132	1056436.64768301\\
133	1056436.64768301\\
134	1056436.64768301\\
135	1056436.64768301\\
136	1056436.64768301\\
137	1056436.64768301\\
138	1056436.64768301\\
139	1056436.64768301\\
140	1056436.64768301\\
141	1056436.64768301\\
};
\addplot [color=mycolor2, line width=2.0pt, forget plot]
  table[row sep=crcr]{%
1	0\\
2	3.13035311771256\\
3	5.99363697996894\\
4	8.81301644771832\\
5	11.7286320997933\\
6	14.8314323952082\\
7	18.1851086922459\\
8	21.8394191967353\\
9	25.8379249437466\\
10	30.2224138428075\\
11	35.0354689152852\\
12	40.3220343379938\\
13	46.1304569486418\\
14	52.5132644113878\\
15	59.5278219235952\\
16	67.2369451701362\\
17	75.7095131997929\\
18	85.0211070106206\\
19	95.2546903059363\\
20	106.501344088687\\
21	118.861064375482\\
22	132.443631216484\\
23	147.369556815285\\
24	163.771120540924\\
25	181.793498842764\\
26	201.595998423233\\
27	223.353401437379\\
28	247.257431934985\\
29	273.518353213499\\
30	302.366706183846\\
31	334.055199242422\\
32	368.860760464201\\
33	407.086763153045\\
34	449.065435869127\\
35	495.160467955635\\
36	545.769821254805\\
37	601.328758073296\\
38	662.313094453735\\
39	729.242686343325\\
40	802.685154216735\\
41	883.259848986073\\
42	971.642058473118\\
43	1068.56744916448\\
44	1174.8367322327\\
45	1291.32053567606\\
46	1418.96445567475\\
47	1558.79424962783\\
48	1711.92112055339\\
49	1879.54702732146\\
50	2062.96993726491\\
51	2263.58891680757\\
52	2482.90893163022\\
53	2722.5452003921\\
54	2984.22691506895\\
55	3269.80010663332\\
56	3581.22939737128\\
57	3920.59834115716\\
58	4290.10801140896\\
59	4692.07345458592\\
60	5128.91758687902\\
61	5603.1620757499\\
62	6117.41471949415\\
63	6674.35282115129\\
64	7276.70205279021\\
65	7920.17996996011\\
66	8599.49067208886\\
67	9311.69095949914\\
68	10053.8795307135\\
69	10823.0577067716\\
70	11608.0025311918\\
71	12388.5356768127\\
72	13146.4968028434\\
73	13872.1116892158\\
74	14560.1363098495\\
75	15207.9295912462\\
76	15814.4340522375\\
77	16379.6111055873\\
78	16904.1113149666\\
79	17383.6503443024\\
80	17810.7286953654\\
81	18179.2904187004\\
82	18486.910461852\\
83	18735.2114477968\\
84	58091.2160215723\\
85	58239.2302122755\\
86	58347.5428684992\\
87	58640.8690192272\\
88	58698.8144473195\\
89	58740.4048417709\\
90	58769.575092253\\
91	58789.5456531122\\
92	58802.910104634\\
93	58811.6746437456\\
94	58817.3239726477\\
95	114010.913165898\\
96	169203.166530296\\
97	201917.517170812\\
98	252715.301291529\\
99	252715.750683093\\
100	252716.000424642\\
101	252716.140903478\\
102	252716.220643192\\
103	252716.266189971\\
104	252716.292315663\\
105	252716.307343337\\
106	252716.316003319\\
107	252716.320999856\\
108	252716.323884994\\
109	252716.32555182\\
110	252716.326515121\\
111	252716.327071962\\
112	252716.327393892\\
113	252716.32758003\\
114	252716.32768766\\
115	252716.327749898\\
116	252716.327785888\\
117	252716.3278067\\
118	252716.327818735\\
119	252716.327825695\\
120	252716.32782972\\
121	252716.327832047\\
122	252716.327833393\\
123	252716.327834172\\
124	252716.327834622\\
125	252716.327834882\\
126	252716.327835033\\
127	252716.32783512\\
128	252716.32783517\\
129	252716.327835199\\
130	252716.327835216\\
131	252716.327835226\\
132	252716.327835232\\
133	252716.327835235\\
134	252716.327835237\\
135	252716.327835238\\
136	252716.327835238\\
137	252716.327835239\\
138	252716.327835239\\
139	252716.327835239\\
140	252716.327835239\\
141	252716.327835239\\
};
\addplot [color=mycolor3, line width=2.0pt, forget plot]
  table[row sep=crcr]{%
1	0\\
2	2.76332081289318\\
3	4.53285380309558\\
4	5.88719090711098\\
5	7.08987683975862\\
6	8.26793475218234\\
7	9.48767244104269\\
8	10.7877399204971\\
9	12.1940512268626\\
10	13.7267865607043\\
11	15.4038243359572\\
12	17.2425132723664\\
13	19.2606512408868\\
14	21.4770771369171\\
15	23.9120724142912\\
16	26.5876705648774\\
17	29.5279254004995\\
18	32.7591655449052\\
19	36.3102506933262\\
20	40.2128390992851\\
21	44.5016725719945\\
22	49.2148836105955\\
23	54.3943284535773\\
24	60.0859494156922\\
25	66.3401697260611\\
26	73.2123240642561\\
27	80.7631280564633\\
28	89.0591901062684\\
29	98.1735690719099\\
30	108.186381448341\\
31	119.185461855576\\
32	131.267080763061\\
33	144.536723481308\\
34	159.109934513798\\
35	175.113231368963\\
36	192.685091866286\\
37	211.97701881112\\
38	233.154685634665\\
39	256.399166169156\\
40	281.908251118668\\
41	309.897852952102\\
42	340.603499839258\\
43	374.281917818236\\
44	411.212698559806\\
45	451.700047810437\\
46	496.07460677053\\
47	544.695335210158\\
48	597.951440945963\\
49	656.264335299236\\
50	720.08958822309\\
51	789.918849824606\\
52	866.281696921875\\
53	949.747353988956\\
54	1040.92622730413\\
55	1140.47117932198\\
56	1249.07845729183\\
57	1367.48817608278\\
58	1496.48424029878\\
59	1636.89357546856\\
60	1789.58452294491\\
61	1955.4642389294\\
62	2135.47492578423\\
63	2330.58871481326\\
64	2541.8010155955\\
65	2767.55696549906\\
66	3005.20043689685\\
67	3252.30688524823\\
68	3507.22619951732\\
69	3768.76753562427\\
70	4034.96587054797\\
71	4301.39523033234\\
72	4562.51418249878\\
73	4813.90583224214\\
74	5052.77013674165\\
75	5277.58631664244\\
76	5487.69717910667\\
77	5682.99567144592\\
78	5863.71981975721\\
79	6029.3398408159\\
80	6178.28076321121\\
81	6308.82679905563\\
82	33961.9538400421\\
83	87171.7039335882\\
84	102907.160613371\\
85	158150.811740558\\
86	158191.571117815\\
87	158278.898396717\\
88	158300.115077864\\
89	158315.370946672\\
90	158326.1920374\\
91	158333.733917266\\
92	158338.892422533\\
93	158342.356600856\\
94	158344.643837624\\
95	158346.131272737\\
96	158347.085819602\\
97	158347.691380542\\
98	158348.070196422\\
99	158348.303292299\\
100	158348.44439158\\
101	158348.528353094\\
102	158348.577799014\\
103	158348.606729723\\
104	158348.623587177\\
105	158348.633383688\\
106	158348.639067056\\
107	158348.642360552\\
108	158348.644267744\\
109	158348.645371638\\
110	158348.646010383\\
111	158348.646379905\\
112	158348.646593651\\
113	158348.646717279\\
114	158348.64678878\\
115	158348.646830132\\
116	158348.646854046\\
117	158348.646867876\\
118	158348.646875874\\
119	158348.6468805\\
120	158348.646883175\\
121	158348.646884721\\
122	158348.646885616\\
123	158348.646886133\\
124	158348.646886432\\
125	158348.646886605\\
126	158348.646886705\\
127	158348.646886763\\
128	158348.646886797\\
129	158348.646886816\\
130	158348.646886827\\
131	158348.646886834\\
132	158348.646886838\\
133	158348.64688684\\
134	158348.646886841\\
135	158348.646886842\\
136	158348.646886842\\
137	158348.646886842\\
138	158348.646886842\\
139	158348.646886843\\
140	158348.646886843\\
141	158348.646886843\\
};
\addplot [color=mycolor4, line width=2.0pt, forget plot]
  table[row sep=crcr]{%
1	0\\
2	0.144981564007577\\
3	1.73622969244256\\
4	3.74124399809646\\
5	5.91788388681085\\
6	8.23224826380028\\
7	10.705820809838\\
8	13.3726209055639\\
9	16.2679132934087\\
10	19.4261844469212\\
11	22.8816278222187\\
12	26.6691162535201\\
13	30.8251211676301\\
14	35.3884849657269\\
15	40.4010712273574\\
16	45.9083369387699\\
17	51.9598649176146\\
18	58.6098848336909\\
19	65.9178032129327\\
20	73.948757249055\\
21	82.7742036541361\\
22	92.4725515611032\\
23	103.129847183013\\
24	114.840517216603\\
25	127.708177640726\\
26	141.846514467332\\
27	157.380243063599\\
28	174.44615281769\\
29	193.19424412363\\
30	213.788964879216\\
31	236.410553895827\\
32	261.256498783352\\
33	288.543115969162\\
34	318.507260506374\\
35	351.408173188735\\
36	387.529472176316\\
37	427.181295800251\\
38	470.702602399393\\
39	518.463631880883\\
40	570.868532112646\\
41	628.358151158607\\
42	691.412993652765\\
43	760.5563361572\\
44	836.357492026533\\
45	919.435210956603\\
46	1010.46119186158\\
47	1110.16367982149\\
48	1219.33110838043\\
49	1338.81573725891\\
50	1469.53722237896\\
51	1612.48603981061\\
52	1768.72666768482\\
53	1939.40041018832\\
54	2125.72772544507\\
55	2329.00989450011\\
56	2550.62984201749\\
57	2792.05189115745\\
58	3054.82020615697\\
59	3340.55564748586\\
60	3650.95073758186\\
61	3987.76241205249\\
62	4352.80221437401\\
63	4747.92358459394\\
64	5175.00589799959\\
65	5626.73340344414\\
66	6090.9382929094\\
67	6554.8537217557\\
68	7013.97794717574\\
69	33229.9531066389\\
70	88870.5339701593\\
71	144479.012004805\\
72	200012.507858573\\
73	255463.781693804\\
74	310838.854253023\\
75	366147.844965924\\
76	384344.449492927\\
77	384362.684443225\\
78	384364.726520369\\
79	384364.955586557\\
80	384364.981616378\\
81	384364.984858984\\
82	384364.985493891\\
83	384364.98578276\\
84	384364.985984073\\
85	384364.986134747\\
86	384364.986247468\\
87	384393.986330501\\
88	384393.986391051\\
89	384393.985237958\\
90	384393.984302052\\
91	384393.983648971\\
92	384393.98321108\\
93	384393.982923841\\
94	384393.982738542\\
95	384393.982620629\\
96	384393.98254644\\
97	384393.9825002\\
98	384393.98247181\\
99	384393.982454724\\
100	384393.982444626\\
101	384393.982438771\\
102	384393.982435385\\
103	384393.982433428\\
104	384393.982432297\\
105	384393.982431643\\
106	384393.982431264\\
107	384393.982431046\\
108	384393.982430919\\
109	384393.982430846\\
110	384393.982430804\\
111	384393.982430779\\
112	384393.982430765\\
113	384393.982430757\\
114	384393.982430752\\
115	384393.98243075\\
116	384393.982430748\\
117	384393.982430747\\
118	384393.982430746\\
119	384393.982430746\\
120	384393.982430746\\
121	384393.982430746\\
122	384393.982430746\\
123	384393.982430746\\
124	384393.982430746\\
125	384393.982430746\\
126	384393.982430746\\
127	384393.982430746\\
128	384393.982430746\\
129	384393.982430746\\
130	384393.982430746\\
131	384393.982430746\\
132	384393.982430746\\
133	384393.982430746\\
134	384393.982430746\\
135	384393.982430746\\
136	384393.982430746\\
137	384393.982430746\\
138	384393.982430746\\
139	384393.982430746\\
140	384393.982430746\\
141	384393.982430746\\
};
\addplot [color=mycolor5, line width=2.0pt, forget plot]
  table[row sep=crcr]{%
1	0\\
2	1.22489296657866\\
3	2.68227660029449\\
4	4.29947341840747\\
5	6.01508638444028\\
6	7.82546637978117\\
7	9.75124124476728\\
8	11.8197401583686\\
9	14.0590108071114\\
10	16.4965388109992\\
11	19.1595218795925\\
12	22.0755358757182\\
13	25.2732178697597\\
14	28.7828723775511\\
15	32.6369980318831\\
16	36.8707561142531\\
17	41.5224047517269\\
18	46.6337189871589\\
19	52.2504125766382\\
20	58.4225737406134\\
21	65.2051244328696\\
22	72.6583108652577\\
23	80.8482318368696\\
24	89.8474106857598\\
25	99.7354162706203\\
26	110.599538196994\\
27	122.535521453989\\
28	135.648365669555\\
29	150.053194285241\\
30	165.87619906361\\
31	183.255665446565\\
32	202.343084356264\\
33	223.304356047664\\
34	246.321091556738\\
35	271.592017112046\\
36	299.334486555881\\
37	329.786106315813\\
38	363.206476732527\\
39	399.879052532184\\
40	440.11312386922\\
41	484.245917586925\\
42	532.64481606601\\
43	585.709688162202\\
44	643.875323167414\\
45	707.613954348502\\
46	777.437853295484\\
47	853.901969910989\\
48	937.60658525255\\
49	1029.1999354563\\
50	1129.38075448861\\
51	1238.90067137215\\
52	1358.56638372621\\
53	1489.24151391002\\
54	1631.8480367949\\
55	1787.36714935457\\
56	1956.83943212491\\
57	2141.36413160007\\
58	2342.09737147929\\
59	2560.24908031613\\
60	2797.078404833\\
61	3053.88736361116\\
62	3332.01248711863\\
63	18573.8141895934\\
64	74088.6636290527\\
65	129583.240090562\\
66	184984.993739166\\
67	202430.901147384\\
68	202463.729813003\\
69	202466.332118465\\
70	202466.538529205\\
71	202466.555027798\\
72	202466.556467509\\
73	202466.556704265\\
74	202466.556835209\\
75	202466.556947322\\
76	202466.557047988\\
77	202466.557138466\\
78	202466.557219589\\
79	202466.557292554\\
80	202466.557357307\\
81	202466.557413184\\
82	202466.557459889\\
83	202466.557497682\\
84	202466.557527315\\
85	202466.557549964\\
86	202466.557566971\\
87	202492.557579506\\
88	202492.557588637\\
89	202492.556183147\\
90	202492.55508898\\
91	202492.554330475\\
92	202492.553820644\\
93	202492.553484738\\
94	202492.553267045\\
95	202492.553127913\\
96	202492.553040026\\
97	202492.552985054\\
98	202492.552951238\\
99	202492.552930851\\
100	202492.552918773\\
101	202492.552911753\\
102	202492.552907684\\
103	202492.552905329\\
104	202492.552903965\\
105	202492.552903177\\
106	202492.55290272\\
107	202492.552902456\\
108	202492.552902304\\
109	202492.552902215\\
110	202492.552902164\\
111	202492.552902135\\
112	202492.552902118\\
113	202492.552902108\\
114	202492.552902102\\
115	202492.552902099\\
116	202492.552902097\\
117	202492.552902096\\
118	202492.552902095\\
119	202492.552902095\\
120	202492.552902095\\
121	202492.552902094\\
122	202492.552902094\\
123	202492.552902094\\
124	202492.552902094\\
125	202492.552902094\\
126	202492.552902094\\
127	202492.552902094\\
128	202492.552902094\\
129	202492.552902094\\
130	202492.552902094\\
131	202492.552902094\\
132	202492.552902094\\
133	202492.552902094\\
134	202492.552902094\\
135	202492.552902094\\
136	202492.552902094\\
137	202492.552902094\\
138	202492.552902094\\
139	202492.552902094\\
140	202492.552902094\\
141	202492.552902094\\
};
\addplot [color=mycolor6, line width=2.0pt, forget plot]
  table[row sep=crcr]{%
1	0\\
2	1.06753654407652\\
3	2.33370294413975\\
4	3.73058349518353\\
5	5.22984641015384\\
6	6.82614802805157\\
7	8.5292322102561\\
8	10.3570337636596\\
9	12.3312440408861\\
10	14.475102410168\\
11	16.8125997139099\\
12	19.3684230875351\\
13	22.1682335176085\\
14	25.239064166563\\
15	28.6097439707015\\
16	32.3113116919143\\
17	36.3774140440252\\
18	40.8446932171313\\
19	45.7531728642508\\
20	51.1466519434767\\
21	57.0731149184523\\
22	63.5851656836453\\
23	70.7404915781311\\
24	78.6023630911511\\
25	87.240174340836\\
26	96.730029081196\\
27	107.155376811026\\
28	118.607703474939\\
29	131.187281221339\\
30	145.003981681138\\
31	160.178157225051\\
32	176.841594619744\\
33	195.138545407538\\
34	215.226837153746\\
35	237.279069410184\\
36	261.48389779975\\
37	288.047408996887\\
38	317.194588518473\\
39	349.170882098255\\
40	384.243849937234\\
41	422.704911235428\\
42	464.871174041799\\
43	511.08734252375\\
44	561.727690161802\\
45	617.198083016708\\
46	677.938031986601\\
47	744.422746758259\\
48	817.165156847042\\
49	896.717856608899\\
50	983.674921304851\\
51	1078.67353014036\\
52	1182.3953196679\\
53	1295.56737707236\\
54	1418.96276778193\\
55	1553.40047580563\\
56	1699.74461858366\\
57	1858.90278152834\\
58	2031.82330164182\\
59	2219.49131569875\\
60	2422.92337786621\\
61	2643.16044601591\\
62	58071.2590374272\\
63	98577.2803654591\\
64	98722.1436468629\\
65	98749.2830848338\\
66	98754.3686474978\\
67	98755.3226866397\\
68	98755.5026963891\\
69	98755.537657331\\
70	98755.5454347316\\
71	98755.5481059315\\
72	98755.5497779045\\
73	98755.5511991429\\
74	98755.5525014376\\
75	98755.5537071857\\
76	98755.5548213458\\
77	98755.5558471829\\
78	98755.5567884856\\
79	98755.5576451088\\
80	98755.558409178\\
81	98755.5590715569\\
82	98755.5596278517\\
83	98755.5600801583\\
84	98755.5604364159\\
85	98755.5607099354\\
86	98755.5609163291\\
87	98831.5610688535\\
88	98831.561180394\\
89	98831.5558208716\\
90	98831.5510563792\\
91	98831.5475951929\\
92	98831.545229863\\
93	98831.5436609966\\
94	98831.5426410258\\
95	98831.541988097\\
96	98831.541575352\\
97	98831.5413171363\\
98	98831.5411586053\\
99	98831.5410633667\\
100	98831.5410071878\\
101	98831.5409747125\\
102	98831.5409560153\\
103	98831.5409452486\\
104	98831.5409390423\\
105	98831.5409354613\\
106	98831.5409333936\\
107	98831.540932199\\
108	98831.5409315086\\
109	98831.5409311096\\
110	98831.5409308789\\
111	98831.5409307455\\
112	98831.5409306683\\
113	98831.5409306237\\
114	98831.540930598\\
115	98831.540930583\\
116	98831.5409305744\\
117	98831.5409305694\\
118	98831.5409305665\\
119	98831.5409305649\\
120	98831.5409305639\\
121	98831.5409305634\\
122	98831.540930563\\
123	98831.5409305628\\
124	98831.5409305627\\
125	98831.5409305627\\
126	98831.5409305627\\
127	98831.5409305626\\
128	98831.5409305626\\
129	98831.5409305626\\
130	98831.5409305626\\
131	98831.5409305626\\
132	98831.5409305626\\
133	98831.5409305626\\
134	98831.5409305626\\
135	98831.5409305626\\
136	98831.5409305626\\
137	98831.5409305626\\
138	98831.5409305626\\
139	98831.5409305626\\
140	98831.5409305626\\
141	98831.5409305626\\
};

\addplot[area legend, draw=none, fill=green, fill opacity=0.1, line width=2.0pt, forget plot]
table[row sep=crcr] {%
x	y\\
98	0\\
150	0\\
150	1200000\\
98	1200000\\
}--cycle;
\addplot [color=green, dashed, line width=2.0pt, forget plot]
  table[row sep=crcr]{%
98	0\\
98	1200000\\
};
\end{axis}
\end{tikzpicture}%

%% file: Figures/MPC_state_4.tex
%
%
\definecolor{mycolor1}{rgb}{0.00000,0.44700,0.74100}%
\definecolor{mycolor2}{rgb}{0.85000,0.32500,0.09800}%
\definecolor{mycolor3}{rgb}{0.92900,0.69400,0.12500}%
\definecolor{mycolor4}{rgb}{0.49400,0.18400,0.55600}%
\definecolor{mycolor5}{rgb}{0.46600,0.67400,0.18800}%
\definecolor{mycolor6}{rgb}{0.30100,0.74500,0.93300}%
\begin{tikzpicture}

\begin{axis}[%
width=7cm,
height=3.5cm,
at={(0cm,0cm)},
scale only axis,
xmin=61,
xmax=140,
ymin=0,
ymax=800,
xmajorgrids,
ymajorgrids,
axis background/.style={fill=white},
axis x line*=bottom,
axis y line*=left
]
\addplot [color=mycolor1, line width=2.0pt, forget plot]
  table[row sep=crcr]{%
1	0\\
2	0.00205352335266638\\
3	0.00525182455328961\\
4	0.00931989200713513\\
5	0.0141180246459586\\
6	0.0195874651244966\\
7	0.0257180252629923\\
8	0.0325297359848644\\
9	0.0400627010959333\\
10	0.048371613840375\\
11	0.0575229110143233\\
12	0.0675934353886098\\
13	0.0786699827885641\\
14	0.0908493900088023\\
15	0.104238973821805\\
16	0.118957216548155\\
17	0.135134641337862\\
18	0.152914847492578\\
19	0.172455692048125\\
20	0.193930613432432\\
21	0.217530099048604\\
22	0.243463302618713\\
23	0.27195981990942\\
24	0.303271633527648\\
25	0.337675239099494\\
26	0.375473966485274\\
27	0.417000510824364\\
28	0.462619689186677\\
29	0.512731439445091\\
30	0.567774078664699\\
31	0.628227838802702\\
32	0.694618697783339\\
33	0.76752252399592\\
34	0.847569551884924\\
35	0.935449205464931\\
36	1.03191528518579\\
37	1.13779153145777\\
38	1.25397757515994\\
39	1.38145528140676\\
40	1.5212954875148\\
41	1.6746651292366\\
42	1.84283474061763\\
43	2.02718630195185\\
44	2.22922139689045\\
45	2.45056962338697\\
46	2.69299718339932\\
47	2.95841555264993\\
48	3.24889010379286\\
49	3.5666485235887\\
50	3.91408882672011\\
51	4.29378672535344\\
52	4.70850206426271\\
53	5.16118397628744\\
54	5.654974352402\\
55	6.19320915544583\\
56	6.77941703785782\\
57	7.41731465352169\\
58	8.11079798486894\\
59	8.86392894254139\\
60	9.6809164412259\\
61	10.566091118119\\
62	11.5238728476621\\
63	12.5587302269204\\
64	13.6751312707161\\
65	14.8731358159354\\
66	16.1489631494888\\
67	17.4972017031715\\
68	18.9115394748405\\
69	20.3344373293719\\
70	21.6920449261607\\
71	22.9595093188366\\
72	24.152153102895\\
73	25.2762602972864\\
74	26.3342147116034\\
75	27.3269947152898\\
76	28.2553363062141\\
77	29.1202324420116\\
78	29.8919482678918\\
79	30.5420627959626\\
80	31.0617663871655\\
81	31.4568221129757\\
82	31.7426360127987\\
83	31.9395197691993\\
84	32.0749460531434\\
85	32.1714270622874\\
86	32.2415783626832\\
87	32.2927459858135\\
88	32.3265634956812\\
89	32.3472799024019\\
90	32.3591490423889\\
91	32.3655368456322\\
92	32.3687643268792\\
93	32.3702832690524\\
94	32.3709337057251\\
95	32.371170494887\\
96	32.3712256060723\\
97	32.3712509236666\\
98	32.3712654799355\\
99	32.3712682252526\\
100	32.3712687079247\\
101	32.3712688968754\\
102	32.3712689979727\\
103	32.3712690553989\\
104	32.3712690883784\\
105	32.3712691073764\\
106	32.3712691183359\\
107	32.3712691246638\\
108	32.3712691283193\\
109	32.3712691304319\\
110	32.371269131653\\
111	32.371269132359\\
112	32.3712691327672\\
113	32.3712691330032\\
114	32.3712691331397\\
115	32.3712691332186\\
116	32.3712691332642\\
117	32.3712691332906\\
118	32.3712691333059\\
119	32.3712691333147\\
120	32.3712691333198\\
121	32.3712691333228\\
122	32.3712691333245\\
123	32.3712691333255\\
124	32.371269133326\\
125	32.3712691333264\\
126	32.3712691333266\\
127	32.3712691333267\\
128	32.3712691333267\\
129	32.3712691333268\\
130	32.3712691333268\\
131	32.3712691333268\\
132	32.3712691333268\\
133	32.3712691333268\\
134	32.3712691333268\\
135	32.3712691333268\\
136	32.3712691333268\\
137	32.3712691333268\\
138	32.3712691333268\\
139	32.3712691333268\\
140	32.3712691333268\\
141	32.3712691333268\\
};
\addplot [color=mycolor2, line width=2.0pt, forget plot]
  table[row sep=crcr]{%
1	0\\
2	0.00792471824437081\\
3	0.0151733311032985\\
4	0.0223107967711348\\
5	0.0296918913897686\\
6	0.0375468576459977\\
7	0.0460369348792874\\
8	0.0552880896329352\\
9	0.0654105998584643\\
10	0.0765102547101886\\
11	0.0886948562262939\\
12	0.102078183882955\\
13	0.116782631241678\\
14	0.13294117593241\\
15	0.150699042154799\\
16	0.17021525241692\\
17	0.191664179076784\\
18	0.215237160967964\\
19	0.241144226783254\\
20	0.269615954754126\\
21	0.300905492122249\\
22	0.335290755126997\\
23	0.373076829240376\\
24	0.414598589375783\\
25	0.460223560350273\\
26	0.510355038751714\\
27	0.56543549840999\\
28	0.625950302802667\\
29	0.69243174887093\\
30	0.765463467820007\\
31	0.845685209468571\\
32	0.933798037525921\\
33	1.03056996373564\\
34	1.13684204903658\\
35	1.25353499964457\\
36	1.38165628511757\\
37	1.52230780387181\\
38	1.67669411907706\\
39	1.84613128414785\\
40	2.03205627189918\\
41	2.23603701453813\\
42	2.45978305265654\\
43	2.70515677985952\\
44	2.97418525513976\\
45	3.26907253705632\\
46	3.59221247161288\\
47	3.94620183881129\\
48	4.33385373049752\\
49	4.75821099360553\\
50	5.2225595275264\\
51	5.73044117140642\\
52	6.28566585611891\\
53	6.89232262602924\\
54	7.55478905730304\\
55	8.27773851258592\\
56	9.06614457712155\\
57	9.92528192018206\\
58	10.8607227203739\\
59	11.8783276874067\\
60	12.9842306111187\\
61	14.1848152774243\\
62	15.4866835187535\\
63	16.8966131238879\\
64	18.42150433135\\
65	20.0505158192848\\
66	21.7702406274207\\
67	23.5732278301586\\
68	25.4521325702614\\
69	27.3993634722607\\
70	29.3865087996392\\
71	31.3624856389416\\
72	33.2813197570452\\
73	35.1182669997973\\
74	36.8600517310025\\
75	38.4999878796172\\
76	40.0353983545932\\
77	41.4661854695157\\
78	42.7939961739777\\
79	44.0079843573436\\
80	45.0891644904391\\
81	46.0222054935652\\
82	46.800968168774\\
83	47.4295603050024\\
84	47.9206970961173\\
85	48.2954058632702\\
86	48.5696072873893\\
87	48.7704272493202\\
88	48.9171203455664\\
89	49.0224094759955\\
90	49.0962561020775\\
91	49.146813039918\\
92	49.1806461280133\\
93	49.2028342008566\\
94	49.2171358908372\\
95	49.2262221964916\\
96	49.2319267535449\\
97	49.2353459989572\\
98	49.2373310579952\\
99	49.2384687256631\\
100	49.2391009646895\\
101	49.2394565971534\\
102	49.2396584640812\\
103	49.239773769089\\
104	49.2398399081929\\
105	49.2398779518509\\
106	49.2398998752301\\
107	49.2399125243279\\
108	49.2399198282665\\
109	49.2399240479598\\
110	49.2399264866274\\
111	49.2399278963089\\
112	49.2399287112995\\
113	49.2399291825214\\
114	49.2399294549954\\
115	49.2399296125542\\
116	49.2399297036653\\
117	49.2399297563528\\
118	49.2399297868211\\
119	49.2399298044406\\
120	49.2399298146298\\
121	49.2399298205221\\
122	49.2399298239296\\
123	49.2399298259002\\
124	49.2399298270397\\
125	49.2399298276987\\
126	49.2399298280798\\
127	49.2399298283002\\
128	49.2399298284277\\
129	49.2399298285014\\
130	49.239929828544\\
131	49.2399298285686\\
132	49.2399298285829\\
133	49.2399298285911\\
134	49.2399298285959\\
135	49.2399298285987\\
136	49.2399298286002\\
137	49.2399298286012\\
138	49.2399298286017\\
139	49.239929828602\\
140	49.2399298286022\\
141	49.2399298286023\\
};
\addplot [color=mycolor3, line width=2.0pt, forget plot]
  table[row sep=crcr]{%
1	0\\
2	0.112690712840141\\
3	0.184853862746458\\
4	0.240084950271744\\
5	0.289131566372381\\
6	0.337173829615426\\
7	0.386915831700138\\
8	0.439933755032252\\
9	0.497284397364492\\
10	0.559790725460741\\
11	0.628181837151906\\
12	0.703165228862507\\
13	0.785466713808157\\
14	0.875854559119352\\
15	0.975155860759522\\
16	1.08426916438193\\
17	1.20417540610817\\
18	1.33594829094\\
19	1.48076474325507\\
20	1.63991581516001\\
21	1.81481830894113\\
22	2.00702730227545\\
23	2.21824973028615\\
24	2.4503591619685\\
25	2.70541190204352\\
26	2.98566454860751\\
27	3.29359313960854\\
28	3.63191402576716\\
29	4.00360661315731\\
30	4.41193812463795\\
31	4.86049053516305\\
32	5.35318984122802\\
33	5.89433782885078\\
34	6.48864650700303\\
35	7.14127537368635\\
36	7.85787167916477\\
37	8.64461384436313\\
38	9.5082581810978\\
39	10.4561890434171\\
40	11.4964725144667\\
41	12.6379136992921\\
42	13.8901176488973\\
43	15.2635538824578\\
44	16.7696244002538\\
45	18.420734986753\\
46	20.2303694880559\\
47	22.2131666070527\\
48	24.3849985892301\\
49	26.7630509680134\\
50	29.365902296615\\
51	32.2136025094385\\
52	35.3277482263393\\
53	38.7315529343103\\
54	42.4499095514307\\
55	46.5094423969074\\
56	50.9385450609718\\
57	55.76740009492\\
58	61.0279758348949\\
59	66.7539950491193\\
60	72.9808694804421\\
61	79.7455937762219\\
62	87.0865917978288\\
63	95.0435079358437\\
64	103.656935889885\\
65	112.863466960701\\
66	122.554781870165\\
67	132.632005507097\\
68	143.027844856501\\
69	153.693736223706\\
70	164.549544199184\\
71	175.414773576662\\
72	186.063439746153\\
73	196.315417757341\\
74	206.056519340813\\
75	215.224725744069\\
76	223.793236050971\\
77	231.757684556417\\
78	239.127778179134\\
79	245.881911881823\\
80	251.955856911109\\
81	257.279641890805\\
82	261.811502761302\\
83	265.553154254181\\
84	268.508002948067\\
85	270.736725162094\\
86	272.398929295837\\
87	273.594922156919\\
88	274.460157532913\\
89	275.082305610325\\
90	275.523599441343\\
91	275.83116409845\\
92	276.041532591651\\
93	276.182804905306\\
94	276.276080481793\\
95	276.336739426075\\
96	276.375666707403\\
97	276.400362028216\\
98	276.415810481151\\
99	276.425316341073\\
100	276.431070496913\\
101	276.434494523102\\
102	276.436510972251\\
103	276.437690792659\\
104	276.438378254815\\
105	276.438777765363\\
106	276.439009538236\\
107	276.439143849951\\
108	276.439221626964\\
109	276.43926664477\\
110	276.439292693365\\
111	276.439307762818\\
112	276.439316479558\\
113	276.439321521221\\
114	276.4393244371\\
115	276.439326123456\\
116	276.439327098713\\
117	276.439327662717\\
118	276.439327988884\\
119	276.439328177508\\
120	276.439328286589\\
121	276.439328349671\\
122	276.439328386151\\
123	276.439328407247\\
124	276.439328419447\\
125	276.439328426502\\
126	276.439328430582\\
127	276.439328432941\\
128	276.439328434306\\
129	276.439328435095\\
130	276.439328435551\\
131	276.439328435815\\
132	276.439328435968\\
133	276.439328436056\\
134	276.439328436107\\
135	276.439328436137\\
136	276.439328436154\\
137	276.439328436163\\
138	276.439328436169\\
139	276.439328436172\\
140	276.439328436174\\
141	276.439328436175\\
};
\addplot [color=mycolor4, line width=2.0pt, forget plot]
  table[row sep=crcr]{%
1	0\\
2	0.00679340447442321\\
3	0.0813545545739133\\
4	0.175303555942114\\
5	0.27729442119742\\
6	0.38573864596288\\
7	0.501642891561894\\
8	0.626601204894404\\
9	0.762265986806404\\
10	0.910253170780997\\
11	1.07216496037678\\
12	1.24963539278765\\
13	1.44437341049702\\
14	1.65819905278897\\
15	1.89307392237178\\
16	2.1511279982941\\
17	2.43468458378129\\
18	2.74628471971661\\
19	3.0887120190501\\
20	3.46501861676365\\
21	3.87855276165622\\
22	4.33298847227015\\
23	4.83235761798685\\
24	5.38108475270394\\
25	5.9840250127223\\
26	6.64650538611481\\
27	7.37436966370559\\
28	8.17402738899736\\
29	9.05250713389798\\
30	10.0175144373313\\
31	11.0774947534215\\
32	12.2417017636408\\
33	13.5202714117946\\
34	14.9243020205465\\
35	16.465940841723\\
36	18.1584773779641\\
37	20.0164437881747\\
38	22.0557226510232\\
39	24.2936623063414\\
40	26.7491999200586\\
41	29.442992320031\\
42	32.3975545229293\\
43	35.6374057106406\\
44	39.1892222120978\\
45	43.0819967960091\\
46	47.3472032738224\\
47	52.0189650419779\\
48	57.134225749159\\
49	62.7329197486622\\
50	68.8581393791485\\
51	75.5562954005784\\
52	82.8772660891817\\
53	90.8745295614481\\
54	99.605272851721\\
55	109.130470115814\\
56	119.514921086473\\
57	130.827239587645\\
58	143.139780558448\\
59	156.528492695141\\
60	171.07268256046\\
61	186.854674926518\\
62	203.959353327587\\
63	222.47356444654\\
64	242.485370214378\\
65	263.651976311597\\
66	285.403235478435\\
67	307.140931382403\\
68	328.65412575739\\
69	349.879193207611\\
70	370.992071077858\\
71	390.600705958241\\
72	406.695899067113\\
73	418.938413165081\\
74	427.610362322916\\
75	433.185912890102\\
76	436.166231672653\\
77	437.020667157563\\
78	437.116352813173\\
79	437.127086172827\\
80	437.128305852724\\
81	437.128457791623\\
82	437.128487541444\\
83	437.128501077011\\
84	437.128510509914\\
85	437.128517570066\\
86	437.128522851829\\
87	437.128526742513\\
88	437.128529579692\\
89	437.128475549202\\
90	437.128431695424\\
91	437.128401093982\\
92	437.128380575727\\
93	437.128367116545\\
94	437.128358434021\\
95	437.128352908948\\
96	437.128349432671\\
97	437.128347265999\\
98	437.128345935766\\
99	437.128345135127\\
100	437.128344661986\\
101	437.12834438763\\
102	437.128344228991\\
103	437.128344137287\\
104	437.128344084272\\
105	437.128344053621\\
106	437.128344035897\\
107	437.128344025649\\
108	437.128344019723\\
109	437.128344016296\\
110	437.128344014314\\
111	437.128344013168\\
112	437.128344012506\\
113	437.128344012122\\
114	437.128344011901\\
115	437.128344011773\\
116	437.128344011698\\
117	437.128344011656\\
118	437.128344011631\\
119	437.128344011616\\
120	437.128344011608\\
121	437.128344011603\\
122	437.128344011601\\
123	437.128344011599\\
124	437.128344011598\\
125	437.128344011598\\
126	437.128344011597\\
127	437.128344011597\\
128	437.128344011597\\
129	437.128344011597\\
130	437.128344011597\\
131	437.128344011597\\
132	437.128344011597\\
133	437.128344011597\\
134	437.128344011597\\
135	437.128344011597\\
136	437.128344011597\\
137	437.128344011597\\
138	437.128344011597\\
139	437.128344011597\\
140	437.128344011597\\
141	437.128344011597\\
};
\addplot [color=mycolor5, line width=2.0pt, forget plot]
  table[row sep=crcr]{%
1	0\\
2	0.150403228269135\\
3	0.329353723796686\\
4	0.527927500303994\\
5	0.738585684808434\\
6	0.960880206144185\\
7	1.19734393360152\\
8	1.45133258629647\\
9	1.72629010807894\\
10	2.0255914272833\\
11	2.35257612004483\\
12	2.71063019551259\\
13	3.10326996731672\\
14	3.53421649283687\\
15	4.00746024260347\\
16	4.52731863078379\\
17	5.09848932972572\\
18	5.72610184990716\\
19	6.41576933196955\\
20	7.17364205210785\\
21	8.00646381519116\\
22	8.92163218574193\\
23	9.92726336088055\\
24	11.0322623996333\\
25	12.246399472573\\
26	13.5803927720856\\
27	15.0459987175852\\
28	16.6561100951696\\
29	18.42486278261\\
30	20.3677517243532\\
31	22.5017568341159\\
32	24.8454795116677\\
33	27.4192904625288\\
34	30.2454895013321\\
35	33.3484779979382\\
36	36.7549445859253\\
37	40.4940646910156\\
38	44.5977143467598\\
39	49.1006986398439\\
40	54.0409949601039\\
41	59.4600110119463\\
42	65.4028572642659\\
43	71.9186331636481\\
44	79.0607259977211\\
45	86.887120757639\\
46	95.4607186951772\\
47	104.849661484066\\
48	115.127656959509\\
49	126.374301306817\\
50	138.675391283023\\
51	152.123218569566\\
52	166.816836658866\\
53	182.862288768066\\
54	200.372783153573\\
55	219.468799887067\\
56	240.278110680968\\
57	262.93569077466\\
58	287.583499295503\\
59	314.370102008337\\
60	343.450108121548\\
61	374.983391030793\\
62	409.134061807454\\
63	446.069164184135\\
64	485.957061109996\\
65	523.355661898608\\
66	549.356653998896\\
67	563.957160777154\\
68	567.988155666947\\
69	568.30768982474\\
70	568.33303476707\\
71	568.335060610716\\
72	568.335237391255\\
73	568.335266462184\\
74	568.335282540734\\
75	568.3352963069\\
76	568.335308667537\\
77	568.335319777326\\
78	568.335329738246\\
79	568.335338697596\\
80	568.335346648576\\
81	568.335353509649\\
82	568.335359244426\\
83	568.335363885048\\
84	568.335367523577\\
85	568.335370304636\\
86	568.335372392983\\
87	568.335373932094\\
88	568.335375053221\\
89	568.33520247473\\
90	568.335068123206\\
91	568.334974987281\\
92	568.334912385624\\
93	568.334871140143\\
94	568.334844409859\\
95	568.334827326026\\
96	568.334816534505\\
97	568.334809784554\\
98	568.334805632327\\
99	568.334803128944\\
100	568.334801645954\\
101	568.334800783938\\
102	568.334800284383\\
103	568.334799995135\\
104	568.334799827731\\
105	568.334799730872\\
106	568.33479967484\\
107	568.33479964243\\
108	568.334799623685\\
109	568.334799612844\\
110	568.334799606574\\
111	568.334799602948\\
112	568.334799600851\\
113	568.334799599638\\
114	568.334799598937\\
115	568.334799598531\\
116	568.334799598297\\
117	568.334799598161\\
118	568.334799598083\\
119	568.334799598037\\
120	568.334799598011\\
121	568.334799597996\\
122	568.334799597987\\
123	568.334799597982\\
124	568.334799597979\\
125	568.334799597977\\
126	568.334799597976\\
127	568.334799597976\\
128	568.334799597975\\
129	568.334799597975\\
130	568.334799597975\\
131	568.334799597975\\
132	568.334799597975\\
133	568.334799597975\\
134	568.334799597975\\
135	568.334799597975\\
136	568.334799597975\\
137	568.334799597975\\
138	568.334799597975\\
139	568.334799597975\\
140	568.334799597975\\
141	568.334799597975\\
};
\addplot [color=mycolor6, line width=2.0pt, forget plot]
  table[row sep=crcr]{%
1	0\\
2	0.244503556456279\\
3	0.534500362278705\\
4	0.854435323353285\\
5	1.19781946023111\\
6	1.56342888589301\\
7	1.95349528858806\\
8	2.37212637225748\\
9	2.82429022243462\\
10	3.31530947487756\\
11	3.85067887945964\\
12	4.43605266173708\\
13	5.07730809355792\\
14	5.78063672348677\\
15	6.55264139570166\\
16	7.40043108245494\\
17	8.3317120690617\\
18	9.35487671065472\\
19	10.4790918367534\\
20	11.7143889550212\\
21	13.0717582016558\\
22	14.563247725596\\
23	16.2020699640599\\
24	18.0027160927585\\
25	19.9810798146047\\
26	22.1545915760018\\
27	24.54236425824\\
28	27.1653513724169\\
29	30.0465187804834\\
30	33.2110309647798\\
31	36.6864528670699\\
32	40.5029683094648\\
33	44.6936159877523\\
34	49.2945439862668\\
35	54.3452836957569\\
36	59.8890439140848\\
37	65.973025765296\\
38	72.6487588755495\\
39	79.9724589829873\\
40	88.0054068194712\\
41	96.8143476699469\\
42	106.471910472677\\
43	117.057044651287\\
44	128.655472046011\\
45	141.360150313286\\
46	155.271742965094\\
47	170.499089796311\\
48	187.159669774203\\
49	205.38004651484\\
50	225.296284225982\\
51	247.054319440386\\
52	270.810270992783\\
53	296.73066751729\\
54	324.992568288939\\
55	355.78354955887\\
56	389.301524727323\\
57	425.754362894738\\
58	465.359266714296\\
59	508.341867286253\\
60	554.934991402866\\
61	605.377055166302\\
62	659.910038332378\\
63	718.776996224931\\
64	751.95580185898\\
65	758.171691273644\\
66	759.336464662281\\
67	759.554973306779\\
68	759.59620189438\\
69	759.604209184322\\
70	759.605990483616\\
71	759.606602282706\\
72	759.606985223545\\
73	759.607310737326\\
74	759.607609008789\\
75	759.607885167667\\
76	759.608140349657\\
77	759.608375302559\\
78	759.608590894113\\
79	759.608787091079\\
80	759.608962089908\\
81	759.609113798046\\
82	759.609241209193\\
83	759.609344803376\\
84	759.609426398954\\
85	759.609489044557\\
86	759.609536316019\\
87	759.609571249485\\
88	759.60959679619\\
89	759.608369276346\\
90	759.607278039388\\
91	759.606485305532\\
92	759.605943561479\\
93	759.60558423569\\
94	759.605350626376\\
95	759.605201082622\\
96	759.605106549449\\
97	759.605047408935\\
98	759.60501109975\\
99	759.60498928674\\
100	759.604976419778\\
101	759.604968981797\\
102	759.60496469948\\
103	759.604962233537\\
104	759.604960812074\\
105	759.604959991893\\
106	759.604959518303\\
107	759.604959244703\\
108	759.604959086585\\
109	759.604958995186\\
110	759.604958942346\\
111	759.604958911794\\
112	759.604958894128\\
113	759.604958883912\\
114	759.604958878005\\
115	759.604958874589\\
116	759.604958872614\\
117	759.604958871472\\
118	759.604958870811\\
119	759.604958870429\\
120	759.604958870208\\
121	759.60495887008\\
122	759.604958870006\\
123	759.604958869964\\
124	759.604958869939\\
125	759.604958869925\\
126	759.604958869916\\
127	759.604958869912\\
128	759.604958869909\\
129	759.604958869907\\
130	759.604958869906\\
131	759.604958869906\\
132	759.604958869905\\
133	759.604958869905\\
134	759.604958869905\\
135	759.604958869905\\
136	759.604958869905\\
137	759.604958869905\\
138	759.604958869905\\
139	759.604958869905\\
140	759.604958869905\\
141	759.604958869905\\
};

\addplot[area legend, draw=none, fill=green, fill opacity=0.1, line width=2.0pt, forget plot]
table[row sep=crcr] {%
x	y\\
98	0\\
150	0\\
150	800\\
98	800\\
}--cycle;
\addplot [color=green, dashed, line width=2.0pt, forget plot]
  table[row sep=crcr]{%
98	0\\
98	800\\
};
\end{axis}
\end{tikzpicture}%

%% file: Figures/NS_ID.tex
%
%
\definecolor{mycolor1}{rgb}{0.00000,0.44700,0.74100}%
\definecolor{mycolor2}{rgb}{0.85000,0.32500,0.09800}%
\begin{tikzpicture}

\begin{axis}[%
width=14cm,
height=3.5cm,
at={(0.cm,0.cm)},
scale only axis,
xmin=61,
xmax=140,
ymin=0,
ymax=6000,
xmajorgrids,
ymajorgrids,
axis background/.style={fill=white},
axis x line*=bottom,
axis y line*=left,legend pos=north east,
legend style={legend cell align=left, align=left, draw=white!15!black}
]

\addplot [color=mycolor1, solid, line width=2pt, mark options={solid, mycolor1}]
  table[row sep=crcr]{%
1	17.1099850236\\
2	19.8856880791202\\
3	21.9513349026862\\
4	23.9524338762432\\
5	26.1160712820975\\
6	28.5260815799926\\
7	31.2186380976052\\
8	34.2164255053757\\
9	37.5407869569796\\
10	41.2158755450339\\
11	45.2699554250118\\
12	49.7357549879889\\
13	54.6505676437297\\
14	60.0563441353094\\
15	65.9998529304634\\
16	72.532925512905\\
17	79.7127842120765\\
18	87.6024458664472\\
19	96.2711952946324\\
20	105.795124572122\\
21	116.257736178939\\
22	127.750609801958\\
23	140.374133881475\\
24	154.238303926188\\
25	169.46359024494\\
26	186.181878108626\\
27	204.537483494813\\
28	224.688247495747\\
29	246.806712185885\\
30	271.081380231419\\
31	297.718059751699\\
32	326.941294868877\\
33	358.995880952776\\
34	394.148461715912\\
35	432.689202958523\\
36	474.933534811378\\
37	521.223950666633\\
38	571.931846500642\\
39	627.459378839036\\
40	688.241313040632\\
41	754.746825717205\\
42	827.481215785527\\
43	906.987467686428\\
44	993.847597526585\\
45	1088.68369814065\\
46	1192.15858220363\\
47	1304.9759034682\\
48	1427.87961496319\\
49	1561.65259969024\\
50	1707.11428428539\\
51	1865.11701978369\\
52	2036.54098684426\\
53	2222.28735674052\\
54	2423.26941573766\\
55	2640.40134136732\\
56	2874.58430741881\\
57	3126.68959377704\\
58	3397.53839191242\\
59	3687.87803200698\\
60	3998.35441922185\\
61	4329.48056080571\\
62	4681.601199082\\
63	4855.01811792483\\
64	4948.64490313376\\
65	5063.79366974142\\
66	5170.37355951942\\
67	5263.73809393022\\
68	5355.22394164076\\
69	5434.32477699097\\
70	5493.3191067962\\
71	5531.47350066723\\
72	5548.56132928754\\
73	5544.57570373068\\
74	5519.76807645445\\
75	5511.73336078135\\
76	5487.1031464827\\
77	5436.7800051905\\
78	5360.73281134243\\
79	5260.17167767056\\
80	5136.71898213148\\
81	4992.31024067477\\
82	4829.14863433956\\
83	4649.65044351635\\
84	4456.38182352613\\
85	4251.99193239183\\
86	4039.14663558777\\
87	3820.46607732706\\
88	3598.46861429048\\
89	3375.52288721956\\
90	3153.80913052922\\
91	2935.29019797003\\
92	2717.80554059024\\
93	2488.64175956481\\
94	2253.46153910791\\
95	2018.12001757157\\
96	1787.96503221157\\
97	1567.51336844029\\
98	1360.31275703679\\
99	1168.91422576299\\
100	994.918615772279\\
101	839.073662642119\\
102	701.402496432269\\
103	581.347162821509\\
104	477.913654009143\\
105	389.80817509919\\
106	315.557682781267\\
107	253.610787304112\\
108	200.982426307901\\
109	151.05993807279\\
110	107.124589742243\\
111	71.3724426096781\\
112	44.4823248921271\\
113	25.8087520647787\\
114	13.8634560491953\\
115	6.85025207386769\\
116	3.09006532079357\\
117	1.26091446666489\\
118	0.460260606931022\\
119	0.148198315730558\\
120	0.0477187476267943\\
121	0.0153651890323076\\
122	0.00494752674774652\\
123	0.0015930842143832\\
124	0.000512966460999492\\
125	0.000165172725852102\\
126	5.31846569651986e-05\\
127	1.71250784843204e-05\\
128	5.51412260137026e-06\\
129	1.77548601547984e-06\\
130	5.71681693609833e-07\\
131	1.84071482222844e-07\\
132	5.92669633407313e-08\\
133	1.90823214812003e-08\\
134	6.14384313629456e-09\\
135	1.97804834116662e-09\\
136	6.36822489506705e-10\\
137	2.05012622773163e-10\\
138	6.59961336653424e-11\\
139	2.1243482380786e-11\\
140	6.83745102305095e-12\\
141	2.20046219289577e-12\\
};
\addlegendentry{Infected}

\addplot [color=mycolor2, solid, line width=2pt, mark options={solid, mycolor2}]
  table[row sep=crcr]{%
1	0\\
2	0.524369143637016\\
3	1.15048765905235\\
4	1.82938201864941\\
5	2.54664104864507\\
6	3.30435589038599\\
7	4.1111529055939\\
8	4.97781174409841\\
9	5.91560401563885\\
10	6.93582666695316\\
11	8.04981956427376\\
12	9.26915509817139\\
13	10.6058707992101\\
14	12.0726973941732\\
15	13.683269437413\\
16	15.4523193448798\\
17	17.3958602090915\\
18	19.531363579679\\
19	21.8779378498595\\
20	24.4565120072393\\
21	27.2900286786151\\
22	30.4036497436293\\
23	33.8249773223633\\
24	37.5842926299677\\
25	41.7148150013933\\
26	46.2529832880467\\
27	51.2387617883737\\
28	56.7159728743404\\
29	62.7326584984647\\
30	69.3414727975869\\
31	76.6001080380416\\
32	84.5717561613106\\
33	93.3256081786581\\
34	102.93739361607\\
35	113.489962114214\\
36	125.073909127442\\
37	137.788247424179\\
38	151.741125748667\\
39	167.050595538144\\
40	183.845425973514\\
41	202.265966844991\\
42	222.465057702043\\
43	244.608980489845\\
44	268.878451308114\\
45	295.469645014131\\
46	324.595244077162\\
47	356.485500320869\\
48	391.389294906391\\
49	429.575178055527\\
50	471.332365539015\\
51	516.971663816729\\
52	566.826289887552\\
53	621.252545383431\\
54	680.630297255368\\
55	745.36320962669\\
56	815.878663170715\\
57	892.627289925666\\
58	976.082043108384\\
59	1066.7367136688\\
60	1165.10379861766\\
61	1271.71162129538\\
62	1387.10060163166\\
63	1511.81857614226\\
64	1616.15125632833\\
65	1692.9224822857\\
66	1754.48443281142\\
67	1804.23140310305\\
68	1843.46984730959\\
69	1877.16801343852\\
70	1907.72393642247\\
71	1935.31442076416\\
72	1959.97230551102\\
73	1981.74004140848\\
74	2000.67937504377\\
75	2016.87145837781\\
76	2032.13201918306\\
77	2046.78597022662\\
78	2060.76354986875\\
79	2073.98490494961\\
80	2086.39061943189\\
81	2097.93949529423\\
82	2108.60562752318\\
83	2118.37661923258\\
84	2127.25236280664\\
85	2135.2439768667\\
86	2142.37271130817\\
87	2148.66875631247\\
88	2154.16995345737\\
89	2158.92043696674\\
90	2162.96924583345\\
91	2166.36895060857\\
92	2169.17433611023\\
93	2171.46905183584\\
94	2173.42104454859\\
95	2175.10333491448\\
96	2176.55288218163\\
97	2177.79391559689\\
98	2178.84677632354\\
99	2179.73090844789\\
100	2180.46554458114\\
101	2181.06954770519\\
102	2181.56101268297\\
103	2181.95687182155\\
104	2182.2725950536\\
105	2182.52200770851\\
106	2182.71721895532\\
107	2182.86864158417\\
108	2182.98508000816\\
109	2183.07367511368\\
110	2183.13928161145\\
111	2183.18550757501\\
112	2183.2161894516\\
113	2183.2352523952\\
114	2183.24627751836\\
115	2183.25217807363\\
116	2183.25508065597\\
117	2183.25638258228\\
118	2183.25690991046\\
119	2183.25710045748\\
120	2183.25716092707\\
121	2183.25718023542\\
122	2183.25718642205\\
123	2183.2571884082\\
124	2183.25718904654\\
125	2183.25718925182\\
126	2183.25718931786\\
127	2183.25718933911\\
128	2183.25718934594\\
129	2183.25718934814\\
130	2183.25718934885\\
131	2183.25718934907\\
132	2183.25718934915\\
133	2183.25718934917\\
134	2183.25718934918\\
135	2183.25718934918\\
136	2183.25718934918\\
137	2183.25718934918\\
138	2183.25718934918\\
139	2183.25718934918\\
140	2183.25718934918\\
141	2183.25718934918\\
};
\addlegendentry{Deceased}

\addplot[area legend, draw=none, fill=green, fill opacity=0.1, line width=2.0pt, forget plot]
table[row sep=crcr] {%
x	y\\
118	0\\
150	0\\
150	6000\\
118	6000\\
}--cycle;

\addplot [color=green, dashed, line width=2pt, mark options={solid, green}]
  table[row sep=crcr]{%
118	0\\
118	6000\\
};
\addlegendentry{Eradication}
\end{axis}
\end{tikzpicture}%

%% file: Figures/MPC_ID.tex
%
%
\definecolor{mycolor1}{rgb}{0.00000,0.44700,0.74100}%
\definecolor{mycolor2}{rgb}{0.85000,0.32500,0.09800}%
\begin{tikzpicture}

\begin{axis}[%
width=14cm,
height=3.5cm,
at={(0.cm,0.cm)},
scale only axis,
xmin=61,
xmax=140,
ymin=0,
ymax=6000,
xmajorgrids,
ymajorgrids,
axis background/.style={fill=white},
legend pos=north east,
legend style={legend cell align=left, align=left, draw=white!15!black}
]

\addplot [color=mycolor1, solid, line width=2pt, mark options={solid, mycolor1}]
  table[row sep=crcr]{%
1	17.1099850236\\
2	19.8856880791202\\
3	21.9513349026862\\
4	23.9524338762432\\
5	26.1160712820975\\
6	28.5260815799926\\
7	31.2186380976052\\
8	34.2164255053757\\
9	37.5407869569796\\
10	41.2158755450339\\
11	45.2699554250118\\
12	49.7357549879889\\
13	54.6505676437297\\
14	60.0563441353094\\
15	65.9998529304634\\
16	72.532925512905\\
17	79.7127842120765\\
18	87.6024458664472\\
19	96.2711952946324\\
20	105.795124572122\\
21	116.257736178939\\
22	127.750609801958\\
23	140.374133881475\\
24	154.238303926188\\
25	169.46359024494\\
26	186.181878108626\\
27	204.537483494813\\
28	224.688247495747\\
29	246.806712185885\\
30	271.081380231419\\
31	297.718059751699\\
32	326.941294868877\\
33	358.995880952776\\
34	394.148461715912\\
35	432.689202958523\\
36	474.933534811378\\
37	521.223950666633\\
38	571.931846500642\\
39	627.459378839036\\
40	688.241313040632\\
41	754.746825717205\\
42	827.481215785527\\
43	906.987467686428\\
44	993.847597526585\\
45	1088.68369814065\\
46	1192.15858220363\\
47	1304.9759034682\\
48	1427.87961496319\\
49	1561.65259969024\\
50	1707.11428428539\\
51	1865.11701978369\\
52	2036.54098684426\\
53	2222.28735674052\\
54	2423.26941573766\\
55	2640.40134136732\\
56	2874.58430741881\\
57	3126.68959377704\\
58	3397.53839191242\\
59	3687.87803200698\\
60	3998.35441922185\\
61	4329.48056080571\\
62	4681.601199082\\
63	4855.02173872579\\
64	4948.84225969198\\
65	5063.97703699162\\
66	5170.54003804567\\
67	5263.90032212604\\
68	5287.92810482711\\
69	5163.96393140387\\
70	4915.6457274557\\
71	4616.939354397\\
72	4302.71715256023\\
73	3988.84376002634\\
74	3682.80212111316\\
75	3388.42486393043\\
76	3107.86770594886\\
77	2795.55277426442\\
78	2428.84734184806\\
79	2030.91825566014\\
80	1634.89587915841\\
81	1268.70556615151\\
82	950.923699539626\\
83	702.571714382669\\
84	519.395760434455\\
85	380.400706819642\\
86	277.203776678016\\
87	194.05288108619\\
88	131.24960933043\\
89	86.2276134133415\\
90	55.3231620173998\\
91	34.8415090528093\\
92	21.6354801187237\\
93	13.2950551961035\\
94	8.10644170738514\\
95	4.91338876598438\\
96	2.98605174698645\\
97	1.78098150457314\\
98	1.036017154756\\
99	0.593630887122527\\
100	0.341775127955673\\
101	0.197117712034829\\
102	0.113796041930225\\
103	0.06573397649758\\
104	0.037985861582004\\
105	0.0219565713626296\\
106	0.0126934362286452\\
107	0.00733906944409359\\
108	0.00424359128138579\\
109	0.00245383965918738\\
110	0.00141896578337002\\
111	0.000820552214323371\\
112	0.000474510804607345\\
113	0.000274403492913551\\
114	0.000158684878684443\\
115	9.1766254508132e-05\\
116	5.30678486228365e-05\\
117	3.06888503024542e-05\\
118	1.77472140608699e-05\\
119	1.02631349236283e-05\\
120	5.93512777330006e-06\\
121	3.43226038483846e-06\\
122	1.98486265580962e-06\\
123	1.14783838770864e-06\\
124	6.63790545757567e-07\\
125	3.83867550917112e-07\\
126	2.21989153676452e-07\\
127	1.28375491195338e-07\\
128	7.42390639206271e-08\\
129	4.29321719485327e-08\\
130	2.48275140044663e-08\\
131	1.43576582794569e-08\\
132	8.30297998237877e-09\\
133	4.80158221903664e-09\\
134	2.77673701248841e-09\\
135	1.60577661492062e-09\\
136	9.28614602900109e-10\\
137	5.37014347428639e-10\\
138	3.10553386194853e-10\\
139	1.79591860350095e-10\\
140	1.0385730035627e-10\\
141	6.00602879042087e-11\\
};
\addlegendentry{Infected}

\addplot [color=mycolor2, solid, line width=2pt, mark options={solid, mycolor2}]
  table[row sep=crcr]{%
1	0\\
2	0.524369143637016\\
3	1.15048765905235\\
4	1.82938201864941\\
5	2.54664104864507\\
6	3.30435589038599\\
7	4.1111529055939\\
8	4.97781174409841\\
9	5.91560401563885\\
10	6.93582666695316\\
11	8.04981956427376\\
12	9.26915509817139\\
13	10.6058707992101\\
14	12.0726973941732\\
15	13.683269437413\\
16	15.4523193448798\\
17	17.3958602090915\\
18	19.531363579679\\
19	21.8779378498595\\
20	24.4565120072393\\
21	27.2900286786151\\
22	30.4036497436293\\
23	33.8249773223633\\
24	37.5842926299677\\
25	41.7148150013933\\
26	46.2529832880467\\
27	51.2387617883737\\
28	56.7159728743404\\
29	62.7326584984647\\
30	69.3414727975869\\
31	76.6001080380416\\
32	84.5717561613106\\
33	93.3256081786581\\
34	102.93739361607\\
35	113.489962114214\\
36	125.073909127442\\
37	137.788247424179\\
38	151.741125748667\\
39	167.050595538144\\
40	183.845425973514\\
41	202.265966844991\\
42	222.465057702043\\
43	244.608980489845\\
44	268.878451308114\\
45	295.469645014131\\
46	324.595244077162\\
47	356.485500320869\\
48	391.389294906391\\
49	429.575178055527\\
50	471.332365539015\\
51	516.971663816729\\
52	566.826289887552\\
53	621.252545383431\\
54	680.630297255368\\
55	745.36320962669\\
56	815.878663170715\\
57	892.627289925666\\
58	976.082043108384\\
59	1066.7367136688\\
60	1165.10379861766\\
61	1271.71162129538\\
62	1387.10060163166\\
63	1511.81857614226\\
64	1616.1518046753\\
65	1692.96644807977\\
66	1754.57033978669\\
67	1804.35550050676\\
68	1843.63000022032\\
69	1879.21862924201\\
70	1914.55919425353\\
71	1948.2791373861\\
72	1978.13503428801\\
73	2003.59093541902\\
74	2024.80403965586\\
75	2042.18080270364\\
76	2056.19365140163\\
77	2067.30846470539\\
78	2076.87399606654\\
79	2085.50317099663\\
80	2093.17940237992\\
81	2099.83159459666\\
82	2105.42819493794\\
83	2109.99544409382\\
84	2113.57695052977\\
85	2116.27693500691\\
86	2118.28354650674\\
87	2119.73156731615\\
88	2120.77734280326\\
89	2121.524042289\\
90	2122.04978244383\\
91	2122.41337537079\\
92	2122.66017956937\\
93	2122.82474486759\\
94	2122.93270354861\\
95	2123.00251343505\\
96	2123.04709158364\\
97	2123.07516341033\\
98	2123.09256968692\\
99	2123.1031908428\\
100	2123.10956289725\\
101	2123.1133341705\\
102	2123.11554764716\\
103	2123.11683998311\\
104	2123.11759197546\\
105	2123.11802860098\\
106	2123.11828176084\\
107	2123.11842841172\\
108	2123.11851331354\\
109	2123.11856244749\\
110	2123.11859087488\\
111	2123.1186073194\\
112	2123.11861683111\\
113	2123.11862233242\\
114	2123.11862551408\\
115	2123.11862735412\\
116	2123.11862841825\\
117	2123.11862903365\\
118	2123.11862938954\\
119	2123.11862959535\\
120	2123.11862971437\\
121	2123.1186297832\\
122	2123.118629823\\
123	2123.11862984602\\
124	2123.11862985933\\
125	2123.11862986703\\
126	2123.11862987148\\
127	2123.11862987405\\
128	2123.11862987554\\
129	2123.1186298764\\
130	2123.1186298769\\
131	2123.11862987719\\
132	2123.11862987735\\
133	2123.11862987745\\
134	2123.11862987751\\
135	2123.11862987754\\
136	2123.11862987756\\
137	2123.11862987757\\
138	2123.11862987757\\
139	2123.11862987758\\
140	2123.11862987758\\
141	2123.11862987758\\
};
\addlegendentry{Deceased}

\addplot[area legend, draw=none, fill=green, fill opacity=0.1, line width=2.0pt, forget plot]
table[row sep=crcr] {%
x	y\\
98	0\\
150	0\\
150	6000\\
98	6000\\
}--cycle;

\addplot [color=green, dashed, line width=2.0pt, mark options={solid, green}]
  table[row sep=crcr]{%
98	0\\
98	6000\\
};
\addlegendentry{Eradication}
\end{axis}
\end{tikzpicture}%

%% file: Figures/NS_input.tex
%
%
\definecolor{mycolor1}{rgb}{0.00000,0.44700,0.74100}%
\definecolor{mycolor2}{rgb}{0.85000,0.32500,0.09800}%
\definecolor{mycolor3}{rgb}{0.92900,0.69400,0.12500}%
\definecolor{mycolor4}{rgb}{0.49400,0.18400,0.55600}%
\definecolor{mycolor5}{rgb}{0.46600,0.67400,0.18800}%
\definecolor{mycolor6}{rgb}{0.30100,0.74500,0.93300}%
\begin{tikzpicture}

\begin{axis}[%
width=14cm,
height=3.5cm,
at={(0cm,0cm)},
scale only axis,
xmin=61,
xmax=140,
ymin=0,
ymax=60000,
xmajorgrids,
ymajorgrids,
axis background/.style={fill=white},
axis x line*=bottom,
axis y line*=left,
legend pos=north east,
legend style={legend cell align=left, align=left, draw=white!15!black}
]

\addplot [const plot, color=mycolor1, solid, line width=2pt, mark options={solid, mycolor1}] table[row sep=crcr] {
1	0\\
2	0\\
3	0\\
4	0\\
5	0\\
6	0\\
7	0\\
8	0\\
9	0\\
10	0\\
11	0\\
12	0\\
13	0\\
14	0\\
15	0\\
16	0\\
17	0\\
18	0\\
19	0\\
20	0\\
21	0\\
22	0\\
23	0\\
24	0\\
25	0\\
26	0\\
27	0\\
28	0\\
29	0\\
30	0\\
31	0\\
32	0\\
33	0\\
34	0\\
35	0\\
36	0\\
37	0\\
38	0\\
39	0\\
40	0\\
41	0\\
42	0\\
43	0\\
44	0\\
45	0\\
46	0\\
47	0\\
48	0\\
49	0\\
50	0\\
51	0\\
52	0\\
53	0\\
54	0\\
55	0\\
56	0\\
57	0\\
58	0\\
59	0\\
60	0\\
61	0\\
62	0\\
63	0\\
64	0\\
65	0\\
66	0\\
67	0\\
68	0\\
69	0\\
70	0\\
71	0\\
72	0\\
73	0\\
74	0\\
75	0\\
76	0\\
77	0\\
78	0\\
79	0\\
80	0\\
81	0\\
82	0\\
83	0\\
84	0\\
85	0\\
86	0\\
87	0\\
88	0\\
89	0\\
90	0\\
91	0\\
92	0\\
93	0\\
94	0\\
95	0\\
96	0\\
97	0\\
98	0\\
99	0\\
100	0\\
101	0\\
102	0\\
103	0\\
104	0\\
105	0\\
106	8963.63625687419\\
107	55191\\
108	55191\\
109	55191\\
110	55191\\
111	55191\\
112	55191\\
113	55191\\
114	55191\\
115	55191\\
116	55191\\
117	55191\\
118	0\\
119	0\\
120	0\\
121	0\\
122	0\\
123	0\\
124	0\\
125	0\\
126	0\\
127	0\\
128	0\\
129	0\\
130	0\\
131	0\\
132	0\\
133	0\\
134	0\\
135	0\\
136	0\\
137	0\\
138	0\\
139	0\\
140	0\\
141	0\\
};
\addlegendentry{0-24}

\addplot [const plot, color=mycolor2, solid, line width=2pt, mark options={solid, mycolor2}]
table[row sep=crcr] {
1	0\\
2	0\\
3	0\\
4	0\\
5	0\\
6	0\\
7	0\\
8	0\\
9	0\\
10	0\\
11	0\\
12	0\\
13	0\\
14	0\\
15	0\\
16	0\\
17	0\\
18	0\\
19	0\\
20	0\\
21	0\\
22	0\\
23	0\\
24	0\\
25	0\\
26	0\\
27	0\\
28	0\\
29	0\\
30	0\\
31	0\\
32	0\\
33	0\\
34	0\\
35	0\\
36	0\\
37	0\\
38	0\\
39	0\\
40	0\\
41	0\\
42	0\\
43	0\\
44	0\\
45	0\\
46	0\\
47	0\\
48	0\\
49	0\\
50	0\\
51	0\\
52	0\\
53	0\\
54	0\\
55	0\\
56	0\\
57	0\\
58	0\\
59	0\\
60	0\\
61	0\\
62	0\\
63	0\\
64	0\\
65	0\\
66	0\\
67	0\\
68	0\\
69	0\\
70	0\\
71	0\\
72	0\\
73	0\\
74	0\\
75	0\\
76	0\\
77	0\\
78	0\\
79	0\\
80	0\\
81	0\\
82	0\\
83	0\\
84	0\\
85	0\\
86	0\\
87	0\\
88	0\\
89	0\\
90	10707.2369549301\\
91	55191\\
92	55191\\
93	55191\\
94	55191\\
95	55191\\
96	55191\\
97	55191\\
98	55191\\
99	55191\\
100	55191\\
101	55191\\
102	55191\\
103	55191\\
104	55191\\
105	55191\\
106	46227.3637431258\\
107	0\\
108	0\\
109	0\\
110	0\\
111	0\\
112	0\\
113	0\\
114	0\\
115	0\\
116	0\\
117	0\\
118	0\\
119	0\\
120	0\\
121	0\\
122	0\\
123	0\\
124	0\\
125	0\\
126	0\\
127	0\\
128	0\\
129	0\\
130	0\\
131	0\\
132	0\\
133	0\\
134	0\\
135	0\\
136	0\\
137	0\\
138	0\\
139	0\\
140	0\\
141	0\\
};
\addlegendentry{25-44}

\addplot [const plot, color=mycolor3, solid, line width=2pt, mark options={solid, mycolor3}]
table[row sep=crcr] {
1	0\\
2	0\\
3	0\\
4	0\\
5	0\\
6	0\\
7	0\\
8	0\\
9	0\\
10	0\\
11	0\\
12	0\\
13	0\\
14	0\\
15	0\\
16	0\\
17	0\\
18	0\\
19	0\\
20	0\\
21	0\\
22	0\\
23	0\\
24	0\\
25	0\\
26	0\\
27	0\\
28	0\\
29	0\\
30	0\\
31	0\\
32	0\\
33	0\\
34	0\\
35	0\\
36	0\\
37	0\\
38	0\\
39	0\\
40	0\\
41	0\\
42	0\\
43	0\\
44	0\\
45	0\\
46	0\\
47	0\\
48	0\\
49	0\\
50	0\\
51	0\\
52	0\\
53	0\\
54	0\\
55	0\\
56	0\\
57	0\\
58	0\\
59	0\\
60	0\\
61	0\\
62	0\\
63	0\\
64	0\\
65	0\\
66	0\\
67	0\\
68	0\\
69	0\\
70	0\\
71	0\\
72	0\\
73	48103.5585726969\\
74	55191\\
75	55191\\
76	55191\\
77	55191\\
78	55191\\
79	55191\\
80	55191\\
81	55191\\
82	55191\\
83	55191\\
84	55191\\
85	55191\\
86	55191\\
87	55191\\
88	55191\\
89	55191\\
90	44483.7630450699\\
91	0\\
92	0\\
93	0\\
94	0\\
95	0\\
96	0\\
97	0\\
98	0\\
99	0\\
100	0\\
101	0\\
102	0\\
103	0\\
104	0\\
105	0\\
106	0\\
107	0\\
108	0\\
109	0\\
110	0\\
111	0\\
112	0\\
113	0\\
114	0\\
115	0\\
116	0\\
117	0\\
118	0\\
119	0\\
120	0\\
121	0\\
122	0\\
123	0\\
124	0\\
125	0\\
126	0\\
127	0\\
128	0\\
129	0\\
130	0\\
131	0\\
132	0\\
133	0\\
134	0\\
135	0\\
136	0\\
137	0\\
138	0\\
139	0\\
140	0\\
141	0\\
};
\addlegendentry{45-64}

\addplot [const plot, color=mycolor4, solid, line width=2pt, mark options={solid, mycolor3}]
table[row sep=crcr] {
1	0\\
2	0\\
3	0\\
4	0\\
5	0\\
6	0\\
7	0\\
8	0\\
9	0\\
10	0\\
11	0\\
12	0\\
13	0\\
14	0\\
15	0\\
16	0\\
17	0\\
18	0\\
19	0\\
20	0\\
21	0\\
22	0\\
23	0\\
24	0\\
25	0\\
26	0\\
27	0\\
28	0\\
29	0\\
30	0\\
31	0\\
32	0\\
33	0\\
34	0\\
35	0\\
36	0\\
37	0\\
38	0\\
39	0\\
40	0\\
41	0\\
42	0\\
43	0\\
44	0\\
45	0\\
46	0\\
47	0\\
48	0\\
49	0\\
50	0\\
51	0\\
52	0\\
53	0\\
54	0\\
55	0\\
56	0\\
57	0\\
58	0\\
59	0\\
60	0\\
61	0\\
62	0\\
63	0\\
64	0\\
65	0\\
66	37693.9876062099\\
67	55191\\
68	55191\\
69	55191\\
70	55191\\
71	55191\\
72	55191\\
73	7087.44142730308\\
74	0\\
75	0\\
76	0\\
77	0\\
78	0\\
79	0\\
80	0\\
81	0\\
82	0\\
83	0\\
84	0\\
85	0\\
86	0\\
87	0\\
88	0\\
89	0\\
90	0\\
91	0\\
92	0\\
93	0\\
94	0\\
95	0\\
96	0\\
97	0\\
98	0\\
99	0\\
100	0\\
101	0\\
102	0\\
103	0\\
104	0\\
105	0\\
106	0\\
107	0\\
108	0\\
109	0\\
110	0\\
111	0\\
112	0\\
113	0\\
114	0\\
115	0\\
116	0\\
117	0\\
118	0\\
119	0\\
120	0\\
121	0\\
122	0\\
123	0\\
124	0\\
125	0\\
126	0\\
127	0\\
128	0\\
129	0\\
130	0\\
131	0\\
132	0\\
133	0\\
134	0\\
135	0\\
136	0\\
137	0\\
138	0\\
139	0\\
140	0\\
141	0\\
};
\addlegendentry{65-74}

\addplot [const plot, color=mycolor5, solid, line width=2pt, mark options={solid, mycolor5}]
table[row sep=crcr] {
1	0\\
2	0\\
3	0\\
4	0\\
5	0\\
6	0\\
7	0\\
8	0\\
9	0\\
10	0\\
11	0\\
12	0\\
13	0\\
14	0\\
15	0\\
16	0\\
17	0\\
18	0\\
19	0\\
20	0\\
21	0\\
22	0\\
23	0\\
24	0\\
25	0\\
26	0\\
27	0\\
28	0\\
29	0\\
30	0\\
31	0\\
32	0\\
33	0\\
34	0\\
35	0\\
36	0\\
37	0\\
38	0\\
39	0\\
40	0\\
41	0\\
42	0\\
43	0\\
44	0\\
45	0\\
46	0\\
47	0\\
48	0\\
49	0\\
50	0\\
51	0\\
52	0\\
53	0\\
54	0\\
55	0\\
56	0\\
57	0\\
58	0\\
59	0\\
60	0\\
61	0\\
62	14795.8746730618\\
63	55191\\
64	55191\\
65	55191\\
66	17497.0123937901\\
67	0\\
68	0\\
69	0\\
70	0\\
71	0\\
72	0\\
73	0\\
74	0\\
75	0\\
76	0\\
77	0\\
78	0\\
79	0\\
80	0\\
81	0\\
82	0\\
83	0\\
84	0\\
85	0\\
86	0\\
87	0\\
88	0\\
89	0\\
90	0\\
91	0\\
92	0\\
93	0\\
94	0\\
95	0\\
96	0\\
97	0\\
98	0\\
99	0\\
100	0\\
101	0\\
102	0\\
103	0\\
104	0\\
105	0\\
106	0\\
107	0\\
108	0\\
109	0\\
110	0\\
111	0\\
112	0\\
113	0\\
114	0\\
115	0\\
116	0\\
117	0\\
118	0\\
119	0\\
120	0\\
121	0\\
122	0\\
123	0\\
124	0\\
125	0\\
126	0\\
127	0\\
128	0\\
129	0\\
130	0\\
131	0\\
132	0\\
133	0\\
134	0\\
135	0\\
136	0\\
137	0\\
138	0\\
139	0\\
140	0\\
141	0\\
};
\addlegendentry{75-84}

\addplot [const plot, color=mycolor6, solid, line width=2pt, mark options={solid, mycolor6}]
table[row sep=crcr] {
1	0\\
2	0\\
3	0\\
4	0\\
5	0\\
6	0\\
7	0\\
8	0\\
9	0\\
10	0\\
11	0\\
12	0\\
13	0\\
14	0\\
15	0\\
16	0\\
17	0\\
18	0\\
19	0\\
20	0\\
21	0\\
22	0\\
23	0\\
24	0\\
25	0\\
26	0\\
27	0\\
28	0\\
29	0\\
30	0\\
31	0\\
32	0\\
33	0\\
34	0\\
35	0\\
36	0\\
37	0\\
38	0\\
39	0\\
40	0\\
41	0\\
42	0\\
43	0\\
44	0\\
45	0\\
46	0\\
47	0\\
48	0\\
49	0\\
50	0\\
51	0\\
52	0\\
53	0\\
54	0\\
55	0\\
56	0\\
57	0\\
58	0\\
59	0\\
60	0\\
61	55191\\
62	40395.1253269382\\
63	0\\
64	0\\
65	0\\
66	0\\
67	0\\
68	0\\
69	0\\
70	0\\
71	0\\
72	0\\
73	0\\
74	0\\
75	0\\
76	0\\
77	0\\
78	0\\
79	0\\
80	0\\
81	0\\
82	0\\
83	0\\
84	0\\
85	0\\
86	0\\
87	0\\
88	0\\
89	0\\
90	0\\
91	0\\
92	0\\
93	0\\
94	0\\
95	0\\
96	0\\
97	0\\
98	0\\
99	0\\
100	0\\
101	0\\
102	0\\
103	0\\
104	0\\
105	0\\
106	0\\
107	0\\
108	0\\
109	0\\
110	0\\
111	0\\
112	0\\
113	0\\
114	0\\
115	0\\
116	0\\
117	0\\
118	0\\
119	0\\
120	0\\
121	0\\
122	0\\
123	0\\
124	0\\
125	0\\
126	0\\
127	0\\
128	0\\
129	0\\
130	0\\
131	0\\
132	0\\
133	0\\
134	0\\
135	0\\
136	0\\
137	0\\
138	0\\
139	0\\
140	0\\
141	0\\
};
\addlegendentry{85+}

\addplot[area legend, draw=none, fill=green, fill opacity=0.1, line width=2.0pt, forget plot]
table[row sep=crcr] {%
x	y\\
118	0\\
150	0\\
150	60000\\
118	60000\\
}--cycle;

\addplot [color=green, dashed, line width=2.0pt, mark options={solid, green}]
  table[row sep=crcr]{%
118	0\\
118	60000\\
};
\addlegendentry{Erad.}

\end{axis}
\end{tikzpicture}%

%% file: Figures/MPC_input.tex
%
%
\definecolor{mycolor1}{rgb}{0.00000,0.44700,0.74100}%
\definecolor{mycolor2}{rgb}{0.85000,0.32500,0.09800}%
\definecolor{mycolor3}{rgb}{0.92900,0.69400,0.12500}%
\definecolor{mycolor4}{rgb}{0.49400,0.18400,0.55600}%
\definecolor{mycolor5}{rgb}{0.46600,0.67400,0.18800}%
\definecolor{mycolor6}{rgb}{0.30100,0.74500,0.93300}%
\begin{tikzpicture}

\begin{axis}[%
width=14cm,
height=3.5cm,
at={(0cm,0cm)},
scale only axis,
xmin=61,
xmax=140,
ymin=0,
ymax=60000,
xmajorgrids,
ymajorgrids,
axis background/.style={fill=white},
axis x line*=bottom,
axis y line*=left,
legend pos=north east,
legend style={legend cell align=left, align=left, draw=white!15!black}
]

\addplot [const plot, color=mycolor1, solid, line width=2pt, mark options={solid, mycolor1}] table[row sep=crcr] {
1	0\\
2	0\\
3	0\\
4	0\\
5	0\\
6	0\\
7	0\\
8	0\\
9	0\\
10	0\\
11	0\\
12	0\\
13	0\\
14	0\\
15	0\\
16	0\\
17	0\\
18	0\\
19	0\\
20	0\\
21	0\\
22	0\\
23	0\\
24	0\\
25	0\\
26	0\\
27	0\\
28	0\\
29	0\\
30	0\\
31	0\\
32	0\\
33	0\\
34	0\\
35	0\\
36	0\\
37	0\\
38	0\\
39	0\\
40	0\\
41	0\\
42	0\\
43	0\\
44	0\\
45	0\\
46	0\\
47	0\\
48	0\\
49	0\\
50	0\\
51	0\\
52	0\\
53	0\\
54	0\\
55	0\\
56	0\\
57	0\\
58	0\\
59	0\\
60	0\\
61	0\\
62	0\\
63	0\\
64	0\\
65	0\\
66	37863\\
67	55190\\
68	29426\\
69	0\\
70	0\\
71	0\\
72	0\\
73	0\\
74	0\\
75	37056\\
76	55189\\
77	55189\\
78	55189\\
79	55189\\
80	55189\\
81	27646\\
82	2071\\
83	363\\
84	0\\
85	55191\\
86	54786\\
87	55190\\
88	55190\\
89	55190\\
90	55190\\
91	55190\\
92	55190\\
93	55190\\
94	0\\
95	0\\
96	22477\\
97	4393\\
98	0\\
99	0\\
100	0\\
101	0\\
102	0\\
103	0\\
104	0\\
105	0\\
106	0\\
107	0\\
108	0\\
109	0\\
110	0\\
111	0\\
112	0\\
113	0\\
114	0\\
115	0\\
116	0\\
117	0\\
118	0\\
119	0\\
120	0\\
121	0\\
122	0\\
123	0\\
124	0\\
125	0\\
126	0\\
127	0\\
128	0\\
129	0\\
130	0\\
131	0\\
132	0\\
133	0\\
134	0\\
135	0\\
136	0\\
137	0\\
138	0\\
139	0\\
140	0\\
141	0\\
};
\addlegendentry{0-24}

\addplot [const plot, color=mycolor2, solid, line width=2pt, mark options={solid, mycolor2}] table[row sep=crcr] {
1	0\\
2	0\\
3	0\\
4	0\\
5	0\\
6	0\\
7	0\\
8	0\\
9	0\\
10	0\\
11	0\\
12	0\\
13	0\\
14	0\\
15	0\\
16	0\\
17	0\\
18	0\\
19	0\\
20	0\\
21	0\\
22	0\\
23	0\\
24	0\\
25	0\\
26	0\\
27	0\\
28	0\\
29	0\\
30	0\\
31	0\\
32	0\\
33	0\\
34	0\\
35	0\\
36	0\\
37	0\\
38	0\\
39	0\\
40	0\\
41	0\\
42	0\\
43	0\\
44	0\\
45	0\\
46	0\\
47	0\\
48	0\\
49	0\\
50	0\\
51	0\\
52	0\\
53	0\\
54	0\\
55	0\\
56	0\\
57	0\\
58	0\\
59	0\\
60	0\\
61	0\\
62	0\\
63	0\\
64	0\\
65	0\\
66	0\\
67	0\\
68	0\\
69	0\\
70	0\\
71	0\\
72	0\\
73	0\\
74	0\\
75	0\\
76	0\\
77	0\\
78	0\\
79	0\\
80	0\\
81	0\\
82	0\\
83	39162\\
84	0\\
85	0\\
86	214\\
87	0\\
88	0\\
89	0\\
90	0\\
91	0\\
92	0\\
93	0\\
94	55190\\
95	55190\\
96	32713\\
97	50797\\
98	0\\
99	0\\
100	0\\
101	0\\
102	0\\
103	0\\
104	0\\
105	0\\
106	0\\
107	0\\
108	0\\
109	0\\
110	0\\
111	0\\
112	0\\
113	0\\
114	0\\
115	0\\
116	0\\
117	0\\
118	0\\
119	0\\
120	0\\
121	0\\
122	0\\
123	0\\
124	0\\
125	0\\
126	0\\
127	0\\
128	0\\
129	0\\
130	0\\
131	0\\
132	0\\
133	0\\
134	0\\
135	0\\
136	0\\
137	0\\
138	0\\
139	0\\
140	0\\
141	0\\
};
\addlegendentry{25-44}

\addplot [const plot, color=mycolor3, solid, line width=2pt, mark options={solid, mycolor3}] table[row sep=crcr] {
1	0\\
2	0\\
3	0\\
4	0\\
5	0\\
6	0\\
7	0\\
8	0\\
9	0\\
10	0\\
11	0\\
12	0\\
13	0\\
14	0\\
15	0\\
16	0\\
17	0\\
18	0\\
19	0\\
20	0\\
21	0\\
22	0\\
23	0\\
24	0\\
25	0\\
26	0\\
27	0\\
28	0\\
29	0\\
30	0\\
31	0\\
32	0\\
33	0\\
34	0\\
35	0\\
36	0\\
37	0\\
38	0\\
39	0\\
40	0\\
41	0\\
42	0\\
43	0\\
44	0\\
45	0\\
46	0\\
47	0\\
48	0\\
49	0\\
50	0\\
51	0\\
52	0\\
53	0\\
54	0\\
55	0\\
56	0\\
57	0\\
58	0\\
59	0\\
60	0\\
61	0\\
62	0\\
63	0\\
64	0\\
65	0\\
66	0\\
67	0\\
68	0\\
69	0\\
70	0\\
71	0\\
72	0\\
73	0\\
74	0\\
75	0\\
76	0\\
77	0\\
78	0\\
79	0\\
80	0\\
81	27542\\
82	53118\\
83	15663\\
84	55189\\
85	0\\
86	58\\
87	0\\
88	0\\
89	0\\
90	0\\
91	0\\
92	0\\
93	0\\
94	0\\
95	0\\
96	0\\
97	0\\
98	0\\
99	0\\
100	0\\
101	0\\
102	0\\
103	0\\
104	0\\
105	0\\
106	0\\
107	0\\
108	0\\
109	0\\
110	0\\
111	0\\
112	0\\
113	0\\
114	0\\
115	0\\
116	0\\
117	0\\
118	0\\
119	0\\
120	0\\
121	0\\
122	0\\
123	0\\
124	0\\
125	0\\
126	0\\
127	0\\
128	0\\
129	0\\
130	0\\
131	0\\
132	0\\
133	0\\
134	0\\
135	0\\
136	0\\
137	0\\
138	0\\
139	0\\
140	0\\
141	0\\
};
\addlegendentry{45-64}

\addplot [const plot, color=mycolor4, solid, line width=2pt, mark options={solid, mycolor4}] table[row sep=crcr] {
1	0\\
2	0\\
3	0\\
4	0\\
5	0\\
6	0\\
7	0\\
8	0\\
9	0\\
10	0\\
11	0\\
12	0\\
13	0\\
14	0\\
15	0\\
16	0\\
17	0\\
18	0\\
19	0\\
20	0\\
21	0\\
22	0\\
23	0\\
24	0\\
25	0\\
26	0\\
27	0\\
28	0\\
29	0\\
30	0\\
31	0\\
32	0\\
33	0\\
34	0\\
35	0\\
36	0\\
37	0\\
38	0\\
39	0\\
40	0\\
41	0\\
42	0\\
43	0\\
44	0\\
45	0\\
46	0\\
47	0\\
48	0\\
49	0\\
50	0\\
51	0\\
52	0\\
53	0\\
54	0\\
55	0\\
56	0\\
57	0\\
58	0\\
59	0\\
60	0\\
61	0\\
62	0\\
63	0\\
64	0\\
65	0\\
66	0\\
67	0\\
68	25763\\
69	55190\\
70	55190\\
71	55190\\
72	55190\\
73	55190\\
74	55190\\
75	18133\\
76	0\\
77	0\\
78	0\\
79	0\\
80	0\\
81	0\\
82	0\\
83	0\\
84	0\\
85	0\\
86	29\\
87	0\\
88	0\\
89	0\\
90	0\\
91	0\\
92	0\\
93	0\\
94	0\\
95	0\\
96	0\\
97	0\\
98	0\\
99	0\\
100	0\\
101	0\\
102	0\\
103	0\\
104	0\\
105	0\\
106	0\\
107	0\\
108	0\\
109	0\\
110	0\\
111	0\\
112	0\\
113	0\\
114	0\\
115	0\\
116	0\\
117	0\\
118	0\\
119	0\\
120	0\\
121	0\\
122	0\\
123	0\\
124	0\\
125	0\\
126	0\\
127	0\\
128	0\\
129	0\\
130	0\\
131	0\\
132	0\\
133	0\\
134	0\\
135	0\\
136	0\\
137	0\\
138	0\\
139	0\\
140	0\\
141	0\\
};
\addlegendentry{65-74}

\addplot [const plot, color=mycolor5, solid, line width=2pt, mark options={solid, mycolor5}] table[row sep=crcr] {
1	0\\
2	0\\
3	0\\
4	0\\
5	0\\
6	0\\
7	0\\
8	0\\
9	0\\
10	0\\
11	0\\
12	0\\
13	0\\
14	0\\
15	0\\
16	0\\
17	0\\
18	0\\
19	0\\
20	0\\
21	0\\
22	0\\
23	0\\
24	0\\
25	0\\
26	0\\
27	0\\
28	0\\
29	0\\
30	0\\
31	0\\
32	0\\
33	0\\
34	0\\
35	0\\
36	0\\
37	0\\
38	0\\
39	0\\
40	0\\
41	0\\
42	0\\
43	0\\
44	0\\
45	0\\
46	0\\
47	0\\
48	0\\
49	0\\
50	0\\
51	0\\
52	0\\
53	0\\
54	0\\
55	0\\
56	0\\
57	0\\
58	0\\
59	0\\
60	0\\
61	0\\
62	14941\\
63	55190\\
64	55190\\
65	55190\\
66	17327\\
67	0\\
68	0\\
69	0\\
70	0\\
71	0\\
72	0\\
73	0\\
74	0\\
75	0\\
76	0\\
77	0\\
78	0\\
79	0\\
80	0\\
81	0\\
82	0\\
83	0\\
84	0\\
85	0\\
86	26\\
87	0\\
88	0\\
89	0\\
90	0\\
91	0\\
92	0\\
93	0\\
94	0\\
95	0\\
96	0\\
97	0\\
98	0\\
99	0\\
100	0\\
101	0\\
102	0\\
103	0\\
104	0\\
105	0\\
106	0\\
107	0\\
108	0\\
109	0\\
110	0\\
111	0\\
112	0\\
113	0\\
114	0\\
115	0\\
116	0\\
117	0\\
118	0\\
119	0\\
120	0\\
121	0\\
122	0\\
123	0\\
124	0\\
125	0\\
126	0\\
127	0\\
128	0\\
129	0\\
130	0\\
131	0\\
132	0\\
133	0\\
134	0\\
135	0\\
136	0\\
137	0\\
138	0\\
139	0\\
140	0\\
141	0\\
};
\addlegendentry{75-84}

\addplot [const plot, color=mycolor6, solid, line width=2pt, mark options={solid, mycolor6}] table[row sep=crcr] {
1	0\\
2	0\\
3	0\\
4	0\\
5	0\\
6	0\\
7	0\\
8	0\\
9	0\\
10	0\\
11	0\\
12	0\\
13	0\\
14	0\\
15	0\\
16	0\\
17	0\\
18	0\\
19	0\\
20	0\\
21	0\\
22	0\\
23	0\\
24	0\\
25	0\\
26	0\\
27	0\\
28	0\\
29	0\\
30	0\\
31	0\\
32	0\\
33	0\\
34	0\\
35	0\\
36	0\\
37	0\\
38	0\\
39	0\\
40	0\\
41	0\\
42	0\\
43	0\\
44	0\\
45	0\\
46	0\\
47	0\\
48	0\\
49	0\\
50	0\\
51	0\\
52	0\\
53	0\\
54	0\\
55	0\\
56	0\\
57	0\\
58	0\\
59	0\\
60	0\\
61	55190\\
62	40249\\
63	0\\
64	0\\
65	0\\
66	0\\
67	0\\
68	0\\
69	0\\
70	0\\
71	0\\
72	0\\
73	0\\
74	0\\
75	0\\
76	0\\
77	0\\
78	0\\
79	0\\
80	0\\
81	0\\
82	0\\
83	0\\
84	0\\
85	0\\
86	76\\
87	0\\
88	0\\
89	0\\
90	0\\
91	0\\
92	0\\
93	0\\
94	0\\
95	0\\
96	0\\
97	0\\
98	0\\
99	0\\
100	0\\
101	0\\
102	0\\
103	0\\
104	0\\
105	0\\
106	0\\
107	0\\
108	0\\
109	0\\
110	0\\
111	0\\
112	0\\
113	0\\
114	0\\
115	0\\
116	0\\
117	0\\
118	0\\
119	0\\
120	0\\
121	0\\
122	0\\
123	0\\
124	0\\
125	0\\
126	0\\
127	0\\
128	0\\
129	0\\
130	0\\
131	0\\
132	0\\
133	0\\
134	0\\
135	0\\
136	0\\
137	0\\
138	0\\
139	0\\
140	0\\
141	0\\
};
\addlegendentry{85+}

\addplot[area legend, draw=none, fill=green, fill opacity=0.1, line width=2.0pt, forget plot]
table[row sep=crcr] {%
x	y\\
98	0\\
150	0\\
150	60000\\
98	60000\\
}--cycle;

\addplot [color=green, dashed, line width=2.0pt, mark options={solid, green}]
  table[row sep=crcr]{%
98	0\\
98	60000\\
};
\addlegendentry{Erad.}

\end{axis}
\end{tikzpicture}%

%% file: Figures/MPC_cost.tex
%
%
\definecolor{mycolor1}{rgb}{0.00000,0.44700,0.74100}%
\definecolor{mycolor2}{rgb}{0.85000,0.32500,0.09800}%
\definecolor{mycolor3}{rgb}{0.92900,0.69400,0.12500}%
\begin{tikzpicture}

\begin{axis}[%
width=14cm,
height=4cm,
at={(0.cm,0.cm)},
scale only axis,
xmin=61,
xmax=140,
ymin=0,
ymax=1000,
xmajorgrids,
ymajorgrids,
axis background/.style={fill=white},
axis x line*=bottom,
axis y line*=left,
legend pos=south east,
legend style={legend cell align=left, align=left, draw=white!15!black}
]

\addplot [color=mycolor1, solid, line width=2pt, mark options={solid, mycolor1}]
  table[row sep=crcr]{%
1	-1271.71162129538\\
2	-1271.18725215174\\
3	-1270.56113363633\\
4	-1269.88223927673\\
5	-1269.16498024673\\
6	-1268.40726540499\\
7	-1267.60046838978\\
8	-1266.73380955128\\
9	-1265.79601727974\\
10	-1264.77579462843\\
11	-1263.6618017311\\
12	-1262.44246619721\\
13	-1261.10575049617\\
14	-1259.63892390121\\
15	-1258.02835185797\\
16	-1256.2593019505\\
17	-1254.31576108629\\
18	-1252.1802577157\\
19	-1249.83368344552\\
20	-1247.25510928814\\
21	-1244.42159261676\\
22	-1241.30797155175\\
23	-1237.88664397302\\
24	-1234.12732866541\\
25	-1229.99680629399\\
26	-1225.45863800733\\
27	-1220.472859507\\
28	-1214.99564842104\\
29	-1208.97896279691\\
30	-1202.37014849779\\
31	-1195.11151325734\\
32	-1187.13986513407\\
33	-1178.38601311672\\
34	-1168.77422767931\\
35	-1158.22165918116\\
36	-1146.63771216794\\
37	-1133.9233738712\\
38	-1119.97049554671\\
39	-1104.66102575723\\
40	-1087.86619532186\\
41	-1069.44565445039\\
42	-1049.24656359334\\
43	-1027.10264080553\\
44	-1002.83316998727\\
45	-976.241976281248\\
46	-947.116377218216\\
47	-915.226120974509\\
48	-880.322326388988\\
49	-842.136443239852\\
50	-800.379255756363\\
51	-754.73995747865\\
52	-704.885331407827\\
53	-650.459075911948\\
54	-591.08132404001\\
55	-526.348411668689\\
56	-455.832958124664\\
57	-379.084331369712\\
58	-295.629578186994\\
59	-204.97490762658\\
60	-106.607822677718\\
61	0\\
62	115.388980336286\\
63	240.106954846879\\
64	344.439635032953\\
65	421.210860990319\\
66	482.772811516042\\
67	532.519781807674\\
68	571.75822601421\\
69	605.456392143141\\
70	636.012315127096\\
71	663.602799468782\\
72	688.260684215644\\
73	710.028420113102\\
74	728.967753748391\\
75	745.159837082433\\
76	760.42039788768\\
77	775.07434893124\\
78	789.051928573373\\
79	802.273283654227\\
80	814.67899813651\\
81	826.227873998852\\
82	836.894006227799\\
83	846.664997937203\\
84	855.540741511264\\
85	863.532355571322\\
86	870.661090012792\\
87	876.957135017093\\
88	882.45833216199\\
89	887.208815671358\\
90	891.257624538076\\
91	894.657329313193\\
92	897.462714814855\\
93	899.757430540458\\
94	901.709423253212\\
95	903.391713619099\\
96	904.841260886253\\
97	906.082294301513\\
98	907.135155028164\\
99	908.019287152515\\
100	908.753923285759\\
101	909.357926409813\\
102	909.849391387589\\
103	910.245250526172\\
104	910.560973758223\\
105	910.810386413132\\
106	911.005597659943\\
107	911.157020288793\\
108	911.273458712784\\
109	911.362053818303\\
110	911.427660316069\\
111	911.473886279629\\
112	911.504568156223\\
113	911.523631099819\\
114	911.534656222981\\
115	911.540556778252\\
116	911.54345936059\\
117	911.544761286903\\
118	911.545288615083\\
119	911.545479162101\\
120	911.545539631687\\
121	911.545558940041\\
122	911.545565126667\\
123	911.545567112824\\
124	911.545567751164\\
125	911.545567956446\\
126	911.545568022482\\
127	911.545568043727\\
128	911.545568050562\\
129	911.545568052761\\
130	911.545568053468\\
131	911.545568053695\\
132	911.545568053769\\
133	911.545568053792\\
134	911.5455680538\\
135	911.545568053802\\
136	911.545568053803\\
137	911.545568053803\\
138	911.545568053803\\
139	911.545568053804\\
140	911.545568053804\\
141	911.545568053804\\
};
\addlegendentry{Deceased Nat.}

\addplot[area legend, draw=none, fill=blue, fill opacity=0.1, line width=2.0pt, forget plot]
table[row sep=crcr] {%
x	y\\
118	0\\
150	0\\
150	1000\\
118	1000\\
}--cycle;

\addplot [color=blue, dashed, line width=2.0pt, mark options={solid, blue}]
  table[row sep=crcr]{%
118	0\\
118	1000\\
};
\addlegendentry{Erad. Nat.}

\addplot [color=mycolor2, solid, line width=2pt, mark options={solid, mycolor2}]
  table[row sep=crcr]{%
1	-1271.71162129538\\
2	-1271.18725215174\\
3	-1270.56113363633\\
4	-1269.88223927673\\
5	-1269.16498024673\\
6	-1268.40726540499\\
7	-1267.60046838978\\
8	-1266.73380955128\\
9	-1265.79601727974\\
10	-1264.77579462843\\
11	-1263.6618017311\\
12	-1262.44246619721\\
13	-1261.10575049617\\
14	-1259.63892390121\\
15	-1258.02835185797\\
16	-1256.2593019505\\
17	-1254.31576108629\\
18	-1252.1802577157\\
19	-1249.83368344552\\
20	-1247.25510928814\\
21	-1244.42159261676\\
22	-1241.30797155175\\
23	-1237.88664397302\\
24	-1234.12732866541\\
25	-1229.99680629399\\
26	-1225.45863800733\\
27	-1220.472859507\\
28	-1214.99564842104\\
29	-1208.97896279691\\
30	-1202.37014849779\\
31	-1195.11151325734\\
32	-1187.13986513407\\
33	-1178.38601311672\\
34	-1168.77422767931\\
35	-1158.22165918116\\
36	-1146.63771216794\\
37	-1133.9233738712\\
38	-1119.97049554671\\
39	-1104.66102575723\\
40	-1087.86619532186\\
41	-1069.44565445039\\
42	-1049.24656359334\\
43	-1027.10264080553\\
44	-1002.83316998727\\
45	-976.241976281248\\
46	-947.116377218216\\
47	-915.226120974509\\
48	-880.322326388988\\
49	-842.136443239852\\
50	-800.379255756363\\
51	-754.73995747865\\
52	-704.885331407827\\
53	-650.459075911948\\
54	-591.08132404001\\
55	-526.348411668689\\
56	-455.832958124664\\
57	-379.084331369712\\
58	-295.629578186994\\
59	-204.97490762658\\
60	-106.607822677718\\
61	0\\
62	115.388980336286\\
63	240.106954846879\\
64	344.440183379926\\
65	421.254826784392\\
66	482.858718491308\\
67	532.643879211384\\
68	571.918378924941\\
69	607.507007946633\\
70	642.847572958149\\
71	676.567516090725\\
72	706.423412992628\\
73	731.879314123637\\
74	753.09241836048\\
75	770.469181408266\\
76	784.482030106247\\
77	795.596843410013\\
78	805.162374771158\\
79	813.791549701253\\
80	821.467781084543\\
81	828.119973301286\\
82	833.716573642559\\
83	838.283822798438\\
84	841.865329234394\\
85	844.565313711532\\
86	846.571925211361\\
87	848.019946020767\\
88	849.065721507885\\
89	849.812420993621\\
90	850.33816114845\\
91	850.701754075416\\
92	850.948558273996\\
93	851.113123572215\\
94	851.221082253232\\
95	851.29089213967\\
96	851.335470288266\\
97	851.363542114949\\
98	851.380948391545\\
99	851.391569547422\\
100	851.397941601867\\
101	851.401712875117\\
102	851.40392635178\\
103	851.405218687727\\
104	851.405970680086\\
105	851.406407305598\\
106	851.406660465465\\
107	851.406807116346\\
108	851.406892018165\\
109	851.40694115211\\
110	851.4069695795\\
111	851.406986024017\\
112	851.406995535731\\
113	851.40700103704\\
114	851.407004218699\\
115	851.407006058743\\
116	851.407007122873\\
117	851.40700773827\\
118	851.407008094157\\
119	851.407008299967\\
120	851.407008418987\\
121	851.407008487816\\
122	851.40700852762\\
123	851.407008550638\\
124	851.40700856395\\
125	851.407008571648\\
126	851.4070085761\\
127	851.407008578674\\
128	851.407008580163\\
129	851.407008581024\\
130	851.407008581522\\
131	851.40700858181\\
132	851.407008581976\\
133	851.407008582072\\
134	851.407008582128\\
135	851.407008582161\\
136	851.407008582179\\
137	851.40700858219\\
138	851.407008582196\\
139	851.4070085822\\
140	851.407008582202\\
141	851.407008582203\\
};
\addlegendentry{Deceased MPC}

\addplot[area legend, draw=none, fill=red, fill opacity=0.1, line width=2.0pt, forget plot]
table[row sep=crcr] {%
x	y\\
98	0\\
150	0\\
150	1000\\
98	1000\\
}--cycle;
\addlegendentry{Erad. MPC}

\addplot [color=red, dashed, line width=2.0pt, mark options={solid, red}]
  table[row sep=crcr]{%
98	0\\
98	1000\\
};

\addplot [const plot, color=teal, solid, line width=2pt, mark options={solid, teal}]
table[row sep=crcr] {%
1	0\\
2	0\\
3	0\\
4	0\\
5	0\\
6	0\\
7	0\\
8	0\\
9	0\\
10	0\\
11	0\\
12	0\\
13	0\\
14	0\\
15	0\\
16	0\\
17	0\\
18	0\\
19	0\\
20	0\\
21	0\\
22	0\\
23	0\\
24	0\\
25	0\\
26	0\\
27	0\\
28	0\\
29	0\\
30	0\\
31	0\\
32	0\\
33	0\\
34	0\\
35	0\\
36	0\\
37	0\\
38	0\\
39	0\\
40	0\\
41	0\\
42	0\\
43	0\\
44	0\\
45	0\\
46	0\\
47	0\\
48	0\\
49	0\\
50	0\\
51	0\\
52	0\\
53	0\\
54	0\\
55	0\\
56	0\\
57	0\\
58	0\\
59	0\\
60	0\\
61	874.632357042548\\
62	735.781560755318\\
63	611.035026349384\\
64	506.699730572134\\
65	429.884893222015\\
66	368.417562801009\\
67	318.519530558364\\
68	279.280095342241\\
69	243.685470832887\\
70	208.3496949666\\
71	174.653265152667\\
72	144.764582019473\\
73	119.460776319973\\
74	98.440511950542\\
75	80.908953308823\\
76	68.1841395469248\\
77	55.7183673095272\\
78	46.2209023080555\\
79	37.4414173717851\\
80	29.7406129983579\\
81	23.1488465011403\\
82	17.5478017806114\\
83	13.0659532534414\\
84	9.37689168932461\\
85	4.49394997294188\\
86	4.91410451416612e-05\\
87	0.239952448410848\\
88	1.38832218711208\\
89	0.990317971639865\\
90	0.71765731905579\\
91	0.706630076402807\\
92	0.459202830773315\\
93	0.294199271609687\\
94	0.185967133025142\\
95	0.115982724079709\\
96	0.0712964745078572\\
97	0.0431602994059961\\
98	0.0431602994059961\\
99	0.0431602994059961\\
100	0.0431602994059961\\
101	0.0431602994059961\\
102	0.0431602994059961\\
103	0.0431602994059961\\
104	0.0431602994059961\\
105	0.0431602994059961\\
106	0.0431602994059961\\
107	0.0431602994059961\\
108	0.0431602994059961\\
109	0.0431602994059961\\
110	0.0431602994059961\\
111	0.0431602994059961\\
112	0.0431602994059961\\
113	0.0431602994059961\\
114	0.0431602994059961\\
115	0.0431602994059961\\
116	0.0431602994059961\\
117	0.0431602994059961\\
118	0.0431602994059961\\
119	0.0431602994059961\\
120	0.0431602994059961\\
121	0.0431602994059961\\
122	0.0431602994059961\\
123	0.0431602994059961\\
124	0.0431602994059961\\
125	0.0431602994059961\\
126	0.0431602994059961\\
127	0.0431602994059961\\
128	0.0431602994059961\\
129	0.0431602994059961\\
130	0.0431602994059961\\
131	0.0431602994059961\\
132	0.0431602994059961\\
133	0.0431602994059961\\
134	0.0431602994059961\\
135	0.0431602994059961\\
136	0.0431602994059961\\
137	0.0431602994059961\\
138	0.0431602994059961\\
139	0.0431602994059961\\
140	0.0431602994059961\\
141	0\\
};
\addlegendentry{MPC cost}

\addplot [color=teal, dashed, line width=2.0pt, mark options={solid, teal}]
  table[row sep=crcr]{%
61	874.632357042548\\
150	874.632357042548\\
};
\addlegendentry{$V_N^0(x(60))$}

\end{axis}
\end{tikzpicture}%

%% file: sections/05_Conclusion.tex
\section{Conclusion}
\label{sec: conclusion}
In this work, a novel Model Predictive Control (MPC) approach is proposed to develop an optimal age-structured vaccination strategy for the first wave of COVID-19, with the primary objective of minimizing deaths. The key contribution of this work lies in providing a detailed proof of the recursive feasibility and asymptotic stability of the proposed MPC algorithm. Through comprehensive simulations, it was demonstrated that the MPC approach outperforms the decreasing age vaccination strategy employed by the Belgian government. Specifically, the MPC strategy results in fewer deaths, fewer infections, faster disease eradication, and more efficient vaccine distribution.

The strength of the MPC approach comes from its predictive capability, which allows it to dynamically adjust to changing conditions, and its model-based nature, which integrates different factors such as contact rates, death rates, and recovery rates across different age groups. Moreover, the proposed MPC framework is not limited to COVID-19. Different models can be used to adapt the proposed strategy to different epidemics. This flexibility enables rapid deployment as soon as vaccines become available. In addition, advances in modeling and data science can benefit the proposed MPC approach by improving the predictive accuracy of the integrated model.

Furthermore, it is assumed that the vaccine exhibits perfect efficacy in this study. In reality, the effectiveness of vaccination is not necessarily absolute \cite{Braeye_2022}. Future work will account for this issue by introducing a parameter to model the efficacy of the vaccine for each age group. It will also incorporate additional constraints, such as the capacity of the daily ICU (Intensive Care Unit), to assess how the MPC adjusts its decisions to satisfy different operational conditions. 



%% file: bibliography.bib
@article{rahmandad2022behavioral,
  title={Behavioral responses to risk promote vaccinating high-contact individuals first},
  author={Rahmandad, Hazhir},
  journal={System Dynamics Review},
  volume={38},
  number={3},
  pages={246--263},
  year={2022},
  publisher={Wiley Online Library}
}

@misc{who_covid_dashboard,
  author       = {{World Health Organization}},
  title        = {World Health Organization Coronavirus {(COVID-19)} Dashboard},
  year         = {2025},
  url          = {https://covid19.who.int},
  note         = {Accessed: 2025-01-22}
}

@article{brauer2008compartmental,
  title={Compartmental models in epidemiology},
  author={Brauer, Fred},
  journal={Mathematical epidemiology},
  pages={19--79},
  year={2008},
  publisher={Springer}
}

@article{tolles2020modeling,
  title={Modeling epidemics with compartmental models},
  author={Tolles, Juliana and Luong, ThaiBinh},
  journal={Jama},
  volume={323},
  number={24},
  pages={2515--2516},
  year={2020},
  publisher={American Medical Association}
}

@article{kermack1991contributions,
  title={Contributions to the mathematical theory of epidemics—{II}. {The }problem of endemicity},
  author={Kermack, William O and McKendrick, Anderson G},
  journal={Bulletin of mathematical biology},
  volume={53},
  number={1-2},
  pages={57--87},
  year={1991},
  publisher={Elsevier}
}

@article{MORATO2020417,
title = {An optimal predictive control strategy for COVID-19 (SARS-CoV-2) social distancing policies in Brazil},
author = {Marcelo M. Morato and Saulo B. Bastos and Daniel O. Cajueiro and Julio E. Normey-Rico},
journal = {Annual Reviews in Control},
volume = {50},
pages = {417-431},
year = {2020}
}

@article{davies2020age,
  title={Age-dependent effects in the transmission and control of {COVID-19} epidemics},
  author={Davies, Nicholas G and Klepac, Petra and Liu, Yang and Prem, Kiesha and Jit, Mark and Eggo, Rosalind M},
  journal={Nature medicine},
  volume={26},
  number={8},
  pages={1205--1211},
  year={2020},
  publisher={Nature Publishing Group US New York}
}

@article{grundel2022much,
  title={How much testing and social distancing is required to control {COVID-19}? {Some} insight based on an age-differentiated compartmental model},
  author={Grundel, Sara and Heyder, Stefan and Hotz, Thomas and Ritschel, Tobias KS and Sauerteig, Philipp and Worthmann, Karl},
  journal={SIAM Journal on Control and Optimization},
  volume={60},
  number={2},
  pages={S145--S169},
  year={2022},
  publisher={SIAM}
}

@article{franco2021covid,
  title={{COVID-19} {Belgium}: Extended {SEIR-QD} model with nursing homes and long-term scenarios-based forecasts},
  author={Franco, Nicolas},
  journal={Epidemics},
  volume={37},
  pages={100490},
  year={2021},
  publisher={Elsevier}
}

@article{alleman2021assessing,
  title={Assessing the effects of non-pharmaceutical interventions on {SARS-CoV-2} transmission in {Belgium} by means of an extended {SEIQRD} model and public mobility data},
  author={Alleman, Tijs W and Vergeynst, Jenna and De Visscher, Lander and Rollier, Michiel and Torfs, Elena and Nopens, Ingmar and Baetens, Jan M and others},
  journal={Epidemics},
  volume={37},
  pages={100505},
  year={2021},
  publisher={Elsevier}
}

@article{abrams2021modelling,
  title={Modelling the early phase of the {Belgian} {COVID-19} epidemic using a stochastic compartmental model and studying its implied future trajectories},
  author={Abrams, Steven and Wambua, James and Santermans, Eva and Willem, Lander and Kuylen, Elise and Coletti, Pietro and Libin, Pieter and Faes, Christel and Petrof, Oana and Herzog, Sereina A and others},
  journal={Epidemics},
  volume={35},
  pages={100449},
  year={2021},
  publisher={Elsevier}
}

@article{cartocci2021compartment,
  title={A compartment modeling approach to reconstruct and analyze gender and age-grouped {COVID-19} Italian data for decision-making strategies},
  author={Cartocci, Alessandra and Cevenini, Gabriele and Barbini, Paolo},
  journal={Journal of Biomedical Informatics},
  volume={118},
  pages={103793},
  year={2021},
  publisher={Elsevier}
}

@article{parino2023model,
  title={A model predictive control approach to optimally devise a two-dose vaccination rollout: a case study on {COVID-19} in {Italy}},
  author={Parino, Francesco and Zino, Lorenzo and Calafiore, Giuseppe C and Rizzo, Alessandro},
  journal={International journal of robust and nonlinear control},
  volume={33},
  number={9},
  pages={4808--4823},
  year={2023},
  publisher={Wiley Online Library}
}

@inproceedings{sauerteig2022model,
  title={Model predictive control tailored to epidemic models},
  author={Sauerteig, Philipp and Esterhuizen, Willem and Wilson, Mitsuru and Ritschel, Tobias KS and Worthmann, Karl and Streif, Stefan},
  booktitle={2022 European Control Conference (ECC)},
  pages={743--748},
  year={2022},
  organization={IEEE}
}

@misc{Sciensano2020, 
  author = {Sciensano},
  year = {2020},
  title = {{COVID-19}. Epistat},
  url = {https://epistat.sciensano.be/covid/}
}

@article{willem2020socrates,
  title={SOCRATES: an online tool leveraging a social contact data sharing initiative to assess mitigation strategies for {COVID-19}},
  author={Willem, Lander and Van Hoang, Thang and Funk, Sebastian and Coletti, Pietro and Beutels, Philippe and Hens, Niel},
  journal={BMC research notes},
  volume={13},
  number={1},
  pages={293},
  year={2020},
  publisher={Springer}
}

@article{statbel2018structuur,
  title={Structuur van de Bevolking},
  author={Statbel},
  year={2018},
  journal={}
}

@book{rawlings2017model,
  title={Model predictive control: theory, computation, and design},
  author={Rawlings, James Blake and Mayne, David Q and Diehl, Moritz and others},
  volume={2},
  year={2017},
  publisher={Nob Hill Publishing Madison, WI}
}

@phdthesis{Sonveaux2023,
    author = {Sonveaux, C.}, 
    title = {Dynamical analysis and feedback control of age-structured epidemic models}, 
    year = {2023}, 
    school = {UNamur}, 
    month = {December}, 
    day = {19}, 
    type = {PhD Thesis}}

@article{FIACCHINI2021500,
title = {The {Ockham’s} razor applied to {COVID-19} model fitting {French} data},
author = {Mirko Fiacchini and Mazen Alamir},
journal = {Annual Reviews in Control},
volume = {51},
pages = {500-510},
year = {2021}
}

@book{seborg2016process,
  title={Process dynamics and control},
  author={Seborg, Dale E and Edgar, Thomas F and Mellichamp, Duncan A and Doyle III, Francis J},
  year={2016},
  publisher={John Wiley \& Sons}
}

@book{world2014infection,
  title={Infection prevention and control of epidemic-and pandemic-prone acute respiratory infections in health care},
  author={{World Health Organization}},
  year={2014},
  publisher={World Health Organization}
}

@article{Braeye_2022, title={{COVID-19} vaccine effectiveness against symptomatic infection and hospitalization in {Belgium, July 2021-April 2022}}, author={Braeye, Toon and van Loenhout, Joris and Brondeel, Ruben and Stouten, Veerle and Hubin, Pierre and Billuart}, journal={Euro Surveill.}, year={2022}}
